\documentclass[10pt]{amsart}
\usepackage[english]{babel}
\usepackage{graphicx}
\usepackage{amsmath,amssymb,hyperref}
\usepackage{srcltx}

\newcommand{\diff}[2]{\mbox{{\rm Diff}{${\,}_{#1}({\mathbb
C}^{#2},0)$}}}
\newcommand{\diffh}[2]{\mbox{$\widehat{\rm Diff}{{\,}_{#1}({\mathbb
C}^{#2},0)}$}}

\newcommand{\cn}[1]{\mbox{(${\mathbb C}^{#1},0$)}}
\newcommand{\Xt}{{\mathcal X}_{tp1}\cn{2}}

\newcommand{\pn}[1]{{\mathbb P}^{#1}({\mathbb C})}

\newcommand{\ox}{\'{o}}

\newtheorem{pro}{Proposition}[section]
\newtheorem{teo}{Theorem}[section]
\newtheorem*{pri}{Principle}

\newtheorem{cor}{Corollary}[section]

\newtheorem{lem}{Lemma}[section]
\newtheorem{rem}{Remark}[section]
\newtheorem{defi}{Definition}[section]
\begin{document}

\title[Multisummability of unfoldings]{Multisummability of unfoldings of tangent to the identity diffeomorphisms}
\author{Javier Rib\ox n}
\address{Instituto de Matem\'{a}tica, UFF, Rua M\'{a}rio Santos Braga S/N
Valonguinho, Niter\'{o}i, Rio de Janeiro, Brasil 24020-140}
\thanks{e-mail address: javier@mat.uff.br}
\thanks{MSC-class. Primary: 37F45; Secondary: 37G10, 37F75, 34E05, 30E15}
\thanks{Keywords: resonant diffeomorphism, bifurcation theory,
asymptotic expansions, Gevrey asymptotics, multisummability,
analytic classification, structural stability}
\date{\today}
\maketitle

\bibliographystyle{plain}
\section*{Abstract}
We prove the multisummability of the infinitesimal generator of 
unfoldings of finite codimension tangent to the identity 
$1$-dimensional local complex analytic diffeomorphisms.
We also prove the multisummability of Fatou coordinates and
extensions of the Ecalle-Voronin invariants associated to these 
unfoldings.
The quasi-analytic nature is related to the parameter variable.
As an application we prove an isolated zeros theorem for the
analytic conjugacy problem.

The proof is based on good asymptotics of Fatou coordinates 
and the introduction of a new auxiliary tool, the so called 
multi-transversal flows.
They provide the estimates and the combinatorics of sectors
typically associated to summability.
The methods are based on the study of
the infinitesimal stability properties of the unfoldings.
\section{Introduction}
  We study  unfoldings
of non-linearizable resonant complex analytic diffeomorphisms. The group of
1-dimensional unfoldings of elements of $\diff{}{}$ is
\[ \diff{p}{2} = \{ \varphi(x,y) \in \diff{}{2} \ : \ x \circ \varphi = x \}
. \]
Most of the time we work in the set $\diff{p1}{2}$ composed by the elements
$\varphi$ of $\diff{p}{2}$ such that
$\varphi_{|x=0}$ is tangent to the identity
(i.e. $j^{1} \varphi_{|x=0} = Id$) but $\varphi_{|x=0} \neq Id$.
The main goal of the paper is providing a rigorous formulation and then
proving the following statement:
\begin{pri}
The infinitesimal generator of an element $\varphi$ of $\diff{p1}{2}$
is multi-summable in the $x$-variable.
\end{pri}
A natural way of studying unfoldings $\varphi \in \diff{p1}{2}$ of tangent to
the identity diffeomorphisms is comparing the dynamics of $\varphi$ and ${\rm exp}(X)$
where $X$ is a vector field whose time $1$ flow ``approximates" $\varphi$.
This point of view has been developed by Glutsyuk  \cite{Gluglu}.
In this way extensions of the Ecalle-Voronin invariants to some sectors
in the parameter space are obtained. On the one hand they are uniquely
defined. On the other hand the sectors have to avoid a finite set of directions,
typically associated (but not exclusively) to small divisors phenomena.

A different point of view was introduced by Shishikura for codimension $1$
unfoldings \cite{Shishi}. The idea is constructing appropriate fundamental domains bounded
by two curves with common ends at
singular points: one curve is the image of the other one.  Pasting the boundary curves by the
dynamics yields (by quasiconformal surgery) a Riemann surface that is conformally
equivalent to the Riemann sphere. The logarithm of an appropriate affine complex
coordinate on the sphere induces a Fatou coordinate for $\varphi$.
These ideas were generalized to higher codimension unfoldings by Oudkerk \cite{Oudkerk}.
In this approach the first curve is a phase
curve of an appropriate vector field transversal to the real flow of $X$. In both cases
the Fatou coordinates provide Lavaurs vector fields $X^{\varphi}$
such that $\varphi= {\rm exp}(X^{\varphi})$ \cite{Lavaurs}. The Shishikura's approach was
used by Mardesic, Roussarie and Rousseau to provide a complete system of invariants
for unfoldings of codimension $1$ tangent to the identity diffeomorphisms \cite{MRR}.
The analytic classification for the finite codimension case was
completed in \cite{JR:mod} by using the Oudkerk's point of view.
On the one hand the constructions are applied to sectors whose union
is a neighborhood of
the origin in the $x$-variable. On the other hand the extensions of
Fatou coordinates, Lavaurs vector fields and Ecalle-Voronin invariants
depend on the choices in the construction. One of the goals of this paper
is explaining how all these objects are intrinsic and can be interpreted
as different sectorial sums of a quasi-analytic formal object.
\subsection{Construction of multi-transversal flows}
In order to study the properties of the infinitesimal generator
of $\varphi \in \diff{p1}{2}$
we construct transversal flows defined in sectorial domains in the variable $x$.
They are of the form $\Re (\aleph^{*} X)$ where
$\aleph^{*}: ([0,\delta) e^{i[u_{0},u_{1}]} \times B(0,\epsilon)) \setminus Sing (X) \to
{\mathbb S}^{1} \setminus \{-1,1\}$ is a continuous function; let us explain how.

The main ideas of the construction can be found in \cite{JR:mod}.
We use the dynamical splitting, we express a neighborhood of the origin
in ${\mathbb C}^{2}$ as a union of basic sets that are associated to $\varphi$
by a desingularization process of the fixed points set $Fix (\varphi)$ of $\varphi$.
There are two types of basic sets, namely exterior and compact-like basic
sets. The exterior sets are dynamically simple and the restriction of
$\Re (\mu X)$ to an exterior set is either a parametrized Fatou flower or
truncated Fatou flower for any $\mu \in {\mathbb S}^{1}$
(see figure (\ref{EVfig6})). Thus the dynamics of
$\Re (\mu X)$ in a neighborhood of the origin is determined by the dynamics
of $\Re(\mu X)$ in the compact-like sets
${\mathcal C}_{1}$, $\hdots$, ${\mathcal C}_{q}$.
We can associate an exponent $e_{j} \in {\mathbb N}$ and a polynomial vector field
$P_{j}(w) \partial / \partial w$ such that the dynamics of
$Re (\mu X)$ in ${\mathcal C}_{j}$ is orbitally equivalent to the dynamics
of $\Re (|x|^{e_{j}} \lambda^{e_{j}} \mu P_{j}(w) \partial / \partial w)$ for
$1 \leq j \leq q$ where $x =|x| \lambda$ (see figure (\ref{EVfig9})). We define
\[ {\mathcal U}_{X}^{j} =
\{ (\lambda,\mu) \in {\mathbb S}^{1} \times {\mathbb S}^{1} :
\Re (\lambda^{e_{j}} \mu  P_{j}(w) \partial / \partial w) \
{\rm is \ not \ stable} \} \]
The definition of stability is borrowed from Douady, Estrada and Sentenac
\cite{DES} (see \cite{JR:mod}).
The real flow of a vector field $P(w) \partial / \partial w \in {\mathbb C}[w]$
is stable if $Re (\mu P(w) \partial / \partial w)$ is orbitally conjugated to
$Re (P(w) \partial / \partial w)$ for any $\mu \in {\mathbb S}^{1}$
in a neighborhood of $1$.
It turns out that $\Re (\mu P(w) \partial / \partial w)$ is stable except for
finitely many directions $\mu \in {\mathbb S}^{1}$.
Then $\Re (\mu X)$ is stable in a neighborhood of the direction
$\lambda {\mathbb R}^{+}$ in the $x$-space if
$(\lambda,\mu) \not \in {\mathcal U}_{X}^{j}$; in other words
there exists a sector $S  = (0,\delta) \lambda e^{i[-u,u]}$ for some
$\delta, u \in {\mathbb R}^{+}$ such that
$\Re (\mu X)_{|x=x_{0}}$ is orbitally conjugated to
$\Re (\mu X)_{|x=x_{1}}$ in ${\mathcal C}_{j}$ for any
$(x_{0},x_{1}) \in S \times S$.
The stability of transversal flows is an important part of our approach
since it guarantees that the objects constructed (Fatou coordinates,...)
depend holomorphically on both variables.

Given a continuous function
$\mu_{j}: e^{i[u_{0},u_{1}]} \to {\mathbb S}^{1} \setminus \{-1,1\}$
such that
$(\lambda,\mu_{j}(\lambda)) \not \in {\mathcal U}_{X}^{j}$ for any
$\lambda \in  e^{i[u_{0},u_{1}]}$
it is natural to consider  $\Re (\aleph^{*} X)$ such that
\[ \Re (\aleph^{*} X)_{|{\mathcal C}_{j} \cap \{x=|x_{0}| \lambda_{0} \}}=
\Re(\mu_{j}(\lambda_{0}) X)_{|{\mathcal C}_{j} \cap \{x=|x_{0}| \lambda_{0} \}} \]
for $0 < |x_{0}| < \delta$ and $\lambda_{0} \in  e^{i[u_{0},u_{1}]}$.
In this way we define $\Re (\aleph^{*} X)_{|{\mathcal C}_{j}}$
for $1 \leq j \leq q$.
Such a vector field would be stable in every compact-like set. Since
compact-like sets collapse when approaching $x=0$
(we have ${\mathcal C}_{j} \cap \{x=0\} = \{(0,0)\}$)
we are requiring conditions of
infinitesimal stability for  $\Re (\aleph^{*} X)$.
The exterior sets are dynamically simple,
we can use them to interpolate the transversal flows defined in
different compact-like sets. We obtain in this way a multi-transversal flow.
Roughly speaking it is a stable flow transversal to $\Re (X)$ that is of the form
$\Re (\mu_{j}(x/|x|) X)$ by restriction to  any compact-like set
${\mathcal C}_{j}$.
Let us remind that in \cite{JR:mod} the functions $\aleph^{*}$ are constant
and that it is not difficult to generalize the constructions there for functions
$\aleph^{*} = \aleph^{*}(x/|x|)$.

Our objects (Fatou coordinates,...) are
defined in regions. A region is a connected component
of the subset $T$ of
$([0,\delta) e^{i[u_{0},u_{1}]} \times B(0,\epsilon)) \setminus Sing (X)$
obtained as the union of the trajectories of $\Re (\aleph^{*} X)$
whose $\alpha$ and $\omega$ limits are both singular points.
The multi-transversal flows have two important
properties:
\begin{itemize}
\item The infinitesimal stability properties allow us to use the same ideas in
\cite{JR:mod} to find Fatou coordinates $\psi_{H}^{\varphi}$ of $\varphi$
defined in regions $H$ of $\Re (\aleph^{*} X)$ such that
$\psi_{H}^{\varphi} - \psi^{X}$ is continuous in $\overline{H}$ and holomorphic
in $\dot{H}$ where $\psi^{X}$ is a Fatou coordinate of $X$. In particular the
function $\psi_{H}^{\varphi} - \psi^{X}$ is bounded.
\item The dynamics of multi-transversal flows and the transitions of the dynamics
between different multi-transversal flows can be described in a combinatorial way.
\end{itemize}
Let us explain succinctly how to use the previous properties to
deduce multi-summability of Fatou coordinates, Lavaurs vector fields and
Ecalle-Voronin invariants.
Consider a petal $L_{j}$ of $\varphi_{|x=0}$.
There exists a unique region $H_{j}$ of $\Re (\aleph^{*} X)$
containing $L_{j} \cap T$.
Consider the region $\tilde{H}_{j}$ obtained in an analogous way
for a multi-transversal flow $\Re (\tilde{\aleph}^{*} X)$
defined in
$[0,\delta) e^{i[\tilde{u}_{0},\tilde{u}_{1}]} \times B(0,\epsilon)$.
The first property implies roughly speaking that
$\psi_{H_{j}}^{\varphi} - \psi_{\tilde{H}_{j}}^{\varphi}$ is bounded in
$H_{j} \cap \tilde{H}_{j}$. Clearly
$\psi_{H_{j}}^{\varphi} - \psi_{\tilde{H}_{j}}^{\varphi}$ is constant in orbits of $\varphi$.
As a consequence the function
\[
(\psi_{H_{j}}^{\varphi} - \psi_{\tilde{H}_{j}}^{\varphi}) \circ (x, e^{2 \pi i \psi_{H_{j}}^{\varphi}})^{-1} \]
is well-defined and bounded in a domain of the form
\[ \{ (x,z) \in
[0,\delta) (( e^{i[{u}_{0},{u}_{1}]}) \cap ( e^{i[\tilde{u}_{0},\tilde{u}_{1}]})) \times {\mathbb C} :
e^{\frac{-C}{|x|^{e}}} < |z| < e^{\frac{C}{|x|^{e}}} \} . \]
The exponent $e$ is deduced from the combinatorial study of multi-transversal flows.
Hence we obtain that up to an additive function of $x$ the function
$\psi_{H_{j}}^{\varphi} - \psi_{\tilde{H}_{j}}^{\varphi}$ is a $O(e^{K/|x|^{e}})$
by using Cauchy's integral formula.
The combinatorics provides the exponentially small estimates and the right sectors to
obtain multi-summability. Multi-summability of Lavaurs vector fields and Ecalle-Voronin
invariants is deduced from the analogous property for Fatou coordinates.

  Since a multi-summable power series in
a direction is a sum of summable ones, the multi-summability levels of
of Fatou coordinates have to appear in an independent way.
The multi-summability is related to the nature of $\varphi$ in compact-like
sets, indeed the multi-summability levels are contained in the set
$\{e_{1}, \hdots, e_{j}\}$.
Imposing all the $\mu_{j}$ functions
($\Re (\aleph^{*} X)_{|{\mathcal C}_{j}}=\Re(\mu_{j} X)_{|{\mathcal C}_{j}} $)
to be equal would result in too small sectors too obtain a multi-summable
object.

Let us remark that the summability of Ecalle-Voronin invariants for generic families unfolding
a codimension $1$ parabolic or resonant diffeomorphism is proved in
\cite{Rousseau:modres} and \cite{rousseau-christopher:modpar}.
The methods are different. In particular it is used the so called compatibility
condition that establishes whether different invariants correspond to the same
diffeomorphism via a translation to Glutsyuk invariants.
We do not need such a condition, the good estimates on the asymptotic of Fatou
coordinates suffice to prove the multi-summability for Ecalle-Voronin invariants
in all finite codimension unfoldings. We also prove the multi-summability of
Fatou coordinates and Lavaurs vector fields.
This last object is specially interesting to us since one of the goals of this
paper is interpreting the sectorial Lavaurs vector fields as sums of the formal
infinitesimal generator.
\subsection{Intrinsic nature of sectorial objects}
As we said earlier we could make other choices of flows transversal to $\Re (X)$.
Consider a case such that the fixed points set $Fix (\varphi)$ of $\varphi$ is a union of
curves of the form $y=\gamma_{j}(x)$ for $j \in \{1,\hdots,p\}$ and
$(\partial \gamma_{j}/\partial x)(0) \neq (\partial \gamma_{k}/\partial x)(0)$ for
$j \neq k$. This is one of the cases providing summable formal objects.
Then we only consider multi-transversal flows that are deformed versions of
the imaginary flow $\Re (i X)$. More precisely our multi-transversal flows
$\Re (\aleph^{*} X)$
are defined in sectors of the form $[0,\delta) e^{i[u_{0},u_{1}]} \times B(0,\epsilon)$
and there exists $\lambda \in e^{i(u_{0},u_{1})}$ such that
$(\aleph^{*})_{|[0,\delta) \lambda \times B(0,\epsilon)} \equiv i$.
The situation is slightly different in the multi-summable case where the
richer combinatorics makes us consider other kind of multi-transversal flows.
Anyway, the imaginary flow is
again the basic ingredient that is present in all multi-transversal flows.
For instance, the directions of non-summability (Stokes directions) are contained in
$\{ \lambda \in {\mathbb S}^{1} : (\lambda ,i) \in \cup_{j=1}^{q} {\mathcal U}_{X}^{j} \}$,
i.e. the sets of directions that are unstable for the restriction of the imaginary flow
to some of the compact-like basic sets. In short,
on the one hand all the multi-transversal flows are related
to the imaginary flow $\Re (iX)$. On the other hand we need to consider more general
vector fields since in this way the sums of the multi-summable objects can be
realized in wide enough sectors.
\subsection{Infinitesimal generator}
The use of normal forms ${\rm exp}(X)$ to study unfoldings
$\varphi \in \diff{p1}{2}$ is classical (see Martinet's paper \cite{Mar:Ast}).
Vector fields are used to model the diffeomorphisms
even if it is clear that generic unfoldings do not behave as nicely as flows.
This paper is one step forward in the direction of justifying such an approach.
The diffeomorphism $\varphi$ is embedded in the
``formal flow"  of its infinitesimal generator. We show that such an
object is of geometric nature and that its sectorial sums provide
analytic vector fields whose time $1$ flow coincides with $\varphi$.
The complexity of the diffeomorphism can be interpreted in a cohomological
way since the Lavaurs vector fields do not coincide when their domains
of definition overlap. Our results allow to obtain Lavaurs vector fields
from the infinitesimal generator and vice versa.

Let us explain in what sense the infinitesimal generator of $\varphi$
is multi-summable in the variable $x$.
Consider a petal $L_{j}$ of $\varphi_{|x=0}$
and the region $H_{j}$ of $\Re (\aleph^{*} X)$ as defined above.
There exists a unique analytic vector field (the Lavaurs vector field)
$X_{H_{j}}^{\varphi} = g_{H_{j}}^{\varphi}(x,y) \partial / \partial y$ defined
in $H_{j}$ such that
$X_{H_{j}}^{\varphi}(\psi_{H_{j}}^{\varphi}) \equiv 1$.
By fixing $L_{j}$ but varying  $\Re (\aleph^{*} X)$ we obtain a family
of functions $\{ g_{H_{j}}^{\varphi}(x,y) \}$ that is multi-summable in the
variable $x$. The common asymptotic development of the family
$\{ X_{H_{j}}^{\varphi}  \}$ is of the form
\[ \hat{X}_{L_{j}}^{\varphi} = \left({ \sum_{k=0}^{\infty} g_{j,k}^{\varphi}(y) x^{k} }\right)
\frac{\partial}{\partial y} \]
where $g_{j,k}^{\varphi}$ is defined in $L_{j}$ for any $k  \in {\mathbb N} \cup \{0\}$.
Now we fix $k \geq 0$. The family $\{ g_{j,k}^{\varphi} \}$ is parametrized by the set
of petals of $\varphi_{|x=0}$. Moreover the functions in the family
$\{ g_{j,k}^{\varphi} \}$ are sums of
a $\nu$ summable power series $\hat{g}_{k}^{\varphi} $ where $2 \nu$ is the number
of petals of $\varphi_{|x=0}$. As a result of this two step process we recover
\[  \log \varphi = \left({ \sum_{k=0}^{\infty} \hat{g}_{k}^{\varphi}(y) x^{k} }\right)
\frac{\partial}{\partial y} \]
the infinitesimal generator of $\varphi$.

Let us remark that  the estimates in this paper for Fatou coordinates are the
generalizations of those in \cite{JR:mod}. Thus, they provide
\begin{itemize}
\item Asymptotic developments of the Lavaurs vector fields $X_{H}^{\varphi}$
until the first non-zero term in the neighborhood of the fixed points \cite{JR:mod}.
\item Gevrey asymptotics in the neighborhood of the bifurcation set $x=0$.
\end{itemize}
\subsection{Isolated zeros theorem for analytic conjugacy}
Given $\varphi \in \diff{p1}{2}$, let us relate the class of analytic
conjugacy of $\varphi$ with the classes of analytic conjugacy
of the $1$-dimensional germs in the family $\{\varphi_{|x=x_{n}}\}$
for some sequence $x_{n} \to 0$. This is an application of the
multi-summability of the Ecalle-Voronin invariants.

Let $\varphi, \eta \in \diff{p1}{2}$.
We denote $\varphi \sim \eta$ if there exists
$\sigma \in \diff{}{2}$ conjugating $\varphi$ and $\eta$
such that $\sigma_{|Fix (\varphi)} \equiv Id$.
\begin{teo}
\label{teo:intro}
Let $\varphi, \eta \in \diff{p1}{2}$
with $Fix (\varphi)= Fix (\eta)$. Suppose there exist $s \in {\mathbb R}^{+}$ and
a sequence $x_{n} \to 0$ contained in $B(0,\delta) \setminus \{0\}$ such that
for any $n \in {\mathbb N}$ the
restrictions $\varphi_{|x=x_{n}}$ and $\eta_{|x=x_{n}}$ are conjugated by an injective
holomorphic mapping $\kappa_{n}$ defined in $B(0,s)$ and fixing the points in
$Fix (\varphi) \cap \{ x=x_{n}\}$. Then we obtain $\varphi \sim \eta$.
\end{teo}
The previous theorem is the corollary \ref{cor:modi} of subsection
\ref{subsec:isozeros}. Indeed we prove the more general theorem
\ref{teo:modi} that analyzes what is the minimum domain of definition
$B(0,s_{n})$ of the mappings $\kappa_{n}$ such that the theorem is still true.
The theorem can be extended to codimension finite resonant diffeomorphisms
(remark \ref{rem:res}).
Theorem \ref{teo:intro} is a generalization of the main theorem
of \cite{JR:mod} where the $\kappa$ mappings where required to exist
for any parameter $x$ in a pointed neighborhood of $0$.

The theorem can be interpreted as an isolated zeros theorem, i.e. if the set
of parameters $x_{0}$ such that  $\varphi_{|x=x_{0}}$ and $\eta_{|x=x_{0}}$
are conjugated by a mapping defined in $B(0,s)$ acumulates the origin then
it contains a neighborhood of the origin. This analytic type property is a consequence
of the quasi-analytic nature of Ecalle-Voronin invariants.
The possibility of having flat Ecalle-Voronin invariants
when $x \to 0$ would be an obstruction to the extension of conjugations even
if defined in open sets of parameters.
\subsection{Remarks}
%
The codimension $\infty$ case can be studied with the techniques in this paper.
Indeed all the transforms of $\varphi$ in the desingularization process
of $Fix (\varphi)$ are codimension $\infty$ unfoldings.
We still obtain exponentially small estimates
when comparing different Fatou coordinates, Lavaurs vector fields or
Ecalle-Voronin invariants. These objects are not multi-summable since
in general they are not bounded in the neighborhood of $x=0$.

Given un unfolding $\varphi(x,y)=(x,f(x,y))$ of a resonant difeomorphism
$\phi \in \diff{}{}$ (i.e. $(\partial \phi/\partial y)(0)$ is a root of the unit of order $p$)
it admits a Jordan decomposition
\[ \varphi = \varphi_{s} \circ \varphi_{u} = \varphi_{u} \circ \varphi_{s}. \]
In fact $\varphi_{s}, \varphi_{u}$ are formal diffeomorphisms
such that $\varphi_{s}$ is semisimple (or equivalently formally linearizable)
and $\varphi_{u}$ is unipotent ($j^{1} \varphi_{u}$ is unipotent).
Since $j^{1} \varphi$ is $p$ periodic then
$\varphi_{s}$ is also $p$ periodic. Thus
we obtain $\varphi^{p} = \varphi_{u}^{p} = {\rm exp}(p \log \varphi_{u})$.
We deduce that $\log \varphi_{u}$ is multi-summable since
$\varphi^{p} \in \diff{p1}{2}$.
The semisimple part $\varphi_{s}$ is the unique formal diffeomorphism such
that $\varphi_{s}^{p} \equiv Id$, it commutes with $\varphi_{u}$ and
$(\partial (y \circ \varphi_{s})/\partial y)(0,0) =  (\partial (y \circ \varphi)/\partial y)(0,0) $.
Therefore $\varphi_{s}$ is totally determined by $\varphi_{u}$
(prop. 5.4 in \cite{JR:mod}). It is possible to use the multi-summability of
$\log \varphi_{u}$ to deduce the multi-summability of $\varphi_{s}$.
\section{Notations and definitions}
\label{sec:notdef}
Let $\diff{}{n}$ be the group of complex analytic germs of
diffeomorphisms at $0 \in {\mathbb C}^{n}$.
Denote $Fix (\varphi)$ the set of fixed points of an element $\varphi$
of $\diff{}{n}$.
\begin{defi}
Let $\varphi: U \to V$ be a holomorphic mapping where $U$ and $V$ are
open sets of ${\mathbb C}^{n}$. We say that  a holomorphic
$\psi: U \to {\mathbb C}$ is a Fatou coordinate of $\varphi$ if
$\psi \circ \varphi \equiv \psi + 1$ in $U \cap \varphi^{-1}(U)$.
\end{defi}
We say that $\varphi \in \diff{}{2}$
is a parametrized diffeomorphism if $\varphi(x,y)$ is of the form
$(x,f(x,y))$. Equivalently $\varphi$ is an unfolding of
$\varphi_{|x=0} \in \diff{}{}$.
We denote $\diff{p}{2}$ the group of parametrized
diffeomorphisms.
Let $\diff{1}{}$ be the subgroup of $\diff{}{}$ of germs
whose linear part is the identity.
\begin{defi}
We define the set
\[ \diff{p1}{2} = \{ \varphi \in \diff{p}{2} :
\varphi_{|x=0} \in \diff{1}{} \setminus \{ Id \} \} . \]
Then $\diff{p1}{2}$ is the set of one dimensional unfoldings of
one dimensional tangent to the identity germs of diffeomorphisms
(excluding the identity).
We consider $\diff{tp1}{2}$ the subset of $\diff{p1}{2}$ such that
$\varphi \in \diff{tp1}{2}$ if all the irreducible components of
$Fix (\varphi)$ are of the form $y=g(x)$.
\end{defi}
We define a formal vector field $\hat{X}$ as a derivation of the
maximal ideal of the ring
${\mathbb C}[[x_{1},\hdots,x_{n}]]$. We also express $\hat{X}$ in
the more conventional form
\[ \hat{X} =
\hat{X}(x_{1}) \partial / \partial{x_{1}} + \hdots +
\hat{X}(x_{n}) \partial / \partial{x_{n}} . \]
We denote $\hat{\mathcal X} \cn{n}$ the set of formal vector fields.
We denote ${\mathcal X} \cn{n}$ the Lie algebra of germs of analytic vector fields in a
neighborhood of $0 \in {\mathbb C}^{n}$. A formal vector field $X$
belongs to ${\mathcal X} \cn{n}$ if and only if
$X({\mathbb C}\{x_{1},\hdots,x_{n}\}) \subset {\mathbb C}\{x_{1},\hdots,x_{n}\}$.
\begin{defi}
Let $X$ be a holomorphic vector field defined in an open set $U$ of
${\mathbb C}^{n}$. We say that  a holomorphic
$\psi: U \to {\mathbb C}$ is a Fatou coordinate of $X$ if
$X(\psi) \equiv 1$.
\end{defi}
\begin{defi}
We denote $\Xt$ the subset of ${\mathcal X} \cn{2}$ of vector fields
of the form
\[  X =
v(x,y) (y-\gamma_{1}(x))^{s_{1}} \hdots (y-\gamma_{p}(x))^{s_{p}}
\partial / \partial{y} \]
where $v, \gamma_{1}, \hdots, \gamma_{p} \in {\mathbb C}\{x,y\}$,
$v(0) \neq 0 = \gamma_{1}(0) = \hdots = \gamma_{p}(0)$
and $s_{1}+\hdots+s_{p} \geq 2$.
We denote $\hat{\mathcal X}_{tp1}\cn{2}$ the set of formal vector fields that
are of the previous form but allowing $\hat{v} \in {\mathbb C}[[x,y]]$.
\end{defi}
Given a vector field $X$ defined in a domain $U \subset {\mathbb C}^{n}$
we denote $\Re (X)$ the real flow of $X$, namely the two dimensional
vector field on ${\mathbb R}^{2n} = {\mathbb C}^{n}$ defined by $X$.

Suppose that $X \in {\mathcal X} \cn{n}$ is singular at $0$.
We denote ${\rm exp}(tX)$ the flow of the vector field $X$, it is
the unique solution
of the differential equation
\[ \frac{\partial}{\partial{t}} {\rm exp}(tX) = X({\rm exp}(tX)) \]
with initial condition ${\rm exp}(0X)=Id$. We define the exponential
${\rm exp}(X)$ of $X$ as
${\rm exp}(1X)$. We can define the exponential operator for
a nilpotent $\hat{X} \in \hat{\mathcal X} \cn{n}$ and in particular for
$\hat{X} \in \hat{\mathcal X}_{tp1}\cn{2}$ as
\[
\begin{array}{rccc}
{\rm exp}(\hat{X}): & {\mathbb C}[[x_{1},\hdots,x_{n}]] & \to &
{\mathbb C}[[x_{1},\hdots,x_{n}]] \\
& g & \to & \sum_{j=0}^{\infty} \frac{\hat{X}^{\circ(j)}}{j!} (g) .
\end{array}
\]
Moreover the definition coincides with the previous one if $\hat{X}$ is
convergent, i.e. $({\rm exp}(X))(g) = g \circ {\rm exp}(X)$ for any
$g \in {\mathbb C}[[x,y]]$.
By the properties of the exponential mapping given a unipotent
$\varphi \in \diff{}{n}$ (i.e. $j^{1} \varphi$ is unipotent)
there exists a unique formal nilpotent vector field
$\log \varphi \in  \hat{\mathcal X} \cn{n}$ such that
$\varphi = {\rm exp}(\log \varphi)$
(see \cite{Ecalle} and \cite{MaRa:aen}).
We say that $\log \varphi$ is
the {\it infinitesimal generator} of $\varphi$.
\begin{defi}
\label{def:res1}
Let $X$ be a holomorphic vector field defined in a connected domain
$U \subset {\mathbb C}$ such that $X \neq 0$. Consider $P \in Sing X$.
There exists a unique meromorphic differential form
$\omega$ in $U$ such that $\omega(X)=1$. We denote
$Res(X,P)$ the residue of $\omega$ at the point $P$.
\end{defi}
\begin{defi}
\label{def:res2}
Let $Y=f(x,y) \partial/\partial y \in \Xt$.
Given $({x}^{0},y^{0}) \in Sing Y$ we define
$Res(Y,({x}^{0},y^{0}))=Res(f({x}^{0},y) \partial/\partial y,y^{0})$.
\end{defi}
\section{Dynamical splitting} \label{subsec:dynsplit}
We introduce a dynamical splitting $\digamma$ associated to an element of $\Xt$ along with
some notation. Most of the concepts were already introduced in \cite{JR:mod}.

Let $X \in \Xt$. We say that $T_{0}=\{(x,y) \in B(0,\delta) \times \overline{B(0,\epsilon)} \}$
is a {\it seed}. We provide a method to divide the set $T_{0}$. At each step of the process we have
a vector $\beta=(0, \beta_{1}, \hdots, \beta_{k}) \in \{0\} \times {\mathbb C}^{k}$ with $k \geq 0$ and a seed
$T_{\beta} = \{ (x,t) \in  B(0,\delta) \times \overline{B(0,\eta)} \}$ in coordinates
$(x,t)$ canonically associated to $T_{\beta}$. We either decide not to split $T_{\beta}$
or we divide it in sets ${\mathcal E}_{\beta}$,
$M_{\beta} = {\mathcal C}_{\beta} \cup \cup_{\zeta \in S_{\beta}} T_{\beta, \zeta}$
where $S_{\beta}$ is a finite subset of ${\mathbb C}$.
The seeds $T_{\beta, \zeta}$ for  $(\beta,\zeta) \in \{0\} \times {\mathbb C}^{k+1}$
with $\zeta \in S_{\beta}$ are divided in ulterior steps of the process.
The sets $T_{\beta}$, $M_{\beta}$, ${\mathcal E}_{\beta}$ and ${\mathcal C}_{\beta}$ are defined by induction on $k$.
Every set $M_{\beta}$   is called
a {\it magnifying glass set}. The sets ${\mathcal E}_{\beta}$ are called
{\it exterior basic sets} whereas the sets ${\mathcal C}_{\beta}$ are called {\it compact-like basic sets}.
At the first step of the process we consider $\beta=0$, $k=0$ and the coordinates $(x,y)$ in $T_{0}$.

Suppose also that
\begin{equation}
\label{for:forx}
 X = x^{e({\mathcal E}_{\beta})}
v(x,t) (t-\gamma_{1}(x))^{s_{1}} \hdots (t-\gamma_{p}(x))^{s_{p}}
\partial / \partial{t}
\end{equation}
in $T_{\beta}$ where $\gamma_{1}(0)=\hdots=\gamma_{p}(0)=0$
and $\{ v=0 \} \cap T_{\beta}=\emptyset$.
We denote
\[ \partial_{e} {\mathcal E}_{\beta} = \{ (x,t) \in B(0,\delta) \times \partial B(0,\eta) \} \ {\rm and} \
\nu({\mathcal E}_{\beta}) = s_{1}+\hdots+s_{p}-1. \]
For $p=1$ we define the {\it terminal exterior set} ${\mathcal E}_{\beta}=T_{\beta}$,
we do not split the terminal seed $T_{\beta}$.
We say that
$\partial_{e} {\mathcal E}_{\beta}$
is the {\it exterior boundary} of ${\mathcal E}_{\beta}$ and $e({\mathcal E}_{\beta})$ is the
{\it exterior exponent} of ${\mathcal E}_{\beta}$. We define
$\iota({\mathcal E}_{\beta}) = e({\mathcal E}_{\beta})$ the {\it interior exponent} of ${\mathcal E}_{\beta}$.
Suppose $p>1$. We define $t=xw$ and the sets
${\mathcal E}_{\beta} = T_{\beta} \cap \{ |t| \geq |x|\rho \}$,
$\tilde{\mathcal E}_{\beta} = T_{\beta} \cap \{ |t| > |x| 2 \rho \}$
and $M_{\beta}= \{ (x,w) \in B(0,\delta) \times \overline{B(0,\rho)} \}$
for some $\rho>>0$.
\begin{defi}
Given an exterior set ${\mathcal E}_{\beta}$ we define
$X_{{\mathcal E}_{\beta}}$ as the vector field defined in $T_{\beta}$ such that
$X= x^{e({\mathcal E}_{\beta})} X_{{\mathcal E}_{\beta}}$.
\end{defi}
We define
\[ \partial_{I} {\mathcal E}_{\beta} = \{ (x,t) \in B(0,\delta) \times \overline{B}(0,\eta) :
|t| = |x| \rho \} \] of ${\mathcal E}_{\beta}$. The sets $\partial_{e} {\mathcal E}_{\beta}$ and
$\partial_{I} {\mathcal E}_{\beta}$ are the {\it exterior} and {\it interior boundaries} of
${\mathcal E}_{\beta}$ respectively. We say that the coordinates $(x,t)$ are {\it adapted} to
$T_{\beta}$ and ${\mathcal E}_{\beta}$. We have
\[ X  =
x^{e({\mathcal E}_{\beta})+s_{1}+\hdots+s_{p}-1}v(x,xw)
{\left({ w -  \gamma_{1}(x) / x }\right)}^{s_{1}} \hdots
{\left({ w - \gamma_{p}(x) / x }\right)}^{s_{p}}
\partial / \partial{w}. \]
We denote $S_{\beta} =  \{ (\partial \gamma_{1}/\partial x)(0), \hdots,
(\partial \gamma_{p}/\partial x)(0) \}$.
We define
\[ {\mathcal C}_{\beta}=\{
(x,w) \in B(0,\delta) \times (\overline{B}(0,\rho) \setminus \cup_{\zeta \in S_{\beta}}
B(\zeta, \eta_{\beta,\zeta}) )  \} \]
where $\eta_{\beta,\zeta}>0$ is small enough for any $\zeta \in S_{\beta}$.  We denote
\[ \partial_{e} {\mathcal C}_{\beta} = \{ (x,w) \in B(0,\delta) \times \partial B(0,\rho) \}, \
\partial_{I} {\mathcal C}_{\beta} = \{ (x,w) \in B(0,\delta) \times
\cup_{\zeta \in S_{\beta}} \partial B(\zeta, \eta_{\beta,\zeta}) \} . \]
We define
\[ \nu({\mathcal C}_{\beta}) =  \nu({\mathcal E}_{\beta}) \ {\rm and} \
\iota({\mathcal E}_{\beta})=e({\mathcal C}_{\beta}) = \iota({\mathcal C}_{\beta}) =
e({\mathcal E}_{\beta})+ \nu({\mathcal E}_{\beta}). \]
We say that $e({\mathcal E}_{\beta})$  and $\iota({\mathcal E}_{\beta})$ are the {\it exterior} and
{\it interior exponents} of
${\mathcal E}_{\beta}$ respectively whereas $e({\mathcal C}_{\beta})$  and $\iota({\mathcal C}_{\beta})$
are the {\it exterior} and {\it interior exponents} of ${\mathcal C}_{\beta}$.
\begin{defi}
Given a compact-like set ${\mathcal C}_{\beta}$ we define
$X_{{\mathcal C}_{\beta}}$ as the vector field defined in $M_{\beta}$ such that
$X= x^{e({\mathcal C}_{\beta})} X_{{\mathcal C}_{\beta}}$.
\end{defi}
\begin{defi}
\label{def:pol}
We define the polynomial vector field
\[ X_{\beta}(\lambda) = \lambda^{e({\mathcal C}_{\beta})} v(0,0)
{\left({ w - (\partial \gamma_{1} / \partial x)(0) }\right)}^{s_{1}} \hdots
{\left({ w - (\partial \gamma_{p} / \partial x)(0) }\right)}^{s_{p}}
\partial / \partial{w}  \]
for $\lambda \in {\mathbb S}^{1}$ (see the equation (\ref{for:forx}), note that $t=xw$)
associated to $X$, $T_{\beta}$ and ${\mathcal C}_{\beta}$.
\end{defi}

Fix $\zeta \in S_{\beta}$. We define the seed
$T_{\beta, \zeta}= \{ (x,t') \in B(0,\delta) \times \overline{B(0,\eta_{\beta,\zeta})} \}$ where
$t'$ is the coordinate $w - \zeta$. By definition $(x,t')$ is the set of adapted coordinates
associated to $T_{\beta, \zeta}$.
We say that the seed $T_{\beta, \zeta}$ is a {\it son} of the seed
$T_{\beta}$. We have
\[ X = x^{e({\mathcal E}_{\beta, \zeta})} h(x,t')
\prod_{(\partial \gamma_{j}/\partial x)(0) = \zeta}
{\left({ t' - \left({
\gamma_{j}(x) / x - \zeta }\right)
}\right)}^{s_{j}}  \partial / \partial{w}  \]
where $e({\mathcal E}_{\beta, \zeta}) = e({\mathcal C}_{\beta})$. We just introduced
a method to divide $|y| \leq \epsilon$ in a union of exterior
and compact-like sets.
\begin{figure}[h]
\begin{center}
\includegraphics[height=6cm,width=6cm]{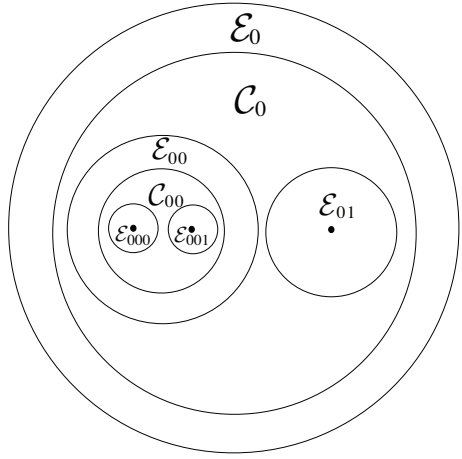}
\end{center}
\caption{Splitting for $X = y(y-{x}^{2})(y-x) \partial/\partial{y}$ in a line $x=x_{0}$} \label{fig5}
\end{figure}

\underline{Example}: Consider $X=y (y-x^{2}) (y-x) \partial / \partial y$.
Denote $w=y/x$. The vector field $X$ has the form $x^{2} w (w-x) (w-1) \partial / \partial w$
in coordinates $(x,w)$. The polynomial vector field $X_{0}(\lambda)$ associated to the seed
$T_{0}=B(0,\delta) \times \overline{B}(0,\epsilon)$ is equal to $\lambda^{2} w^{2}  (w-1) \partial / \partial w$.

The exterior and compact-like sets
associated to $T_{0}$ are ${\mathcal E}_{0}=T_{0} \cap \{ |y| \geq \rho_{0} |x| \}$ and
${\mathcal C}_{0}=\{ |y| \leq \rho_{0} |x| \} \setminus (\{ |w|<\eta_{00} \} \cup \{ |w-1|<\eta_{01} \} )$
respectively. The sons of $T_{0}$ are the seeds $T_{00}=\{ |w| \leq \eta_{00} \}$ and
${\mathcal E}_{01} = T_{01}= \{ |w-1| \leq \eta_{01} \}$.
The seed $T_{01}$ is terminal since it only contains one irreducible component of
$Sing X$.

Denote $w'=w/x$. We have
$X=x^{3} w' (w'-1) (xw'-1) \partial / \partial w'$
in coordinates $(x,w')$. Thus
$- \lambda^{3} w' (w'-1) \partial / \partial w'$
is the polynomial vector field $X_{00}(\lambda)$
associated to $T_{00}$. The seed $T_{00}$ contains an exterior set
${\mathcal E}_{00} = T_{00} \cap \{ |w| \geq \rho_{00} |x| \}$ for
$\rho_{00}>>1$, a compact-like set
${\mathcal C}_{00} =
\{ |w'| \leq \rho_{00} \} \setminus ( \{ |w'|<\eta_{000} \} \cup \{ |w'-1|<\eta_{001} \} )$
and two terminal seeds ${\mathcal E}_{000}=T_{000}=\{ |w'| \leq \eta_{000} \}$
and ${\mathcal E}_{001}=T_{001}= \{ |w'-1| \leq \eta_{001} \}$ for some
$0 < \eta_{000},\eta_{001} <<1$. We have $e({\mathcal E}_{0})=0$,
$\iota({\mathcal E}_{0})=e({\mathcal C}_{0})=e({\mathcal E}_{00})=e({\mathcal E}_{01})=2$
and
$\iota({\mathcal E}_{00})=e({\mathcal C}_{00})=e({\mathcal E}_{000})=e({\mathcal E}_{001})=3$.
\begin{rem}
The dynamical splitting $\digamma$ depends on the choice of the constants determining
the size of the basic sets. Anyway given two dynamical splittings
$\digamma_{1}$ and $\digamma_{2}$ there are bijective correspondences
${\mathcal E}_{\beta}^{1} \leftrightarrow {\mathcal E}_{\beta}^{2}$ and
${\mathcal C}_{\beta}^{1} \leftrightarrow {\mathcal C}_{\beta}^{2}$
between exterior and compact-like sets.
\end{rem}
\begin{defi}
Given two dynamical splittings $\digamma$ and $\digamma'$
we say that $\digamma$ is a refinement of $\digamma'$
if we have ${\mathcal E}_{\beta} \subset {\mathcal E}_{\beta}'$ for any exterior set
${\mathcal E}_{\beta}$ of $\digamma$ and
${\mathcal C}_{\beta} \supset {\mathcal C}_{\beta}'$ for any compact-like set
${\mathcal C}_{\beta}$ of $\digamma$.
\end{defi}
\begin{defi}
\label{def:refin}
Given two dynamical splittings $\digamma_{1}$ and $\digamma_{2}$ we can consider a refinement
$\digamma_{1} \cup \digamma_{2}$ of both of them.
More precisely if ${\mathcal E}_{\beta}^{1}$ and ${\mathcal E}_{\beta}^{2}$ are
exterior sets associated to $\digamma_{1}$ and $\digamma_{2}$ respectively then
${\mathcal E}_{\beta} = {\mathcal E}_{\beta}^{1} \cap {\mathcal E}_{\beta}^{2}$
is associated to $\digamma_{1} \cup \digamma_{2}$. Analogously suppose that
${\mathcal C}_{\beta}^{1}$ and ${\mathcal C}_{\beta}^{2}$ are
compact-like sets associated to $\digamma_{1}$ and $\digamma_{2}$ respectively.
Then ${\mathcal C}_{\beta} = {\mathcal C}_{\beta}^{1} \cup {\mathcal C}_{\beta}^{2}$
is associated to $\digamma_{1} \cup \digamma_{2}$.
\end{defi}
\section{Multi-transversal flows}
\label{sec:tramul}
Let $X \in  \Xt$.
This section is intended to define multi-transversal flows and describe their main
properties. Roughly speaking a multi-transversal flow is of the form
$\Re ({\aleph}_{\mathcal B} X)$
by restriction to any basic set ${\mathcal B}$ of the dynamical splitting.
The function $\aleph_{\mathcal B}(r  \lambda, y)$ is continuous, takes values in
${\mathbb S}^{1} \setminus \{1,-1\}$ and  depends only on $\lambda \in {\mathbb S}^{1}$.
The choice of $\aleph_{\mathcal B}$ is related to make the dynamics of
$\Re ({\aleph}_{\mathcal B} X)$ stable with respect to the direction
$\lambda {\mathbb R}^{+}$ in the parameter space.

The dynamical behavior of a transversal flow
$\Re (\mu X)_{| {\mathcal E} \cap (\{x\} \times B(0,\epsilon))}$ does not depend on
$x \in B(0,\delta) \setminus \{0\}$ or $\mu \in {\mathbb S}^{1} \setminus \{1,-1\}$
for any exterior basic set ${\mathcal E}$. Indeed it basically depends on
$\nu ({\mathcal E})$. Therefore in order to study the stability properties of transversal
or multi-transversal flows we can focus on compact-like sets where we can use
the associated polynomial vector fields defined in  subsection \ref{subsec:dynsplit}.
The notion of stability for polynomial vector fields, namely the absence of homoclinic
trajectories, was introduced in     \cite{DES}. We can define an analogous concept for
any compact-like set. Since the intersection of a compact-like set and the line $x=0$
contains only the origin our concept of stability is of infinitesimal type. It imposes a
set of restrictions in compact-like sets. Another way of making sense of the infinitesimal
label in stability is by noticing that there are as many
compact-like sets as steps in a minimal desingularizion of
the singular set of $X$ via blow-ups.

The subsection \ref{subsec:polvecfield} is devoted to introduce some of the
concepts and definitions that are required to define infinitesimal stability for
a multi-transversal flow.
The definition of $\aleph_{\mathcal B}$ is explained in subsections
\ref{subsec:stamuldir} and \ref{subsec:consmult}.
The construction of the multi-transversal flows is introduce in subsection
\ref{subsec:consmult}. The subsection \ref{subsec:traflowbs}
reviews the properties of transversal flows that are described in \cite{JR:mod}
and whose analogues for multi-transversal flows are studied in subsections
\ref{subsec:mtraflowpes} and \ref{subsec:mtraflowbs}.
Subsections \ref{subsec:npmtraflow} and \ref{subsec:sizereg} deal with some
quantitative properties of the constructions that will be used later on.
\subsection{Polynomial vector fields}
\label{subsec:polvecfield}
A vector field $X \in \Xt$ is of the form
\[ X = u(x,y) (y - g_{1}(x))^{n_{1}} \hdots (y - g_{p}(x))^{n_{p}} \frac{\partial}{\partial y} \]
where $u \in {\mathbb C}\{x,y\}$ is a unit. Given a non-degenerate element
$h(x,y) \partial / \partial y$
of ${\mathcal X}_{p1} \cn{2}$ there exists $k \in {\mathbb N}$ such that
$h(x^{k},y) \partial / \partial y$ belongs to $\Xt$.
\begin{defi}
Let $Y=P(w) \partial / \partial w$ be a polynomial vector field. We define
$\nu (Y) = \deg (P) -1$.
\end{defi}
\begin{defi}
\label{def:trY}
Let $Y=P(w) \partial / \partial w \in {\mathbb C}[w] \partial / \partial w$ with $\nu (Y) \geq 1$.
We define $Tr_{\to \infty}(Y)$ as the set of trajectories $\gamma :(c,d) \to {\mathbb C}$
of $\Re (Y)$ with $c \in {\mathbb R} \cup \{-\infty\}$ and $d \in {\mathbb R}$ such that
$\lim_{\zeta \to d} \gamma (\zeta) = \infty$. Analogously we define
$Tr_{\leftarrow \infty} (Y)= Tr_{\to \infty}(-Y)$.
\end{defi}
\begin{defi}
We say that $\Re(Y)$ has $\infty$-connections or homoclinic trajectories if
$Tr_{\to \infty}(Y) \cap Tr_{\leftarrow \infty}(Y) \neq \emptyset$.
Then there exists a trajectory $\gamma:(c_{-},c_{+}) \to {\mathbb C}$
of $\Re(Y)$ such that $c_{-},c_{+} \in {\mathbb R}$ and
$\lim_{\zeta \to c_{-}} \gamma(\zeta)= \infty = \lim_{\zeta \to c_{+}} \gamma(\zeta)$.
This notion has been introduced in \cite{DES}
for the study of deformations of elements of $\diff{1}{}$ (see also \cite{JR:mod}).
\end{defi}
\begin{rem}
Let $Y=P(w) \partial / \partial w$ be a polynomial vector field with
$\deg (P)  \geq 2$. We define the set $S \subset {\mathbb S}^{1}$ defined by
$\mu \in S$ if $\Re(\mu Y)$ and $\Re(\mu' Y)$ are orbitally equivalent
for any $\mu'  \in {\mathbb S}^{1}$ in  a  neighborhood  of $\mu$.
The set $S$ is the set of directions $\mu$ in which the dynamics of
$\Re (\mu X)$ is stable with respect to $\mu$. It turns out that $S$ coincides with
the set
$\{ \mu \in {\mathbb S}^{1} : \Re (\mu X) \ {\rm has \ no \ homoclinic \ trajectories} \}$
\cite{DES}.
 \end{rem}
\begin{defi}
\label{def:noinfcon}
We denote ${\mathcal X}_{\infty}\cn{}$ the set of polynomial vector
fields in ${\mathcal X} \cn{}$ such that $\nu(Y) \geq 1$ and
$2 \pi i \sum_{P \in S} Res(Y,P) \not \in {\mathbb R} \setminus \{0\}$
for any subset $S$ of $Sing Y$.
\end{defi}
\begin{rem}
Let $Y \in {\mathcal X}_{\infty}\cn{}$. The vector field $\Re (Y)$ has no homoclinic
trajectories. Thus $Re(1 \cdot Y)$ is stable at $1$ (see \cite{DES} or \cite{JR:mod}).
\end{rem}
\begin{defi}
\label{def:levels}
  Let $X \in \Xt$. Consider the compact-like sets
${\mathcal C}_{1}$, $\hdots$, ${\mathcal C}_{q}$ associated to $X$.
Let $X_{j}(\lambda) = \lambda^{e({\mathcal C}_{j})} P_{j}(w) \frac{\partial}{\partial w}$ be the
polynomial vector field associated to ${\mathcal C}_{j}$ and $X$ for $1 \leq j \leq q$.
We define
\[ {\mathcal U}_{X}^{j} = \left\{{ (\lambda, \mu) \in {\mathbb S}^{1} \times {\mathbb S}^{1} :
\lambda^{e({\mathcal C}_{j})} \mu P_{j}(w) \frac{\partial}{\partial w} \not \in {\mathcal X}_{\infty}\cn{} }\right\} \]
The vector field $\mu X_{j}(\lambda)$ has no homoclinic trajectories for
$(\lambda, \mu) \not \in {\mathcal U}_{X}^{j}$ \cite{DES} (see also \cite{JR:mod}).
We define $\Xi_{X}^{j} = \{ \lambda \in {\mathbb S}^{1} : (\lambda,i) \in {\mathcal U}_{X}^{j} \}$
for $1 \leq j \leq q$. We enumerate $\tilde{e}_{1} < \tilde{e}_{2} < \hdots < \tilde{e}_{\tilde{q}}$ the elements of
$\cup_{\{ j \in \{1,\hdots,q\} : \ {\mathcal U}_{X}^{j} \neq \emptyset \}} \{ e({\mathcal C}_{j}) \}$.
By convention we denote $\tilde{e}_{0}=0$ and $\tilde{e}_{\tilde{q}+1}=\infty$. We define
\[ \tilde{\mathcal U}_{X}^{k} =
\cup_{\{ j \in \{1,\hdots,q\} : \ e({\mathcal C}_{j}) = \tilde{e}_{k} \}} {\mathcal U}_{X}^{j}
\ \ {\rm and} \ \ \tilde{\Xi}_{X}^{k} =
\cup_{\{ j \in \{1,\hdots,q\} : \ e({\mathcal C}_{j}) = \tilde{e}_{k} \}} \Xi_{X}^{j} \]
for $1 \leq k \leq \tilde{q}$. We define
${\mathcal U}_{X} = \cup_{j=1}^{q} {\mathcal U}_{X}^{j} = \cup_{k=1}^{\tilde{q}} \tilde{\mathcal U}_{X}^{k}$.
\end{defi}
The notations in the previous definition are fixed from now on.
\begin{defi}
We say that $\tilde{\Xi}_{X}^{k}$ is the set of singular directions of level
$\tilde{e}_{k}$. Clearly the set $\tilde{\Xi}_{X}^{k}$ is finite for any $1 \leq k \leq \tilde{q}$.
\end{defi}
Later on we will see that given $\varphi \in \diff{p1}{2}$ we can associate a vector field
$X \in \Xt$, namely a normal form. We will see that
$  \{ \tilde{e}_{1}, \hdots, \tilde{e}_{\tilde{q}} \}$ is the set of levels of
multi-summability of the infinitesimal generator of $\varphi$. Moreover we will prove
that the set of singular directions of level $\tilde{e}_{k}$ is contained in
$\tilde{\Xi}_{X}^{k}$ for any $1 \leq k \leq \tilde{q}$.
\subsection{Stable multi-directions}
\label{subsec:stamuldir}
Let $X \in \Xt$. As a generalization of transversal flows of the form $\Re (\mu X)$  for
$\mu \in {\mathbb S}^{1} \setminus \{1,-1\}$ we are
going to construct transversal flows of the form $\Re (\aleph^{*} X)$ where
$\aleph^{*} : (B(0,\delta) \times B(0,\epsilon)) \setminus Sing X \to e^{i(0,\pi)}$ is a
$C^{\infty}$ function. We need them to capture the multiple summability levels associated
to Fatou coordinates of elements of $\diff{p1}{2}$.

Consider the notations at the beginning of the section.
\begin{defi}
We say that a multi-direction
$(\tilde{\mu}_{1}, \hdots, \tilde{\mu}_{\tilde{q}}) \in (e^{i(0,\pi)})^{\tilde{q}}$ is stable
at a direction $\lambda {\mathbb R}^{+}$ in the parameter space $x$ if
$(\lambda, \tilde{\mu}_{k}) \not \in \tilde{\mathcal U}_{X}^{k}$ for any $1 \leq k \leq \tilde{q}$.
We say that a multi-direction
$({\mu}_{1}, \hdots, {\mu}_{q}) \in (e^{i(0,\pi)})^{q}$ is stable at a direction
$\lambda {\mathbb R}^{+}$ if
$(\lambda, {\mu}_{j}) \not \in {\mathcal U}_{X}^{j}$ for any $1 \leq j \leq q$.

A stable multi-direction $(\tilde{\mu}_{1}, \hdots, \tilde{\mu}_{\tilde{q}})$ induces
a stable multi-direction $({\mu}_{1}, \hdots, {\mu}_{q})$. We define
$\mu_{j} = \tilde{\mu}_{k}$ if $e({\mathcal C}_{j}) = \tilde{e}_{k}$ and $\mu_{j} = i$ if
$e({\mathcal C}_{j}) \not \in  \{ \tilde{e}_{1}, \hdots, \tilde{e}_{\tilde{q}} \}$.
\end{defi}
\begin{defi}
Let $X \in \Xt$.
Let $I$ a closed arc $e^{i[u_{0},u_{1}]}$ of ${\mathbb S}^{1}$.
We say that a function $\aleph: I \to (e^{i(0,\pi)})^{\tilde{q}}$ is a stable multi-direction at
$I$ if
\[ \aleph(e^{iu}) \equiv (e^{i \tilde{\theta}_{1}(u)}, \hdots, e^{i \tilde{\theta}_{\tilde{q}}(u)}) \]
for some continuous decreasing functions
$\tilde{\theta}_{1}, \hdots, \tilde{\theta}_{\tilde{q}}:[u_{0},u_{1}] \to (0,\pi)$
and $\aleph(\lambda)$ is stable at $\lambda {\mathbb R}^{+}$ for any $\lambda \in e^{i[u_{0},u_{1}]}$.
\end{defi}
\begin{defi}
Consider $\lambda \in {\mathbb S}^{1}$ and $\upsilon \in {\mathbb R}^{+} \cup \{0\}$. We define
\[ I_{j}(\lambda, \upsilon) =
\lambda e^{i[-\frac{\pi}{2 \tilde{e}_{j}}- \upsilon, \frac{\pi}{2 \tilde{e}_{j}}+ \upsilon]} \]
for $1 \leq j \leq \tilde{q}$ and $I_{0}(\lambda, \upsilon)={\mathbb S}^{1}$.
\end{defi}
\begin{defi}
\label{def:msing}
We define ${\mathcal M} \subset ({\mathbb S}^{1})^{\tilde{q}}$ as the set whose elements
$(\lambda_{1}, \hdots, \lambda_{\tilde{q}})$ satisfy
$\lambda_{1} \not \in \tilde{\Xi}_{X}^{1}$, $\hdots$, $\lambda_{\tilde{q}} \not \in \tilde{\Xi}_{X}^{\tilde{q}}$
and $I_{j+1}(\lambda_{j+1},0) \subset I_{j}(\lambda_{j},0)$ for any $1 \leq j <\tilde{q}$.
\end{defi}
In the language of summability $(\tilde{e}_{1},\hdots, \tilde{e}_{\tilde{q}})$ and
$(\lambda_{1},\hdots , \lambda_{\tilde{q}})$ are admissible parameter vectors.

Let $\Lambda=(\lambda_{1}, \hdots, \lambda_{\tilde{q}}) \in {\mathcal M}$. Since
$\lambda_{j}=e^{i \theta_{j}} \not \in  \tilde{\Xi}_{X}^{j}$, given $\upsilon >0$ small enough
there exists a continuous function $\tilde{\mu}_{j}: I_{j}(\lambda_{j},\upsilon ) \to e^{i(0,\pi)}$ such that
$\tilde{\mu}_{j}(\lambda_{j})=i$ and
\[ (\lambda, \tilde{\mu}_{j}(\lambda)) \not \in \tilde{\mathcal U}_{X}^{j}   \ \
\forall \lambda \in I_{j}(\lambda_{j},\upsilon )  . \]
Moreover, if we choose a lift
$\tilde{\theta}_{j}: \theta_{j} + [- \pi / (2 \tilde{e}_{j})-\upsilon , \pi / (2 \tilde{e}_{j})+\upsilon ] \to (0,\pi)$
of $\tilde{\mu}_{j}$, i.e. $e^{i \tilde{\theta}_{j}(\theta)} = \tilde{\mu}_{j}(e^{i \theta})$
we choose $\tilde{\mu}_{j}$ such that $\tilde{\theta}_{j}$ is a decreasing (maybe non-strictly decreasing) function.
By considering a smaller $\upsilon >0$, if necessary, we obtain
\[ (\lambda', \mu) \not \in \tilde{\mathcal U}_{X}^{j} \ \ \ \forall 1 \leq j \leq \tilde{q} \ \ \forall \lambda
\in I_{j}(\lambda_{j},0) \ \ \forall \lambda' \in \lambda e^{i[-\upsilon ,\upsilon ]} \ \
\forall \mu \in \tilde{\mu}_{j}(\lambda e^{i[-\upsilon ,\upsilon ]}) . \]
By taking a smaller $\upsilon >0$, we obtain that for any $1 \leq j \leq \tilde{q}$ there exist
compact sets $I^{j,1}, \hdots, I^{j,s_{j}} \subset {\mathbb S}^{1}$ and complex numbers
${\mu}_{j,1}, \hdots, {\mu}_{j,s_{j}} \in e^{i[\pi/4,3 \pi/4]}$ such that
\begin{itemize}
\item $I^{j,1} \cup  \hdots \cup I^{j,s_{j}} = {\mathbb S}^{1}$.
\item $(\lambda',\mu_{j,l})  \not \in \tilde{\mathcal U}_{X}^{j}$ for all
$\lambda' \in I^{j,l} e^{i[-\upsilon ,\upsilon ]}$ and $l \in \{1,\hdots,s_{j}\}$.
\end{itemize}
Consider $1 \leq j \leq \tilde{q}$ and $\lambda \in {\mathbb S}^{1}$.
We choose $1 \leq l \leq s_{j}$ such that $\lambda \in I^{j,l}$.
We define $\tilde{\mu}_{j}^{*}(\lambda) = \mu_{j,l}$. Denote $\upsilon_{\Lambda}=\upsilon$.
\begin{defi}
\label{def:alepk}
Let $\Lambda=(\lambda_{1}, \hdots, \lambda_{\tilde{q}}) \in {\mathcal M}$.
Consider $\lambda \in  I_{k}(\lambda_{k},\upsilon_{\Lambda})$ if $k \neq 0$
or $\lambda \in {\mathbb S}^{1}$ if $k=0$. The formula
\[ \aleph_{k,\Lambda,\lambda}(\lambda')=
(\tilde{\mu}_{1}(\lambda'), \hdots, \tilde{\mu}_{k}(\lambda'),
\tilde{\mu}_{k+1}^{*}(\lambda), \hdots, \tilde{\mu}_{\tilde{q}}^{*}(\lambda)) \]
defines a stable multi-direction
$\aleph_{k,\Lambda,\lambda}: \lambda e^{i[-\upsilon_{\Lambda},\upsilon_{\Lambda}]}
\to (e^{i(0,\pi)})^{\tilde{q}}$ at
$\lambda e^{i[-\upsilon_{\Lambda},\upsilon_{\Lambda}]}$.
We denote $\aleph_{k,\Lambda,\lambda}^{j}$ the projection in the $j$ coordinate
of the image of $\aleph_{k,\Lambda,\lambda}$.
\end{defi}
\begin{rem}
\label{rem:alext}
Every multi-direction in
$\aleph_{k,\Lambda,\lambda}^{1} \times \hdots \times \aleph_{k,\Lambda,\lambda}^{\tilde{q}}$
is stable at
$\lambda' {\mathbb R}^{+}$ for any
$\lambda' \in \lambda e^{i[-\upsilon_{\Lambda},\upsilon_{\Lambda}]}$.
Therefore $\mu_{j}' X_{j}(\lambda')$ belongs to ${\mathcal X}_{\infty}\cn{}$
for all $\mu_{j}' \in \aleph_{k,\Lambda,\lambda}^{j}$ and
$\lambda' \in \lambda e^{i[-\upsilon_{\Lambda},\upsilon_{\Lambda}]}$.
\end{rem}
\subsection{Dynamics of transversal flows in basic sets}
\label{subsec:traflowbs}
Let us remind the reader some properties of transversal flows
before defining multi-transversal flows.
We will adapt these properties to the multi-transversal setting.
Further details and proofs can be found in \cite{JR:mod}.

Let us introduce some notations.
\begin{defi}
\label{def:contsets}
We consider coordinates $(x,y) \in {\mathbb C} \times {\mathbb C}$ or
$(r,\lambda,y) \in
{\mathbb R}_{\geq 0} \times {\mathbb S}^{1} \times {\mathbb C}$
in ${\mathbb C}^{2}$. Given a set $F \subset {\mathbb C}^{2}$ we denote
$F(x_{0})$ the set $F \cap \{ x=x_{0} \}$ and by $F(r_{0},\lambda_{0})$
the set $F \cap \{ (r,\lambda)=(r_{0},\lambda_{0}) \}$.
\end{defi}
\begin{defi}
\label{def:traj}
Let $\gamma_{P}(s)$ be the trajectory of $\Re (Z)$ such that $\gamma_{P}(0)=P$.
We define ${\mathcal I}(Z, P,F)$ the maximal interval where
$\gamma_{P}(s)$ is well-defined and belongs to $F$ for any $s \in {\mathcal I}(Z, P,F)$ whereas
$\gamma_{P}(s)$ belongs to ${F}^{\circ}$ for
any $s \neq 0$ in the interior of ${\mathcal I}(Z, P,F)$.
We denote $\Gamma(Z,P,F)=\gamma_{P}({\mathcal I}(Z, P,F))$. We define
\[ \partial {\mathcal I}(Z, P,F) = \{ \inf({\mathcal I}(Z, P,F)), \sup({\mathcal I}(Z, P,F)) \}
\subset {\mathbb R} \cup \{-\infty, \infty\}. \]
We denote $\Gamma(Z,P,F)(s) = \gamma_{P}(s)$.
\end{defi}
\begin{defi}
Let $X \in \Xt$. Let ${\mathcal E}$ be an exterior set associated to $X$. We say that
${\mathcal E}$ is parabolic if $\nu ({\mathcal E}) \geq 1$. Every non-terminal
exterior set is parabolic.
\end{defi}
The qualitative behavior of a transversal flow $\Re(\mu X)$
($\mu \in {\mathbb S}^{1} \setminus \{-1,1\}$) in an exterior set
${\mathcal E}$ depends on the nature of the set of tangent points between
$\Re(\mu X)$  and $\partial {\mathcal E}$. Since we want to reproduce the
same ideas for multi-transversal flows we introduce these concepts.
\begin{defi}
Let $X \in \Xt$. Consider an exterior set
\[ {\mathcal E} = \{(x,t) \in B(0,\delta) \times {\mathbb C} :  \eta \geq |t| \geq \rho|x| \} \]
associated to $X$ with $0 < \eta <<1$ and $\rho \geq 0$.
We define $T{\mathcal E}_{\mu X}^{\eta}(r, \lambda)$ the set of tangent points between
$|t|=\eta$ and $\Re(\lambda^{e({\mathcal E})} \mu X_{\mathcal E})_{|x=r \lambda}$
for $(r,\lambda,\mu) \in {\mathbb R}_{\geq 0} \times {\mathbb S}^{1} \times {\mathbb S}^{1}$.
We denote $T_{\mu X}^{\epsilon}(r \lambda)= T {\mathcal E}_{\mu X}^{\epsilon}(r, \lambda)$
for the particular case ${\mathcal E}= {\mathcal E}_{0}$.
\end{defi}
\begin{defi}
\label{def:taintpt0}
Let $X \in \Xt$. Consider a compact-like set
\[ {\mathcal C} = \{ (x,w) \in
B(0,\delta) \times (\overline{B}(0,\rho) \setminus \cup_{\zeta \in
S_{\mathcal C}} B(\zeta, \eta_{{\mathcal C},\zeta}) )  \} \]
associated to $X$. We denote $T{\mathcal C}_{\mu X}^{\rho}(r, \lambda)$ the set of
tangent points between
$|w|=\rho$ and $\Re(\lambda^{e({\mathcal C})} \mu X_{\mathcal C})_{|x=r \lambda}$.
\end{defi}
\begin{defi}
Let ${\mathcal B}$ a basic set.
We say that a point $y_{0} \in T{\mathcal B}_{\mu X}(r, \lambda)$
is {\it convex} if the germ of trajectory of
$\Re(\lambda^{e({\mathcal B})} \mu X_{\mathcal B})_{|x=r \lambda}$ passing through $y_{0}$
is contained in ${\mathcal B}$.
\end{defi}
\begin{lem}
\cite{JR:mod}
\label{lem:tgpt20}
Let $X \in \Xt$ and an exterior set
${\mathcal E}=\{ \eta \geq |t| \geq \rho|x| \}$ associated to $X$
with $0 < \eta <<1$ and $\rho \geq 0$. Then the set
$T{\mathcal E}_{\mu X}^{\eta}(r, \lambda)$ is composed of
$2 \nu({\mathcal E})$ convex points for all
$(\lambda,\mu) \in {\mathbb S}^{1} \times {\mathbb S}^{1}$
and $r$ close to $0$. Each connected component of
$\{ t \in \partial B(0,\eta) \} \setminus T{\mathcal E}_{\mu X}^{\eta}(r, \lambda)$
contains a unique point of $T{\mathcal E}_{\mu' X}^{\eta}(r, \lambda)$ for any
$\mu' \in {\mathbb S}^{1} \setminus \{ -\mu, \mu \}$.
\end{lem}
\begin{lem} \cite{JR:mod}
\label{lem:tgpt30}
Let $X \in \Xt$ and a compact-like set
\[ {\mathcal C} = \{ (x,w) \in
B(0,\delta) \times (\overline{B}(0,\rho) \setminus \cup_{\zeta
\in S_{\mathcal C}} B(\zeta, \eta_{{\mathcal C},\zeta}) )  \} \]
associated to $X$ with $\rho >>0$. Then $T{\mathcal C}_{\mu X}^{\rho}(r, \lambda)$
is composed of $2 \nu({\mathcal C})$ convex points for all
$(\lambda,\mu) \in {\mathbb S}^{1} \times {\mathbb S}^{1}$
and $r$ close to $0$. Moreover each connected component of
$\{ |w|=\rho \} \setminus T{\mathcal C}_{\mu X}^{\rho}(r, \lambda)$
contains a unique point of $T {\mathcal C}_{\mu' X}^{\rho}(r, \lambda)$ for
any point $\mu' \in {\mathbb S}^{1} \setminus \{ -\mu, \mu \}$.
\end{lem}
\subsubsection{Parabolic exterior sets}
\label{subsub:parextset}
Let us analyze the quantitative and the qualitative dynamical behavior of a transversal
flow  $\Re(\mu X)$  ($\mu \in {\mathbb S}^{1} \setminus \{-1,1\}$) in a parabolic
exterior set ${\mathcal E}$. We want to apply the same ideas in the more general setting
of multi-transversal flows.
\begin{rem}
\label{rem:qbeps}
The qualitative behavior of $\Re(\mu X)$
($\mu \in {\mathbb S}^{1} \setminus \{-1,1\}$) in a parabolic exterior set
${\mathcal E}={\mathcal E}_{\beta}$ is described in
propositions 6.1 and 6.2 and corollary 6.1 of \cite{JR:mod}. It is a truncated Fatou
flower (see figure (\ref{EVfig6})). The proof is based on:
\begin{itemize}
\item Tangent points between $\Re (\mu X)$ and $\partial {\mathcal E}$ are convex.
\item $\sharp T{\mathcal E}_{\mu X}^{\eta}(r,\lambda) = \nu ({\mathcal E})$
for any $(r,\lambda) \in [0,\delta) \times {\mathbb S}^{1}$.
\item Suppose ${\mathcal E}$ is non-terminal and denote
${\mathcal C}={\mathcal C}_{\beta}$. Then we have
\[ \sharp T{\mathcal E}_{\mu X}^{\eta}(r,\lambda)=
\sharp T{\mathcal C}_{\mu X}^{\rho}(r,\lambda) =\nu ({\mathcal E}) \]
for any $(r,\lambda) \in [0,\delta) \times {\mathbb S}^{1}$.
\end{itemize}
\end{rem}
The previous properties are guaranteed by lemmas
\ref{lem:tgpt20} and \ref{lem:tgpt30}.
The results hold true for
$0<\eta \leq \eta^{0}$ for some $\eta^{0} \in {\mathbb R}^{+}$.
We also need $\rho \geq \rho^{0}$ for some $\rho^{0} \in {\mathbb R}^{+}$ if
${\mathcal E}$ is non-terminal. These properties are preserved by refinement of the
dynamical splitting $\digamma$.
Let us remark that the choice of $\eta^{0}$ and $\rho^{0}$ does not depend
on $\mu \in {\mathbb S}^{1}$.

   Next, we introduce the other ingredient in \cite{JR:mod} to describe the
dynamics in exterior sets. It is of quantitative nature.
\begin{defi}
\label{def:psiext}
Let $X \in \Xt$.
Let ${\mathcal E} = \{   \eta \geq |t| \geq \rho|x| \}$
be a parabolic exterior set associated to a seed $T$.
The vector field $X_{\mathcal E}$ is of the form
\[ X_{\mathcal E} =v(x,t) (t-\gamma_{1}(x))^{s_{1}} \hdots (t-\gamma_{p}(x))^{s_{p}}
\partial / \partial{t} \]
where $v$ is a function never vanishing in $T$. Denote $\gamma_{\mathcal E}=\gamma_{1}$.
Denote $\psi_{\mathcal E}^{0}$  a holomorphic integral of the time form of
$v(0,t-\gamma_{\mathcal E}(x)) (t-\gamma_{\mathcal E}(x))^{\nu({\mathcal E})+1}
\partial / \partial{t}$
defined in the neighborhood of ${\mathcal E} \setminus Sing X$.
\end{defi}
\begin{rem}
\label{rem:psiext}
The function $\psi_{\mathcal E}^{0}$ is of the form
\[ \psi_{\mathcal E}^{0} = \frac{-1}{\nu({\mathcal E}) v(0,0)} \frac{1}{(t-\gamma_{\mathcal E}(x))^{\nu({\mathcal E})}}
+ Res(X_{\mathcal E},(0,0)) \ln (t-\gamma_{\mathcal E}(x)) + h(t-\gamma_{\mathcal E}(x)) + b(x) \]
where $h(z)$ is a $O(1/z^{\nu({\mathcal E})-1})$ meromorphic function
and $b(x)$ is a holomorphic function in the neighborhood of $0$. Thus given $\zeta>0$ there exists
$C_{\zeta} \in {\mathbb R}^{+}$ such that
\[ \frac{1}{C_{\zeta}}  \frac{1}{|t-\gamma_{\mathcal E}(x)|^{\nu({\mathcal E})}} \leq
|\psi_{\mathcal E}^{0}|(x,t) \leq  C_{\zeta} \frac{1}{|t-\gamma_{\mathcal E}(x)|^{\nu({\mathcal E})}} \]
in ${\mathcal E} \cap \{t - \gamma_{\mathcal E}(x) \in {\mathbb R}^{+} e^{i[-\zeta, \zeta]}\}
\cap \{x \in B(0,\delta(\zeta)) \}$.
\end{rem}
\begin{defi}
\label{def:psiE}
Let ${\mathcal E} = \{   \eta \geq |t| \geq \rho|x| \}$
be a parabolic exterior set associated to $X \in \Xt$.
Denote $\psi_{\mathcal E}$  a holomorphic integral of the time form of
$X_{\mathcal E}$ defined in the neighborhood of ${\mathcal E} \setminus Sing X$ such that
$\psi_{\mathcal E}(0,y) \equiv \psi_{\mathcal E}^{0}(0,y)$. The function $\psi_{\mathcal E}$ is
multi-valued.
\end{defi}
\begin{lem}
(lemma 6.5 \cite{JR:mod})
\label{lem:itf}
Let ${\mathcal E} = \{(x,t) \in B(0,\delta) \times {\mathbb C} :  \eta \geq |t| \geq \rho|x| \}$
be a parabolic exterior set associated to $X \in \Xt$.
Let $\upsilon>0$, $\zeta>0$. Suppose ${\mathcal E}$ is terminal.
Then $|\psi_{\mathcal E}/\psi_{\mathcal E}^{0} -1| \leq \upsilon$ in
${\mathcal E} \cap \{t - \gamma_{\mathcal E}(x) \in {\mathbb R}^{+} e^{i[-\zeta, \zeta]}\}
\cap \{x \in B(0,\delta(\upsilon,\zeta)) \}$ for some $\delta(\upsilon,\zeta) \in {\mathbb R}^{+}$.
The same inequality is true for a non-terminal ${\mathcal E}$ if $\rho >0$ is big enough.
\end{lem}

\begin{rem}
The lemma \ref{lem:itf} implies that the qualitative behavior of
\[ \Re (\mu X) \ {\rm and} \
\Re(\mu x^{e({\mathcal E})} v(0,t-\gamma_{\mathcal E}(x)) (t-\gamma_{\mathcal E}(x))^{\nu({\mathcal E})}
\partial / \partial{t}) \]
is very similar for all exterior set ${\mathcal E}$ and $\mu \in {\mathbb S}^{1}$.
We say that the properties in remark \ref{rem:qbeps} and
lemma \ref{lem:itf} are the stability properties for the behavior of
transversal flows in parabolic exterior sets. They are preserved by refinement.
Moreover they do not depend on
the transversal flow $\Re (\mu X)$ whose dynamics we are studying.
\end{rem}
\subsubsection{Non-parabolic exterior sets}
\label{subsub:norparextset}
A non-parabolic exterior set ${\mathcal E}={\mathcal E}_{\beta, \zeta}$ is terminal.
We have
\[ {\mathcal E}={\mathcal E}_{\beta, \zeta} = \{(x,t) \in B(0,\delta) \times {\mathbb C} :  |t| \leq \eta  \} \]
and ${\mathcal C}_{\beta}={\mathcal C}_{j}$ for some $1 \leq j \leq q$.
Consider a compact set $I \subset {\mathbb S}^{1}$.
\begin{rem}
The stability properties associated to a transversal flow $\Re (\mu X)$ and a
non-parabolic exterior set ${\mathcal E}$ are:
\begin{itemize}
\item $(\lambda, \mu) \not \in {\mathcal U}_{X}^{j}$ for any $\lambda \in I$.
\item $\Re (\mu X)$ is transversal to $\partial_{e} {\mathcal E}$
in $\{ (x,t) \in  (0,\delta)I \times \partial B(0,\eta) \}$.
\end{itemize}
\end{rem}
The first property implies the second one for  $0 < \eta<<1$
by the argument in subsection (6.4.2) of \cite{JR:mod}.
We can choose  $0<\eta \leq \eta_{0}$ for some $\eta_{0} \in {\mathbb R}^{+}$.
The choice of $\eta_{0}>0$ depends on $I$ and $\mu$.
Given $I$, we can choose the same $\eta_{0}>0$ for any
$\mu' \in {\mathbb S}^{1}$ in a neighborhood of $\mu$.
There exists a dynamical splitting $\digamma'$ whose every refinement
$\digamma$ satisfies
${\mathcal E}_{\beta, \zeta} \subset \{(x,t) \in B(0,\delta) \times {\mathbb C} :  |t| \leq \eta_{0}  \}$.

Let us remark that the unique singular point $(x_{0}, \gamma_{\mathcal E}(x_{0}))$
in ${\mathcal E}(x_{0})$ is attracting or repelling for $\Re (\mu X)_{|x=x_{0}}$
and any $x_{0} \in (0,\delta)  I$ (subsection 6.4.2 of \cite{JR:mod}).
\subsubsection{Compact-like sets}
\label{subsub:comlikset}
Let
\[ {\mathcal C} = {\mathcal C}_{j}= {\mathcal C}_{\beta}=\{
(x,w) \in B(0,\delta) \times (\overline{B}(0,\rho) \setminus
\cup_{\zeta \in S_{\beta}} B(\zeta, \eta_{\beta,\zeta}) )  \} \]
be a compact-like set. Consider a compact set $I \subset {\mathbb S}^{1}$.
\begin{rem}
The stability properties associated to a transversal flow $\Re (\mu X)$ and a
compact-like set ${\mathcal C}$ are:
\begin{itemize}
\item $(\lambda, \mu) \not \in {\mathcal U}_{X}^{j}$ for any $\lambda \in I$.
\item $\sharp T {\mathcal C}_{\mu' X}^{\rho} (r,\lambda) = \nu ({\mathcal C}) \ {\rm and} \
\sharp T ({\mathcal E}_{\beta,\zeta})_{\mu' X}^{\eta_{\beta,\zeta}} (r,\lambda) =
\nu({\mathcal E}_{\beta,\zeta})
\ \forall \zeta \in  S_{\beta} \ \forall \mu' \in {\mathbb S}^{1}$.
\item If there is a trajectory $\gamma$ of $\Re (\mu X_{j}(\lambda))$
for some $\lambda \in I$ whose
$\alpha$ and $\omega$ limits are singletons contained in $Sing (X_{j}(1))$ then
there exists a trajectory
$\gamma'$ of $\Re (\mu X_{j}(\lambda))$ contained in $B(0,\rho)$
such that $\alpha(\gamma)=\alpha (\gamma')$ and $\omega(\gamma)=\omega (\gamma')$.
\end{itemize}
\end{rem}
The second property can be obtained by  choosing $\rho \geq \rho^{0} >0$ and
$0<\eta_{\beta,\zeta} \leq \eta_{\beta,\zeta}^{0}$
for any $\zeta \in  S_{\beta}$.
The above properties imply that the behavior of $\Re (\mu X_{j}(\lambda))$ in
${\mathbb C}$ and $B(0,\rho)$ are analogous (see equation (4) and proposition 6.6 in \cite{JR:mod}).
They  are the ingredients that we use to describe the behavior of $\Re (\mu X)$
in the compact-like set ${\mathcal C}_{\beta}$.
In fact the dynamics of
$\Re (\mu X)$ in ${\mathcal C}_{\beta}$ is analogous to the dynamics of the stable polynomial vector field
$\Re (\mu X_{j}(\lambda))$ (see section 6.5 in \cite{JR:mod} and figure (\ref{EVfig9})).
Given $I$ we can consider the same
choice of $\eta_{0}$ and ${\{ \eta_{\beta,\zeta}^{0} \}}_{\zeta \in S_{\beta}}$ for any $\mu' \in {\mathbb S}^{1}$
in the neighborhood of $\mu$.
\begin{figure}[h]
\begin{center}
\includegraphics[height=4cm,width=10cm]{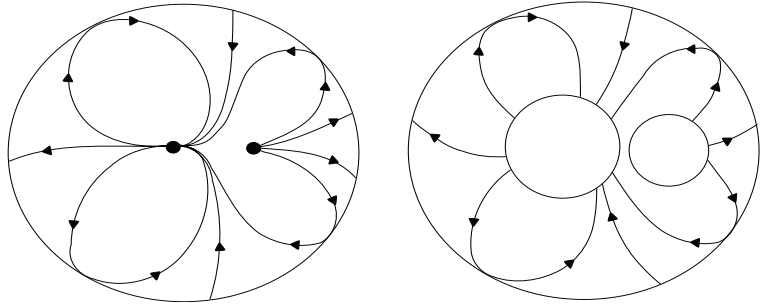}
\end{center}
\caption{Dynamics of $\Re(X_{0}(1))_{|B(0,\rho)}$ and $\Re (X)_{|{\mathcal C}_{0}(x_{0})}$
for $x_{0}$ in ${\mathbb R}^{+}$ and $X=y^{2}(y-x) \partial / \partial y$} \label{EVfig9}
\end{figure}
%
\begin{rem}
The stability properties that we demand to exterior and compact-like sets are compatible
since they are preserved by refinement.
\end{rem}
Next remark is basically lemma 6.13 in \cite{JR:mod}. It is helpful to make pictures
of the dynamics of $\Re (\mu X)$ in ${\mathcal C}_{j}$.
\begin{rem}
\label{rem:cldel}
Consider $P =(r,\lambda,w) \in \partial_{e} {\mathcal C}_{j}$ with $\lambda \in I$
and $(\lambda, \mu) \not \in {\mathcal U}_{X}^{j}$. If
$\Re (\mu X)$ points towards the interior of ${\mathcal C}_{j}$ at $P$ or
$\Re (\mu X)$ is tangent to $\partial_{e} {\mathcal C}_{j}$ at $P$ then
$\Gamma(s) \in \partial_{I} {\mathcal C}_{j}$
where
$\Gamma = \Gamma (\lambda^{e({\mathcal C}_{j})} \mu X_{{\mathcal C}_{j}}, P, {\mathcal C}_{j})$
and $s = \sup ({\mathcal I}(\Gamma))$.
A possible dynamical behavior for the trajectories of $\Re (\aleph X)_{|{\mathcal C}_{j}}$
is represented in figure (\ref{EVfig9}).
\end{rem}
\subsection{Construction of the multi-transversal flow}
\label{subsec:consmult}
Consider $\lambda \in {\mathbb S}^{1}$ and a stable multi-direction
$\aleph = (\tilde{\mu}_{1}, \hdots, \tilde{\mu}_{\tilde{q}}): e^{i[u_{0},u_{1}]} \to (e^{i(0,\pi)})^{\tilde{q}}$.
We denote
\[ (\mu_{1},\hdots,\mu_{q}): e^{i[u_{0},u_{1}]} \to (e^{i(0,\pi)})^{q} \]
the multi-direction induced by $\aleph$. We define
\begin{itemize}
\item $\aleph_{{\mathcal C}_{j}} \equiv \mu_{j}$ for a
a compact-like set ${\mathcal C}_{\beta}={\mathcal C}_{j}$.
\item $\aleph_{{\mathcal E}_{0}} \equiv i$ for the first exterior set.
\item $\aleph_{\mathcal E} \equiv \aleph_{{\mathcal C}_{\beta}}$ for
an exterior seed ${\mathcal E}$ whose father we denote $T_{\beta}$.
\end{itemize}
We defined a function
$\aleph_{{\mathcal B}} : e^{i[u_{0},u_{1}]} \to e^{i(0,\pi)}$ for any basic set ${\mathcal B}$.

We say that a dynamical splitting $\digamma$ is associated to $\aleph$ if
\begin{itemize}
\item $\Re (\mu X)$ is stable at ${\mathcal E}$ for any $\mu \in {\mathbb S}^{1}$ and
any parabolic exterior set ${\mathcal E}$.
\item $\Re (\aleph_{\mathcal E} X)$ is stable at
$\cup_{(r,\theta) \in [0,\delta) \times [u_{0},u_{1}]} {\mathcal E}(r,e^{i \theta})$
for any non-parabolic exterior set ${\mathcal E}$.
\item $\Re (\aleph_{\mathcal C} X)$ is stable at
$\cup_{(r,\theta) \in [0,\delta) \times [u_{0},u_{1}]} {\mathcal C}(r,e^{i \theta})$
for any compact-like set ${\mathcal C}$.
\end{itemize}
The stability properties are introduced in subsections \ref{subsub:parextset},
\ref{subsub:norparextset} and \ref{subsub:comlikset}.
We always consider associated dynamical splittings.
The first condition is obtained just by considering small exterior sets
(see subsection \ref{subsub:parextset}). Given $u \in [u_{0},u_{1}]$ there exist
a dynamical splitting and a neighborhood $I^{u}$ of $u$ in $[u_{0},u_{1}]$ such that
$\Re (\aleph_{\mathcal B} X)$ is stable at
$\cup_{(r,\theta) \in [0,\delta) \times I^{u}} {\mathcal B}(r,e^{i \theta})$
for any non-parabolic basic set ${\mathcal B}$
(see subsections \ref{subsub:norparextset} and \ref{subsub:comlikset}).
The dynamical splitting is obtained by doing successive refinements.
Since $[u_{0},u_{1}]$ is compact we can find the
same dynamical splitting satisfying the conditions above.

We want to define a continuous function
\[ \aleph^{*} : (({\mathbb R}^{+} \cup \{ 0 \}) e^{i[u_{0},u_{1}]} \times B(0,\epsilon)) \setminus Sing X \to
e^{i(0,\pi)} . \]
We obtain a flow
$\Re (\aleph^{*} X)$ defined in $[0,\delta) e^{i[u_{0},u_{1}]} \times B(0,\epsilon)$.
We require $\Re (\aleph^{*} X)$ to fulfill the following properties:
\begin{itemize}
\item We define $\aleph_{|{\mathcal C}_{j}}^{*}=\aleph_{{\mathcal C}_{j}}=\mu_{j}$ and
$(\aleph X)_{{\mathcal C}_{j}} = \mu_{j} X_{{\mathcal C}_{j}}$ for any $1 \leq j \leq q$.
\item Let ${\mathcal E}$ be a terminal exterior set.
We define $\aleph_{|\mathcal E}^{*}=\aleph_{\mathcal E}$ and
$(\aleph X)_{\mathcal E} = \aleph_{|\mathcal E}^{*} X_{\mathcal E}$.
\item Consider a non-terminal exterior set
${\mathcal E}_{\beta}= \{(x,t) \in {\mathbb C}^{2} : \rho |x| \leq |t| \leq \eta \}$ for some
$\rho ,\eta \in {\mathbb R}^{+}$. We require
\begin{itemize}
\item $\Re (\aleph^{*} X)_{|\tilde{\mathcal E}_{\beta}} \equiv \Re (\aleph_{{\mathcal E}_{\beta}} X)$.
\item $\Re (\aleph^{*} X)_{|\partial_{I} {\mathcal E}_{\beta}} \equiv \Re (\aleph_{{\mathcal C}_{\beta}} X)$.
\end{itemize}
\end{itemize}
Denote $\aleph_{{\mathcal E}_{\beta}} = e^{i \theta_{0}}$ and
$\aleph_{{\mathcal C}_{\beta}}= e^{i \theta_{1}}$ for
$\theta_{0}, \theta_{1}:e^{i[u_{0},u_{1}]} \to (0,\pi)$. Let $\varsigma: {\mathbb R} \to [0,1]$ be a
$C^{\infty}$ function such that $\varsigma(- \infty,1+1/4] = \{1\}$ and $\varsigma [2-1/4,\infty) = \{0\}$.
We define
\[ {\aleph}^{*}(r,\lambda,t) =
\aleph_{{\mathcal E}_{\beta}}(\lambda) e^{i (\theta_{1}-\theta_{0})(\lambda) \varsigma (|t|/(\rho r))}
\ \forall (r,\lambda,t) \in {\mathcal E}_{\beta} \cap ([0,\delta) \times e^{i[u_{0},u_{1}]} \times {\mathbb C}) \]
and $(\aleph X)_{{\mathcal E}_{\beta}} \equiv {\aleph}^{*} X_{{\mathcal E}_{\beta}}$.
\begin{defi}
Let $\aleph: I \to (e^{i(0,\pi)})^{\tilde{q}}$ be a stable multi-direction at $I$.
We denote $\Re (\aleph X)$ the flow $\Re (\aleph^{*} X)$ defined in
$[0,\delta)I \times B(0,\epsilon)$.
In particular we can define $\Re (\aleph_{k,\Lambda,\lambda} X)$ in
$[0,\delta) \lambda e^{i[-\upsilon_{\Lambda},\upsilon_{\Lambda}]} \times B(0,\epsilon)$ for
any function $\aleph_{k,\Lambda,\lambda}$.
We say that $\Re (\aleph X)$ is a multi-transversal flow.
\end{defi}
\begin{rem}
\label{rem:unifspl}
Let $X \in \Xt$ and $\Lambda=(\lambda_{1}, \hdots, \lambda_{\tilde{q}}) \in {\mathcal M}$.
Any multi-direction $\aleph_{k,\Lambda,\lambda}$ is of the form
\[ \aleph_{k,l_{k+1},\hdots,l_{\tilde{q}}}(\lambda')=
 (\tilde{\mu}_{1}(\lambda'), \hdots, \tilde{\mu}_{k}(\lambda'),
\mu_{k+1, l_{k+1}}, \hdots, \mu_{\tilde{q}, l_{\tilde{q}}}) \]
where $1 \leq l_{j} \leq s_{j}$ for any $k <j \leq \tilde{q}$. We denote
\[ D_{k,l_{k+1},\hdots,l_{\tilde{q}}}=I_{k}(\lambda_{k},0) \cap \cap_{j=k+1}^{\tilde{q}} I^{j,l_{j}}  . \]
Indeed
$\aleph_{k,l_{k+1},\hdots,l_{\tilde{q}}}:D_{k,l_{k+1},\hdots,l_{\tilde{q}}} e^{i[-\upsilon_{\Lambda} ,\upsilon_{\Lambda} ]}
\to (e^{i(0,\pi)})^{\tilde{q}}$
is a stable multi-direction.
We construct a dynamical splitting $\digamma_{k,l_{k+1},\hdots,l_{\tilde{q}}}$ associated to
$\aleph_{k,l_{k+1},\hdots,l_{\tilde{q}}}$.
By taking a smaller $\upsilon_{\Lambda} >0$ and a compactness argument we can suppose that
$\digamma_{k,l_{k+1},\hdots,l_{\tilde{q}}}$ is associated to the constant multi-direction
$\aleph_{k,l_{k+1},\hdots,l_{\tilde{q}}}(\lambda') : \lambda e^{i[-\upsilon_{\Lambda} ,\upsilon_{\Lambda} ]}
\to (e^{i(0,\pi)})^{\tilde{q}}$
for any $\lambda \in D_{k,l_{k+1},\hdots,l_{\tilde{q}}}$ and any $\lambda' \in \lambda e^{i[-\upsilon_{\Lambda} ,\upsilon_{\Lambda} ]}$.
Since the choices of sequences $(k,l_{k+1},\hdots,l_{\tilde{q}})$ are finite there exists
a dynamical splitting $\digamma_{\Lambda}$ associated to any $\aleph_{k,\Lambda,\lambda}$.
It is obtained by taking a common refinement of every splitting $\digamma_{k,l_{k+1},\hdots,l_{\tilde{q}}}$.
\end{rem}
\begin{rem}
Let $\lambda \in {\mathbb S}^{1}$ and $\mu \in e^{i(0,\pi)}$.
Consider $\aleph (\lambda) = (\mu, \hdots, \mu)$ such that
$(\lambda, \mu) \not \in \tilde{\mathcal U}_{X}^{k}$ for any $1 \leq k \leq \tilde{q}$.
Then $\Re (\aleph(\lambda) X)$ is a multi-transversal flow at $\lambda {\mathbb R}^{+}$
if and only if the transversal flow
$\Re (\mu X)$ is stable at $\lambda {\mathbb R}^{+}$ in the sense in \cite{JR:mod}
(which happens by definition if
$(\lambda, \mu) \not \in {\mathcal U}_{X} = \cup_{k=1}^{\tilde{q}} \tilde{\mathcal U}_{X}^{k}$).
Multi-transversal flows are then a natural generalization of transversal flows and they share
their good properties.
\end{rem}
\subsection{Dynamics of multi-transversal flows in basic sets}
\label{subsec:mtraflowpes}
Let $X \in \Xt$. Consider a multi-transversal flow $\Re (\aleph X)$ and a basic set ${\mathcal B}$.
If ${\mathcal B}$ is a compact-like or a terminal exterior set then we have
$\Re (\aleph X)_{|\mathcal B} \equiv \Re (\aleph_{\mathcal B} X)_{|\mathcal B}$.
Thus the dynamics of
a multi-transversal flow in ${\mathcal B}$ is the dynamics of a transversal flow.
The dynamics of $\Re (\aleph X)$ in a terminal exterior set ${\mathcal E}$ is described
in subsections \ref{subsub:parextset} and \ref{subsub:norparextset}. It is a
Fatou flower dynamics in the parabolic case. Otherwise it is an attractor or a repellor.

Consider a stable direction $\aleph: I \to (e^{i(0,\pi)})^{\tilde{q}}$ at a closed arc $I$.
Let ${\mathcal C}_{1}$, $\hdots$, ${\mathcal C}_{q}$ be the compact-like sets associated
to $X$. We denote $(\mu_{1}, \hdots, \mu_{q}): I \to (e^{i(0,\pi)})^{q}$ the multi-direction
induced by $\aleph$.
We have $\Re (\aleph X)_{|{\mathcal C}_{j}} \equiv \Re (\mu_{j} X)$.
Since $(\lambda, \mu_{j}(\lambda)) \not \in {\mathcal U}_{X}^{j}$ for any $\lambda \in I$
then the dynamics of $\Re (\aleph X)_{|{\mathcal C}_{j}}$ is as described in section (6.4)
of \cite{JR:mod} for any $1 \leq j \leq q$. The  dynamics of
$\Re (\mu X)_{|{\mathcal C}_{j}(r \lambda)}$ is analogous to the dynamics of
a stable polynomial vector field
$\Re (\mu X_{j}(\lambda))$ for $(\lambda, \mu) \not \in {\mathcal U}_{j}^{X}$.

Next, we see that, even if  ${\mathcal B}$ is a non-terminal
exterior set, the dynamics of $\Re (\aleph X)_{|\mathcal B}$ is still analogous to the dynamics
of a transversal flow, i.e. a truncated Fatou flower.
\begin{defi}
Let $X \in \Xt$. Consider an exterior set
\[ {\mathcal E} = \{(x,t) \in B(0,\delta) \times {\mathbb C} :  \eta \geq |t| \geq \rho|x| \} \]
associated to $X$ with $0 < \eta <<1$ and $\rho \geq 0$.
Given a multi-transversal flow $\Re (\aleph X)$
we denote $T{\mathcal E}_{\aleph X}^{\eta}(r, \lambda)=
T {\mathcal E}_{\aleph_{\mathcal E}(\lambda) X}^{\eta}(r, \lambda)$ the set of tangent points
between $\Re (\aleph X)$ and $\partial_{e} {\mathcal E}$.
We denote $T_{\aleph X}^{\epsilon}(r \lambda)= T {\mathcal E}_{\aleph X}^{\epsilon}(r, \lambda)$
for the particular case ${\mathcal E}= {\mathcal E}_{0}$.
\end{defi}
\begin{defi}
\label{def:taintpt}
Let $X \in \Xt$. Consider a compact-like set
\[ {\mathcal C} = \{ (x,w) \in
B(0,\delta) \times (\overline{B}(0,\rho)
\setminus \cup_{\zeta \in S_{\mathcal C}} B(\zeta, \eta_{{\mathcal C}, \zeta})  \} \]
associated to $X$. We denote $T{\mathcal C}_{\aleph X}^{\rho}(r, \lambda)=
T {\mathcal C}_{\aleph_{\mathcal C}(\lambda) X}^{\rho}(r, \lambda)$
\end{defi}
\begin{defi}
Let ${\mathcal B}$ a basic set. Given a multi-transversal flow $\Re (\aleph X)$ we say
that $y_{0} \in T{\mathcal B}_{\aleph X}(r, \lambda)$
is {\it convex} if  the germ of trajectory of
$\Re(\lambda^{e({\mathcal B})} (\aleph X)_{\mathcal B})_{|x=r \lambda}$ passing through $y_{0}$
is contained in ${\mathcal B}$. Equivalently $y_{0} \in T{\mathcal B}_{\aleph X}(r, \lambda)$
is {\it convex} if it is a convex point of
$T{\mathcal B}_{\aleph_{\mathcal B}(\lambda) X}(r, \lambda)$.
\end{defi}
\begin{figure}[h]
\begin{center}
\includegraphics[height=4cm,width=10cm]{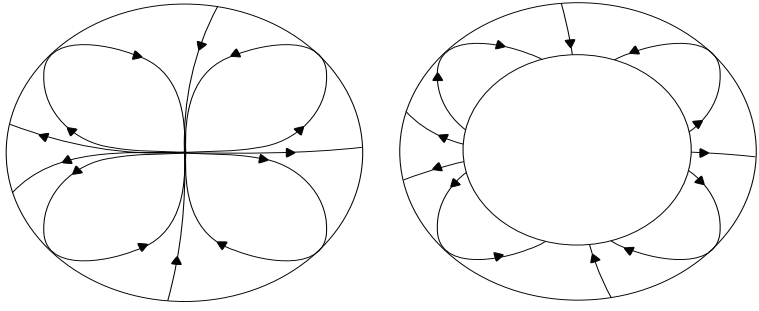}
\end{center}
\caption{Parabolic exterior sets} \label{EVfig6}
\end{figure}
\begin{rem}
\label{rem:qbeps2}
Let ${\mathcal E}={\mathcal E}_{\beta}$ be a parabolic exterior set and a
multi-transversal flow $\Re (\aleph X)$ . We have:
\begin{itemize}
\item Tangent points between $\Re (\aleph X)$ and $\partial {\mathcal E}$ are convex.
\item $\sharp T{\mathcal E}_{\aleph X}^{\eta}(r,\lambda) = \nu ({\mathcal E})$
for any $(r,\lambda) \in [0,\delta) \times {\mathbb S}^{1}$.
\item Suppose ${\mathcal E}$ is non-terminal and denote
${\mathcal C}={\mathcal C}_{\beta}$. Then we have
\[ \sharp T{\mathcal E}_{\aleph X}^{\eta}(r,\lambda)=
\sharp T{\mathcal C}_{\aleph X}^{\rho}(r,\lambda) =\nu ({\mathcal E}) \]
for any $(r,\lambda) \in [0,\delta) \times {\mathbb S}^{1}$.
\end{itemize}
The properties are a consequence of lemmas \ref{lem:tgpt20} and \ref{lem:tgpt30}.
Analogously as for transversal flows (see remark \ref{rem:qbeps})
the qualitative behavior of $\Re(\aleph X)$ in a parabolic exterior set
${\mathcal E}={\mathcal E}_{\beta}$ is a truncated Fatou
flower (see figure (\ref{EVfig6})). Lemma \ref{lem:itf} is also satisfied
since it does not depend on the choice of $\aleph$.
\end{rem}
Next we see that the behavior of a multi-transversal flow in a parabolic exterior set is
analogous to a Fatou flower also from a quantitative point of view. In particular we prove
that the spiraling behavior is bounded in exterior basic sets.
\begin{pro}
\label{pro:estext}
Let $X \in \Xt$ and let
${\mathcal E} = \{ \eta \geq |t| \geq \rho |x| \}$ be a parabolic exterior set
associated to $X$.  Consider a trajectory
$\Gamma = \Gamma( \lambda^{e({\mathcal E})} (\aleph X)_{\mathcal E},(r, \lambda,t),{\mathcal E})$.
for $r \lambda$ in a neighborhood of $0$ and a transversal multi-direction
$\aleph \in (e^{i(0,\pi)})^{\tilde{q}}$ at $\lambda {\mathbb R}^{+}$. Then $\Gamma$ is contained
in a sector centered at $t=\gamma_{\mathcal E}(r \lambda)$ (see def. \ref{def:psiext}) of angle less
than $\zeta$ for some $\zeta>0$ independent of $r, \lambda$, $\Gamma$ and $\aleph$.
\end{pro}
Let us explain the statement. Consider the universal covering
\[ (r,\lambda,\gamma_{\mathcal E}(r \lambda) + e^{z}): {\mathcal E}^{\flat} \to {\mathcal E} \setminus Sing X . \]
Let ${\Gamma}^{\flat}$ the lifting of $\Gamma$ by $(r,\lambda,e^{z})$.
We claim that the set $(Im(z))({\Gamma}^{\flat})$ is contained in an interval of length $\zeta$.
\begin{proof}
We have
\[ X = x^{e({\mathcal E})} v(x,t) (t-\gamma_{1}(x))^{s_{1}} \hdots (t-\gamma_{p}(x))^{s_{p}}
\partial / \partial{t}  \]
where we consider $\gamma_{\mathcal E} \equiv \gamma_{1}$.
%
%
We denote
\[ \psi_{\mathcal E}^{00} =
\frac{-1}{\nu({\mathcal E}) v(0,0)} \frac{1}{(t-\gamma_{\mathcal E}(x))^{\nu({\mathcal E})}} =
\frac{-1}{\nu({\mathcal E}) v(0,0)} \frac{1}{(t-\gamma_{1}(x))^{\nu({\mathcal E})}}. \]
Given $\upsilon>0$ and $\zeta_{0}>0$ we can consider $\eta>0$ small to obtain that
$|\psi_{\mathcal E}^{0}/\psi_{\mathcal E}^{00} -1| < \upsilon$ in the set
$\{ (x,t) \in {\mathcal E} :  |\arg(t-\gamma_{\mathcal E}(x))| \leq \zeta_{0} \}$
(see remark \ref{rem:psiext}).
%
%
Therefore we obtain
$|\psi_{\mathcal E}/\psi_{\mathcal E}^{00} -1| < \upsilon$ in
$\{ (x,t) \in {\mathcal E} :  |\arg(t-\gamma_{1}(x))| \leq \zeta_{0} \}$
by considering $\eta >0$ small enough and $\rho >0$ big enough if ${\mathcal E}$ is not terminal
(lemma \ref{lem:itf}).
%

We have
$(\aleph X)_{\mathcal E} \equiv \aleph_{|\mathcal E}^{*} X_{\mathcal E}$. Since
$\aleph^{*}(x,t) \in e^{i(0,\pi)}$ for any $(x,t) \in {\mathcal E}$ then we have that either
\[ (\psi_{\mathcal E}/\lambda^{e({\mathcal E})})(\Gamma) \cap ({\mathbb R}^{+} \cup \{0\}) = \emptyset \
{\rm or} \ (\psi_{\mathcal E}/\lambda^{e({\mathcal E})})(\Gamma) \cap ({\mathbb R}^{-} \cup \{0\}) = \emptyset . \]
Thus $(\psi_{\mathcal E}/\lambda^{e({\mathcal E})})(\Gamma)$ lies in a sector of angle of angle $2 \pi$.
Since $\psi_{\mathcal E} / \psi_{\mathcal E}^{00} \sim 1$ then $\Gamma$  lies in a sector of
center $t=\gamma_{\mathcal E}(r, \lambda)$ and angle close to $2 \pi/\nu({\mathcal E})$.
\end{proof}
\subsection{Dynamics of multi-transversal flows in a neighborhood of the origin}
\label{subsec:mtraflowbs}
Let $X \in \Xt$.
The dynamical properties of $Re(\mu X)_{|\lambda ({\mathbb R} \cup \{0\}) \times B(0,\epsilon)}$
for $(\lambda, \mu) \not \in {\mathcal U}_{X}$ are obtained by pasting
the dynamics in exterior and compact-like basic sets. The important facts are that
$\Re (\mu X)_{|{\mathcal E}}$ is a Fatou flower dynamics for any parabolic exterior set
${\mathcal E}$, an attractor or a repellor for any non-parabolic exterior set
and that the dynamics of $\Re (\mu X)_{|{\mathcal C}_{j}(r \lambda)}$
is analogous to the dynamics of
a stable polynomial vector field
$\Re (\mu X_{j}(\lambda))$ for $(\lambda, \mu) \not \in {\mathcal U}_{j}^{X}$.
These properties are preserved for a multi-transversal flow $\Re (\aleph X)$. Namely,
the dynamics of
$\Re (\aleph X)$ is a Fatou flower, an attractor or a repellor in exterior sets and
since $(\lambda, \mu_{j}(\lambda)) \not \in {\mathcal U}_{j}^{X}$ then the dynamics of
$\Re (\aleph X)_{|{\mathcal C}_{j} (r\lambda)} =
\Re (\mu_{j}(\lambda) X)_{|{\mathcal C}_{j} (r \lambda)}$
is analogous to the dynamics of a stable polynomial vector field
$\Re (\mu_{j}(\lambda) X_{j}(\lambda))$.

Next we introduce the generalizations for multi-transversal flows of some definitions and theorem
for transversal flows. We skip the proofs since they are straightforward generalizations of their
analogue counterparts in \cite{JR:mod}.
\begin{lem}
\label{lem:goins}
Let $X \in \Xt$. Consider the multi-transversal flow $\Re (\aleph X)$
for some stable multi-direction $\aleph: I \to (e^{i(0,\pi)})^{\tilde{q}}$.
Let $P_{0} \in [0,\delta) I \times \partial{B(0,\epsilon)}$
such that $\Re (\aleph X)$ does not point towards
${\mathbb C} \setminus \overline{B}(0,\epsilon)$ at $P_{0}$.
Denote $\Gamma = \Gamma (\aleph X, P_{0},T_{0})$. Then $[0,\infty)$ is
contained in ${\mathcal I}(\Gamma)$ and $\lim_{\zeta \to \infty} \Gamma (\zeta) \in Sing X$.
Moreover the intersection of $\Gamma [0, \infty)$ with
every compact-like or exterior set is connected.
\end{lem}
The previous lemma is a generalization of lemma 6.13 in \cite{JR:mod}.
The proof of the previous lemma is obtained by using that $\Re (\aleph X)$ is a Fatou
flower in parabolic exterior sets, an attractor or a repellor in non-parabolic exterior sets
and remark \ref{rem:cldel} in compact-like sets.
\begin{defi}
\label{def:regeao}
Let $X \in \Xt$. Consider a multi-transversal flow $\Re (\aleph X)$ for some
stable multi-direction $\aleph: I \to (e^{i(0,\pi)})^{\tilde{q}}$.
We define the set of regions $Reg^{*}(\epsilon, \aleph X, I)$ associated to
 $\Re (\aleph X)$  in $B(0,\delta) \times B(0,\epsilon)$ as the set of
connected components of
\[  ([0,\delta) I \times B(0,\epsilon)) \setminus (Sing X
\cup_{x \in [0,\delta) I}
\cup_{P \in  T_{\aleph X}^{\epsilon}(x)} \Gamma(\aleph X, P, T_{0})) . \]
We define $\alpha^{\aleph X}(P)$ as the $\alpha$-limit of $\Gamma(\aleph X, P, T_{0})$
for any $P \in B(0,\delta) \times \overline{B(0,\epsilon)}$
such that ${\mathcal I}(\aleph X,P,|y| \leq \epsilon)$
contains $(- \infty,0)$. Otherwise we define $\alpha^{\aleph X}(P)=\infty$.
We define $\omega^{\aleph X}(P)$ in an analogous way.
\end{defi}
Figures (1) and (2) in \cite{JR:mod} are examples of partitions
in regions.
The dynamical behavior of $\Re (\aleph X)$ in terminal exterior sets implies that
the singular points of $\Re (\aleph X)$ in $[0,\delta) I \times B(0,\epsilon)$ are
attracting, repelling or parabolic. We have:
\begin{lem}
Let $X \in \Xt$. Consider a multi-transversal flow $\Re (\aleph X)$ for some
stable multi-direction $\aleph: I \to (e^{i(0,\pi)})^{\tilde{q}}$. Given
$P \in B(0,\delta) \times \overline{B}(0,\epsilon)$ then either
$\alpha^{\aleph X}(P)=\infty$ (resp. $\omega^{\aleph X}(P)=\infty$)
or $\alpha^{\aleph X}(P)$ (resp. $\omega^{\aleph X}(P)$) is a singleton contained in
$Sing X$.
\end{lem}
The basins of attraction and repulsion of
the singular points are open sets. The functions $\alpha^{\aleph X}$
and $\omega^{\aleph X}$ are locally constant in
\[  (\{ x \} \times B(0,\epsilon)) \setminus (Sing X \cup
\cup_{P \in  T_{\aleph X}^{\epsilon}(x)} \Gamma(\aleph X, P, T_{0}))  \]
for any $x \in [0,\delta) I$. Given
$H \in Reg^{*}(\epsilon, \aleph X, I)$ the function $\omega^{\aleph X}_{|H}$
is either identically $\infty$ or $\omega^{\aleph X}(H(x))$ is a singleton
$\{ (x,g(x)) \}$ for any $x \in [0,\delta) I$
where $y=g(x)$ is one of the irreducible components of $Sing X$.
We denote $\omega^{\aleph X}(H)$ the curve $y=g(x)$. An analogous property
holds true for $\alpha^{\aleph X}$.
\begin{defi}
\label{def:regt}
Denote
\[ Reg_{\infty}(\epsilon, \aleph X, I)= Reg^{*}(\epsilon,\aleph X, I)
\cap ((\alpha^{\aleph X})^{-1}(\infty) \cup (\omega^{\aleph X})^{-1}(\infty)) \]
and
$Reg(\epsilon, \aleph X, I) = Reg^{*}(\epsilon, \aleph X, I) \setminus Reg_{\infty}(\epsilon, \aleph X,I)$.
We define
\[ Reg_{j}(\epsilon, \aleph X, I) = \{ H \in Reg(\epsilon, \aleph X, I) :
\sharp \{\alpha^{\aleph X}(H), \omega^{\aleph X}(H) \} =j \} \]
for $j \in \{1,2\}$. We define
$Reg_{1}^{*}(\epsilon, \aleph X, I)= Reg_{1}(\epsilon, \aleph X, I) \cup Reg_{\infty}(\epsilon, \aleph X, I)$
and $Reg_{2}^{*}(\epsilon, \aleph X, I)= Reg_{2}(\epsilon, \aleph X, I)$.
\end{defi}
\begin{rem}
Let $X \in \Xt$. Consider a multi-transversal flow $\Re (\aleph X)$ for some
stable multi-direction $\aleph: I \to (e^{i(0,\pi)})^{\tilde{q}}$.
The domains represented by points of  $Reg(\epsilon, \aleph X, I)$ are $\Re (\aleph X)$ invariant.
\end{rem}
\begin{rem}
\label{rem:setLR}
Let $X \in \Xt$. Consider a multi-transversal flow $\Re (\aleph X)$ for some
stable multi-direction $\aleph: I \to (e^{i(0,\pi)})^{\tilde{q}}$.
The set $H(x)$ (see def. \ref{def:contsets}) is connected for
$H \in Reg^{*}(\epsilon, \aleph X,I)$ and $x \in (0,\delta)I$.
The set $H(0)$ is connected for $H \in Reg_{1}^{*}(\epsilon, \aleph X,I)$
whereas otherwise $H(0)$ has two connected components.
\end{rem}
\begin{defi}
Let $H \in Reg^{*}(\epsilon, \aleph X,I)$ (see def. \ref{def:regt}).
We denote ${\mathcal P}(H)$ the set of connected components of $H(0)$
(see def. \ref{def:contsets}). Given $L \in {\mathcal P}(H)$ we define
$H_{L}= (H \setminus H(0)) \cup L$.
We define ${\mathcal P}_{X}^{\epsilon} = \cup_{H \in Reg(\epsilon, \aleph X,I)} {\mathcal P}(H)$.
The set ${\mathcal P}_{X}^{\epsilon}$ does not depend on $I$ or $\aleph$.
\end{defi}
\begin{rem}
\label{rem:LR}
Let $H \in Reg_{j}^{*}(\epsilon, \aleph X, I)$. We have $\sharp {\mathcal P} (H) =j$
(see remark \ref{rem:setLR}).
\end{rem}
\begin{defi}
\label{def:defLR2}
Let $L \in {\mathcal P}_{X}^{\epsilon}$.
We denote $L_{i X}^{\epsilon}(x)$ the unique continuous section of $T_{i X}^{\epsilon}$
such that $L_{i X}^{\epsilon}(0) \in \overline{L}$.
\end{defi}
\begin{defi}
\label{def:defLR}
Fix $\epsilon_{0}>0$ small enough. Let $\epsilon \in (0, \epsilon_{0}/2)$.
Consider $L \in {\mathcal P}_{X}^{\epsilon}$
and the region $H \in Reg(\epsilon, \aleph X,I)$ containing $L$.
There exist $H_{0} \in Reg(\epsilon_{0}, \aleph X,I)$ containing $H$ and
$L_{0} \in {\mathcal P}(H_{0})$ containing $L$.
Let $\psi$ be an integral of the time form of $X$ defined in
a neighborhood of $(L_{0})_{i X}^{\epsilon_{0}}(0)$ in ${\mathbb C}^{2}$
such that $\psi(x, y_{0}) \equiv 0$ where $(0,y_{0})$ is the point
$(L_{0})_{i X}^{\epsilon_{0}}(0)$.
By analytic continuation of $\psi$ we obtain a continuous integral of the time
form $\psi_{H,L}^{X}$ of $X$ in $H_{L}$. We denote $\psi_{H,L}^{X}$ by $\psi_{L}^{X}$
if the data $\epsilon$, $\aleph$ and $I$ are implicit.
\end{defi}
\begin{rem}
The choice of $\epsilon_{0}$ is intended to make the definition of Fatou coordinates
as independent of $\epsilon$ as possible. Given
$H \in Reg(\epsilon, \aleph X,I)$, $L \in {\mathcal P}(H)$ and
$H' \in Reg(\epsilon', \aleph X,I)$, $L' \in {\mathcal P}(H)$ with $H \subset H'$ and
$L \subset L'$ we have that $\psi_{H,L}^{X} \equiv (\psi_{H',L'}^{X})_{|H}$.
\end{rem}
\begin{defi}
\label{def:lj}
There exist $2 \nu({\mathcal E}_{0})$ continuous sections
\[ T_{i X}^{\epsilon, 1}(r,\lambda), \ \hdots, \
 T_{i X}^{\epsilon, 2 \nu({\mathcal E}_{0})}(r,\lambda), \
 T_{i X}^{\epsilon, 2 \nu({\mathcal E}_{0})+1}(r,\lambda) =
T_{i X}^{\epsilon, 1}(r,\lambda) \]
of $T_{i X}^{\epsilon}(r,\lambda)$. The sections are ordered in counter clock wise
sense. Let $L_{j}$ be the element of ${\mathcal P}_{X}^{\epsilon}$
such that $T_{i X}^{\epsilon,j}(0) \in \overline{L_{j}}$.
\end{defi}
\begin{defi}
Consider $H \in Reg^{*}(\epsilon, \aleph X,I) \setminus Reg(\epsilon, \aleph X,I)$ and
$L \in {\mathcal P}(H)$. Consider $L^{1} \in {\mathcal P}_{X}^{\epsilon}$ such that
$(L^{1})_{i X}^{\epsilon}(x) \in \overline{H}$ for any $x \in [0,\delta)I$.
We define $\psi_{L}^{X}$ as the analytic continuation of $\psi_{L^{1}}^{X}$
to $H$. Let us remark that there are two choices of $L^{1}$ and then two
possible definitions of $\psi_{L}^{X}$ since
$\sharp(T_{i X}^{\epsilon}(x) \cap \overline{H}) =2$ for any $x \in [0,\delta)I$.
\end{defi}
Consider $H \in Reg(\epsilon, \aleph X,I)$ and $L \in {\mathcal P}(H)$.
By construction $\psi_{L}^{X}$ is holomorphic in $H^{\circ}$.
Moreover $(x,\psi_{L}^{X})$ is injective in $H$
since $\psi_{L}^{X}(H_{L}(x))$ is simply connected for any $x \in [0,\delta)I$.
By construction $|\psi_{L}^{X}|$ is bounded by below by a positive constant in $H$.
The function $\psi_{L}^{X} + d(x)$ is also a Fatou coordinate of $X$ in $H$ for any
$d \in {\mathbb C}\{x\}$.

 We call subregion of a region
$H \in Reg^{*}(\epsilon, \aleph X, I)$  every set of the form
$H \cap {\mathcal E}$ or $H \cap {\mathcal C}$ where ${\mathcal E}$ is an
exterior set and ${\mathcal C}$ is a compact-like set.
We say that all the subregions of $H \in  Reg_{1}^{*}(\epsilon, \aleph X, I)$
are L-subregions where ${\mathcal P}(H)=\{L\}$.
Consider $H \in Reg_{2}^{*}(\epsilon, \aleph X, I)=Reg_{2}(\epsilon, \aleph X, I)$ with
${\mathcal P}(H)=\{L, R\}$. There exists a unique seed $T_{\beta}$ such that the curves
$\alpha^{\mu X}(H)$ and $\omega^{\mu X}(H)$ are contained in $T_{\beta}$
but they are not contained in the same son of $T_{\beta}$.
A subregion of $H$ contained in $M_{\beta}$ is a L-subregion.
A subregion in the same connected component
of $H \setminus (M_{\beta}^{\circ} \cup \{(0,0)\})$ as $L$
is also a L-subregion. We define ${H}^{L}$ the union of the L-subregions of $H$.
Clearly we have $H=H^{L} \cup H^{R}$ by lemma \ref{lem:goins}.
\begin{figure}[h]
\begin{center}
\includegraphics[height=5cm,width=10.5cm]{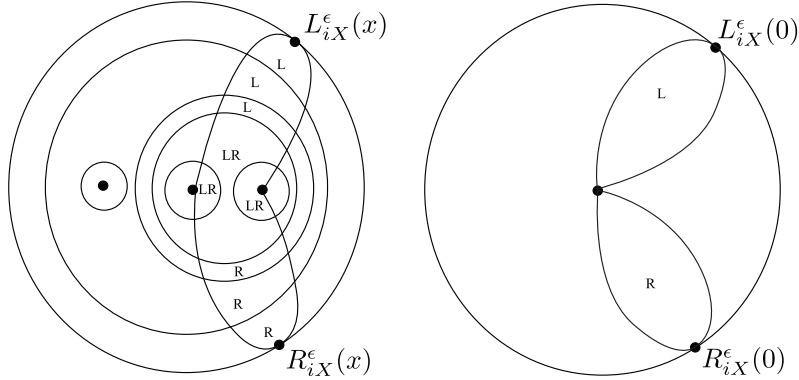}
\end{center}
\caption{Subregions of a region $H$ given ${\mathcal P}(H)=\{L,R\}$} \label{EVfig12}
\end{figure}
\subsection{Non-spiraling properties of multi-transversal flows}
\label{subsec:npmtraflow}
This subsection introduces some of framework leading to the proof of the
summability estimates. We need them in order to prove the multi-summability of
the infinitesimal generator of an element of $\diff{p1}{2}$. The task
requires an analysis of the dynamics of multi-transversal flows in basic sets.
This infinitesimal study is key to capture all levels of multi-summability.

Let $X \in \Xt$. Consider a multi-transversal flow $\Re (\aleph X)$ for some
stable multi-direction $\aleph: I \to (e^{i(0,\pi)})^{\tilde{q}}$.
Let $H \in Reg^{*}(\epsilon, \aleph X,I)$. In this subsection we study the
behavior of $H(r,\lambda) \cap {\mathcal B}$ when
$(r,\lambda) \to (0,\lambda_{0})$ for a basic set ${\mathcal B}$ and
$\lambda_{0} \in I$. We see that such limit exists in adapted coordinates.
\begin{defi}
Let $A$ be a subset of ${\mathbb C} \setminus \{0\}$. Consider the blow-up mapping
$\pi : ({\mathbb R}^{+} \cup \{0\}) \times {\mathbb S}^{1} \to {\mathbb C}$
defined by $\pi (r,\lambda)=r \lambda$. We denote $A_{\pi}$ the subset
$\overline{\pi^{-1}(A)}$ of $({\mathbb R}^{+} \cup \{0\}) \times {\mathbb S}^{1}$.
\end{defi}
\begin{defi}
Let $X \in \Xt$ and let
${\mathcal E} = \{ \eta \geq |t| \geq \rho |x| \}$ be an exterior set associated to $X$.
Consider $A \subset {\mathbb C} \setminus \{0\}$ and a continuous section
$\tau: A \to \partial_{e} {\mathcal E}$ (i.e. $\tau(x) \in \partial_{e} {\mathcal E}(x)$ for
any $x \in A$). The map $\tau$ is of the form $(x,\tau_{\mathcal E}(x))$ in
adapted coordinates $(x,t)$.
We say that $\tau$ is asymptotically continuous if $\tau_{\mathcal E}$ admits a continuous
extension to $A_{\pi}$.
\end{defi}
\begin{defi}
Let $X \in \Xt$ and let
\[ {\mathcal C} = \{ (x,w) \in
B(0,\delta) \times (\overline{B}(0,\rho) \setminus
\cup_{\zeta \in S_{\mathcal C}} B(\zeta, \eta_{{\mathcal C}, \zeta}) )  \} \]
be a compact-like set associated to $X$.
Consider $A \subset {\mathbb C} \setminus \{0\}$ and a continuous section
$\tau: A \to \partial_{e} {\mathcal C}$. The map $\tau$ is of the form $(x,\tau_{\mathcal C}(x))$
in adapted coordinates $(x,w)$. We say that $\tau$ is asymptotically continuous if $\tau_{\mathcal C}$
admits a continuous extension to $A_{\pi}$.
\end{defi}
\underline{Example}: Let $X = y(y-{x}^{2})(y-x) \partial/\partial{y}$. We have
that $\partial_{I} {\mathcal E}_{0} = \partial_{e} {\mathcal C}_{0}$ is of the form
$\{ (x,w) \in B(0,\delta) \times \partial B(0,\rho) \}$ with $w=y/x$. Consider the
sections $\tau_{0}: {\mathbb R}^{+} \to \partial_{e} {\mathcal C}_{0}$ and
$\tau_{1}: {\mathbb R}^{+} \to \partial_{e} {\mathcal C}_{0}$ defined by
$\tau_{0}(x)=(x,x \rho)$ and $\tau_{1}(x)=(x,x e^{i/x} \rho)$ in $(x,y)$ coordinates.
Both sections admit a continuous extension to $x=0$ by defining $\tau_{0}(0)=(0,0)=\tau_{1}(0)$.
Nevertheless $\tau_{0}$ is asymptotically continuous whereas $\tau_{1}$ is not.
\begin{defi}
Let $X \in \Xt$. Consider a multi-transversal flow $\Re (\aleph X)$ for some
stable multi-direction $\aleph: I \to (e^{i(0,\pi)})^{\tilde{q}}$. Let ${\mathcal B}$ be a non-terminal basic
set associated to $X$. Consider a continuous section $\tau: (0,\delta)I \to \partial_{e} {\mathcal B}$.
Suppose that $\Re (\aleph X)$ does not point towards the exterior of
${\mathcal B}$ at $\tau (x)$ for any $x \in  (0,\delta) I$. The dynamics of $\Re (\aleph X)$
implies $\sup {\mathcal I}(\Gamma_{x}) < \infty$ and
$\Gamma_{x}(\sup {\mathcal I}(\Gamma_{x})) \in \partial_{I} {\mathcal B}$
for $\Gamma_{x} = \Gamma (\aleph X, \tau(x), {\mathcal B})$
and any $x \in  (0,\delta) I$. The formula $\partial \tau (x) = \Gamma_{x}(\sup {\mathcal I}(\Gamma_{x}))$
defines a continuous section $\partial \tau: (0,\delta) I \to \partial_{I} {\mathcal B}$
whose image is contained in a connected component of $\partial_{I} {\mathcal B}$.
We define $\psi(\partial \tau(x))-\psi(\tau(x))$ by considering
a Fatou coordinate $\psi$ of $X$ defined in a neighborhood of
$\Gamma_{x}[0,\sup {\mathcal I}(\Gamma_{x})]$. The definition does not depend on the choice of $\psi$.
\end{defi}
\begin{pro}
\label{pro:hit}
Let $X \in \Xt$. Consider a multi-transversal flow $\Re (\aleph X)$ for some
continuous function $\aleph: I \to (e^{i(0,\pi)})^{\tilde{q}}$. Let ${\mathcal B}$ be a non-terminal basic
set associated to $X$. Consider a continuous section $\tau: (0,\delta)I \to \partial_{e} {\mathcal B}$.
Suppose that $\Re (\aleph X)$ does not point towards the exterior of
${\mathcal B}$ at $\tau (x)$ for any $x \in  (0,\delta) I$. Suppose that
$\tau$ is asymptotically continuous. Then $\partial \tau$ is asymptotically continuous.
Moreover the function $F: (0,\delta) \times I \to {\mathbb C}$ defined by
\[   |r|^{\iota({\mathcal B})} (\psi(\partial \tau(r,\lambda))-\psi(\tau(r,\lambda)))  \]
admits a continuous extension to $[0,\delta) \times I$ such that $F(0,\lambda) \in {\mathbb H}$
for any $\lambda \in I$.
\end{pro}
\begin{proof}
Suppose that ${\mathcal B}$ is a compact-like set
\[ {\mathcal C} =\{
(x,w) \in B(0,\delta) \times (\overline{B}(0,\rho)
\setminus \cup_{\zeta \in S_{\mathcal C}} B(\zeta, \eta_{{\mathcal C}, \zeta}) )  \} .\]
Denote $X^{\mathcal C}(\lambda)$ the polynomial vector field associated to ${\mathcal C}$.
We denote
\[ \Gamma_{\lambda}= \Gamma (\aleph_{\mathcal C}(\lambda) X^{\mathcal C}(\lambda),
\tau_{\mathcal C}(0,\lambda), {\mathcal C}). \]
We have that $\Re (\lambda^{e({\mathcal C})} (\aleph X)_{{\mathcal C}})$
does not point towards the exterior of
$\overline{B}(0,\rho)$ at $(r,\lambda, \tau_{\mathcal C}(r,\lambda))$ for any
$(r, \lambda) \in (0,\delta) \times I$.
By continuity  $\Re (\lambda^{e({\mathcal C})}  (\aleph X)_{{\mathcal C}})$ does not point towards the exterior of
$\overline{B}(0,\rho)$ at $(r,\lambda, \tau_{\mathcal C}(r,\lambda))$ for any
$(r, \lambda) \in [0,\delta) \times I$. The vector field $\Re( \aleph X)_{|{\mathcal C}}$
is of the form $\Re (|x|^{e({\mathcal C})}\lambda^{e({\mathcal C})}  (\aleph X)_{{\mathcal C}})$.
We have
\[  {\left({ \lambda^{e({\mathcal C})}  (\aleph X)_{{\mathcal C}} }\right)}_{|(r,\lambda)=
(0,\lambda_{0})} \equiv \aleph_{\mathcal C} (\lambda_{0}) X^{\mathcal C}(\lambda_{0}) \ \
\forall \lambda_{0} \in I . \]
Hence the remark (\ref{rem:cldel}) implies that
$\sup ({\mathcal I}(\Gamma_{\lambda}))< \infty$ and
$\Gamma_{\lambda}(\sup ({\mathcal I}(\Gamma_{\lambda}))) \in \partial_{I} {\mathcal C}$ for any
$\lambda \in I$.
We deduce that $\partial \tau$ extends continuously to $((0,\delta) I)_{\pi}$ by defining
\[ \partial \tau (0,\lambda) = \Gamma_{\lambda}(\sup ({\mathcal I}(\Gamma_{\lambda})))  \]
in $(x,w)$ coordinates for any $\lambda \in I$.
Consider a Fatou coordinate $\psi^{\mathcal C}$
of $X^{\mathcal C}(\lambda_{0})$ defined in a neighborhood of
$\Gamma_{\lambda_{0}}[0,\sup {\mathcal I}(\Gamma_{\lambda_{0}})]$. We obtain
\[ \lim_{(r,\lambda) \to (0,\lambda_{0})} |r|^{e({\mathcal C})}
(\psi(\partial \tau(r,\lambda))-\psi(\tau(r,\lambda))) =
\psi^{\mathcal C}(\partial \tau(0,\lambda_{0})) - \psi^{\mathcal C}(\tau(0,\lambda_{0}))  \]
for any $\lambda_{0} \in I$. Clearly $F(0,\lambda_{0})$ belongs to
$\aleph_{\mathcal C} (\lambda_{0}) {\mathbb R}^{+} \subset {\mathbb H}$.

Suppose now that ${\mathcal B}$ is an exterior set
\[  {\mathcal E} = \{(x,t) \in B(0,\delta) \times {\mathbb C} :  \eta \geq |t| \geq \rho|x| \} . \]
Let ${\mathcal C}$ be the compact-like set such that
$\partial_{I} {\mathcal E} = \partial_{e} {\mathcal C}$. There exist $2 \nu({\mathcal E})$
continuous sections
\[ T {\mathcal E}_{\aleph X}^{\eta, 1}(r,\lambda), \ \hdots, \
T {\mathcal E}_{\aleph X}^{\eta,2 \nu({\mathcal E}) }(r,\lambda), \
T {\mathcal E}_{\aleph X}^{\eta,2 \nu({\mathcal E})+1 }(r,\lambda) =
T {\mathcal E}_{\aleph X}^{\eta, 1}(r,\lambda) \]
of $T {\mathcal E}_{\aleph X}^{\eta}(r,\lambda)$. The sections are ordered in counter clock wise
sense. We denote $arc_{j}(r,\lambda)$ the closed arc going from
$T {\mathcal E}_{\aleph X}^{\eta, j}(r,\lambda)$ to
$T {\mathcal E}_{\aleph X}^{\eta, j+1}(r,\lambda)$
in counter clock wise sense. The choice of the order implies
\[ arc_{j}(r, \lambda) \cap T {\mathcal E}_{\aleph X}^{\eta}(r,\lambda)=
\{ T {\mathcal E}_{\aleph X}^{\eta, j}(r,\lambda), T {\mathcal E}_{\aleph X}^{\eta,j+1}(r,\lambda) \}
\ \ \forall (r,\lambda) \in [0,\delta) \times I. \]
There exists $k \in {\mathbb Z}/(2 \nu({\mathcal E}) {\mathbb Z})$ such that
$\tau_{\mathcal E}(r,\lambda) \in arc_{k}(r,\lambda)$ for any $(r,\lambda) \in [0,\delta) \times I$.
In the rest of the proof we are going to show that $\partial \tau$ and $F$
only depend on $k$.

Consider coordinates $(x,w)$ with $t=wx$. We denote
\[ {\mathcal C}(\rho_{1}) =
{\mathcal C} \cup \{(x,w) \in B(0,\delta) \times {\mathbb C} :  \rho_{1}  \geq |w| \geq \rho \} \]
and
$T {\mathcal C}_{\aleph X}^{\rho_{1}} (r,\lambda) =
T {{\mathcal C}(\rho_{1})}_{\aleph_{\mathcal E} X}^{\rho_{1}} (r,\lambda)$
for $\rho_{1} \geq  2 \rho$. Fix $\rho_{1} \geq 2 \rho$, we define
\[ \Gamma_{P}^{\rho_{1}} = \Gamma (\lambda^{e({\mathcal E})} (\aleph X)_{\mathcal E}, P ,
{\mathcal E} \setminus {\mathcal C}(\rho_{1})) . \]
We have that $\sup {\mathcal I}(\Gamma_{P}^{\rho_{1}}) < \infty$ and
$\Gamma_{P}^{\rho_{1}} (\sup {\mathcal I}(\Gamma_{P}^{\rho_{1}})) \in
\partial_{e} {\mathcal C}(\rho_{1}) \setminus T {\mathcal C}_{\aleph X}^{\rho_{1}} (r,\lambda)$
for all $P \in arc_{k}(r,\lambda)$ and $(r,\lambda) \in (0,\delta) \times I$.
All the points of the form $\Gamma_{P}^{\rho_{1}} (\sup {\mathcal I}(\Gamma_{P}^{\rho_{1}}))$
are in a unique connected component $A_{\rho_{1}}$ of
$\partial_{e} {\mathcal C}(\rho_{1}) \setminus T {\mathcal C}_{\aleph X}^{\rho_{1}}$. We have
\[  {\left({ \lambda^{e({\mathcal C})} \aleph_{\mathcal E}(\lambda) X_{{\mathcal C}}
}\right)}_{|({\mathcal E} \setminus {\mathcal C}(\rho_{1}))(0,\lambda_{0})}
\equiv \aleph_{\mathcal E} (\lambda_{0}) X^{\mathcal C}(\lambda_{0}) \ \ \forall \lambda_{0} \in I . \]
Thus $A_{\rho_{1}}(0,\lambda)$ is an open arc transverse to
$\Re (\aleph_{\mathcal E} (\lambda) X^{\mathcal C}(\lambda))$
and whose closure contains two points of $T {\mathcal C}_{\aleph X}^{\rho_{1}}(0,\lambda)$
for any $\lambda \in I$.
Moreover $\aleph_{\mathcal E} (\lambda) X^{\mathcal C}(\lambda)$ is conjugated to
$w^{\nu({\mathcal C})+1} \partial / \partial w$ in the neighborhood of $w=\infty$.
Thus there exists a unique trajectory $\phi_{\lambda}:(0,s_{\rho_{1},\lambda}] \to {\mathbb C}$
of $\Re(\aleph_{\mathcal E} (\lambda) X^{\mathcal C}(\lambda))$ such that $\lim_{s \to 0} \phi_{\lambda}(s)=\infty$,
$\phi_{\lambda}(0,s_{\rho_{1},\lambda}) \subset {\mathbb C} \setminus \overline{B}(0,\rho_{1})$
and $\phi_{\lambda}(s_{\rho_{1},\lambda}) \in A_{\rho_{1}}(\lambda)$ for any $\lambda \in I$.
The germ of the curve $\phi_{\lambda}$ in the neighborhood of $\infty$
does not depend on $\rho_{1}$.

Consider $\rho_{2} > \rho_{1}$ and $(r,\lambda) \in [0,\delta) \times I$.
We define
$\Gamma_{r, \lambda}^{\rho_{2},\rho_{1}}:A_{\rho_{2}}(r,\lambda) \to A_{\rho_{1}}(r,\lambda)$,
it is the mapping given by the formula
\[ \Gamma_{r, \lambda}^{\rho_{2},\rho_{1}}(P) = \Gamma_{P}^{\rho_{2},\rho_{1}}
(\sup {\mathcal I}(\Gamma_{P}^{\rho_{2},\rho_{1}}))  \]
where $\Gamma_{P}^{\rho_{2},\rho_{1}} = \Gamma ( \aleph_{\mathcal E} X^{\mathcal C}, P,
{\mathcal C}(\rho_{2}) \setminus  {\mathcal C}(\rho_{1}) )$.
Let us remark that given $\rho_{2} > \rho_{1}$ we have
$\Gamma_{P}^{\rho_{2}} (\sup {\mathcal I}(\Gamma_{P}^{\rho_{2}})) \subset
A_{\rho_{2}}(r,\lambda)$
for all $P \in arc_{k}(r,\lambda)$ and $(r,\lambda) \in (0,\delta) \times I$.
We deduce that
$\Gamma_{P}^{\rho_{1}} (\sup {\mathcal I}(\Gamma_{P}^{\rho_{1}})) \in
\Gamma_{r, \lambda}^{\rho_{2},\rho_{1}}(A_{\rho_{2}}(r,\lambda))$
for all $P \in arc_{k}(r,\lambda)$ and $(r,\lambda) \in (0,\delta) \times I$.
\begin{figure}[h]
\begin{center}
\includegraphics[height=6cm,width=4cm]{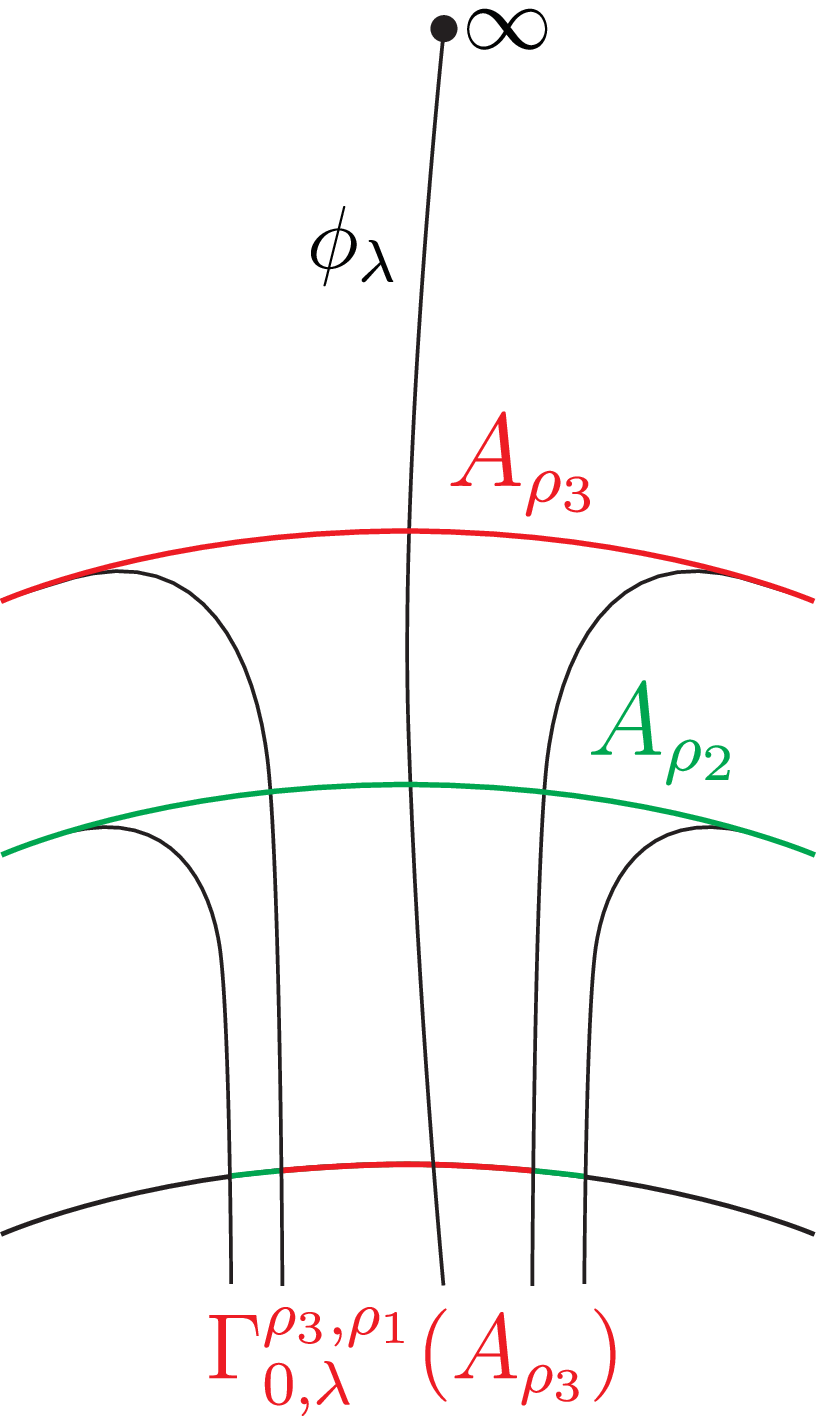}
\end{center}
\label{EVf1}
\end{figure}
We have
\[ \cap_{\rho_{2} > \rho_{1}} \Gamma_{0, \lambda}^{\rho_{2},\rho_{1}}(A_{\rho_{2}}(0,\lambda))
= \{ \phi_{\lambda}(s_{\rho_{1},\lambda}) \} \ \ \forall \lambda \in I. \]
We can consider $\rho_{1}=2 \rho$. We define
\[ \partial \tau(0,\lambda)=(0, \tilde{\Gamma}_{\lambda}(\sup ({\mathcal I}(\tilde{\Gamma}_{\lambda}))) \]
in $(x,w)$ coordinates for $\lambda \in I$ where
$\tilde{\Gamma}_{\lambda}=
\Gamma (\lambda^{e({\mathcal C})} \aleph^{*} X_{\mathcal C}, (0, \phi_{\lambda}(s_{2 \rho,\lambda})), {\mathcal E})$.
We obtain in this way a continuous extension of $\partial \tau$ to $((0,\delta)I)_{\pi}$.

Consider a Fatou coordinate $\psi^{\mathcal C}$
of $X^{\mathcal C}(\lambda_{0})$ defined in a neighborhood of the set
$\tilde{\Gamma}_{\lambda_{0}} \cup \phi_{\lambda}(0,s_{2 \rho, \lambda_{0}}] \cup \{\infty\}$.
Denote
$w_{r,\lambda,\rho_{1}}=
\Gamma_{\tau(r,\lambda)}^{\rho_{1}} (\sup {\mathcal I}(\Gamma_{\tau(r,\lambda)}^{\rho_{1}}))$.
We have
\[ \psi(\partial \tau(r,\lambda)) - \psi(\tau(r,\lambda)) = [\psi(\partial \tau(r,\lambda)) -
\psi (w_{r,\lambda,\rho_{1}})] + [\psi (w_{r,\lambda,\rho_{1}}) - \psi(\tau(r,\lambda))] \]
and
\[ lim_{(r,\lambda) \to (0,\lambda_{0})} |r|^{\iota({\mathcal E})}
(\psi(\partial \tau(r,\lambda))-\psi(w_{r,\lambda,\rho_{1}}))  = \psi^{\mathcal C}(\partial \tau(0,\lambda)) -
\psi^{\mathcal C}(\phi_{\lambda}(s_{\rho_{1},\lambda})). \]
We have $t \sim t-\gamma_{\mathcal E}(x)$ in ${\mathcal E}$.
By using lemma \ref{lem:itf} and remark \ref{rem:psiext}  we obtain
\[ |\psi_{\mathcal E} (w_{r,\lambda,\rho_{1}}) - \psi_{\mathcal E} (\tau(r,\lambda))| \leq
\frac{C'}{\rho_{1}^{\nu({\mathcal E})} r^{\nu({\mathcal E})}} + C'' \]
for some constants $C',C'' \in {\mathbb R}^{+}$. We can consider $\psi=\psi_{\mathcal E}/x^{e({\mathcal E})}$
to get
\[ |r|^{\iota({\mathcal E})} |\psi  (w_{r,\lambda,\rho_{1}}) - \psi  (\tau(r,\lambda))|
\leq \frac{C'}{\rho_{1}^{\nu({\mathcal E})}  } + C'' r^{\nu({\mathcal E})} \]
whose limit when $(\rho_{1},r) \to (\infty,0)$ is equal 0. Therefore we obtain
\[ \lim_{(r,\lambda) \to (0,\lambda_{0})} |r|^{\iota({\mathcal B})}
(\psi(\partial \tau(r,\lambda))-\psi(\tau(r,\lambda)))  = \psi^{\mathcal C}(\partial \tau(0,\lambda)) -
\psi^{\mathcal C}(\infty) \]
for any $\lambda_{0} \in I$. Moreover $F(0,\lambda_{0}) \in {\mathbb H}$ for
any $\lambda_{0} \in I$ since the image of $\aleph^{*}$ is contained in $e^{i(0,\pi)}$.
\end{proof}
The previous proposition makes simpler the analysis of regions of multi-transversal flows.
The proof of the following corollary is contained in the previous one.
\begin{cor}
\label{cor:hit}
Consider the setting of the previous proposition. Suppose that ${\mathcal B}$ is an exterior set.
The set $\partial_{\downarrow} {\mathcal B}$ of points in $\partial_{e} {\mathcal B}$
where $\Re (\aleph X)$ does not point towards the exterior of ${\mathcal B}$ has
$\nu ({\mathcal B})$ connected components. Then $(\partial \tau)_{|r=0}$
only depends on the component of $\partial_{\downarrow} {\mathcal B}$ containing
$\tau ((0,\delta)I)$.
\end{cor}
\begin{cor}
\label{cor:intbou}
Let $X \in \Xt$. Consider a multi-transversal flow $\Re (\aleph X)$ for some
stable multi-direction $\aleph: I \to (e^{i(0,\pi)})^{\tilde{q}}$.
Denote $\Gamma_{x}^{j}=\Gamma (\aleph X, T_{\aleph X}^{\epsilon,j}(x), T_{0})$ for $x \in (0,\delta) I$
and $j \in {\mathbb Z}/(2 \nu({\mathcal E}_{0}) {\mathbb Z})$ (see def. \ref{def:lj}).
Given a basic set ${\mathcal B}$ associated to $X$ we have that either
$\Gamma_{x}^{j}[0,\infty) \cap \partial_{e} {\mathcal B}$ is empty for any $x \in (0,\delta) I$ or
$\Gamma_{x}^{j}[0,\infty) \cap \partial_{e} {\mathcal B}$ contains a unique point for any $x \in (0,\delta) I$.
Suppose that we are in the latter case. Then the mapping
$x \to \Gamma_{x}^{j}[0,\infty) \cap \partial_{e} {\mathcal B}$
is asymptotically continuous in $(0,\delta) I$.
Moreover the function $F: (0,\delta) \times I \to {\mathbb C}$ defined by
$(|r|^{e({\mathcal B})} \psi_{L_{j}}^{X})(\Gamma_{r,\lambda}^{j}[0,\infty) \cap \partial_{e} {\mathcal B})$
admits a continuous extension to $[0,\delta) \times I$ such that $F(0,\lambda) \in {\mathbb H}$
for any $\lambda \in I$.
\end{cor}
\begin{pro}
\label{pro:epsiext}
Let $X \in \Xt$. Consider a multi-transversal flow $\Re (\aleph X)$ for some
stable multi-direction $\aleph: I \to (e^{i(0,\pi)})^{\tilde{q}}$.
Let ${\mathcal E} = \{ \eta \geq |t| \geq \rho|x| \}$
be a parabolic exterior set associated to $X \in \Xt$. Consider
$H \in Reg^{*}(\epsilon, \aleph X, I)$. Then there exists $C_{\mathcal E} \in {\mathbb R}^{+}$ such that
\[ \frac{1}{C_{\mathcal E}}  \frac{1}{|t-\gamma_{\mathcal E}(x)|^{\nu({\mathcal E})}} \leq
|\psi_{\mathcal E}|(x,t) \leq  C_{\mathcal E}  \frac{1}{|t-\gamma_{\mathcal E}(x)|^{\nu({\mathcal E})}} \]
in every sub-region of $H$ contained in ${\mathcal E}$.
\end{pro}
\begin{proof}
We define $\Gamma_{0}=\partial{H} \setminus Sing X$. The set $\Gamma_{0}(x)$ is the union of
the trajectories of $\Re(\aleph X)$ bounding $H(x)$ for $x \in [0,\delta)I$.
The corollary \ref{cor:intbou} implies that $\Gamma_{0} \cap \partial_{e} {\mathcal E}$
is the union of a finite number of asymptotically continuous sections.
Therefore there exists $\zeta \in {\mathbb R}^{+}$
such that
\[ \overline{H} \cap \partial_{e} {\mathcal E} \subset
\{t - \gamma_{\mathcal E}(x) \in {\mathbb R}^{+} e^{i[-\zeta, \zeta]}\} \]
by proposition \ref{pro:estext}.
The result is a consequence of lemma \ref{lem:itf} and remark \ref{rem:psiext}.
\end{proof}
\begin{pro}
\label{pro:ext1g}
Let $X \in \Xt$. Consider a multi-transversal flow $\Re (\aleph X)$ for some
stable multi-direction $\aleph: I \to (e^{i(0,\pi)})^{\tilde{q}}$. Consider
$H \in Reg^{*}(\epsilon, \aleph X, I)$ and $L \in {\mathcal P}(H)$.
Then there exists $C_{1} \in {\mathbb R}^{+}$ such that
\begin{equation}
\label{equ:unieo}
 \frac{1}{C_{1}}  \frac{1}{|y|^{\nu({\mathcal E}_{0})}} \leq
|\psi_{L}^{X}|(x,y) \leq  C_{1} \frac{1}{|y|^{\nu({\mathcal E}_{0})}}
\end{equation}
for any $(x,y)$ that belongs to the L-subregion contained in ${\mathcal E}_{0}$.
The constant $C_{1}$ depends on $X$
but it does not depend on $I$, $\aleph$, $H$ or $L$.
\end{pro}
\begin{proof}
We have $y \sim (y-\gamma_{{\mathcal E}_{0}}(x))$ in ${\mathcal E}_{0}$.
The proposition \ref{pro:estext} implies that the subregions of $H$ contained in
$H \cap {\mathcal E}_{0}$ are contained in
$\{y   \in {\mathbb R}^{+} e^{i[-\zeta, \zeta]}\}$ for some $\zeta \in {\mathbb R}^{+}$.
In fact $\zeta$ does not depend on $I$, $\aleph$ or $H$. Thus we can apply
lemma \ref{lem:itf} and remark \ref{rem:psiext} in order to obtain equation (\ref{equ:unieo})
for some $C_{1} \in {\mathbb R}^{+}$ that does not depend on $I$, $\aleph$, $H$ or $L$.
\end{proof}
%
%
%
\subsection{Size of regions}
\label{subsec:sizereg}
Given $\varphi \in \diff{p1}{2}$ we define Fatou coordinates in the regions of
$Reg(\epsilon, \aleph X, I)$ for convenient choices of $X$, $I$ and $\aleph$
(section \ref{sec:comphicnf}). The Fatou coordinates are the main ingredients
that we use in the theorem of analytic conjugation  \ref{teo:modi}.
In order for such an approach to analytic conjugacy to work it is required that regions
be somehow preserved by analytic conjugations. Very roughly speaking we have to
prove that regions are not too small. This subsection is devoted to make the previous
ideas more precise and to introduce the
concepts and results that will be used in section \ref{sec:app}.

Let $X \in \Xt$. Consider a subset $H$ of $T_{0}$. Given a point $P \in H(x)$
we consider $\Gamma_{P} = \Gamma(X, P,H)$. We define
\[ width_{H} (P) = length ( {\mathcal I} (\Gamma_{P} )) \]
where $length ( {\mathcal I} (\Gamma_{P} ))$ is by definition the length of the interval of
definition of ${\mathcal I} (\Gamma_{P} )$. It can be eventually equal to $\infty$. We define
\[ width_{H}^{\min} (x) = \min_{P \in  H(x)} width_{H} (P), \ \
width_{H}^{\max} (x) = \max_{P \in H(x)} width_{H} (P) . \]
\begin{defi}
\label{def:eh}
Let $H \in Reg_{2}(\epsilon, \aleph X, I)$. There exists a seed $T_{\beta}$ containing
${\alpha}^{\aleph X} (H)$ and ${\omega}^{\aleph X} (H)$ but such that
${\alpha}^{\aleph X} (H)$ and ${\omega}^{\aleph X} (H)$ are contained in different sons
of $T_{\beta}$. We denote ${\mathcal C}_{H}={\mathcal C}_{\beta}$.
We define $e(H)=e({\mathcal C}_{H})$.
\end{defi}
\begin{pro}
\label{pro:widreg}
Let $X \in \Xt$. Consider a multi-transversal flow $\Re (\aleph X)$ for some
stable multi-direction $\aleph: I \to (e^{i(0,\pi)})^{\tilde{q}}$. Suppose
$H \in Reg_{1}(\epsilon, \aleph X, I)$, then
\[ width_{H}^{\min}(x) = width_{H}^{\max}(x) = \infty \ \ \forall x \in [0,\delta)I . \]
Suppose $H \in Reg_{2}(\epsilon, \aleph X, I)$, then there exists $J \in {\mathbb R}^{+}$ such that
\[ width_{H}^{min}(x) \geq \frac{J}{|x|^{e(H)}} \ \  \forall x \in (0,\delta)I . \]
\end{pro}
\begin{proof}
The result is obvious if $H \in Reg_{1}(\epsilon, \aleph X, I)$.
Suppose $H \in Reg_{2}(\epsilon, \aleph X, I)$. Denote ${\mathcal C} = {\mathcal C}_{H}$.
We have ${\mathcal C} = \{ (x,w) \in
B(0,\delta) \times (\overline{B}(0,\rho)
\setminus \cup_{\zeta \in S_{\mathcal C}} B(\zeta, \eta_{{\mathcal C}, \zeta}) )\}$.
Consider the magnifying glass $M =\{ (x,w) \in B(0,\delta) \times \overline{B}(0,\rho)  \}$.
We define
\[ H_{\mathcal C}(x) = \{ P \in H(x) : \Gamma( \aleph X, P, T_{0}) \subset M \} \]
for any $x \in (0,\delta)I$.  We define
\[ width_{\mathcal C}^{*}(P) = length( {\mathcal I} (\Gamma
(\lambda^{e(H)} X_{\mathcal C}, P, H_{\mathcal C}))) \]
for $P \in H_{\mathcal C}(r,\lambda)$ and $(r,\lambda) \in [0,\delta) \times I$.
It suffices to prove that $width_{\mathcal C}^{*} > J_{0}$ in $H_{\mathcal C}$ for
some $J_{0} \in {\mathbb R}^{+}$. The rest of the proof is devoted to show this result.

We have
$H_{\mathcal C}(r, \lambda) \cap \partial_{e} {\mathcal C}_{H} =
\{ T{\mathcal C}_{\aleph X}^{\rho, 1}(r, \lambda), T{\mathcal C}_{\aleph X}^{\rho, 2}(r, \lambda) \}$
for any  $(r,\lambda) \in (0,\delta) \times I$ where
$T{\mathcal C}_{\aleph X}^{\rho, 1}(r, \lambda)$ and
$T{\mathcal C}_{\aleph X}^{\rho, 2}(r, \lambda)$
are continuous sections of $T{\mathcal C}_{\aleph X}^{\rho}(r, \lambda)$ defined for
$(r,\lambda) \in [0,\delta) \times I$. Up to exchange
$T{\mathcal C}_{\aleph X}^{\rho, 1}$ and $T{\mathcal C}_{\aleph X}^{\rho, 2}$
if necessary we suppose that $\Re (X)$ points towards
$H_{\mathcal C}$ at $T{\mathcal C}_{\aleph X}^{\rho, 1}(x)$ for any
$x \in (0,\delta)I$. We claim that $H_{\mathcal C}(r,\lambda)$ extends
continuously to $[0,\delta) \times I$, i.e. there exists
$H_{\mathcal C}(0,\lambda_{0}) \subset M(0,\lambda_{0}) \setminus Sing X$ such that
\[ \lim_{(r,\lambda) \to (0,\lambda_{0})} (H_{\mathcal C} \cup Sing X_{\mathcal C})(r,\lambda)
= (H_{\mathcal C} \cup Sing X_{\mathcal C})(0,\lambda_{0}) \]
for any $\lambda_{0} \in I$. The limit is considered
in the Hausdorff topology for compact sets in
$\{ (r,\lambda,w) \in [0,\infty) \times {\mathbb S}^{1} \times {\mathbb C} \}$.
The claim is a consequence of the analogous property for
$T{\mathcal C}_{\aleph X}^{\rho, 1}(r, \lambda)$ and $T{\mathcal C}_{\aleph X}^{\rho, 2}(r, \lambda)$.
Consider a Fatou coordinate $\psi_{\mathcal C}$ of $\lambda^{e(H)} X_{\mathcal C}$ in
$H_{\mathcal C}$. There exists a continuous function $c : [0,\delta) \times I \to {\mathbb R}^{+}$ such that
we have the equality of sets
\[ \psi_{\mathcal C} (T{\mathcal C}_{\aleph X}^{\rho, 2}(r, \lambda)) + \aleph_{\mathcal C}(\lambda)
{\mathbb R} = c(r,\lambda) + \psi_{\mathcal C} (T{\mathcal C}_{\aleph X}^{\rho, 1}(r, \lambda)) +
\aleph_{\mathcal C}(\lambda){\mathbb R} \]
for any $(r,\lambda) \in [0,\delta) \times I$. Thus there exists $c_{0} \in {\mathbb R}^{+}$
such that $c_{0} \leq c(r,\lambda)$ for any $(r,\lambda) \in [0,\delta) \times I$. We denote
\[ \Gamma_{r, \lambda}^{j} = \Gamma (\lambda^{e(H)}
(\aleph X)_{\mathcal C},T{\mathcal C}_{\aleph X}^{\rho, j}(r, \lambda), M) \]
for $j \in \{1,2\}$ and $(r,\lambda) \in [0,\delta) \times I$.
It suffices to prove that $width_{\mathcal C}^{*}(P) > J_{0}$ for some constant
$J_{0} \in {\mathbb R}^{+}$ and any $P \in \cup_{x \in (0,\delta)I} \Gamma_{x}^{1}(-\infty,\infty)$.
\begin{figure}[h]
\begin{center}
\includegraphics[height=7cm,width=7cm]{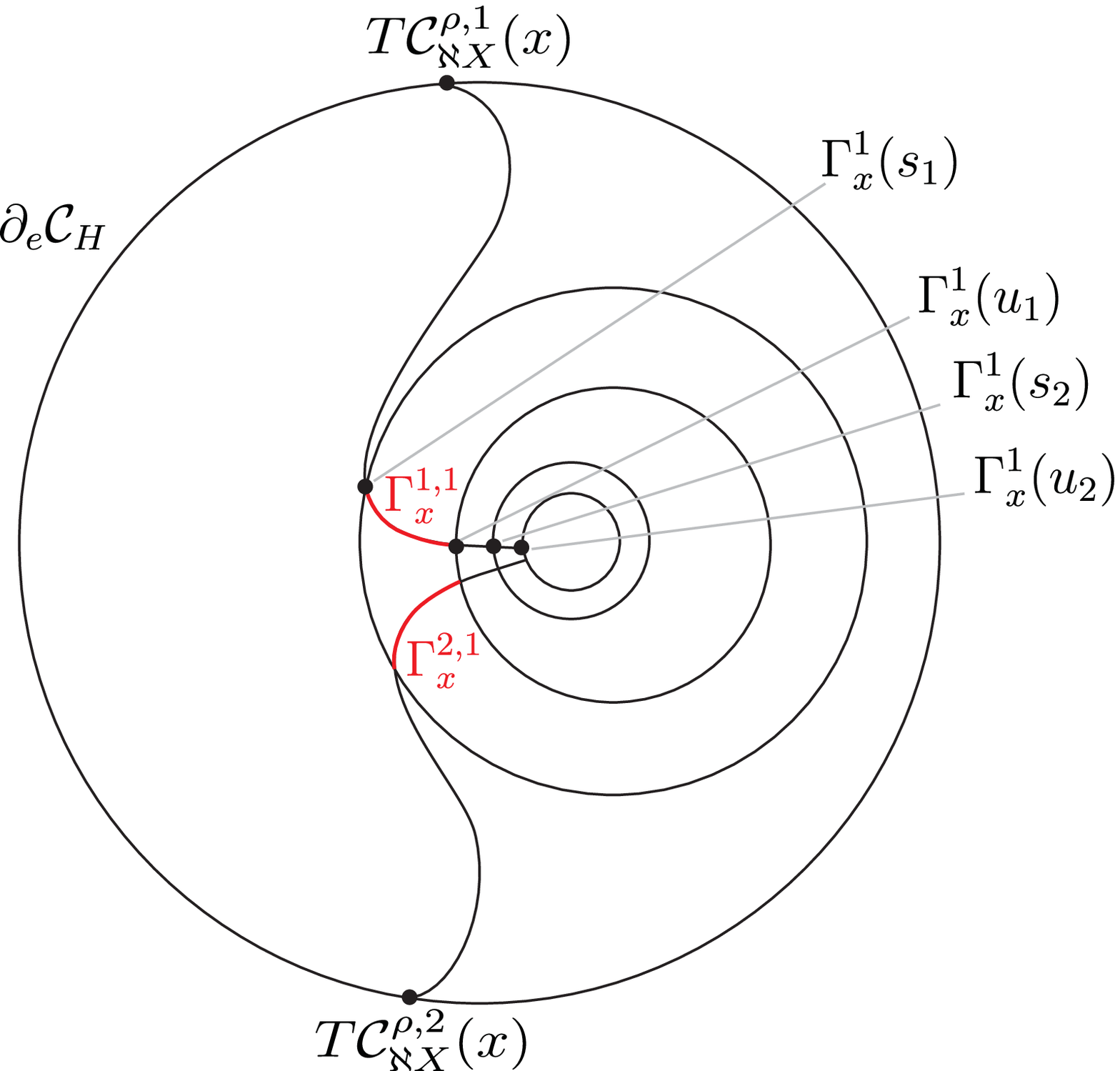}
\end{center}
\label{EVf2}
\end{figure}
Let ${\mathcal E}_{1}$, $\hdots$, ${\mathcal E}_{a}$ be the sequence of
non-terminal exterior sets intersected by $\Gamma_{x}^{j}(0,\infty)$.
We have
\[ {\mathcal E}_{k} =  \{(x,t_{k}) \in B(0,\delta) \times {\mathbb C} :  \eta_{k} \geq |t_{k}| \geq \rho_{k} |x| \} \]
in adapted coordinates $(x,t_{k})$ for any $1 \leq k \leq a$.
Consider coordinates $(x,w_{k}) \in {\mathbb C}^{2}$ such that $t_{k} = x w_{k}$.
Denote
\[ {\mathcal E}_{k}^{*} = \{(x,t_{k}) \in B(0,\delta) \times {\mathbb C} :  \rho_{k} \leq |w_{k}| \leq 2 \rho_{k}  \}
\ {\rm and} \ \partial_{e} {\mathcal E}_{k}^{*} =  {\mathcal E}_{k} \cap \{ |w_{k}| = 2 \rho_{k}  \} . \]
Given $x \in (0,\delta)I$ we denote $s_{k}(x)$ and $u_{k}(x)$ the positive real numbers
such that $\Gamma_{x}^{1}(s_{k}(x)) \in \partial_{e} {\mathcal E}_{k}^{*}$  and
$\Gamma_{x}^{1}(u_{k}(x)) \in \partial_{I} {\mathcal E}_{k}$
for $1 \leq k \leq a$. By reordering the sequence ${\mathcal E}_{1}$, $\hdots$, ${\mathcal E}_{a}$
we suppose $s_{1} < u_{1} < \hdots < s_{a} < u_{a}$.
We denote $\Gamma_{x}^{1,k}$ the compact set $(w_{k} \circ \Gamma_{x}^{1}) [s_{k}(x), u_{k}(x)]$.
Analogously as corollary \ref{cor:intbou} the mappings
$w_{k} \circ \Gamma_{x}^{1}(s_{k}(x))$ and   $w_{k} \circ \Gamma_{x}^{1}(u_{k}(x))$
extend continuously to $(r,\lambda) \in [0,\delta) \times I$.
Therefore the function $x \to \Gamma_{x}^{1,k}$
defined in $(0,\delta) \times I$ can be extended continuously to $[0,\delta) \times I$.
In an analogous way we obtain that
$\Gamma_{x}^{2,k} = \Gamma_{x}^{2}(0,\infty) \cap {\mathcal E}_{k}^{*}$
extends continuously to $(r,\lambda) \in [0,\delta) \times I$. Moreover since
$\Gamma_{r, \lambda}^{1}(0,\infty) \cap \partial_{e} {\mathcal E}_{1}$ and
$\Gamma_{r, \lambda}^{2}(0,\infty) \cap \partial_{e} {\mathcal E}_{1}$
intersect $\partial_{e} {\mathcal E}_{1}$ in the same component of
$\partial_{e} {\mathcal E}_{1} \setminus T ({\mathcal E}_{1})_{\aleph X}^{\eta_{1}}(r,\lambda)$
then prop. \ref{pro:hit} and cor. \ref{cor:hit} imply that
$\Gamma_{0, \lambda}^{1,1} \equiv \Gamma_{0,\lambda}^{2,1}$ for any $\lambda \in I$.
Then we obtain $\Gamma_{0, \lambda}^{1,k} \equiv \Gamma_{0,\lambda}^{2,k}$ for
all $1 \leq k \leq a$ and $\lambda \in I$. The negative trajectory $\Gamma_{x}^{1}(-\infty,0)$
can be analyzed analogously.

Let us study the behavior of $width_{\mathcal C}^{*}$ in
\[ \Gamma_{x}^{1}[0,s_{1}(x)], \  \Gamma_{x}^{1}[s_{1}(x), u_{1}(x)], \ \hdots, \
 \Gamma_{x}^{1}[s_{a}(x), u_{a}(x)], \  \Gamma_{x}^{1}[u_{a}(x), \infty) . \]
Denote $u_{0} \equiv 0$ and $s_{a+1} \equiv \infty$.
Let ${\mathcal C}_{k}$ be the compact-like set associated to the same seed as
${\mathcal E}_{k}$.
We claim that given $x \in (0,\delta) I$ the function $width_{\mathcal C}^{*}$ is constant in
$\Gamma_{x}^{1}[u_{k}(x), s_{k+1}(x)]$ for any $0 \leq k \leq a$.
Suppose $k \geq 1$. Then $\Gamma_{x}^{1}(u_{k})$ and
$\Gamma_{x}^{2}({\mathbb R}^{+}) \cap \partial_{I} {\mathcal E}_{k}$
(respectively $\Gamma_{x}^{1}(s_{k})$ and
$\Gamma_{x}^{2}({\mathbb R}^{+}) \cap \partial_{e} {\mathcal E}_{k}^{*}$)
are asymptotically continuous sections whose extensions to $r=0$ coincide.
Thus we have $\aleph^{*} \equiv \aleph_{{\mathcal C}_{k}}(x)$ in
$\Gamma (\lambda^{e(H)} X_{\mathcal C}, P, H_{\mathcal C})$.
for any $P \in \Gamma_{x}^{1}[u_{k}(x), s_{k+1}(x)]$ and the claim follows.
Suppose $k=0$. Let $s_{-1}(x)$ be the negative real number such that
$\Gamma_{x}^{1}(s_{-1}(x)) \in \partial_{e} {\mathcal E}_{1}^{*}$.
By arguing as in the previous case we obtain that $width_{\mathcal C}^{*}$ is
constant in $\Gamma_{x}^{1}[s_{-1}(x), s_{1}(x)]$.
Moreover
$width_{\mathcal C}^{*} (\Gamma_{x}^{1}[s_{-1}(x),s_{1}(x)]) \subset [c_{0}, \infty)$ for any
$x \in (0,\delta) I$. It suffices to prove that $width_{\mathcal C}^{*}$ decreases at most by
a multiplicative constant in $\Gamma_{x}^{1}[s_{k}(x), u_{k}(x)]$ for all $x \in (0,\delta) I$
and $1 \leq k \leq a$.

The sets $\Gamma_{r,\lambda}^{1,k}[s_{k}(r,\lambda), u_{k}(r,\lambda)]$ and
$\Gamma_{0,\lambda}^{1,k}$ for $(r,\lambda) \in (0,\delta) \times I$ are contained
in trajectories of
$\Re (\lambda^{\iota({\mathcal E}_{k})} \aleph^{*} X_{{\mathcal C}_{k}})$
that we denote $\tilde{\Gamma}_{r,\lambda}^{1,k}[0,v(r,\lambda)]$.
The vector field $Re (\lambda^{\iota({\mathcal E}_{k})} X_{{\mathcal C}_{k}})$ is transversal
to $\Re (\lambda^{\iota({\mathcal E}_{k})} \aleph^{*} X_{{\mathcal C}_{k}})$
in $\cup_{(r,\lambda) \in [0,\delta) \times I} \tilde{\Gamma}_{r,\lambda}^{1,k}[0,v(r,\lambda)]$.
We can define an holonomy mapping in the neighborhood of
$\cup_{(r,\lambda) \in [0,\delta) \times I} \tilde{\Gamma}_{r,\lambda}^{1,k}[0,v(r,\lambda)]$.
More precisely consider $(r,\lambda) \in [0,\delta) \times I$ and points $h_{1},h_{2} \in [0,v(r,\lambda)]$.
Let us define $hol_{r, \lambda, h_{1}, h_{2}}(z)$ for $z \in {\mathbb R}$ in a neighborhood of $0$ such that
\[ \Gamma(\lambda^{\iota({\mathcal E}_{k})} X_{{\mathcal C}_{k}},
\tilde{\Gamma}_{r,\lambda}^{1,k}(h_{2}), T_{0})(hol_{r, \lambda, h_{1}, h_{2}}(z)) \]
is the point of intersection of $\Gamma(\lambda^{\iota({\mathcal E}_{k})} X_{{\mathcal C}_{k}},
\tilde{\Gamma}_{r,\lambda}^{1,k}(h_{2}), T_{0})$ and the trajectory of the vector field
$\Re (\lambda^{\iota({\mathcal E}_{k})} \aleph^{*} X_{{\mathcal C}_{k}})$
passing through $\Gamma(\lambda^{\iota({\mathcal E}_{k})} X_{{\mathcal C}_{k}},
\tilde{\Gamma}_{r,\lambda}^{1,k}(h_{1}), T_{0})(z)$.
The function $hol_{r, \lambda, h_{1}, h_{2}}(z)$ is continuous in $(r,\lambda,h_{1},h_{2},z)$
and real analytic in $z$. Since
$\cup_{(r,\lambda) \in [0,\delta) \times I} (r,\lambda) \times [0,v(r,\lambda)]^{2}$ is compact
then there exists $J_{k} \in {\mathbb R}^{+}$ such that
\begin{equation}
\label{equ:reghol}
 |hol_{r, \lambda, h_{1}, h_{2}}(z)| \geq \frac{|z|}{J_{k}}
\end{equation}
for all $h_{1},h_{2} \in [0,v(r,\lambda)]$ and $z$ in a neighborhood of $0$ and any
$(r,\lambda) \in [0,\delta) \times I$. We denote
$J = c_{0} / \prod_{k=1}^{a} J_{k}$. Consider $1 \leq k \leq a$ and $\lambda_{0} \in I$.
The compact set $H_{\mathcal C}(r,\lambda) \cap {\mathcal E}_{k}^{*}$ tends to the curve
$\Gamma_{0,\lambda_{0}}^{1,k} = \Gamma_{0,\lambda_{0}}^{2,k}$ when $(r,\lambda) \to (0,\lambda_{0})$.
Thus the property (\ref{equ:reghol}) on the behavior of the holonomy in a neighborhood of
$\cup_{\lambda \in I} \Gamma_{0, \lambda}^{1,\lambda}$ implies
$width_{\mathcal C}^{*} (\Gamma_{x}^{1}[0,\infty)) \geq J$ for any $x \in (0,\delta)I$.
We do the same analysis with $\Gamma_{x}^{1}(-\infty,0]$ and consider a smaller
$J>0$ if necessary to obtain $width_{\mathcal C}^{*} (\Gamma_{x}^{1}(-\infty,\infty)) \geq J$
for any $x \in (0,\delta)I$. This implies
$width_{H}^{\min}(x) \geq J/|x|^{e(H)}$ for any $x \in (0,\delta)I$.
\end{proof}
Let $\varphi, \eta \in \diff{p1}{2}$ with the same fixed points set.
Suppose that $\varphi_{|x=x_{0}}$ is analytically conjugated to
$\eta_{|x=x_{0}}$ by an injective mapping $\kappa_{x_{0}}$ whose fixed points
set contains the fixed points set of  $\varphi_{|x=x_{0}}$ for any $x_{0}$
in a pointed  neighborhood of $0$.  If $\kappa_{x_{0}}$ is defined in some
$B(0,\epsilon)$ for any $x_{0} \neq 0$ and some $\epsilon>0$ independent of
$x_{0}$ then  $\varphi$ is analytically conjugated to $\eta$
(main theorem in \cite{JR:mod}). The family $\kappa_{x_{0}}$ is not required to
depend analytically or even continually on $x_{0}$.
It turns out that if we drop the hypothesis on
the common domain of definition $B(0,\epsilon)$ then the result is no longer true.
There are counterexamples where $\varphi$ and $\eta$ are not analytically conjugated
and $\kappa_{x_{0}}$ is defined in
$B(0,C_{0}/\sqrt[\nu({\mathcal E}_{0})]{|\ln x_{0}|})$ for some
$C_{0} \in {\mathbb R}^{+}$ and any $x_{0} \neq 0$.
One of the goals of this paper is proving that such a counterexample is optimal:
$\varphi$ is analytically conjugated to $\eta$ if the domain of definition is
any ``bigger" than $\{ |y| < \sqrt[\nu({\mathcal E}_{0})]{|\ln x|}) \}$, i.e.
if it is of the form $\{ |y| < s (x) \}$ where $s$ is a
$\nu({\mathcal E}_{0})$ slow decaying function (see theorem \ref{teo:modi}
for a precise statement).
\begin{defi}
\label{def:slow}
Let us consider a bounded function $s(x) : B(0,\delta) \setminus \{0\} \to {\mathbb R}^{+}$. We say
that $s$ is a {\it slow decaying} function if $\lim_{x \to 0} s^{-1}(x) |x|^{\tau}=0$ for any
$\tau \in {\mathbb R}^{+}$. Let $\nu \in {\mathbb N}$. We say that $s$ is a $\nu$ slow decaying function if
$\lim_{x \to 0} s(x) \sqrt[\nu]{|\ln|x||} = \infty$.
\end{defi}
\begin{rem}
A $\nu$ slow decaying function is a slow decaying function. The constant fuctions
are examples of $\nu$ slow decaying functions for any $\nu \in {\mathbb N}$.
Moreover both concepts of decay are preserved if we replace $s(x)$ with
$s(x) \tau$ for $\tau \in {\mathbb R}^{+}$.
\end{rem}
\begin{defi}
Let $X \in \Xt$. Consider a multi-transversal flow $\Re (\aleph X)$ for some
stable multi-direction $\aleph: I \to (e^{i(0,\pi)})^{\tilde{q}}$.
Let $s$ be a slow decaying function. We define $Reg(s, \aleph X, I)$ the set of connected
components of
\[ \{ (x,y) \in (0,\delta) I \times B(0,\epsilon) :
\Gamma(\aleph X, (x,y), T_{0}) \subset \{x\} \times B(0,s(x)) \} \setminus Sing X . \]
If $s$ satisfies $s<\epsilon$ there is a bijective correspondence between
$Reg(s, \aleph X, I)$ and $Reg(\epsilon, \aleph X, I)$ and any element of
$Reg(\epsilon, \aleph X, I)$ contains an element of $Reg(s, \aleph X, I)$.
\end{defi}
The next result is used to prove the proposition \ref{pro:conimeq}. It is
one of the ingredients of the proof of theorem \ref{teo:modi} .
\begin{pro}
\label{pro:widcon}
Let $X \in \Xt$. Consider a multi-transversal flow $\Re (\aleph X)$ for some
stable multi-direction $\aleph: I \to (e^{i(0,\pi)})^{\tilde{q}}$.
Let $s$ be a slow decaying function with $s < \epsilon$ and $\tau \in (0,1)$. Consider
$H \in Reg(\epsilon, \aleph X, I)$. Denote $H_{s \tau}$ the element of
$Reg(s \tau, \aleph X, I)$ contained in $H$. There exists $J_{0},J_{1} \in {\mathbb R}^{+}$
such that ${\rm exp}(tX)$ is well-defined in $H_{s \tau}(x)$ and
${\rm exp}(tX)(H_{s \tau}(x)) \subset H(x)$ for all $x \in (0,\delta)I $
and $t \in B(0,J_{0}/(s^{\nu({\mathcal E}_{0})}(x) \tau^{\nu({\mathcal E}_{0})}) - J_{1})$.
In particular we obtain
\[ width_{H \setminus H_{s \tau}}^{\min}(x) \geq
\frac{J_{0}}{s^{\nu({\mathcal E}_{0})}(x) \tau^{\nu({\mathcal E}_{0})}} - J_{1} \ \
\forall x \in (0,\delta)I . \]
\end{pro}
\begin{proof}
Denote $\nu=\nu({\mathcal E}_{0})$.
We have $H \in Reg_{j}(\epsilon, \aleph X, I)$.
The set $H \setminus H_{s \tau}$ has $j$ connected components. Consider
$L \in {\mathcal P}(H)$.
We denote $L_{s \tau}$ the element of ${\mathcal P}(H_{s \tau})$ such that
$L_{s \tau} \subset L$.
There exists a unique connected component $G_{s \tau}$ of
$\overline{H} \setminus (H_{s \tau} \cup Sing X)$
such that $L_{i X}^{\epsilon}(x) \in G_{s \tau}$ for any $x \in [0,\delta)I$.
It suffices to prove
$width_{G_{s \tau}}^{\min}(x) \geq  J_{0}/(s(x) \tau)^{\nu} - J_{1}$ for any $x \in (0,\delta)I$.

We can suppose that $\Re (X)$ points towards the interior of $T_{0}$ at $L_{i X}^{\epsilon}(x)$
for any $x \in [0,\delta)I$ by replacing $X$ with $-X$ if necessary. Denote
\[ \Gamma_{x}^{L} =
\Gamma (X, L_{i X}^{\epsilon}(x), \overline{B}(0,\epsilon) \setminus B(0,s(x) \tau)). \]
Denote $\psi=\psi_{L}^{X}$.
The interval ${\mathcal I}(\Gamma_{x}^{L})$ is of the form $[0,\upsilon (x)]$ for
$x \in (0,\delta)I$. Let us remark that $\aleph X \equiv i X$ in ${\mathcal E}_{0}$. We have
$C_{1}^{-1}/|y|^{\nu} \leq |\psi(x,y)| \leq C_{1}/|y|^{\nu}$
in ${\mathcal E}_{0}$
for some positive constant $C_{1} \geq 1$ (prop. \ref{pro:ext1g}).
Denote $D(x,\tau)=C_{1}^{-1}/(s(x) \tau)^{\nu} - C_{1}/\epsilon^{\nu}$. We deduce
\[ \upsilon (x)=|\psi(\Gamma_{x}^{L}( \upsilon (x)) - \psi(\Gamma_{x}^{L}(0))| \geq
D(x,\tau) \ \ \forall x \in (0,\delta)I. \]
Denote ${\mathcal C}={\mathcal C}_{0}$,
\[ \Gamma_{x}^{\epsilon} = \Gamma (\aleph X, L_{i X}^{\epsilon}(x), T_{0}), \ \
 \Gamma_{x}^{s \tau} = \Gamma (\aleph X, \Gamma_{x}^{L}( \upsilon (x)), T_{0}). \]
Consider the notations in the proof of prop. \ref{pro:hit}.
Given $\rho_{1} \geq 2 \rho_{0}$ the points
$\Gamma_{x}^{\epsilon}(0,\infty) \cap (\{x\} \times \partial B(0,\rho_{1}|x|))$ and
$\Gamma_{x}^{s \tau}(0,\infty) \cap (\{x\} \times \partial B(0,\rho_{1} |x|))$
belong to the same connected component $A_{\rho_{1}}$ of
$\partial_{e} {\mathcal C}(\rho_{1}) \setminus T {\mathcal C}_{\aleph X}^{\rho_{1}}$
for any $x \in (0,\delta) I$ in a neighborhood of $0$.
Analogously as in proposition \ref{pro:hit} the sections
\[ x \to \Gamma_{x}^{\epsilon}(0,\infty) \cap (\{x\} \times \partial B(0,\rho_{1}|x|)) \ \ {\rm and} \ \
x \to \Gamma_{x}^{s \tau}(0,\infty) \cap (\{x\} \times \partial B(0,\rho_{1} |x|)) \]
are asymptotically continuous and their value at $(r,\lambda)=(0,\lambda_{0})$
coincides and is equal to the element of the set
$\cap_{\rho_{2} > \rho_{1}} \Gamma_{0, \lambda_{0}}^{\rho_{2},\rho_{1}}(A_{\rho_{2}}(0,\lambda_{0}))$
for any $\lambda_{0} \in I$. An analogous result holds true for
the negative trajectories $\Gamma_{x}^{\epsilon}(-\infty,0)$ and
$\Gamma_{x}^{s \tau}(-\infty,0)$. We have
\[ width_{G_{s \tau}}(L_{i X}^{\epsilon}(x))   \geq \upsilon (x) \geq
D(x,\tau) \ \ \forall x \in [0,\delta) I . \]
Holonomy arguments, analogous to those in
the proof of prop. \ref{pro:widreg}, imply that there exists $J \in {\mathbb R}^{+}$
such that
$width_{G_{s \tau}}^{\min} \geq J D(x,\tau)$
for any $x \in [0,\delta) I$.
Moreover there exists $\zeta \in {\mathbb R}^{+}$ such that
$\aleph^{*}(P) \in e^{i[\zeta,\pi-\zeta]}$ for any
$P \in ([0,\delta)I \times B(0,\epsilon)) \setminus Sing X$.
Thus we can deduce that ${\rm exp}(t X)(x,y)$ is well defined and belongs to
$H$ for all $(x,y) \in H_{s \tau}$ and
$t \in B(0, J \sin (\zeta) D(x,\tau))$.
\end{proof}
\section{Comparing $\varphi \in \diff{tp1}{2}$ and a convergent normal form}
\label{sec:comphicnf}
The goal of this section is constructing Fatou coordinates for
$\varphi \in \diff{tp1}{2}$  and analyzing their asymptotic properties.
The first step of such a project is choosing an element $X \in  \Xt$ such that
$\varphi$ is tangent to ${\rm exp}(X)$ at a high enough order in the neighborhood of
the fixed points.
Given a region $H \in Reg(\epsilon, \aleph X, I)$ the space of orbits of
${\rm exp}(X)_{|H(x)}$ is biholomorphic to ${\mathbb C}^{*}$ for any
$x \in (0,\delta) I$. Such a property also holds true for the
orbit space of $\varphi_{|H(x)}$. A consequence is the possibility of building
holomorphic Fatou coordinates and Lavaurs vector fields in the region $H$.
This approach was introduced in Lavaurs thesis \cite{Lavaurs}
and developped in several papers \cite{Shishi} \cite{Oudkerk} \cite{MRR}
\cite{JR:mod}.

One of the main properties proved in \cite{JR:mod} is that the difference between
Fatou coordinates of $\varphi$ and ${\rm exp}(X)$ in any given region of a transversal
flow is always bounded. We prove the analogue for multi-transversal flows.
Such a property and the study of the intersection of regions associated to
different choices of multi-transversal flows suffice to prove the exponentially small
estimates required for multi-summability (see subsection \ref{subsec:comfatcor}).
The boundness of the difference of Fatou coordinates of $\varphi$ and ${\rm exp}(X)$
is proved in subsection \ref{subsec:defatcor}.  The main difficulty is that in order to
capture the higher levels of summability the estimates have to be expressed
in adapted coordinates (prop. \ref{pro:bddconf}). The shape of the regions and their
intersections is studied in subsections \ref{subsec:floconset} and
\ref{subsec:comdifmflo}. The regions associated to the same multi-transversal flow
do not intersect and then a priori we can not compare Fatou coordinates defined in them.
But we can acomplish that goal by extending Fatou coordinates to slightly bigger domains
(subsection \ref{subsec:extfatcor}). Given the Lavaurs vector fields associated to
$\varphi$ we prove that they correspond to a multi-summable object by using a
cohomological approach a la Ramis-Sibuya \cite{RamSib:Fou}. Anyway, in order to
prove that the infinitesimal generator of $\varphi$ is the asymptotic
development of the Lavaurs vector fields  we  construct Fatou coordinates whose
asymptotic development coincides with the power expansion of the infinitesimal generator
up to an arbitrary order (subsection \ref{subsec:asyfatcor}); the proof is completed by
using the flatness arguments in  subsection \ref{subsec:comfatcor}.
\begin{rem}
The constructions in this paper can be generalized to elements $\varphi$ of
$\diff{p1}{2}$. Indeed there exists $k \in {\mathbb N}$ such that
$(x^{1/k},y) \circ \varphi \circ (x^{k},y)$ belongs to $\diff{tp1}{2}$.
The results can be translated to
$\diff{p1}{2}$ via a ramification in the parameter space.
We chose $\diff{tp1}{2}$ since the presentation is simpler.
\end{rem}
\subsection{Normal forms}
The infinitesimal generator $\log \varphi$ of $\varphi \in \diff{p1}{2}$ is of the form
$(y \circ \varphi -y) \hat{u} \partial / \partial{y}$ for some unit
$\hat{u} \in {\mathbb C}[[x,y]]$ (see prop. (3.2) of \cite{JR:mod}).
\begin{defi}
Let $\varphi= {\rm exp}((y \circ \varphi -y) \hat{u} \partial / \partial{y}) \in \diff{p1}{2}$.
We say that $\Upsilon \in \diff{p1}{2}$ is a $k$-convergent normal form of $\varphi$
if $\log \Upsilon=(y \circ \varphi -y) u \partial / \partial{y}$ for some
$u \in {\mathbb C}\{x,y\}$ and $\hat{u}-u$ belongs to the ideal $(y \circ \varphi -y)^{k}$.
The last condition is equivalent to $y \circ \varphi - y \circ \Upsilon \in (y \circ \varphi -y)^{k+1}$.
We say that $\Upsilon$ is a convergent normal form of $\varphi$ if $\Upsilon$ is a $1$-convergent
normal form.
\end{defi}
\begin{defi}
\label{def:resd}
Let $\varphi \in \diff{tp1}{2}$. Consider a convergent normal form
$\Upsilon$ of $\varphi$. We define
$Res(\varphi,P)=Res(\log \Upsilon,P)$ for $P \in Fix (\varphi)$ (see def. \ref{def:res2}).
The definition does not depend on the choice of $\Upsilon$ \cite{UPD}.
\end{defi}
\begin{pro}
\cite{UPD}
Let $\varphi \in \diff{p1}{2}$ and $k \in {\mathbb N}$. Then there exists a
$k$-convergent normal form $\Upsilon$.
\end{pro}
Let $\varphi \in \diff{p1}{2}$. Denote $f= y \circ \varphi -y$. Fix a $k$-convergent normal form
$\Upsilon$ of $\varphi$. Consider an integral
of the time form $\psi$ of $\log \Upsilon$. We define (see section 7.1 of \cite{JR:mod})
\[ \Delta_{\varphi}= \psi \circ \varphi - \psi \circ \Upsilon = \psi \circ \varphi - (\psi +1) . \]
The definition depends on $\Upsilon$ but it does not depend on $\psi$.
A priori the function $\Delta_{\varphi}$ is defined only outside of $Fix (\varphi)$.
Nevertheless, by Taylor's formula we have
\[ \Delta_{\varphi} \sim
\left({ \frac{\partial \psi}{\partial y} \circ \Upsilon }\right) (y \circ \varphi - y \circ \Upsilon) =
O \left({ \left({ \frac{1}{f} \circ \Upsilon }\right) f^{k+1} }\right) =
O(f^{k}).\]
We obtain
\begin{lem}
\label{lem:deltahol}
Let $\varphi \in \diff{p1}{2}$ with fixed $k$-convergent normal form. Then $\Delta_{\varphi}$
belongs to the ideal $(y \circ \varphi -y)^{k}$ of the ring ${\mathbb C}\{x,y\}$.
\end{lem}
Next we estimate $\Delta_{\varphi}$ in terms of the Fatou coordinates of regions.
We translate $\Delta_{\varphi}=O(f^{k})$ to an estimate in adapted coordinates
by using the asymptotical good behavior of regions.
\begin{pro}
\label{pro:bddconf}
Let $\varphi \in \diff{tp1}{2}$. Let $\Upsilon={\rm exp}(X)$ be a $k$-convergent normal form.
Consider a multi-transversal flow $\Re (\aleph X)$ for some
stable multi-direction $\aleph: I \to (e^{i(0,\pi)})^{\tilde{q}}$. Take
$H \in Reg^{*}(\epsilon, \aleph X, I)$ and $L \in {\mathcal P}(H)$. Then there exists
$K \in {\mathbb R}^{+}$ such that
\begin{equation}
\label{equ:delta}
|\Delta_{\varphi}(x,y)| \leq \frac{K}{(1+|\psi_{H,L}^{X}(x,y)|)^{k}} \ \ \forall (x,y)
\in H^{L} .
\end{equation}
The constant $K$ depends only on $\Upsilon$ but does not depend on
$I$, $\aleph$, $H$ or $L$.
\end{pro}
\begin{proof}
Denote $f=X(y)$.
Let us prove the result for a L-subregion $J$.
Lemma \ref{lem:goins} implies that
there exists a sequence $B_{0}$, $\hdots$, $B_{k}=J$ of L-subregions
of $H$ such that $\beta(0)=0$ and
\begin{itemize}
\item $B_{2j} \subset {\mathcal E}_{\beta(2j)}$ for any $0 \leq 2j \leq k$.
\item $B_{2j+1} \subset {\mathcal C}_{\beta(2j)}$ for any $0 \leq 2j+1 \leq k$.
\item $\beta(2j+2) = (\beta(2j), \kappa(j))$ for
some $\kappa(j) \in {\mathbb C}$ for any $0 \leq 2j+2 \leq k$.
\end{itemize}
We denote $E_{2j}={\mathcal E}_{\beta(2j)}$ and $E_{2j+1}={\mathcal C}_{\beta(2j)}$.
We define $\partial_{e} B_{0}= \partial_{e} {\mathcal E}_{0} \cap \overline{B_{0}}$ and
$\partial_{e} B_{j} = \overline{B_{j}} \cap \partial_{I}{E_{j-1}}$ for $j \geq 1$.
As in the proof of proposition \ref{pro:epsiext} we obtain that
$\partial H \cap \partial_{e} B_{j}$ is the union of a finite number of
asymptotically continuous sections
$\tau_{1}, \hdots, \tau_{b}: [0,\delta) \times I \to \partial H \cap \partial_{e} B_{j}$.
%
%
Consider adapted coordinates $(x,t)$ for $E_{j}$.
Corollary \ref{cor:intbou} implies that $|r|^{e(E_{j})} \psi_{H,L}^{X}(r,\lambda,t)$ is continuous in
$\overline{B_{j}} \setminus Sing X$ and satisfies
$|r|^{e(E_{j})} |\psi_{H,L}^{X}| \leq A_{j}$ in
$\partial_{e} B_{j}$ for some $A_{j}>0$. Denote $\psi_{j}=x^{e(E_{j})} \psi_{H,L}^{X}$.
%
%

Let $T=\{ (x,t) \in B(0,\delta) \times \overline{B(0,\eta)} \}$ be the seed
associated to $E_{j}$.
  The construction of the splitting implies that
$f \in ({x}^{e(E_{j})+ [(j+1)/2]})$
in $E_{j}$ for any $0 \leq j \leq k$ (let us remark that $[(j+1)/2]$
is the integer part of $(j+1)/2$).

Suppose that $E_{j}= \{ \eta \geq |t| \geq \rho |x| \}$ is a parabolic exterior set, we obtain
\begin{equation}
\label{equ:equsou}
 f = O \left({
\frac{x^{e(E_{j}) +[(j+1)/2]}}{(1+|\psi_{j}|)^{1 + 1/\nu(E_{j})}}
}\right) = O \left({
\frac{x^{e(E_{j}) +[(j+1)/2]}}{(1+|\psi_{j}|)^{1 + 1/\nu({\mathcal E}_{0})}} }\right)
\end{equation}
by prop. \ref{pro:epsiext} since $\nu(E_{j}) \leq \nu({\mathcal E}_{0})$.
We obtain
\begin{equation}
\label{equ:equsou2}
 f = O \left({
\frac{x^{e(E_{j}) +[(j+1)/2] -e(E_{j}) (1+1/\nu({\mathcal E}_{0}))}}
{(1+|\psi_{H,L}^{X}|)^{1 + 1/\nu({\mathcal E}_{0})}}
}\right) = O \left({
\frac{1}{(1+|\psi_{H,L}^{X}|)^{1 + 1/\nu({\mathcal E}_{0})}} }\right)
\end{equation}
in $B_{j}$ since $e(E_{j}) \leq [(j+1)/2] \nu({\mathcal E}_{0})$ by construction.

Suppose that $E_{j}= \{ |t| \leq \eta \}$ is a non-parabolic exterior set,
this implies $j=k$. We have
\[ X= x^{e(E_{k})} X_{E_{k}} = x^{e(E_{k})} v(x,t) (t- \gamma(x)) \partial / \partial{t} \]
where $v$ is a holomorphic function never vanishing in $E_{k}$.
The seed $E_{k}$ is the son of a seed $T_{\beta}$. We obtain
${\mathcal C}_{\beta} = {\mathcal C}_{l}$ for any $1 \leq l \leq q$. We have
$X_{{\mathcal C}_{\beta}}=(x,w-\zeta)^{*} X_{E_{k}}$
for some $\zeta \in {\mathbb C} \cap  S_{\beta}$ by construction
of the dynamical splitting. This implies
\[ v(0,0)^{-1}=Res(X_{E_{k}},(0,0))= Res(X_{{\mathcal C}_{\beta}},(0,\zeta))=Res(X_{l}(1),\zeta). \]
Since $(\lambda, \aleph_{E_{k}}(\lambda)) \not \in {\mathcal U}_{X}^{l}$ (see section \ref{sec:tramul})
for any $\lambda \in I$ we deduce that
\[ \aleph_{E_{k}}(\lambda) \lambda^{e(E_{k})} X_{l}(1)=
\aleph_{E_{k}}(\lambda) X_{l}(\lambda)  \in {\mathcal X}_{\infty}\cn{} \]
for any $\lambda \in I$. The definition of ${\mathcal X}_{\infty}\cn{}$ implies
\[ \lambda^{- e(E_{k})} \aleph_{E_{k}}(\lambda)^{-1} v(0,0)^{-1} =
Res(\aleph_{E_{k}}(\lambda) X_{l}(\lambda), \zeta) \not \in i {\mathbb R}  \]
for any $\lambda \in I$.
The function
$\lambda^{- e(E_{k})} \aleph_{E_{k}}(\lambda)^{-1} \psi_{k}(r,\lambda,t)$
is an integral of the time form of $\Re (\lambda^{e(E_{k})} (\aleph X)_{E_{k}})$. We define
$D(r,\lambda)= \lambda^{- e(E_{k})} \aleph_{E_{k}}(\lambda)^{-1} v(0,0)^{-1}$ and
\[ F(r,\lambda,t)=\lambda^{- e(E_{k})} \aleph_{E_{k}}(\lambda)^{-1}
[\psi_{k}(r,\lambda,t)  - v(r \lambda,\gamma(r \lambda))^{-1} \ln (t-\gamma(r \lambda))]. \]
The function $\partial F/\partial t$ satisfies
\[ \frac{\partial F}{\partial t}(r,\lambda,t)= \frac{1}{\lambda^{e(E_{k})} \aleph_{E_{k}}(\lambda)}
\left({ \frac{1}{v(r \lambda,t) (t- \gamma(r \lambda))} -
\frac{1}{v(r \lambda,\gamma(r \lambda)) (t- \gamma(r \lambda))} }\right) . \]
It is bounded in $B_{k}$ and then so is $F$.
There exists $\upsilon>0$ such that
\[ \arg(D(r,\lambda))
\in (-\pi/2+\upsilon,\pi/2-\upsilon) \]
for any $(r,\lambda) \in [0,\delta) \times I$ if
$B_{k}$ is a basin of repulsion, otherwise we have that
$\arg(D(r,\lambda)) \in (\pi/2+\upsilon,3 \pi/2 - \upsilon)$
for any $(r,\lambda) \in [0,\delta) \times I$. We deduce that
\[ f=O(x^{e(E_{k}) + [(k+1)/2]} (t-\gamma(x))) =
O(x^{e(E_{k}) + [(k+1)/2]} e^{-A |\psi_{k}|}) \]
in $B_{k}$ for some $A>0$. This implies equation
(\ref{equ:equsou}) and then equation (\ref{equ:equsou2}).

Finally suppose that $E_{j}$ is a compact-like set.
The asymptotically continuous character of $\partial H \cap \partial_{e} E_{j}$
implies that $E_{j}$ is compact in the coordinates $(r,\lambda,w)$ associated
to $E_{j}$. We have that
$\psi_{j}$ is bounded in $B_{j}$. Hence $f= O(x^{e(E_{j}) +[(j+1)/2]})$
implies equation (\ref{equ:equsou}) and then equation (\ref{equ:equsou2}).

We proved that there exists $K' \in {\mathbb R}^{+}$ such that
\begin{equation}
\label{equ:delaux}
|\Delta_{\varphi}(x,y)| \leq \frac{K'}{(1+|\psi_{H,L}^{X}(x,y)|)^{k+k/\nu({\mathcal E}_{0})}} \ \
\forall (x,y) \in H^{L} .
\end{equation}
Proposition \ref{pro:ext1g} implies
\[ |\Delta_{\varphi}(x,y)| \leq \frac{K}{(1+|\psi_{H,L}^{X}(x,y)|)^{k+ k/\nu({\mathcal E}_{0})}} \ \ \forall (x,y)
\in H^{L} \cap {\mathcal E}_{0} \]
where $K$ depends only on $\Upsilon$. Since we have
\[ \lim_{\delta' \to 0} \left({ \inf_{x \in [0,\delta')I, \ (x,y) \in H^{L} \setminus {\mathcal E}_{0}}
|\psi_{H,L}^{X}(x,y)| }\right)= \infty \]
then equation (\ref{equ:delta}) holds true for any $(x,y) \in H^{L}$ with $x$ close to $0$.
\end{proof}
Both in subsections \ref{subsec:extfatcor} and \ref{subsec:asyfatcor}
will be necessary to extend Fatou coordinates $\psi^{\varphi}$ of $\varphi$
by using the equation $\psi^{\varphi} \circ \varphi = \psi^{\varphi} +1$.
The next lemma assures that the asymptotic properties of Fatou coordinates
are preserved when extending $\psi^{\varphi}$ along long orbits.
\begin{defi}
\label{def:W}
Let $\theta \in (0,\pi/2)$ and $M \in {\mathbb R}^{+} \cup \{0\}$. We define the set
$W_{\theta, M} \subset {\mathbb C}$ given by
\[ W_{\theta, M} = \{ z \in {\mathbb C} : Re (z) >0 \} \cup
\{ z \in {\mathbb C} : |Im (z)| + \tan (\theta) Re (z) -M >0  \} . \]
We define $W_{\pi/2, M} = \{ z \in {\mathbb C} : Re (z) >0 \}$.
\end{defi}
\begin{figure}[h]
\begin{center}
\includegraphics[height=3cm,width=9cm]{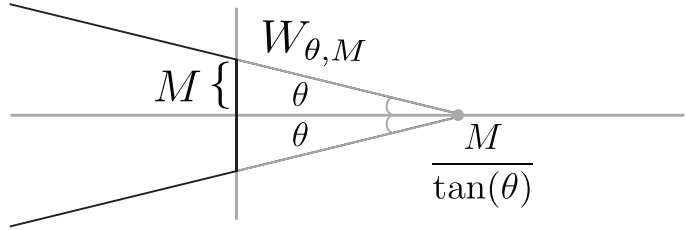}
\end{center}
\caption{Picture of $W_{\theta,M}$}
\label{EVf3}
\end{figure}
\begin{lem}
\label{lem:techsum}
Let $\varphi \in \diff{tp1}{2}$.
Let $\theta \in (0,\pi/2]$, $M \geq 0$ and $k \geq 2$.
Let $\Upsilon$ be a $k$-convergent normal form.
Consider a orbit $\vartheta = \{ P, \varphi(P), \hdots, \varphi^{j}(P) \}$ and
a function $\psi$ defined in $\vartheta$ such that
$\psi(\varphi(Q)) =  \psi + 1 + \Delta_{\varphi}(Q)$ for any $Q \in \vartheta$.
Suppose that  $\psi(P) \in W_{\theta, M}$,  $| \Delta_{\varphi}(Q)| \leq 1/2$ for
any $Q \in \vartheta$ and
$|\Delta_{\varphi}(Q)| \leq \sin (\theta)/2$ for any $Q \in \vartheta$ such that
$|Im(\psi(Q))| \geq M$. Consider a function $\Pi:\vartheta \to {\mathbb C}$ such that
\[ |\Pi(Q)| \leq  \frac{D}{(1+|\psi(Q)|)^{k}}  \ \ \forall Q \in \vartheta. \]
Then, we have
\[ \sum_{l=0}^{j} |\Pi(\varphi^{l}(P))| \leq \frac{4^{k} \sqrt{2}^{k} D k}{c^{k} (k-1)}
\frac{1}{(1+|\psi(P)|)^{k-1}} \]
where $c^{2} = (1-\cos (\theta))/2$.
\end{lem}
\begin{proof}
We claim that $\psi(\vartheta)$ is contained in $W_{\theta, M}$.
We denote $z=z_{1}+iz_{2}$.
The set $W_{\theta, M}$ is a union of the sets $W_{\pi/2, M}$,
\[ E_{1}= \{ z_{2} + (\tan (\theta)) z_{1} >M \} \ {\rm and}
 \ E_{2}= \{ -z_{2} + (\tan (\theta)) z_{1} >M \}. \]
The distance from $Re(z)=0$ to $Re(z)=1$ is $1$. Thus $|(\Delta_{\varphi})|_{\vartheta}| \leq 1/2$
implies that a point $Q \in \vartheta$ such that $\psi(Q) \in  W_{\pi/2, M}$ satisfies
$\psi(\varphi(Q)) \in  W_{\pi/2, M}$. A point
$Q \in \vartheta$ such that $\psi(Q) \in E_{1}  \setminus W_{\pi/2, M}$ satisfies
$|Im(\psi(Q))| \geq M$. Since the distance between the lines
$\partial E_{1}$ and $1+\partial E_{1}$ is $\sin (\theta)$
we obtain that $\psi(\varphi(Q)) \in E_{1}$. Analogously we prove
$\psi (\varphi(Q)) \in E_{2}$ for $Q \in \vartheta$ such that $\psi(Q) \in E_{2} \setminus W_{\pi/2, M}$.
We deduce that $\psi(\vartheta)$ is contained in $W_{\theta, M}$.

Let us deal with the case of an orbit whose image by $\psi$ does not intersect $W_{\pi/2, M}$.
We denote $\tau_{l}=\psi(\varphi^{l}(P))$ for any $l \in \{0,\hdots,j\}$.
Take $l_{1} \in \{0,\hdots,j-1\}$ such that $\{\tau_{0}, \hdots, \tau_{l_{1}}\}$ is contained
in $(E_{1} \cup E_{2}) \setminus W_{\pi/2, M}$.
Given $\tau \in W_{\theta, 0}$ and $l \in {\mathbb N} \cup \{0\}$ we have
$|\tau+l| \geq \sin (\theta) l$. We obtain
\begin{equation}
\label{equ:invteta}
|\tau_{l_{2}}| \geq |\tau_{0}+l_{2}| - \sum_{a=0}^{l_{2}-1} |\Delta_{\varphi}(\varphi^{a}(P))| \geq
|\tau_{0}+l_{2}| - \frac{l_{2}}{2} \sin (\theta) \geq \frac{|\tau_{0}+l_{2}|}{2}
\end{equation}
for any $l_{2} \in \{0, \hdots, l_{1}+1\}$.

Now we consider the case $\psi(\vartheta) \cap W_{\pi/2, M} \neq \emptyset$.
We define
\[ l_{0} = \min \{l \in \{0,\hdots,j\} : \tau_{l} \in W_{\pi/2, M} \}. \]
We have $\tau_{l} \in W_{\pi/2, M}$ for any $l_{0} \leq l \leq j$.
Fix $l_{0} \leq l \leq j$. We obtain
\[ \tau_{l} - (\tau_{l_{0}} + (l-l_{0})) = \sum_{a=l_{0}}^{l-1} |\Delta_{\varphi}(\varphi^{a}(P))| \implies
|\tau_{l} - (\tau_{l_{0}} + (l-l_{0}))| \leq \frac{l-l_{0}}{2} . \]
The property $|\tau_{l_{0}} + (l-l_{0})| \geq l-l_{0}$ is a consequence of $Re (\tau_{l_{0}}) >0$.
It implies
\[ |\tau_{l}| \geq \frac{|\tau_{l_{0}} + (l-l_{0})|}{2} \geq
\frac{\sqrt{|\tau_{l_{0}}|^{2} + (l-l_{0})^{2}}}{2} \geq
\frac{|\tau_{l_{0}}| + (l-l_{0})}{2 \sqrt{2}}
. \]
We put $l_{2}=l_{0}$ in equation (\ref{equ:invteta}) and simplify the previous inequality to get
\begin{equation}
\label{equ:invflat}
 |\tau_{l}| \geq \frac{|\tau_{0}+l_{0}| + (l-l_{0})}{4 \sqrt{2}} \geq
\frac{|\tau_{0}+l|}{4 \sqrt{2}} .
\end{equation}
Equations (\ref{equ:invteta}) and (\ref{equ:invflat}) imply that
$|\tau_{l}| \geq |\tau_{0}+l|/(4 \sqrt{2})$ for any $l \in \{0, \hdots, j\}$.

We claim that $|\tau + l| \geq c (|\tau|+l)$ for all $\tau \in W_{\theta,0}$ and
$l \in {\mathbb N} \cup \{0\}$. It suffices to prove that
\[ |\tau|^{2} + l^{2} + 2 |\tau| l \cos (\theta_{0}) = |\tau+l|^{2} \geq c^{2} (|\tau|+l)^{2}  \]
where $\theta_{0}$ is the angle enclosed by the vectors $\psi$ and $l$. Since $\theta_{0} \leq \pi-\theta$
we deduce that $\cos (\theta_{0}) \geq - \cos (\theta)$. It suffices to prove
\[ |\tau|^{2} + l^{2} - 2 |\tau| l \cos (\theta) \geq c^{2} (|\tau| +l)^{2}
\Leftrightarrow
(|\tau|^{2} + l^{2})(1-c^{2}) \geq  2 |\tau| l ( \cos (\theta) + c^{2}) .  \]
The last inequality can be deduced from $c^{2} = (1-\cos (\theta))/2$ and
$|\tau|^{2} + l^{2} \geq 2 |\tau| l$. The previous properties imply
\[  |\Pi(\varphi^{l}(P))| \leq
\frac{D}{(1+|\tau_{l}|)^{k}} \leq \frac{4^{k} \sqrt{2}^{k} D}{(1+|\tau_{0}+l|)^{k}} \leq
\frac{4^{k} \sqrt{2}^{k} D}{c^{k}}  \frac{1}{(|\tau_{0}|+l+1)^{k}} . \]
We estimate the right hand side to obtain
\[ \sum_{l=0}^{j} |\Pi(\varphi^{l}(P))| \leq
\frac{4^{k} \sqrt{2}^{k} D}{c^{k}} \left({
\frac{1}{(|\tau_{0}|+1)^{k}} + \int_{0}^{\infty} \frac{1}{(|\tau_{0}|+1+t)^{k}} dt
}\right) . \]
This leads us to
\[ \sum_{l=0}^{j} |\Pi(\varphi^{l}(P))| \leq
\frac{4^{k} \sqrt{2}^{k} D}{c^{k}} \left({
\frac{1}{(|\tau_{0}|+1)^{k}} + \frac{1}{k-1} \frac{1}{(|\tau_{0}|+1)^{k-1}}
}\right)  \]
and then to
\[ \sum_{l=0}^{j} |\Pi(\varphi^{l}(P))| \leq
\frac{4^{k} \sqrt{2}^{k} k D}{(k-1)c^{k}} \left({ \frac{1}{(|\psi(P)|+1)^{k-1}} }\right)  \]
as we intended to prove.
\end{proof}
\subsection{Defining Fatou coordinates}
\label{subsec:defatcor}
The construction of Fatou coordinates of elements in $\diff{tp1}{2}$ is based on
building quasiconformal homeomorphisms $\sigma$ in regions
$H \in  Reg(\epsilon, \aleph X,I)$ conjugating $\varphi$ and one of its normal forms.
The mapping $\sigma$ induces a quasiconformal
conjugation between the space of orbits of $\varphi_{|H(x)}$ and ${\mathbb C}^{*}$
for $x \in (0,\delta)I$.  This conjugation can
be turned into a holomorphic one by using the Ahlfors-Bers theorem. As a result,
we obtain holomorphic Fatou coordinates.

Let $\varphi \in \diff{tp1}{2}$. Let $\Upsilon={\rm exp}(X)$ be a $2$-convergent normal form.
Consider $\Lambda=(\lambda_{1}, \hdots, \lambda_{\tilde{q}}) \in {\mathcal M}$ and
the dynamical splitting $\digamma_{\Lambda}$ in remark \ref{rem:unifspl}.
\begin{defi}
\label{def:dlam}
Let $\lambda \in {\mathbb S}^{1}$. We define
\[ I_{\Lambda}^{\lambda}= \lambda e^{i[-\upsilon_{\Lambda}, \upsilon_{\Lambda}]} \ {\rm and} \
d_{\Lambda}^{\lambda}= \max \{ j \in \{0, \hdots, \tilde{q} \} : \lambda \in I_{j}(\lambda_{j},0) \} \]
and $\aleph_{\Lambda, \lambda}= \aleph_{d_{\Lambda}^{\lambda},\Lambda,\lambda}$ (see def. \ref{def:alepk}).
\end{defi}
Given $H$ in $Reg(\epsilon, \aleph_{\Lambda, \lambda} X, I_{\Lambda}^{\lambda})$
we construct Fatou coordinates of $\varphi$ in $H$. The construction is analogous
to the one in \cite{JR:mod} where detailed proofs can be found.

Let $L \in {\mathcal P}(H)$. Consider a point $P=(x_{0},y_{0}) \in H^{L}$. Let
$\gamma$ be the trajectory of $\Re (\aleph_{\Lambda, \lambda} X)$ passing through $P$.
The subset of $H_{L}$ enclosed by $\gamma$ and ${\rm exp}(X)(\gamma)$, i.e.
\[ B_{X}(P) \stackrel{def}{=} \psi_{L}^{X}(x_{0},y)^{-1} (\psi_{L}^{X}(\gamma) + [0,1]) \]
satisfies that $B_{X}(P) \setminus {\rm exp}(X)(\gamma)$ is
a fundamental domain for ${\rm exp}(X)_{|H_{L}(x_{0})}$.

We want to prove that the set $B_{\varphi}(P)$ enclosed by $\gamma$ and $\varphi(\gamma)$
satisfies that $B_{\varphi}(P) \setminus \varphi(\gamma)$ is a fundamental
domain for $\varphi_{|H_{L}(x_{0})}$. The proof relies on showing that there exists a quasi-conformal
homeomorphism $\sigma$ defined in a neighborhood of $B_{X}(P)$ in $\{ x_{0} \} \times B(0,\epsilon)$
conjugating ${\rm exp}(X)$, $\varphi$ such that $\sigma(B_{X}(P)) = B_{\varphi}(P)$.

{\bf Step 1.}
Since $\Re (\aleph_{\Lambda, \lambda} X)$ is transversal to $\Re (X)$ then
$\psi_{L}^{X}(\gamma) \cap \{ Im(z)=b \}$ is a singleton whose element we
denote $a(b)+ i b$. The curve $\gamma$ is parametrized by $Im (\psi_{L}^{X})$.
Let $\varrho: {\mathbb R} \to [0,1]$ a $C^{\infty}$ increasing function such that
$\varrho(-\infty,1/3] = \{0\}$ and $\varrho [2/3,\infty) = \{1\}$. We define
\[ \sigma_{0}(z)=\sigma_{0}(z_{1}+i z_{2})=
z + \varrho (z_{1}-a(z_{2})) (\Delta_{\varphi} \circ (\psi_{L}^{X}(x_{0},y))^{-1}(z-1)) \]
and
\[ \sigma(x_{0},y)= (\psi_{L}^{X}(x_{0},y))^{-1} \circ \sigma_{0} \circ \psi_{L}^{X}(x_{0},y) . \]
Let us prove that $\sigma$ is a q.c. diffeomorphism from a neighborhood of $B_{X}(P)$ onto
a neighborhood of $B_{\varphi}(P)$ such that $\sigma \circ {\rm exp}(X)= \varphi \circ \sigma$ and
$\sigma(B_{X}(P)) = B_{\varphi}(P)$.

Clearly $\sigma$ is the identity in a neighborhood of $\gamma$ and
$\sigma= \varphi \circ {\rm exp}(X)^{-1}(x_{0},y)$ in a neighborhood of ${\rm exp}(X)(\gamma)$.
Thus we obtain $\sigma \circ {\rm exp}(X)= \varphi \circ \sigma$ in a neighborhood of $\gamma$.
By construction $\sigma$ is a $C^{\infty}$ mapping.

{\bf Step 2.}
In order to prove that $\sigma$ is a q.c. diffeomorphism
we divide $H^{L}$ in two parts, in one of them $\Re (\aleph_{\Lambda,\lambda} X)$
is ``very transversal" to $\Re (X)$ whereas in the other one $\sigma$ is very
close to $Id$. We will use different estimates in both kind of sets in order to
analyze the properties of $\sigma$.

Denote $\aleph =\aleph_{\Lambda,\lambda}$.
Consider ${\mathcal E}_{0} =\{ (x,y) \in B(0,\delta) \times B(0,\epsilon) : |y| \geq \eta_{0} |x| \}$.
Next, we prove that $\psi_{L}^{X}$ is big outside
${\mathcal E}_{0}'=\{ (x,y) \in B(0,\delta) \times B(0,\epsilon) : |y| \geq 2 \eta_{0} |x| \}$.
By prop. \ref{pro:ext1g} we obtain
\begin{equation}
\label{equ:step21}
|\psi_{L}^{X}(x,y)| \geq \frac{1}{C_{1} (2 \eta_{0})^{\nu({\mathcal E}_{0})} |x|^{\nu({\mathcal E}_{0})}}
\end{equation}
for any $(x,y) \in H^{L}$ such that $|y| = 2 \eta_{0} |x|$.
Denote
$\Gamma_{x}= \Gamma(\aleph X, L_{i X}^{\epsilon}(x),{\mathcal E}_{0}')$
and ${\mathcal I} (\Gamma_{x}) = [h_{1}(x), h_{2}(x)]$. Since
$\Gamma_{x}(h_{j}(x)) \in \{ |y| = 2 \eta_{0} |x| \}$ and
\[ \psi_{L}^{X}(\Gamma_{x}(h_{j}(x))) - \psi_{L}^{X}(L_{i  X}^{\epsilon}(x))
\in i {\mathbb R} \]
for all $x \in (0,\delta) I_{\Lambda}^{\lambda}$ and $j \in \{1,2\}$ then we obtain
\begin{equation}
\label{equ:step22}
|Im (\psi_{L}^{X}(\Gamma_{x}(h_{j}(x))))|
\geq \frac{1}{2 C_{1} (2 \eta_{0})^{\nu({\mathcal E}_{0})} |x|^{\nu({\mathcal E}_{0})}}
\end{equation}
for all $x \in (0,\delta) I_{\Lambda}^{\lambda}$ and $j \in \{1,2\}$
by considering a smaller $\delta>0$ if necessary.
Let $\tilde{\upsilon}_{x}$ be the open arc in $\partial_{I} {\mathcal E}_{0}'$ contained in
$H$ and such that $\partial \tilde{\upsilon}_{x}= \{\Gamma_{x}(h_{1}(x)), \Gamma_{x}(h_{2}(x)) \}$.
Suppose that $\Re (X)$ points towards $H$ at $L_{i X}^{\epsilon}(0)$
without lack of generality.
Denote $\upsilon_{x}$ the curve
\[ \Gamma_{x}(-\infty,h_{1}(x)] \cup \upsilon_{x} \cup \Gamma_{x}[h_{2}(x),\infty) . \]
Given $(x,y) \in H_{L} \setminus  {\mathcal E}_{0}'$ the point $\psi_{L}^{X}(x,y)$
is to the right of the curve $\psi_{L}^{X}(\upsilon_{x})$.
Equations (\ref{equ:step21}) and (\ref{equ:step22}) imply that
the ball $B(0,C_{2}/|x|^{\nu({\mathcal E}_{0})})$ does not intersect $\psi_{L}^{X}(\upsilon_{x})$
where $C_{2}=1/(2 C_{1} (2 \eta_{0})^{\nu({\mathcal E}_{0})})$.
Thus $B(0,C_{2}/|x|^{\nu({\mathcal E}_{0})})$ is to the left of $\psi_{L}^{X}(\upsilon_{x})$.
We deduce that
\[ |\psi_{L}^{X}(x,y)| \geq \frac{C_{2}}{  |x|^{\nu({\mathcal E}_{0})}}
\ \ \forall (x,y) \in H_{L} \setminus  {\mathcal E}_{0}' . \]
We obtain $\aleph_{\Lambda,\lambda}^{*}(x,y)=i$ for any $(x,y) \in H^{L}$
such that $|\psi_{L}^{X}(x,y)| < C_{2}/|x|^{\nu({\mathcal E}_{0})}$.

{\bf Step 3.}
\begin{pro}
The mapping $\sigma$ is a diffeomorphism between neighborhoods of
$B_{X}(P)$ and $B_{\varphi}(P)$ in $\{x_{0}\} \times {\mathbb C}$.
\end{pro}
\begin{proof}
Denote
\[ {\Delta}_{0}=\Delta_{\varphi} \circ (\psi_{L}^{X}(x_{0},y))^{-1}(z-1) \ \ {\rm and} \ \
{\mathcal J} (h) =  \left({
\begin{array}{cc}
\frac{\partial Re (h)}{\partial z_{1}} & \frac{\partial Re (h)}{\partial z_{2}} \\
\frac{\partial Im (h)}{\partial z_{1}} & \frac{\partial Im (h)}{\partial z_{2}}
\end{array}
}\right) .\]
We have that $({\mathcal J} (\sigma_{0}))(z)  - Id  - \varrho (z_{1}-a(z_{2})) ({\mathcal J} (\Delta_{0}))(z)$
is equal to
\[
\left({
\begin{array}{cc}
Re (\Delta_{0})(z) \frac{\partial \varrho}{\partial t}(z_{1}-a(z_{2})) &
-Re (\Delta_{0})(z) \frac{\partial \varrho}{\partial t}(z_{1}-a(z_{2})) \frac{\partial a}{\partial t}(z_{2}) \\
Im (\Delta_{0})(z) \frac{\partial \varrho}{\partial t}(z_{1}-a(z_{2})) &
-Im (\Delta_{0})(z) \frac{\partial \varrho}{\partial t}(z_{1}-a(z_{2})) \frac{\partial a}{\partial t}(z_{2})
\end{array}
}\right) .
\]
By lemma \ref{lem:deltahol} we have that $\Delta_{\varphi}(0,0)=0$.
By using Cauchy's integral formula
\[ \frac{\partial \Delta_{0}}{\partial z}(z) = \frac{1}{2 \pi i}
\int_{|w|=1} \frac{\Delta_{0}(w)}{(w-z)^{2}} dw \]
we can suppose that $|\Delta_{\varphi}(Q)| < (\sup_{\mathbb R} |\partial \varrho/\partial t|)^{-1}/16$ and
$|\partial \Delta_{0}/\partial z|(\psi_{L}^{X}(Q)) < 1/16$ for any $Q$ in a neighborhood of
$B_{X}(P)$ in $\{x_{0}\} \times B(0,\epsilon)$. Let us remark that we could need to consider a
smaller domain of definition $B(0,\delta) \times B(0,\epsilon)$ but the domain with the good
estimates only depend on $\varphi$ and $X$ (and not on $P$ for example).

Denote $\tau(z)= \Delta_{0}(z) (\partial a / \partial t)(z_{2})$.
Let us estimate $\tau(z)$.
There exist continuous functions
$\theta_{1}, \hdots, \theta_{r}: I_{\Lambda}^{\lambda} \to (0,\pi)$ such
that $\theta_{j}$ is decreasing on $\arg \lambda'$ and
$\aleph_{\Lambda,\lambda}(\lambda') = (e^{i \theta_{1}}, \hdots, e^{i \theta_{r}})(\lambda')$
for any $\lambda' \in I_{\Lambda}^{\lambda}$. We define
\[ \theta' = \min_{1 \leq j \leq r, \ \lambda' \in I_{\Lambda}^{\lambda}}
\min (\pi - \theta_{j}(\lambda'), \theta_{j}(\lambda')) . \]
We have $\theta'>0$. The tangent vector to $\psi_{L}^{X}(\gamma)$ at any point
$\psi_{L}^{X}(Q)$ for $Q$ in $\gamma$ belongs to ${\mathbb R}^{+} e^{i[\theta',\pi-\theta']}$.
We deduce that $|\partial a/\partial t| \leq 1 / \tan (\theta')$. The equation (\ref{equ:step22})
implies that $a(t)$ is constant in
$(-C_{2}/|x_{0}|^{\nu({\mathcal E}_{0})}, C_{2}/|x_{0}|^{\nu({\mathcal E}_{0})})$.
We obtain that $\tau(z)=0$ if $|z_{2}| < C_{2}/|x_{0}|^{\nu({\mathcal E}_{0})}$.
Suppose now that $|z_{2}| \geq C_{2}/|x_{0}|^{\nu({\mathcal E}_{0})}$. We obtain
\begin{equation}
\label{equ:later}
 |\Delta_{0}(z)| \leq \frac{K_{1}}{(1+|z|)^{2}} \leq
\frac{K_{1} |x_{0}|^{2 \nu({\mathcal E}_{0})}}{C_{2}^{2}} <
\frac{\tan (\theta')}{16 \sup_{\mathbb R} |\partial \varrho/\partial t|}
\end{equation}
by considering $\delta>0$ small enough. We obtain
$|\tau(z)| < (\sup_{\mathbb R} |\partial \varrho/\partial t|)^{-1}/16$.
The coefficients of the matrix $({\mathcal J} (\sigma_{0}))(z)  - Id$ belong to
$(-1/8,1/8)$ for any $z$ in a neighborhood of $\psi_{L}^{X}(B_{X}(P))$.
Since $||({\mathcal J} (\sigma_{0}))(z)  - Id|| < 1/4$ for the spectral norm then $\sigma_{0}$ is a
diffeomorphism in the neighborhood of $\psi_{L}^{X}(B_{X}(P))$. We deduce that $\sigma$ is a
diffeomorphism in the neighborhood of $B_{X}(P)$.
\end{proof}
%
\begin{rem}
Consider $Q \in \gamma$. Denote $z' =  \psi_{L}^{X}(Q)+1$.
Let $e^{i \theta_{1}}$ be the unit tangent vector to $\psi_{L}^{X}(\gamma)$
at $\psi_{L}^{X}(Q)$.
The vector $({\mathcal J} (\sigma_{0})(z'))(e^{i \theta_{1}})$ is tangent to
$\psi_{L}^{X}(\varphi(\gamma))$ at $\psi_{L}^{X}(\varphi(Q))$.
We claim that
\[ ({\mathcal J} (\sigma_{0})(z'))(e^{i \theta_{1}}) \in
{\mathbb R}^{+} e^{i[\theta'/2, \pi - \theta'/2]} \cup {\mathbb R}^{+} i e^{i[-\theta_{0}, \theta_{0}]} \]
where $\theta_{0} = \arctan (1/4)$. Indeed  if
$|\psi_{L}^{X}(Q)| < C_{2}/|x_{0}|^{\nu({\mathcal E}_{0})}$ we have that
$\theta_{1}=\pi/2$ and $|({\mathcal J} (\sigma_{0})(z'))(i) - i| \leq 1/4$
imply $({\mathcal J} (\sigma_{0})(z'))(i) \in {\mathbb R}^{+} i e^{i[-\theta_{0}, \theta_{0}]}$.
Otherwise $|{\mathcal J} (\sigma_{0}) - Id| = O(|x|^{\nu({\mathcal E}_{0})})$ and we obtain
$({\mathcal J} (\sigma_{0})(z'))(e^{\theta_{1}}) \in {\mathbb R}^{+} e^{i[\theta'/2, \pi - \theta'/2]}$.
Anyway  $\Re (X)$ is transversal to $\varphi(\gamma)$ at $Q'$ for any $Q' \in \varphi(\gamma)$.
\end{rem}
{\bf Step 4.} Analogously we can prove
\begin{pro}
The mapping $\sigma$ is quasiconformal.
\end{pro}
\begin{proof}
We have
\begin{equation}
\label{equ:dilats}
\chi_{\sigma_{0}}=
\frac{\frac{\partial \sigma_{0}}{\partial \overline{z}}}{\frac{\partial \sigma_{0}}{\partial z}} (z)=
\frac{\frac{\partial \varrho(z_{1}-a(z_{2}))}{\partial \overline{z}} \Delta_{0}(z)}{
1+ \varrho(z_{1}-a(z_{2})) \frac{\partial \Delta_{0}}{\partial z}(z) + \Delta_{0}(z)
\frac{\partial \varrho(z_{1}-a(z_{2}))}{\partial z}}
\end{equation}
and
\[ \frac{\partial \varrho(z_{1}-a(z_{2}))}{\partial z_{1}} = \frac{\partial \varrho}{\partial t} (z_{1}-a(z_{2})),
\  \frac{\partial \varrho(z_{1}-a(z_{2}))}{\partial z_{2}} = -\frac{\partial \varrho}{\partial t} (z_{1}-a(z_{2}))
\frac{\partial a}{\partial t}(z_{2}). \]
We obtain
\[ |\chi_{\sigma_{0}}| =
\left|{ \frac{\frac{\partial \sigma_{0}}{\partial \overline{z}}}{\frac{\partial \sigma_{0}}{\partial z}} (z)
}\right| \leq \frac{\frac{\sqrt{2}}{2} \frac{1}{16}}{1- \frac{1}{16}-\frac{\sqrt{2}}{2} \frac{1}{16}} <
\frac{1}{14}   \]
in the neighborhood of $\psi_{L}^{X}(B_{X}(P))$. Therefore $\sigma$ is a $15/13$ q.c. mapping in a neighborhood of
$B_{X}(P)$.
\end{proof}
{\bf Step 5.} We prove now that the domain enclosed by $\gamma$ and $\varphi(\gamma)$ is a
fundamental domain for $\varphi_{|H_{L}(x_{0}) \cup \varphi(H_{L}(x_{0}))}$.

The vector field $\Re (X)$ is
transversal to $\varphi(\gamma)$ at $Q'$ for any $Q' \in \varphi(\gamma)$.
Hence given $z_{2} \in {\mathbb R}$ there exists a unique
point $c(z_{2})+i z_{2}$ in $\psi_{L}^{X}(\varphi(\gamma)) \cap \{Im(z)=z_{2}\}$.

We define (see Step 1 for the definition of the function $a$)
\[ B_{\varphi}(P) = \{ Q \in H_{L}(x_{0}) : Re (\psi_{L}^{X}(Q)) \in
[a(Im(\psi_{L}^{X}(Q))),c(Im(\psi_{L}^{X}(Q)))] \} . \]
Let us prove that $B_{\varphi}(P) \setminus \varphi(\gamma)$ is a fundamental domain for
$\varphi_{|H_{L}(x_{0}) \cup \varphi(H_{L}(x_{0}))}$.
Step 3 implies that $B_{\varphi}(P) \setminus \varphi(\gamma)$
is a fundamental domain for $\varphi$ in the neighborhood of $B_{\varphi}(P)$
and that orbits of $\varphi_{|H_{L}(x_{0}) \cup \varphi(H_{L}(x_{0}))}$ intersects
$B_{\varphi}(P) \setminus \varphi(\gamma)$ at most once.
It suffices to prove that every orbit of a point $Q$ in
$H_{L}(x_{0}) \cup \varphi(H_{L}(x_{0}))$ intersects $B_{\varphi}(P)$.

Let $Q \in H_{L}(x_{0}) \cup \varphi(H_{L}(x_{0}))$. Denote
$z_{1}+i z_{2} = \psi_{L}^{X}(Q)$. If $z_{1} \in [a(z_{2}), c(z_{2})]$ then
$Q \in B_{\varphi}(P)$ and there is nothing to prove. Suppose without lack of generality
that $z_{1} < a(z_{2})$. We define
\[ A= \{  Q \in H_{L}(x_{0}) : Re (\psi_{L}^{X}(Q)) < a(Im(\psi_{L}^{X}(Q))) \} . \]
It suffices to prove that there exists $j \in {\mathbb N}$ such that
$Q, \hdots, \varphi^{j-1}(Q) \in A$ and $\varphi^{j}(Q) \not \in A$ since then
$\varphi^{j}(Q)$ belongs to $B_{\varphi}(P)$. Let us argue by contradiction, we suppose that
$\varphi^{j}(Q) \in A$ for any $j \in {\mathbb N}$.
Denote $\psi_{L}^{X}(\varphi^{j}(Q)) = s_{j}^{1} + i s_{j}^{2}$. We have
$|\Delta_{\varphi}(\varphi^{j}(P))| \leq 1/16$ and then $s_{j+1}^{1} - s_{j}^{1} \geq 15/16$
for any $j \geq 0$. Since $|\partial a/ \partial t| \leq 1/\tan(\theta')$ then
$Re (\psi_{L}^{X})$ is bounded by above in
\[ \{ (x_{0},y) \in A : |Im (\psi_{L}^{X})|(x_{0},y) < M \} \]
for any $M >0$. Therefore we obtain
$\lim_{j \to \infty} |Im (\psi_{L}^{X})|(\varphi^{j}(Q)) = \infty$.
In particular the limit $\lim_{j \to \infty} \varphi^{j}(Q)$ exists and is equal to
a point $Z \in Sing X$. We deduce that
$\lim_{j \to \infty} \Delta_{\varphi}(\varphi^{j}(Q))=0$.
 We have
\[ |s_{j+1}^{2} -s_{j}^{2}| \leq |\Delta_{\varphi}(\varphi^{j}(Q))| \Rightarrow
 |a(s_{j+1}^{2}) -a(s_{j}^{2})| \leq \frac{|\Delta_{\varphi}(\varphi^{j}(Q))|}{\tan (\theta')} \]
and then
\[ (s_{j+1}^{1} - a(s_{j+1}^{2}))  -(s_{j}^{1} - a(s_{j}^{2})) \geq \frac{15}{16} -
\frac{|\Delta_{\varphi}(\varphi^{j}(Q))|}{\tan (\theta')} > 1/2 \]
for any $j \in {\mathbb N}$ big enough. This is impossible since
$\varphi^{j}(P) \in A$ implies $s_{j}^{1} - a(s_{j}^{2})<0$ for any $j \in {\mathbb N}$.

{\bf Step 6.}  We denote
$B_{\varphi}^{*}(P)$ the space of orbits of $\varphi_{|H_{L}(x_{0})}$.
We construct a biholomorphism from $B_{\varphi}^{*}(P)$ to ${\mathbb C}^{*}$.

The mapping $\xi = e^{2 \pi i z} \circ \psi_{L}^{X} \circ \sigma^{-1}$ is a
diffeomorphism from $B_{\varphi}^{*}(P)$ onto ${\mathbb C}^{*}$.
The mapping $\xi$ depends on the point $P$. The function
$\psi_{L}^{X} \circ \sigma^{-1}$ is a $C^{\infty}$ Fatou coordinate of $\varphi$ in $B_{\varphi}(P)$.
Equations (\ref{equ:delta}) and (\ref{equ:delaux}) imply
\begin{equation}
\label{equ:delfund}
 |\Delta_{0}(z)| \leq \frac{K_{1}}{(1+|z|)^{2}} \ {\rm and} \
|\Delta_{0}(z)| \leq \frac{K'}{(1+|z|)^{2+2/\nu({\mathcal E}_{0})}} \ {\rm in} \  H^{L}(x_{0})
\end{equation}
where $K_{1}$ depends only on $X$, $\varphi$ and $K'>0$ does not depend on $P$ or $x_{0}$
but depends on $\aleph_{\Lambda,\lambda}$.
Consider the notations in Step 3. Let $z=z_{1}+i z_{2} \in \psi_{L}^{X}(B_{X}(P))$.
If $|z_{2}| < C_{2}/|x_{0}|^{\nu({\mathcal E}_{0})}$ we obtain $(\partial a/\partial t)(z_{2})=0$;
the complex dilatation $\chi_{\sigma_{0}}$ of $\sigma_{0}$ satisfies
\[ |\chi_{\sigma_{0}}|(z) =
\left|{ \frac{\frac{\partial \sigma_{0}}{\partial \overline{z}}}{\frac{\partial \sigma_{0}}{\partial z}} (z)
}\right| \leq \frac{\frac{\sup_{\mathbb R} |\partial \varrho/\partial t| K_{1} }{2 (1+|z|)^{2}}}{1-\frac{1}{16} - \frac{\sqrt{2}}{2} \frac{1}{16} }=
\frac{C_{3} K_{1} }{(1+|z|)^{2}} . \]
Suppose that $|z_{2}| \geq C_{2}/|x_{0}|^{\nu({\mathcal E}_{0})}$, we have
\[ |\chi_{\sigma_{0}}|(z) \leq
\frac{\frac{1}{2} \left({ \sup_{\mathbb R} |\partial \varrho/\partial t|
+ \frac{\sup_{\mathbb R} |\partial \varrho/\partial t|}{\tan (\theta')} }\right) \frac{K'}{K_{1}}
\frac{1}{(1+|z|)^{2/\nu({\mathcal E}_{0})}}
\frac{K_{1}}{(1+|z|)^{2}}}{1-\frac{1}{16} - \frac{\sqrt{2}}{2} \frac{1}{16}} . \]
Consider $x_{0} \in B(0,\delta)$ and $\delta>0$ small enough. Since
\[ \frac{1}{(1+|z|)^{2/\nu({\mathcal E}_{0})}} \leq \frac{1}{|z_{2}|^{2/\nu({\mathcal E}_{0})}}
\leq \frac{|x_{0}|^{2}}{C_{2}^{2/\nu({\mathcal E}_{0})}} \]
the previous calculations and Step 4 lead us to
\[ |\chi_{\sigma_{0}}|(z) \leq  \min \left({ \frac{C_{3} K_{1} }{(1+|z|)^{2}}, \frac{1}{14} }\right)
\ \ \forall z \in \psi_{L}^{X}(B_{X}(P)) . \]
Since $\xi^{-1}$ is equal to
$(\psi_{L}^{X})^{-1} \circ \sigma_{0} \circ ((1/2\pi i) \ln z)$ then
\begin{lem}
\begin{equation}
\label{equ:dilat}
|\chi_{\xi^{-1}}|(z) \leq  \min \left({ \frac{C_{3} K_{1} }{(1+2^{-1} \pi^{-1} |\ln z|)^{2}},
\frac{1}{14} }\right)
\end{equation}
for any $z \in e^{2 \pi i w} \circ \psi_{L}^{X}(B_{X}(P)) = {\mathbb C}^{*}$.
\end{lem}
There exists a quasi-conformal homeomorphism $\tilde{\rho}: \pn{1} \to \pn{1}$
such that $\chi_{\tilde{\rho}} = \chi_{\xi^{-1}}$. Since
$||\chi_{\xi^{-1}}||_{\infty} = \sup_{z \in {\mathbb C}^{*}} |\chi_{\xi^{-1}}(z)| \leq 1/14$
it is a consequence of Ahlfors-Bers theorem. The choice of $\tilde{\rho}$ is unique if we require
the normalizing conditions $\tilde{\rho}(0)=0$, $\tilde{\rho}(1)=1$ and $\tilde{\rho}(\infty)=\infty$.
By construction $\tilde{\rho} \circ \xi$ is a biholomorphism from $B_{\varphi}^{*}(P)$ to ${\mathbb C}^{*}$.

{\bf Step 7.} Let us construct a holomorphic Fatou coordinate of $\varphi$ defined in a
neighborhood of $B_{\varphi}(P)$ in $\{x_{0}\} \times B(0,\epsilon)$. The construction and its properties
are analogous to those in section 7.2 of \cite{JR:mod}. In both cases the key
ingredient is the equation (\ref{equ:dilat}). Further details can be found in  \cite{JR:mod}.

We define
\[ J(r) = \frac{2}{\pi} \int_{|z| < r}
\frac{C_{3} K_{1}}{(1+ 2^{-1} \pi^{-1} |\ln |z||)^2}
\frac{1}{|z|^{2}} d \sigma  \]
for $r \in {\mathbb R}^{+}$. We have that $J(r) < \infty$ for
any $r \in {\mathbb R}^{+}$.
\begin{lem}
\label{lem:Le-Vi}
The mapping $\tilde{\rho}$ is conformal at $0$ and at $\infty$.
Moreover we have
\[ \left|{ \frac{\tilde{\rho}(z)}{z} -
\frac{\partial \tilde{\rho}}{\partial z}(0) }\right| \leq \left|{
\frac{\partial \tilde{\rho}}{\partial z}(0)
}\right| \jmath (|z|) \ {\rm and} \
 \left|{ \frac{z}{\tilde{\rho}(z)} -
{\frac{\partial \tilde{\rho}}{\partial z}(\infty) }^{-1}
}\right| \leq {\left|{
\frac{\partial \tilde{\rho}}{\partial z}(\infty)
}\right|}^{-1} \jmath(1/|z|) \]
where $\jmath$ is a function depending only on $\varphi$, $X$; it satisfies
$\lim_{|z| \to 0} \jmath(|z|)=0$. We have
\[ min_{|z|=1} |\tilde{\rho}(z)| e^{-J(1)} \leq
|\partial \tilde{\rho}/\partial z|(0), |\partial \tilde{\rho}/
\partial z|(\infty) \leq  max_{|z|=1} |\tilde{\rho}(z)| e^{J(1)} . \]
\end{lem}
Since $\tilde{\rho}$ is conformal at $0$ we can define
$\rho=\tilde{\rho}/(\partial \tilde{\rho}/\partial z)(0)$.
The q.c. mapping $\rho: \pn{1} \to \pn{1}$ is the unique solution
of $\chi_{\rho} = \chi_{\xi^{-1}}$ such that
$\rho(0)=0$, $\rho(\infty)=\infty$ and $(\partial \rho/\partial z)(0)=1$.
We define the function
\begin{equation}
\label{def:Fat}
\psi_{H,L,P}^{\varphi} = \frac{1}{2 \pi i} \ln z \circ
\rho \circ {e}^{2 \pi i z} \circ \psi_{H,L}^{X} \circ \sigma^{-1} .
\end{equation}
It is a holomorphic Fatou coordinate of $\varphi$ defined in the neighborhood of
$B_{\varphi}(P)$ in $\{x_{0}\} \times {\mathbb C}$.
The set $B_{\varphi}(P)$ contains a fundamental domain of
$\varphi_{|H_{L}(x_{0}) \cup \varphi(H_{L}(x_{0}))}$ by Step 5.
Thus we can extend $\psi_{H,L,P}^{\varphi}$ to $H_{L}(x_{0})$ by using
$\psi_{H,L,P}^{\varphi} \circ \varphi \equiv \psi_{H,L,P}^{\varphi} +1$.
Consider the points $Z_{\pm} \in Sing X$ such that
$Z_{\pm}= \lim_{Q \in H_{L}(x_{0}), \ Im(\psi_{L}^{X}(Q)) \to \pm \infty} Q$.
By construction we have
\[ |\psi_{L}^{X} \circ \sigma - \psi_{L}^{X}|(x_{0},y) \leq
\frac{K_{1}}{(1+|\psi_{L}^{X}(x_{0},y)|)^{2}} \]
for any $(x_{0},y)$ in a neighborhood of $B_{X}(P)$. We deduce that
\[ \lim_{Q \in B_{\varphi}(P), \ |Im (\psi_{L}^{X}(Q))| \to \infty}
|\psi_{L}^{X} \circ \sigma^{-1} - \psi_{L}^{X}|(Q)=0. \]
The mapping $\rho$ is conformal at $0$ and $\infty$, hence there exist
$\kappa_{+}, \kappa_{-} \in {\mathbb C}$ such that
\[ \lim_{Q \in B_{\varphi}(P), \ Im (\psi_{L}^{X}(Q)) \to \pm \infty}
|\psi_{H,L,P}^{\varphi} - \psi_{L}^{X}|(Q)=\kappa_{\pm} . \]
Moreover $(\partial \rho/ \partial z)(0)=1$ implies $\kappa_{+}=0$.
Since $\Delta_{\varphi}(Z_{+})=\Delta_{\varphi}(Z_{-})=0$  we obtain
\[ \lim_{Q \in B_{\varphi}(P'), \ Im (\psi_{L}^{X}(Q)) \to \pm \infty}
|\psi_{H,L,P}^{\varphi} - \psi_{L}^{X}|(Q)= \kappa_{\pm}  \]
for any $P' \in H^{L}(x_{0})$.
The function
$(\psi_{H,L,P}^{\varphi} - \psi_{H,L,P'}^{\varphi}) \circ (e^{2 \pi i w} \circ \psi_{H,L,P}^{\varphi})^{-1}(z)$
is a bounded holomorphic function
defined in ${\mathbb C}^{*}=e^{2 \pi i w} \circ \psi_{H,L,P}^{\varphi}(B_{\varphi} (P'))$ and then
constant. Since its value at $z=0$ is $0$ we deduce that
$\psi_{H,L,P}^{\varphi} \equiv \psi_{H,L,P'}^{\varphi}$ for all $P,P' \in H^{L}(x_{0})$.
\begin{defi}
We denote $\psi_{H, L}^{\varphi}$ any of the functions $\psi_{H, L,P}^{\varphi}$ defined
in $H_{L}(x_{0})$.
We denote $\psi_{L}^{\varphi}=\psi_{H, L}^{\varphi}$ if the choice of the domain of definition and
multi-transversal flow is implicit.
\end{defi}
Consider
$H \in Reg_{2}(\epsilon, \aleph_{\Lambda, \lambda} X, I_{\Lambda}^{\lambda})$.
We denote ${\mathcal P}(H)=\{L,R\}$. Consider $x_{0} \neq 0$. Proceeding in an analogous way as
above we obtain
\[ \psi_{L}^{\varphi} - \psi_{L}^{X} \equiv \psi_{R}^{\varphi} - \psi_{R}^{X}
\ \ {\rm in} \ H_{L}(x_{0})=H_{R}(x_{0}). \]
\begin{defi}
We denote $\psi_{H}^{\varphi}-\psi_{H}^{X}$ the function
defined in $H(x_{0})$ and given by the formula $\psi_{H,L}^{\varphi} - \psi_{H,L}^{X}$ in $H_{L}(x_{0})$
for $L \in {\mathcal P}(H)$.
\end{defi}
{\bf Step 8.} We introduce the main results concerning Fatou coordinates.
\begin{pro}
\label{pro:Fatpar}
Let $\varphi \in \diff{tp1}{2}$. Let $\Upsilon$ be a $2$-convergent normal form.
Consider $\Lambda=(\lambda_{1}, \hdots, \lambda_{\tilde{q}}) \in {\mathcal M}$ and
$\lambda \in {\mathbb S}^{1}$.
Let $H \in Reg(\epsilon, \aleph_{\Lambda, \lambda} X, I_{\Lambda}^{\lambda})$
and $L \in {\mathcal P}(H)$. Then the mapping $(x,\psi_{L}^{\varphi})$ is holomorphic in
$H^{\circ}$ and continuous and injective in $H_{L}$.
\end{pro}
We constructed holomorphic Fatou coordinates in $H_{L}(x)$ for any
$x \in [0,\delta) I_{\Lambda}^{\lambda}$.
Next, we explain why the dependence of $\psi_{L}^{\varphi}(x,y)$ on $x$ is continuous.
The holomorphic part of the statement is proved in Step 11.

Consider $P=(x_{0},y_{0}) \in H^{L}(x_{0})$. Given
$x \in [0,\delta) I_{\Lambda}^{\lambda}$ in a neighborhood of $x_{0}$
we define the continuous section $P(x) \in H_{L}(x)$ such that
$\psi_{L}^{X}(P(x)) \equiv \psi_{L}^{X}(P)$ and $P(x_{0})=P$. We consider the trajectory
$\gamma_{x}=\Gamma(\aleph_{\Lambda,\lambda} X, P(x), T_{0})$ and then we define the function
$a(x,z_{2})$ as in Step 1.

We define
\[ \sigma_{0}(x,z)=\sigma_{0}(x,z_{1}+i z_{2})=
(x,z + \varrho (z_{1}-a(x,z_{2})) (\Delta_{\varphi} \circ (x, \psi_{L}^{X})^{-1}(x,z-1))) \]
and
\[ \sigma(x,y)= (x, \psi_{L}^{X})^{-1} \circ \sigma_{0} \circ (x,\psi_{L}^{X})(x,y) . \]
The functions $a(x,z_{2})$ and $\partial a(x,z_{2}) /\partial z_{2}$ are continuous. Therefore
the complex dilatations $\chi_{\sigma_{0}}$ and $\chi_{\xi^{-1}}$ depend continuously on
$(x,z)$ (see Step 4 and in particular equation (\ref{equ:dilats})). We deduce that
$\tilde{\rho}$ and $\rho$ depend continuously on $(x,z)$. By definition of $\psi_{L}^{\varphi}$
(see equation (\ref{def:Fat})) we obtain that $\psi_{L}^{\varphi}$ is continuous in
$\cup_{x \in V} B_{\varphi}(P(x))$ for some neighborhood $V$ of $x_{0}$ in
$[0,\delta) I_{\Lambda}^{\lambda}$.
As a consequence $\psi_{L}^{\varphi}$ is continuous in $\cup_{x \in V} H_{L}(x)$
and then in $H_{L}$.

A proof for the next proposition can be found by replicating the proof of the
analogous result in \cite{JR:mod}.
\begin{pro}
\label{pro:bddcon}
Let $\varphi \in \diff{tp1}{2}$. Let $\Upsilon$ be a $2$-convergent normal form.
Consider $\Lambda=(\lambda_{1}, \hdots, \lambda_{\tilde{q}}) \in {\mathcal M}$ and
$\lambda \in {\mathbb S}^{1}$.
Let $H \in Reg(\epsilon, \aleph_{\Lambda, \lambda} X, I_{\Lambda}^{\lambda})$;
the function $\psi_{H}^{\varphi} - \psi_{H}^{X}$ is continuous in $\overline{H}$.
\end{pro}
\begin{cor}
\label{cor:Lavvf}
Let $\varphi \in \diff{tp1}{2}$. Let $\Upsilon$ be a $2$-convergent normal form.
Consider $\Lambda=(\lambda_{1}, \hdots, \lambda_{\tilde{q}}) \in {\mathcal M}$ and
$\lambda \in {\mathbb S}^{1}$.
Let $H \in Reg(\epsilon, \aleph_{\Lambda, \lambda} X, I_{\Lambda}^{\lambda})$.
There exists a unique vector field
$X_{H}^{\varphi}= X_{H}^{\varphi}(y) \partial / \partial y =
(e^{2 \pi i \psi_{H}^{\varphi}})^{*} (2 \pi i z \partial / \partial z)$,
the so called Lavaurs vector field, which is
continuous in $H$, holomorphic in ${H}^{\circ}$ and satisfies
$X_{H}^{\varphi}(\psi_{H}^{\varphi}) \equiv 1$.
Moreover $X_{H}^{\varphi}(y)/X(y)-1$ is continuous in
$\overline{H}$ and vanishes at $\overline{H} \cap Fix (\varphi)$.
\end{cor}
Proposition \ref{pro:bddcon} implies all the statements in corollary \ref{cor:Lavvf}
except the holomorphic character of $X_{H}^{\varphi}$ (see \cite{JR:mod}). This last property is a consequence
of the analogous result for Fatou coordinates (prop. \ref{pro:Fatpar}) and it will be proved later on.

{\bf Step 9.}
We prove that Fatou coordinates are holomorphic if multi-directions are constant.

Consider the notations at the beginning of this section.
Let $\lambda_{0} \in I_{\Lambda}^{\lambda}$.
%
%
Then $\Re (\aleph_{\Lambda,\lambda}(\lambda_{0}) X)$ is a multi-transversal flow
in $[0,\delta) I_{\Lambda}^{\lambda}$ by remark \ref{rem:alext}.
The multi-direction $\aleph_{\Lambda,\lambda}(\lambda_{0})$ associated to
$\Re (\aleph_{\Lambda,\lambda}(\lambda_{0}) X)$ is constant.

Let $H \in Reg (\epsilon, \aleph_{\Lambda,\lambda} X, I_{\Lambda}^{\lambda})$
and $L \in {\mathcal P}(H)$.
Let $H_{\lambda_{0}}$ be the element of the set
$Reg (\epsilon, \aleph_{\Lambda,\lambda}(\lambda_{0}) X, I_{\Lambda}^{\lambda})$
such that $L \subset H_{\lambda_{0}}$.

Let $x_{0} \in (0,\delta) \lambda e^{i(-\upsilon_{\Lambda},\upsilon_{\Lambda})}$.
Consider $P \in H_{\lambda_{0}}^{L}(x_{0})$ and
$\gamma = \Gamma(\aleph_{\Lambda,\lambda}(\lambda_{0}) X,P,T_{0})$.
Let $P(x)$ be the section defined in Step 8.
We consider the continuous family of curves $\{\gamma_{x}\}_{x \in V}$ defined
for an open neighborhood $V$ of $x_{0}$ in
$(0,\delta) \lambda e^{i(-\upsilon_{\Lambda},\upsilon_{\Lambda})}$
by $\gamma_{x_{0}}=\gamma$ and
$\psi_{L}^{X}(\gamma_{x})=\psi_{L}^{X}(\gamma)$
for any $x \in V$.
Since the multi-direction of $\Re (\aleph_{\Lambda,\lambda}(\lambda_{0}) X)$ is constant
we obtain that $\gamma_{x}$ is contained in $H_{\lambda_{0}}$ for any $x \in V$
(by considering a smaller $V$ if necessary), indeed we have
\[ \lim_{x \to x_{0}}
| \psi_{L}^{X}(\gamma_{x}(s)) - \psi_{L}^{X}(\Gamma(\aleph_{\Lambda,\lambda}(\lambda_{0}) X,P(x),T_{0})(s))| =0 \]
uniformly on $s \in {\mathbb R}$.
This property does not hold true in general if we replace $\aleph_{\Lambda,\lambda}(\lambda_{0})$
with $\aleph_{\Lambda,\lambda}$ since then for instance we could have
\[ \sup_{s \in {\mathbb R}} |\psi_{L}^{X}(\Gamma(\aleph_{\Lambda,\lambda} X,P,T_{0})(s)),
\psi_{L}^{X}(\Gamma(\aleph_{\Lambda,\lambda} X,P(x),T_{0})(s))| =\infty \]
for any $x \in V \setminus \{x_{0}\}$.

The function $a(x,z_{2})$ (see Steps 1 and 8) does not depend on $x$.
The complex dilation $\chi_{\sigma_{0}}$ depends holomorphically on $x$
(see equation (\ref{equ:dilats})) and then $\chi_{\xi^{-1}}$ is also holomorphic on
$x$. As a consequence $\tilde{\rho}$ and $\rho$ are holomorphic on $x$.
Therefore $\psi_{H_{\lambda_{0}},L}^{\varphi}$ is holomorphic in a neighborhood of
$\cup_{x \in V} B_{\varphi} (P(x))$. The set $B_{\varphi} (P(x))$ contains a fundamental domain of
$\varphi$ in $H_{\lambda_{0}}(x)$. By using
$\psi_{H_{\lambda_{0}},L}^{\varphi} \circ \varphi \equiv \psi_{H_{\lambda_{0}},L}^{\varphi}+1$
we obtain that $\psi_{H_{\lambda_{0}},L}^{\varphi}$ is holomorphic in $\cup_{x \in V} H_{\lambda_{0}}(x)$.
Thus $\psi_{H_{\lambda_{0}},L}^{\varphi}$ is holomorphic in
$H_{\lambda_{0}, L}^{\circ} = H_{\lambda_{0}}^{\circ}$.

{\bf Step 10.} We prove that $\psi_{H_{\lambda},L}^{\varphi}$ can be extended to $H$ by iteration.
We deduce that $\psi_{H_{\lambda},L}^{\varphi}$ is holomorphic in $H^{\circ}$.
Step 11 is dedicating to prove $\psi_{H,L}^{\varphi} \equiv \psi_{H_{\lambda},L}^{\varphi}$ in $H$.

The difficulty is that in general $H_{\lambda_{0}} \neq H$ for any
$\lambda_{0} \in I_{\Lambda}^{\lambda}$. We could try to consider
$\cup_{\lambda_{0} \in I_{\Lambda}^{\lambda}} H_{\lambda_{0}, L}$ since
$H_{L} \subset \cup_{\lambda_{0} \in I_{\Lambda}^{\lambda}} H_{\lambda_{0}, L}$
but in general $(x,\psi_{L}^{X})$ is not injective in
$\cup_{\lambda_{0} \in I_{\Lambda}^{\lambda}} H_{\lambda_{0}, L}$. We define
\[ G(x)  = \cup_{\lambda_{0} \in I_{\Lambda}^{\lambda}}
\psi_{L}^{X}(H_{\lambda_{0}, L}(x)) \ {\rm and} \
G=\cup_{x \in [0,\delta)  I_{\Lambda}^{\lambda}} (\{x\} \times G(x)) . \]
Consider the continuous functions
$\theta_{1}, \hdots, \theta_{r}: I_{\Lambda}^{\lambda} \to (0,\pi)$
defined in Step 3. We have that $\theta_{j}(\lambda')$ is decreasing on $\arg \lambda'$. We denote
$(\aleph_{\Lambda,\lambda}(\lambda_{0}))^{*} =e^{i \theta_{\lambda_{0}}}$ where
\[ \theta_{\lambda_{0}}:
([0,\delta)  I_{\Lambda}^{\lambda} \times B(0,\epsilon)) \setminus Sing X
\to (0,\pi) \]
is a continuous function. The construction of multi-transversal flows implies that given
$P \in ([0,\delta)  I_{\Lambda}^{\lambda} \times B(0,\epsilon)) \setminus Sing X$
the function $\upsilon \mapsto \theta_{\lambda e^{i \upsilon}}(P)$ is decreasing in
$[-\upsilon_{\Lambda}, \upsilon_{\Lambda}]$.
We define
\[ \Gamma_{x, \lambda_{0}}= \Gamma(\aleph_{\Lambda,\lambda}(\lambda_{0}) X, L_{i X}^{\epsilon}(x), T_{0}). \]
\begin{figure}[h]
\begin{center}
\includegraphics[height=6cm,width=6cm]{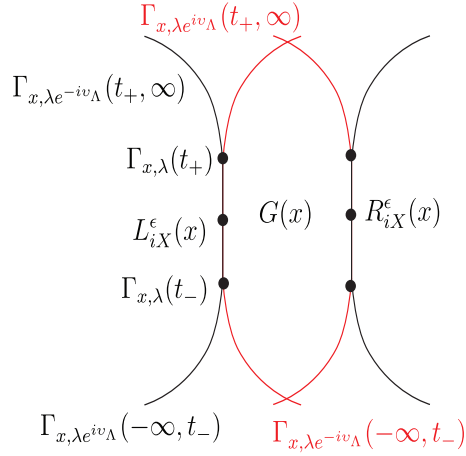}
\end{center}
\caption{Picture of $G(x)$ in the coordinate $\psi_{L}^{X}$}
\label{EVf4}
\end{figure}
Let $\upsilon_{1}, \upsilon_{2} \in [-\upsilon_{\Lambda}, \upsilon_{\Lambda}]$ with
$\upsilon_{1} < \upsilon_{2}$
and $x \in [0,\delta) I_{\Lambda}^{\lambda}$. Denote
$\lambda_{1}= \lambda  e^{i \upsilon_{1}}$ and $\lambda_{2}= \lambda  e^{i \upsilon_{2}}$.
The decreasing character of
$\upsilon \mapsto \theta_{\lambda e^{i \upsilon}}(P)$ implies that there exist
$t_{-}, t_{+} \in {\mathbb R}^{+} \cup \{\infty\}$ such that
$\Gamma_{x, \lambda_{1}}(t) = \Gamma_{x, \lambda_{2}}(t)$ if and only if $-t_{-} \leq t \leq t_{+}$.
Moreover if $t_{+}<\infty$ we obtain that
$\psi_{L}^{X}(\Gamma_{x, \lambda_{2}}(t_{+}, \infty))$ is to the right of
$\psi_{L}^{X}(\Gamma_{x, \lambda_{1}})$ whereas if $t_{-} < \infty$ then
$\psi_{L}^{X}(\Gamma_{x, \lambda_{2}}(-\infty, t_{-}))$ is to the left of
$\psi_{L}^{X}(\Gamma_{x, \lambda_{1}})$. We deduce that
$\cup_{\lambda_{0} \in I_{\Lambda}^{\lambda}} \psi_{L}^{X}(\Gamma_{x, \lambda_{0}})$
is a simply connected subset whose boundary is the union of the curves
$\psi_{L}^{X}(\Gamma_{x, \lambda e^{-i \upsilon_{\Lambda}}})$ and
$\psi_{L}^{X}(\Gamma_{x, \lambda e^{i \upsilon_{\Lambda}}})$ (see figure (\ref{EVf4})).
Therefore $G(x)$ is a simply connected open set in ${\mathbb C}$.
Suppose that $\Re (X)$ points towards the interior of $H(0)$ at
$L_{i X}^{\epsilon}(0)$ without lack of generality. Then the curve
\begin{equation}
\label{equ:cur}
 \psi_{L}^{X}(\Gamma_{x, \lambda e^{-i \upsilon_{\Lambda}}}[0,\infty)) \cup
\psi_{L}^{X}(\Gamma_{x, \lambda e^{i \upsilon_{\Lambda}}}(-\infty,0])
\end{equation}
is in the boundary of $G(x)$. If
$H \in Reg_{1}(\epsilon, \aleph_{\Lambda, \lambda}X, I_{\Lambda}^{\lambda})$
or $x=0$ there are no other points in the boundary. Otherwise let $R$ be the element in
${\mathcal P}(H) \setminus \{L\}$. The boundary of $G(x)$ is the union
of the curve (\ref{equ:cur}) and a curve, passing through the point
$\psi_{L}^{X}(R_{i X}^{\epsilon}(x))$, whose description is analogous (see figure (\ref{EVf4})).

Since $G(x)$ is simply connected for any $x \in [0,\delta)  I_{\Lambda}^{\lambda}$
we have that
\[ (x, \psi_{L}^{X})^{-1} : G \to B(0,\delta) \times B(0,\epsilon) \]
is a well-defined  continuous function. It is holomorphic if we restrict the parameter $x$
to $(0,\delta)  \lambda e^{i(-\upsilon_{\Lambda}, \upsilon_{\Lambda})}$.
Consider the lift $\tilde{\varphi}$ of $\varphi$ to $G$
and $\tilde{\Delta}_{\varphi}=\Delta_{\varphi} \circ (x, \psi_{L}^{X})^{-1}$.
The diffeomorphism $\tilde{\varphi}$ is of the form
\[ (x,z) \mapsto (x,z+1+\tilde{\Delta}_{\varphi}(x,z)). \]
We define
\[ G^{L} =
\cup_{x \in [0,\delta)  I_{\Lambda}^{\lambda}, \
\lambda_{0} \in I_{\Lambda}^{\lambda}}
(\{x\} \times \psi_{L}^{X}(H_{\lambda_{0}}^{L}(x))) \]
Since $I_{\Lambda}^{\lambda}$ is a compact set
the proof of prop. \ref{pro:bddconf} implies
\begin{equation}
\label{equ:inegl}
 |\tilde{\Delta}_{\varphi}(x,z)|
\leq \frac{K}{(1+|z|)^{2}} \ \ \forall (x,z) \in G^{L}.
\end{equation}
More precisely, there are two main ingredients in the proof of prop. \ref{pro:bddconf},
namely we use that
$\partial H_{\lambda_{0}} \cap \partial {\mathcal B}$ is the union of a finite number of
asymptotically continuous sections for any basic set ${\mathcal B}$ and prop. \ref{pro:epsiext}.
The proof can be adapted since the asymptotically continuous sections
depend continuously on $\lambda_{0} \in I_{\Lambda}^{\lambda}$
and prop.  \ref{pro:epsiext} is still valid.

Fix $x_{0}=r_{0} \lambda_{0}$ with $r_{0} \in [0,\delta)$ and
$\lambda_{0} \in  I_{\Lambda}^{\lambda}$.
Suppose $\sharp {\mathcal P}(H)=1$ or $x_{0}=0$. Then we obtain $G^{L}(x_{0})=G(x_{0})$
and
\begin{equation}
\label{equ:reldp}
\lim_{z \in G(x_{0}), \ |z| \to \infty} \tilde{\Delta}_{\varphi}(x_{0},z) =0 .
\end{equation}
Suppose $\sharp {\mathcal P}(H)=2$ and $x_{0} \neq 0$. Let $R$ be the element in
${\mathcal P}(H) \setminus \{L\}$. We can define $G^{R}$ in an analogous way as $G^{L}$
We obtain
\[ |\Delta_{\varphi} \circ (x, \psi_{H,R}^{X})^{-1}(x,z')|
\leq \frac{K}{(1+|z'|)^{2}} \ \ \forall (x,z') \in G^{R}. \]
We have $\psi_{R}^{X} \circ (\psi_{L}^{X}(x_{0},y))^{-1} = z+ \tau$
for some $\tau \in {\mathbb C}$. Therefore we obtain
$|\tilde{\Delta}_{\varphi}|(x_{0},z) \leq K/(1+|z+\tau|)^{2}$ for any $z \in G^{R}(x_{0})$.
This inequality and (\ref{equ:inegl}) imply equation (\ref{equ:reldp}).
Consider the set $B_{\varphi} (L_{i X}^{\epsilon}(x_{0}))$ as defined in Step 5.
It is enclosed by
$\Gamma_{x_{0}, \lambda_{0}}$ and $\varphi(\Gamma_{x_{0}, \lambda_{0}})$.
Analogously consider the set $B_{\varphi} ^{\lambda_{1}}$ enclosed by
$\Gamma_{x_{0}, \lambda_{1}}$ and $\varphi(\Gamma_{x_{0}, \lambda_{1}})$.
Denote
$B_{\varphi}^{\flat}=
\cup_{\lambda_{1} \in I_{\Lambda}^{\lambda}} \psi_{L}^{X}(B_{\varphi} ^{\lambda_{1}})$.
By proceeding as in Step 5 we obtain that
the set $\psi_{L}^{X}(B_{\varphi} ^{\lambda_{1}} \setminus \varphi(\Gamma_{x_{0}, \lambda_{1}}))$
is a fundamental domain of
$\tilde{\varphi}_{|B_{\varphi}^{\flat}}$ for any
$\lambda_{1} \in I_{\Lambda}^{\lambda}$.
%
%
Therefore we can extend $\psi_{H_{\lambda_{1}},L}^{\varphi}$ by iteration to $G$
for any $\lambda_{1} \in I_{\Lambda}^{\lambda}$.
We deduce that $\psi_{H_{\lambda}, L}^{\varphi}$ is defined in $H^{L}$ and holomorphic in $H^{\circ}$.

{\bf Step 11}.
In order to complete the proofs of proposition \ref{pro:Fatpar} and corollary \ref{cor:Lavvf}
it suffices to show
$(\psi_{H,L}^{\varphi})_{|x=x_{0}} \equiv (\psi_{H_{\lambda}, L}^{\varphi})_{|x=x_{0}}$ for
$x_{0} = r_{0} \lambda_{0} \neq 0$.

By arguing as in Step 2 (see equation (\ref{equ:step22})) we obtain
$C_{4} \in {\mathbb R}^{+}$ such that
\[ B_{\varphi} ^{\lambda_{1}}  \cap
\left\{{ |Im (\psi_{L}^{X})| \leq \frac{C_{4}}{|x_{0}|^{\nu({\mathcal E}_{0})}} }\right\} =
B_{\varphi} ^{\lambda}  \cap
\left\{{ |Im (\psi_{L}^{X})| \leq \frac{C_{4}}{|x_{0}|^{\nu({\mathcal E}_{0})}} }\right\} \]
for any $\lambda_{1} \in I_{\Lambda}^{\lambda}$.
The constant $C_{4}$ is independent of $\lambda_{1}$ and $x_{0}$.

%
%
%
%
%
Consider $P \in B_{\varphi} ^{\lambda_{0}} \setminus B_{\varphi} ^{\lambda}$.
Denote $\tilde{z}= \psi_{L}^{X}(P)$. We obtain
$|Im (\tilde{z})| > C_{4}/|x_{0}|^{\nu({\mathcal E}_{0})}$. There exist $j \in {\mathbb Z}$ and
a orbit $\tilde{z}, \hdots, \tilde{\varphi}^{j}(\tilde{z})$ contained in
$B_{\varphi}^{\flat} \cap \{ |Im(z)| > C_{4} / |x|^{\nu({\mathcal E}_{0})} \}$
such that $\psi_{L}^{X}(\tilde{\varphi}^{j}(\tilde{z})) \in \psi_{L}^{X}(B_{\varphi} ^{\lambda})$.
We have
\[ \psi_{H_{\lambda},L}^{\varphi}(\tilde{z}) - \tilde{z} =
\psi_{H_{\lambda},L}^{\varphi}(\tilde{\varphi}^{j}(\tilde{z})) - \tilde{\varphi}^{j}(\tilde{z})
- \sum_{s=0}^{|j|-1} \tilde{\Delta}_{\varphi}(\tilde{\varphi}^{j+s}(\tilde{z})) \]
if $j<0$ and
\[ \psi_{H_{\lambda},L}^{\varphi}(\tilde{z}) - \tilde{z} =
\psi_{H_{\lambda},L}^{\varphi}(\tilde{\varphi}^{j}(\tilde{z})) - \tilde{\varphi}^{j}(\tilde{z})
+ \sum_{s=0}^{j-1} \tilde{\Delta}_{\varphi}(\tilde{\varphi}^{s}(\tilde{z})) \]
for $j>0$. Suppose $j>0$ without lack of generality.
Since $B_{\varphi}^{\flat}$ is contained in $G^{L}$ and
$|Im(\tilde{\varphi}^{s}(\tilde{z}))|> C_{4}/|x|^{\nu({\mathcal E}_{0})}$ the eq. (\ref{equ:inegl}) implies
that $|\tilde{\Delta}_{\varphi}(\tilde{\varphi}^{s}(\tilde{z}))| \leq K|x_{0}|^{2 \nu({\mathcal E}_{0})}/C_{4}^{2}$
for any $0 \leq s <j$. Thus $|\tilde{\Delta}_{\varphi}(\tilde{\varphi}^{s}(\tilde{z}))|$ is
as small as desired for $0 \leq s <j$ by considering a smaller $\delta>0$ if necessary.
By applying lemma \ref{lem:techsum} we obtain
\[ |(\psi_{H_{\lambda},L}^{\varphi}(\tilde{z}) - \tilde{z}) -
(\psi_{H_{\lambda},L}^{\varphi}(\tilde{\varphi}^{j}(\tilde{z})) - \tilde{\varphi}^{j}(\tilde{z}))|
\leq \frac{C_{5}}{1+|\tilde{z}|} \]
where $C_{5} \in {\mathbb R}^{+}$ does not depend on $\tilde{z}$ or $j$. Moreover we deduce
\[ |Im(\tilde{\varphi}^{j}(\tilde{z}))  - Im(\tilde{z})| \leq \frac{C_{5}}{1+|\tilde{z}|} . \]
By definition of $\psi_{H_{\lambda},L}^{\varphi}$ we get
\[ \lim_{\tilde{z} \in \psi_{L}^{X}(B_{\varphi} ^{\lambda_{0}}), \ Im(\tilde{z}) \to \pm \infty}
(\psi_{H_{\lambda},L}^{\varphi}(\tilde{\varphi}^{j}(\tilde{z})) - \tilde{\varphi}^{j}(\tilde{z}))=
\kappa_{\pm} \]
for $\kappa_{+}=0$ and some $\kappa_{-} \in {\mathbb C}$. We obtain
\[ \lim_{Q \in B_{\varphi} ^{\lambda_{0}}, \ Im(\psi_{L}^{X}(Q)) \to \pm \infty}
(\psi_{H_{\lambda},L}^{\varphi}(Q) - \psi_{H_{\lambda},L}^{X}(Q))=
\kappa_{\pm} . \]
The function
$(\psi_{H,L}^{\varphi} - \psi_{H_{\lambda},L}^{\varphi}) \circ (e^{2 \pi i \psi_{H,L}^{\varphi}(x_{0},y)})^{-1}(z)$
is holomorphic in ${\mathbb C}^{*}=e^{2 \pi i z}(B_{\varphi}^{\lambda_{0}})$
and extends continuously to $\pn{1}$. Moreover it takes the value $0$ at $0$.
Therefore we obtain $\psi_{H,L}^{\varphi} \equiv \psi_{H_{\lambda},L}^{\varphi}$.
The Fatou coordinate $\psi_{H,L}^{\varphi}$ is holomorphic in $H^{\circ}$.
\begin{rem}
Apparently the constructions depend on the choice of the $2$-convergent normal
form $\Upsilon={\rm exp}(X)$. Anyway, the polynomial vector fields associated to
compact-like sets are the same. Then ${\mathcal M}$ is independent of
$\Upsilon$ and so is $\aleph_{\Lambda,\lambda}$
for all $\Lambda \in {\mathcal M}$ and $\lambda \in {\mathbb S}^{1}$.
It is easy to see that the dynamical splitting $\digamma_{\Lambda}$
is independent of the choice of normal form too. Of course the regions
in $Reg (\epsilon, \aleph_{\Lambda,\lambda} X, I_{\Lambda}^{\lambda})$
depend on $X$ and then on $\Upsilon$. Anyway, the different Fatou coordinates
that we obtain in slightly different regions
can be interpreted as particularizations of the same object in an analogous way
as described in Steps 10 and 11.
\end{rem}
\subsection{Flow-convex sets}
This subsection is of technical importance for the results presented in subsections
\ref{subsec:rj1}, \ref{subsec:extfatcor} and \ref{subsec:asyfatcor}.
The language required to study
properties related to the shape of subregions
contained in the first exterior set is introduced below.
We use those properties
in order to analyze or extend Fatou coordinates.
\label{subsec:floconset}
\begin{defi}
Let us consider a domain $D \subset {\mathbb C}$ such that there exists a homeomorphism
$\sigma : D \to {\mathbb D}$  extending to a homeomorphism
$\tilde{\sigma}: \overline{D} \to \overline{{\mathbb D}}$.
Let $Z$ be a holomorphic vector
field defined in the neighborhood of $\overline{D}$ such that $D \cap Sing Z=\emptyset$.
We say that $(Z,D)$ is a regular pair.
\end{defi}
\begin{defi}
\label{def:flowcon}
Let $(Z,D)$ be a regular pair.
We say that $D$ is $\Re (Z)$-convex if
$\Gamma_{P}=\Gamma(Z, P, \overline{D})$ satisfies that
$\{ t \in {\mathcal I}(P): \Gamma_{P}(t) \in D \}$
is connected for any $P \in D$.
In other words we have
$\Gamma(Z, P, \overline{D}) \cap D = \Gamma(Z, P, D)$ for any $P \in D$.
\end{defi}
\begin{defi}
\label{def:transv}
Let $(Z,D)$ be a regular pair and $Q \in \partial{D} \setminus Sing Z$.
We say that $\Re(Z)$ is almost transversal to $\partial D$ at $Q$
if there exists $s_{0} \in {\mathbb R}^{+}$ such that either
${\rm exp}(s X)(Q) \not \in \overline{D}$ for any $s \in (0,s_{0})$ or
${\rm exp}(-s X)(Q) \not \in \overline{D}$ for any $s \in (0,s_{0})$.
If the curve $\partial D$ is smooth in the neighborhood of $Q$ then transversal
implies almost transversal.
\end{defi}
The following result is straightforward.
\begin{lem}
\label{lem:transv}
Let $(Z,D)$ be a regular pair and $P \in D$.
Suppose that $\Re(Z)$ is almost transversal to $\partial D$ at
the points in $\Gamma(Z, P, D)(\partial {\mathcal I}(\Gamma(Z, P, D)))$.
Then we obtain
${\mathcal I}(\Gamma(Z, P, \overline{D})) = \overline{{\mathcal I}(\Gamma(Z, P, D))}$.
In particular $D$ is $\Re (Z)$-convex if
$\Re(Z)$ is almost transversal to $\partial D$ at $Q$ for any
$Q \in \overline{D} \setminus Sing Z$.
\end{lem}
\begin{defi}
Since $\overline{D} \setminus Sing Z$ is simply connected there exists
a holomorphic Fatou coordinate $\psi$ of $Z$ defined in a neighborhood of $\overline{D} \setminus Sing Z$.
We say that $\psi$ is a Fatou coordinate of the pair $(Z,D)$.
\end{defi}
\begin{lem}
\label{lem:flcolp}
Consider a regular pair $(Z,D)$ such that $D$ is $\Re (Z)$-convex.
Let $\psi$ a Fatou coordinate of $(Z,D)$. Consider
continuous paths $\tau_{1},\tau_{2}:[0,1] \to D$ such that
\[ Im(\psi(\tau_{1}(s))) = Im(\psi(\tau_{2}(s))) \ \forall s \in [0,1] \ {\rm and} \
\tau_{1}(1)=\tau_{2}(1) . \]
Then $\tau_{2}(0)$ belongs to $\Gamma(Z,\tau_{1}(0), D)$.
\end{lem}
\begin{proof}
Consider $F= \{ s \in [0,1]:  \tau_{2}(s) \in \Gamma(Z,\tau_{1}(s), D) \}$.
It suffices to prove $F=[0,1]$. Suppose $F \neq [0,1]$.
Denote $s_{0} = \sup \{s \in [0,1] : s \not \in F \}$.
Since $F$ is open and contains the point $1$ we deduce that $s_{0} < 1$,
$s_{0} \not \in F$ and $s \in F$ for any $s \in (s_{0},1]$.
We obtain
\[ \tau_{2}(s) \in \Gamma(Z,\tau_{1}(s), D)(a(s))  \ \forall s \in (s_{0},1] \]
where $a(s) = \psi(\tau_{2}(s)) - \psi(\tau_{1}(s))$.
Therefore we have
\[ \tau_{2}(s_{0}) \in \Gamma(Z,\tau_{1}(s_{0}), \overline{D})(a(s_{0})). \]
As a consequence of definition \ref{def:flowcon} the point $\tau_{2}(s_{0})$ belongs to
$\Gamma(Z,\tau_{1}(s_{0}), D)$ and then $s_{0} \in F$. This is a contradiction.
We obtain $F=[0,1]$.
\end{proof}
\begin{pro}
\label{pro:Finj1}
Consider a regular pair $(Z,D)$ such that $D$ is $\Re (Z)$-convex.
Let $\psi$ a Fatou coordinate of $(Z,D)$ and $P \in D$.
Then there exists a continuous path
$\gamma_{Q}:[0,1] \to D$ with $\gamma_{Q}(0)=P$, $\gamma_{Q}(1)=Q$ such that
$Im(\psi \circ \gamma_{Q}):[0,1] \to {\mathbb R}$ is injective for any $Q \in D \setminus \Gamma(Z,P,D)$.
\end{pro}
\begin{proof}
We define the set $E$ of points $Q \in D$ satisfying that
there exists a continuous path
$\gamma_{Q}:[0,1] \to D$ with $\gamma_{Q}(0)=P$, $\gamma_{Q}(1)=Q$ such that
$Im(\psi \circ \gamma_{Q}):[0,1] \to {\mathbb R}$ is either an injective or a constant function.
A point $Q \in E$ belongs to $\Gamma(Z,P,D)$
if $Im(\psi \circ \gamma_{Q})$ is a constant function.
It suffices to prove that $E=D$.

Let $\tau_{1},\tau_{2}:[0,1] \to D$ be continuous paths such that
$\tau_{1}(0)=P$, $\tau_{1}(1)=\tau_{2}(0)$ and $\tau_{2}(1)=Q'$.
We claim that if $Im(\psi \circ \tau_{1})$ is constant and
$Im(\psi \circ \tau_{2})$ is injective or $Im(\psi \circ \tau_{1})$ is injective and
$Im(\psi \circ \tau_{2})$ is constant then $Q'$ belongs to $E$.
For instance in the latter case we consider $a \in (0,1)$ near $1$ and we define
a path
$\tau:[0,1] \to D$ such that $\tau(s)=\tau_{1}(s)$ for $s \in [0,a]$ and
$\psi(\tau((1-s)a+s))= (1-s) \psi(\tau_{1}(a)) + s \psi(Q')$ for $s \in [0,1]$.
Clearly $Im(\psi \circ \tau)$ is injective.

Given $Q \in E$ any point $Q'$ in a neighborhood of $Q$ satisfies either $Q' \in \Gamma(Z,P,D)$
or we can find paths $\tau_{1}, \tau_{2}$ as in the previous paragraph.
As a consequence the set $E$ is open in $D$. Since $D$ is connected it
suffices to prove that $E$ is closed in $D$.

Let $Q \in \overline{E} \cap D$.
Consider a sequence $Q_{n} \to Q$ of points in $E \cap D$.
Choose $n>>1$. We can choose
a continuous path $\tilde{\gamma}_{n}:[0,1] \to D$
with $\tilde{\gamma}_{n}(0)=Q$, $\tilde{\gamma}_{n}(1)=Q_{n}$ such that
$Im (\psi \circ \tilde{\gamma}_{n})$ is injective or constant.
We can suppose that $Im(\psi(Q_{n}))$ does not belong to the closed interval whose
ends are $Im(\psi(P))$ and $Im(\psi(Q))$. Otherwise the argument in the second paragraph
implies $Q \in E$. Without lack of generality we can suppose
$Im(\psi(Q_{n})) > \max (Im(\psi(P)), Im(\psi(Q)))$. We obtain
$Im(\psi(Q)) \geq Im(\psi(P))$.

Suppose $Im(\psi(P)) = Im(\psi(Q))$. Lemma \ref{lem:flcolp}
implies $Q \in \Gamma(Z,P,D)$ and then $Q \in E$. Now suppose
$Im(\psi(P)) < Im(\psi(Q))$. There exists $P' \in \gamma_{Q_{n}}[0,1]$
such that $Im(\psi(P')) = Im(\psi(Q))$. We obtain $Q \in \Gamma(Z,P',D)$
by lemma \ref{lem:flcolp} and then
$Q \in E$ (see second paragraph). The set $E$ is closed in $D$ as we wanted to prove.
\end{proof}
\begin{cor}
\label{cor:flconcl}
Consider a regular pair $(Z,D)$ such that $D$ is $\Re (Z)$-convex.
Let $\psi$ a Fatou coordinate of $(Z,D)$ and $P,Q \in \overline{D} \setminus Sing Z$.
Then $Im(\psi(P))=Im(\psi(Q))$ implies that $Q \in \Gamma(Z, P, \overline{D})$.
In particular $\psi$ is injective in $\overline{D} \setminus Sing Z$.
\end{cor}
\begin{proof}
Proposition \ref{pro:Finj1} implies that $\psi$ is injective in $D$.
Denote $E= Im(\psi)(D)$ and $E'=Im(\psi)(\overline{D} \setminus Sing Z)$.
We have $E' \subset \overline{E}$. The set $E$ is an open interval.

Suppose $Im(\psi(P)) \in E$.
There exists $P' \in D$ such that $Im(\psi(P'))=Im(\psi(P))$.
There exists a sequence $P_{n} \to P$ of points in $D$. We choose a sequence
$P_{n}' \to P'$ of points in $D$ such that $Im(\psi(P_{n}'))= Im(\psi(P_{n}))$
for any $n \in {\mathbb N}$. Proposition \ref{pro:Finj1}  implies
\[ P_{n} \in \Gamma(Z,P_{n}',D)(a_{n}) \ \forall n \in {\mathbb N} \]
where $a_{n}=\psi(P_{n}) -\psi(P_{n}')$. We obtain
$P \in \Gamma(Z,P',\overline{D})(\psi(P) -\psi(P'))$. Analogously
$Q$ belongs to $\Gamma(Z,P',\overline{D})$.
We deduce $Q = \Gamma(Z,P,\overline{D})(\psi(Q) -\psi(P))$.

Suppose $Im(\psi(P)) \in E' \setminus E$. Then $Im(\psi(P))$ is either the supremum
or the infimum or $Im(\psi)$ at $\overline{D} \setminus Sing Z$. Suppose without lack
of generality that we are in the former case. There exists a continuous path
$\gamma:[0,1] \to \overline{D}$ such that $\gamma[0,1) \subset D$ and $\gamma(1)=P$.
We deduce that
$Im(\psi)(\gamma[0,1))= [Im(\psi(P))-a, Im(\psi(P)))$ for some $a \in {\mathbb R}^{+}$.
An analogous property holds true for $Q$. Therefore there exist sequences
$P_{n} \to P$ and $Q_{n} \to Q$ of points in $D$ such that $Im(\psi(P_{n}))= Im(\psi(Q_{n}))$
for any $n \in {\mathbb N}$. Since
$Q_{n} \in \Gamma(Z,P_{n},D)(\psi(Q_{n}) -\psi(P_{n}))$ for any $n \in {\mathbb N}$ we obtain
$Q \in \Gamma(Z,P,\overline{D})(\psi(Q) -\psi(P))$.
\end{proof}
\subsection{Comparing multi-transversal flows}
\label{subsec:comdifmflo}
Let $\varphi \in \diff{tp1}{2}$ with $2$-convergent normal form $\Upsilon={\rm exp}(X)$.
Consider $\Lambda,\Lambda' \in {\mathcal M}$ and
the dynamical splittings $\digamma_{\Lambda}$ and $\digamma_{\Lambda'}$ in remark \ref{rem:unifspl}.
Let $\lambda, \lambda' \in {\mathbb S}^{1}$.
\begin{defi}
\label{def:hjlam}
Let $X \in \Xt$, $\Lambda \in {\mathcal M}$ and
$\lambda \in {\mathbb S}^{1}$.
We define $H_{\Lambda,j}^{\lambda}$ the element in
$Reg(\epsilon, \aleph_{\Lambda, \lambda} X, I_{\Lambda}^{\lambda})$
such that $L_{j} \subset H_{\Lambda,j}^{\lambda}$ (see def. \ref{def:lj}). We denote
$\tilde{\psi}_{j,\Lambda,\lambda}^{\varphi}= \psi_{H_{\Lambda,j}^{\lambda},L_{j}}^{\varphi}$
or just $\tilde{\psi}_{j,\lambda}^{\varphi}$ if $\Lambda$ is implicit.
\end{defi}
\begin{defi}
Let $X \in \Xt$, $\Lambda, \Lambda' \in {\mathcal M}$ and
$\lambda, \lambda' \in {\mathbb S}^{1}$. We denote $d_{\Lambda,\Lambda'}^{\lambda,\lambda'}=0$
if $\Lambda \neq \Lambda'$. We denote
\[ d_{\Lambda,\Lambda}^{\lambda,\lambda'}=\min (d_{\Lambda}^{\lambda}, d_{\Lambda}^{\lambda'}) \ {\rm and} \
I_{\Lambda,\Lambda'}^{\lambda,\lambda'}=
I_{\Lambda}^{\lambda} \cap I_{\Lambda'}^{\lambda'},\]
see def. \ref{def:dlam}.
Denote $R_{\Lambda,\Lambda',j}^{\lambda,\lambda'}=2$ if
$H_{\Lambda,j}^{\lambda} \in
Reg_{2}(\epsilon, \aleph_{\Lambda, \lambda} X, I_{\Lambda}^{\lambda})$
and $e(H) \leq \tilde{e}_{d_{\Lambda,\Lambda'}^{\lambda,\lambda'}}$
(see definitions \ref{def:levels} and \ref{def:eh}).
Otherwise we define $R_{\Lambda,\Lambda',j}^{\lambda,\lambda'}=1$.
\end{defi}
\begin{rem}
The equation $R_{\Lambda,\Lambda',j}^{\lambda,\lambda'}=2$
implies
$H_{\Lambda,j}^{\lambda} \in
Reg_{2}(\epsilon, \aleph_{\Lambda, \lambda} X, I_{\Lambda}^{\lambda})$
and
$H_{\Lambda',j}^{\lambda'} \in
Reg_{2}(\epsilon, \aleph_{\Lambda', \lambda'} X, I_{\Lambda'}^{\lambda'})$.
Moreover, we have
\[
\partial {\mathcal E}_{0} \cap \overline{H_{\Lambda,j}^{\lambda}(x)}=
\partial {\mathcal E}_{0} \cap \overline{H_{\Lambda',j}^{\lambda'}(x)}=
\{ T_{iX}^{\epsilon,j} (x) ,  T_{iX}^{\epsilon,k} (x) \}
\]
for some $k \in {\mathbb Z}/(2 \nu({\mathcal E}_{0}) {\mathbb Z})$ and any
$x \in [0,\delta) ( I_{\Lambda}^{\lambda} \cap  I_{\Lambda'}^{\lambda'})$.
The equation $R_{\Lambda,\Lambda',j}^{\lambda,\lambda'}=1$ is less
restrictive, for instance $d_{\Lambda,\Lambda'}^{\lambda,\lambda'}=0$
implies $R_{\Lambda,\Lambda',j}^{\lambda,\lambda'}=1$
for any $j \in {\mathbb Z}/(2 \nu({\mathcal E}_{0}) {\mathbb Z})$.
\end{rem}
\begin{figure}[h]
\begin{center}
\includegraphics[height=8cm,width=7cm]{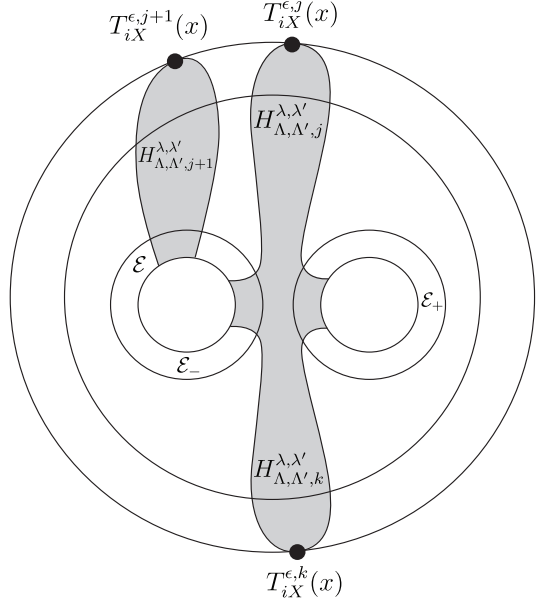}
\end{center}
\caption{Examples: $R_{\Lambda,\Lambda',j+1}^{\lambda,\lambda'}=1$ and
$R_{\Lambda,\Lambda',j}^{\lambda,\lambda'}=2$,
$H_{\Lambda,\Lambda',j}^{\lambda,\lambda'} =H_{\Lambda,\Lambda',k}^{\lambda,\lambda'} $}
\label{EVf5}
\end{figure}
We want to estimate
$\tilde{\psi}_{j,\Lambda,\lambda}^{\varphi} - \tilde{\psi}_{j,\Lambda',\lambda'}^{\varphi}$.
It is defined in  $H_{\Lambda,j}^{\lambda} \cap H_{\Lambda',j}^{\lambda'}$
but this set is not well suited to work with since for instance it is not
necessarily connected. We work with the set
$H_{\Lambda,\Lambda',j}^{\lambda,\lambda'} $ (see def. \ref{def:widint}). It is contained
in $H_{\Lambda,j}^{\lambda} \cap H_{\Lambda',j}^{\lambda'}$ and
the restrictions of $\Re (\aleph_{\Lambda, \lambda}^{*} X)$ and
$\Re (\aleph_{\Lambda', \lambda'}^{*} X)$ to $H_{\Lambda,\Lambda',j}^{\lambda,\lambda'} $
coincide.
Roughly speaking the orbit space of $\varphi$ restricted to
$H_{\Lambda,\Lambda',j}^{\lambda,\lambda'} (x_{0})$ is biholomorphic to the subset
$B(0,\kappa_{-}(x_{0})) \setminus \overline{B}(0,\kappa_{+}(x_{0}))$ of
${\mathbb C}^{*}$. The idea is studying the subregions of both
 $H_{\Lambda,j}^{\lambda} $ and $H_{\Lambda',j}^{\lambda'}$
 contained in $H_{\Lambda,j}^{\lambda} \cap H_{\Lambda',j}^{\lambda'}$
 to prove that $\kappa_{-}$ and $\kappa_{+}$ are exponentially small of the right order.
\begin{defi}
Let $En_{\Lambda,\Lambda', j}^{\lambda, \lambda'}$ be the set of basic sets ${\mathcal B}$ of
$\digamma_{\Lambda} \cup \digamma_{\Lambda'}$ (see def. \ref{def:refin})
such that
$e({\mathcal B}) \leq \iota({\mathcal B})
< \tilde{e}_{d_{\Lambda,\Lambda'}^{\lambda,\lambda'}+1}$
and $H_{\Lambda,j}^{\lambda} \cap {\mathcal B} \neq \emptyset$.
\end{defi}
We have $\aleph_{\Lambda,\lambda}^{*} \equiv \aleph_{\Lambda',\lambda'}^{*}$
in $\cup_{(r,\lambda_{0}) \in [0,\delta) \times I_{\Lambda,\Lambda'}^{\lambda,\lambda'}}
{\mathcal B}'(r,\lambda_{0})$
for any ${\mathcal B}' \in En_{\Lambda,\Lambda' j}^{\lambda, \lambda'}$
by definition. We obtain
\[ (H_{\Lambda,j}^{\lambda})^{L_{j}} \cap {\mathcal B}(r,\lambda_{0}) = (H_{\Lambda',j}^{\lambda'})^{L_{j}} \cap
{\mathcal B}(r,\lambda_{0}) \ \forall (r,\lambda_{0}) \in [0,\delta) \times I_{\Lambda,\Lambda'}^{\lambda,\lambda'}
\ \forall {\mathcal B} \in En_{\Lambda,\Lambda',j}^{\lambda,\lambda'}. \]
We deduce that
$En_{\Lambda,\Lambda',j}^{\lambda,\lambda'} \subset En_{\Lambda',\Lambda, j}^{\lambda', \lambda}$.
We can permute the roles of $\lambda$ and $\lambda'$ to get
$En_{\Lambda,\Lambda',j}^{\lambda,\lambda'} = En_{\Lambda',\Lambda, j}^{\lambda', \lambda}$.
\begin{defi}
Let $Ex_{\Lambda,\Lambda',j}^{\lambda,\lambda'}$ be the set of basic sets ${\mathcal B}$ of
$\digamma_{\Lambda} \cup \digamma_{\Lambda'}$
satisfying the properties $H_{\Lambda,j}^{\lambda} \cap {\mathcal B} \neq \emptyset$ and
$e({\mathcal B})  < \tilde{e}_{d_{\Lambda,\Lambda'}^{\lambda,\lambda'}+1}  \leq
 \iota({\mathcal B})$.
Since $e({\mathcal B}) < \iota({\mathcal B})$ the elements of $Ex_{\Lambda,\Lambda',j}^{\lambda,\lambda'}$
are non-terminal exterior sets.
\end{defi}
Analogously as in the previous paragraph we obtain
\[ (H_{\Lambda,j}^{\lambda})^{L_{j}} \cap \tilde{\mathcal E}(r,\lambda_{0}) = (H_{\Lambda',j}^{\lambda'})^{L_{j}} \cap
\tilde{\mathcal E}(r,\lambda_{0}) \ \forall (r,\lambda_{0}) \in [0,\delta) \times I_{\Lambda,\Lambda'}^{\lambda,\lambda'}
\ \forall {\mathcal E} \in Ex_{\Lambda,\Lambda',j}^{\lambda,\lambda'}, \]
where $\tilde{\mathcal E}$ is defined in section \ref{subsec:dynsplit}.
We have $\sharp Ex_{\Lambda,\Lambda',j}^{\lambda,\lambda'} \leq 2$.
Moreover we obtain
$Ex_{\Lambda,\Lambda',j}^{\lambda,\lambda'} \subset
Ex_{\Lambda', \Lambda, j}^{\lambda', \lambda}$
and then $Ex_{\Lambda,\Lambda',j}^{\lambda,\lambda'} = Ex_{\Lambda',\Lambda,j}^{\lambda', \lambda}$.
\begin{defi}
\label{def:widint}
Let $X \in \Xt$, $\Lambda,\Lambda' \in {\mathcal M}$ and
$\lambda, \lambda' \in {\mathbb S}^{1}$.
We define
\[ H_{\Lambda,\Lambda',j}^{\lambda,\lambda'} = [(H_{\Lambda,j}^{\lambda})^{L_{j}} \cap
(\cup_{{\mathcal B} \in En_{\Lambda,\Lambda',j}^{\lambda,\lambda'}} {\mathcal B}  \cup
\cup_{{\mathcal E} \in Ex_{\Lambda,\Lambda',j}^{\lambda,\lambda'}} \tilde{\mathcal E})]
\cap \{ x \in [0,\delta) I_{\Lambda,\Lambda'}^{\lambda,\lambda'} \}  \]
if $R_{\Lambda,\Lambda',j}^{\lambda,\lambda'}=1$. We define
\[ H_{\Lambda,\Lambda',j}^{\lambda,\lambda'} = [H_{\Lambda,j}^{\lambda} \cap
(\cup_{{\mathcal B} \in En_{\Lambda,\Lambda',j}^{\lambda,\lambda'}} {\mathcal B}  \cup
\cup_{{\mathcal E} \in Ex_{\Lambda,\Lambda',j}^{\lambda,\lambda'}} \tilde{\mathcal E})] \cap
\{ x \in [0,\delta) I_{\Lambda,\Lambda'}^{\lambda,\lambda'} \} \]
for $R_{\Lambda,\Lambda',j}^{\lambda,\lambda'}=2$ (see figure (\ref{EVf5})).
\end{defi}
We have $H_{\Lambda,\Lambda',j}^{\lambda,\lambda'} \subset H_{\Lambda,j}^{\lambda} \cap H_{\Lambda',j}^{\lambda'}$
and $\Re (\aleph_{\Lambda,\lambda}X) \equiv  \Re (\aleph_{\Lambda',\lambda'}X)$
in $H_{\Lambda,\Lambda',j}^{\lambda,\lambda'}$.
\begin{defi}
\label{def:exexr}
Let $X \in \Xt$, $\Lambda,\Lambda' \in {\mathcal M}$ and
$\lambda, \lambda' \in {\mathbb S}^{1}$. Denote
$\Gamma_{x}=\Gamma(\aleph_{\Lambda,\lambda}X ,T_{i X}^{\epsilon,j}(x),T_{0})$
(see def. \ref{def:lj}).
Either $\Gamma_{x}[0,\infty)$ intersects
an element ${\mathcal E}_{+}$ of $Ex_{\Lambda,\Lambda',j}^{\lambda,\lambda'}$ for any
$x \in (0,\delta) I_{\Lambda,\Lambda'}^{\lambda,\lambda'}$ or $\Gamma_{x} \cap {\mathcal E}=\emptyset$
for all ${\mathcal E} \in Ex_{\Lambda,\Lambda',j}^{\lambda,\lambda'}$ and
$x \in (0,\delta) I_{\Lambda,\Lambda'}^{\lambda,\lambda'}$. In the former case we define
$e_{\Lambda,\Lambda',j}^{\lambda,\lambda',+}=\iota({\mathcal E}_{+})$. Otherwise we define
$e_{\Lambda,\Lambda',j}^{\lambda,\lambda',+}= \infty$. Analogously we define
${\mathcal E}_{-}$ and $e_{\Lambda,\Lambda',j}^{\lambda,\lambda',-}$ by considering $\Gamma_{x}(-\infty,0]$ (see figure (\ref{EVf5})).
\end{defi}
\begin{defi}
\label{def:dirvf}
Let $\varphi \in \diff{tp1}{2}$ with $2$-convergent normal form ${\rm exp}(X)$.
We define ${\mathcal D}(\varphi)={\mathbb Z}/(2 \nu({\mathcal E}_{0}) {\mathbb Z})$.
We define
\[ {\mathcal D}_{1}(\varphi)=\{ j \in {\mathbb Z}/(2 \nu({\mathcal E}_{0}) {\mathbb Z}) :
\Re (X) \ {\rm points \ towards} \ B(0,\delta) \times B(0,\epsilon) \ {\rm at} \ T_{iX}^{\epsilon,j}(0) \}  \]
and ${\mathcal D}_{-1}(\varphi) =  {\mathcal D}(\varphi) \setminus  {\mathcal D}_{1}(\varphi)$.
\end{defi}
\begin{defi}
Let $\varphi \in \diff{tp1}{2}$ with $2$-convergent normal form ${\rm exp}(X)$.
Let $\Lambda,\Lambda' \in {\mathcal M}$ and $\lambda, \lambda' \in {\mathbb S}^{1}$. We define
\[ \psi_{j,\lambda,\lambda'}^{\varphi}=
(\tilde{\psi}_{j,\Lambda,\lambda}^{\varphi}- \tilde{\psi}_{j,\Lambda',\lambda'}^{\varphi})
\circ (x, e^{2 \pi i \tilde{\psi}_{j,\Lambda,\lambda}^{\varphi}})^{-1}. \]
\end{defi}
The values of $\Lambda$ and $\Lambda'$ in $\psi_{j,\lambda,\lambda'}^{\varphi}$ are implicit.
We want to prove that
$\psi_{j,\lambda,\lambda'}^{\varphi}$ is defined in the space of orbits of
$\varphi_{|H_{\Lambda,\Lambda',j}^{\lambda,\lambda'}}$ in order to estimate
$\tilde{\psi}_{j,\lambda}^{\varphi}- \tilde{\psi}_{j,\lambda'}^{\varphi}$ in $H_{\Lambda,\Lambda',j}^{\lambda,\lambda'}$.
\begin{pro}
\label{pro:flasp}
Let $\varphi \in \diff{tp1}{2}$ with $2$-convergent normal form ${\rm exp}(X)$.
Let $\Lambda,\Lambda' \in {\mathcal M}$. Consider
$\lambda, \lambda' \in {\mathbb S}^{1}$ and $j \in {\mathcal D}(\varphi)$.
Then we have
\[ (x, e^{2 \pi i \tilde{\psi}_{j,\Lambda,\lambda}^{\varphi}})(H_{\Lambda,\Lambda',j}^{\lambda,\lambda'}) \subset
\cup_{(r,\tilde{\lambda}) \in [0,\delta) \times I_{\Lambda,\Lambda'}^{\lambda,\lambda'}} (\{r \tilde{\lambda} \} \times
[B(0,\kappa_{-}(r,\tilde{\lambda})) \setminus \overline{B}(0,\kappa_{+}(r,\tilde{\lambda}))]) \]
where $\kappa_{\pm}(r,\tilde{\lambda}) \equiv e^{\mp C_{\pm}(\lambda)/|r|^{e_{\Lambda,\Lambda',j}^{\lambda,\lambda',\pm}}}$
for some continuous function $C_{\pm}:I_{\Lambda,\Lambda'}^{\lambda,\lambda'} \to {\mathbb R}^{+}$.
Moreover there exists $\zeta \in {\mathbb R}^{+}$ such that $\psi_{j,\lambda,\lambda'}^{\varphi}$ is well
defined in
\[ \cup_{(r,\tilde{\lambda}) \in [0,\delta) \times I_{\Lambda,\Lambda'}^{\lambda,\lambda'}} (\{r \tilde{\lambda}\} \times
[B(0,\tilde{\kappa}_{-}(r,\tilde{\lambda})) \setminus \overline{B}(0,\tilde{\kappa}_{+}(r,\tilde{\lambda}))]) \]
and holomorphic outside $x=0$ where
$\tilde{\kappa}_{\pm}(r,\tilde{\lambda}) = e^{\mp (C_{\pm}(\tilde{\lambda})+\zeta)/|r|^{e_{\Lambda,\Lambda',j}^{\lambda,\lambda',\pm}}}$.
\end{pro}
Let us remark that in the previous proposition
$\tilde{\kappa}_{+} \equiv \kappa_{+} \equiv 0$ if $e_{\Lambda,\Lambda',j}^{\lambda,\lambda',+}=\infty$
and $\tilde{\kappa}_{-} \equiv \kappa_{-} \equiv \infty$ if $e_{\Lambda,\Lambda',j}^{\lambda,\lambda',-}=\infty$.

The subsections \ref{subsec:rj1} and \ref{subsec:rj2} are intended to prove proposition \ref{pro:flasp}
in the cases $R_{\Lambda,\Lambda',j}^{\lambda,\lambda'}=1$ and $R_{\Lambda,\Lambda',j}^{\lambda,\lambda'}=2$ respectively.
The subsection \ref{subsec:rj1} is though a little more ambitious since it introduces the setup that
will be used in subsections \ref{subsec:extfatcor} and \ref{subsec:asyfatcor}.
\subsubsection{Proof of proposition \ref{pro:flasp} in the case $R_{\Lambda,\Lambda',j}^{\lambda,\lambda'}=2$}
\label{subsec:rj2}
The exterior sets ${\mathcal E}_{-}$ and ${\mathcal E}_{+}$ (see def. \ref{def:exexr}) are different if
$e_{\Lambda,\Lambda',j}^{\lambda,\lambda',-}  \neq \infty$ and
$e_{\Lambda,\Lambda',j}^{\lambda,\lambda',+}  \neq \infty$.
Thus the cardinal of the set $Ex_{\Lambda,\Lambda',j}^{\lambda,\lambda'}$ coincides with the cardinal of
$\{e_{\Lambda,\Lambda',j}^{\lambda,\lambda',-}, e_{\Lambda,\Lambda',j}^{\lambda,\lambda',+} \} \setminus \{ \infty \}$.

Denote $\psi = \psi_{H_{\Lambda,j}^{\lambda},L_{j}}^{X}$, $\aleph=\aleph_{\Lambda,\lambda}$ and
$\Gamma_{x}=\Gamma(\aleph_{\Lambda,\lambda}X ,T_{i X}^{\epsilon,j}(x),T_{0})$.
Suppose that $\Re (X)$ points towards $B(0,\delta) \times B(0,\epsilon)$ at $T_{i X}^{\epsilon,j}(0)$
without lack of generality.

Suppose that $e_{\Lambda,\Lambda',j}^{\lambda,\lambda',+} \neq \infty$. We have
\[ {\mathcal E}_{\beta}= {\mathcal E}_{+}=
\{(x,t) \in B(0,\delta) \times {\mathbb C} :  \eta_{+} \geq |t| \geq \rho_{+} |x| \} . \]
We define
\[ {\mathcal E}_{+}'= \{(x,t) \in B(0,\delta) \times {\mathbb C} :  \eta_{+} \geq |t| \geq (2-1/4) \rho_{+} |x| \} .\]
The construction of multi-transversal flows implies that
$\Re (\aleph_{\Lambda,\lambda}X) \equiv  \Re (\aleph_{\Lambda',\lambda'}X)$ in
${\mathcal E}_{+}' \cap \{ x \in (0,\delta) I_{\Lambda,\Lambda'}^{\lambda,\lambda'}\}$.
Given $k \in \{1,2\}$ we consider the continuous section
$\tau_{k,+}: (0,\delta) I_{\Lambda,\Lambda'}^{\lambda,\lambda'} \to {\mathcal E}_{+}$ defined by
\[ \tau_{k,+}(x) = \Gamma_{x}[0,\infty) \cap \{ |t|= (2-k/16) |x| \} . \]
\begin{figure}[h]
\begin{center}
\includegraphics[height=6cm,width=13cm]{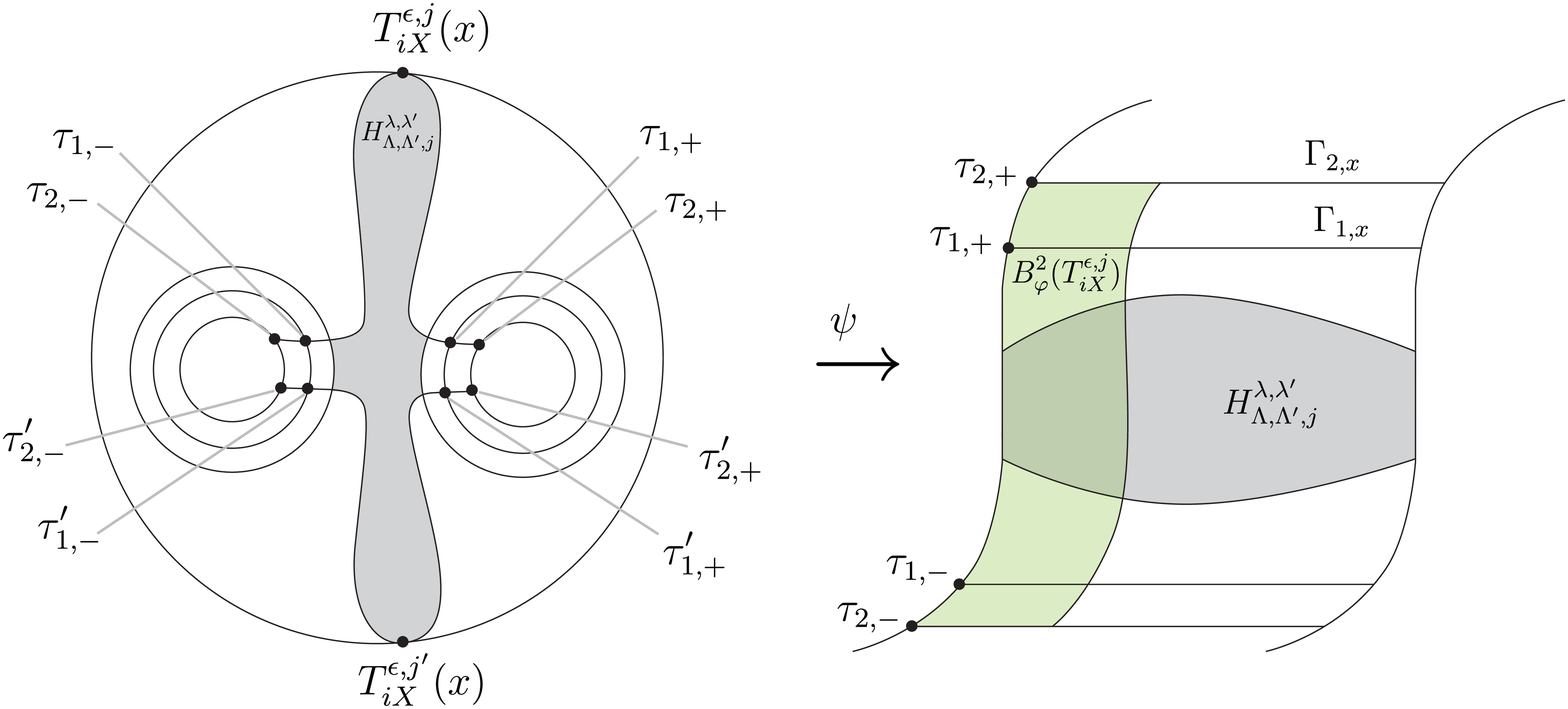}
\end{center}
\caption{Case $R_{\Lambda,\Lambda',j}^{\lambda,\lambda'}=2$}
\label{EVf6}
\end{figure}
We can define $\tau_{1,-}$ and $\tau_{2,-}$ by replacing $\Gamma_{x}[0,\infty)$
with $\Gamma_{x}(-\infty,0]$ if $e_{\Lambda,\Lambda',j}^{\lambda,\lambda',-} \neq \infty$.
We claim that
$H_{\Lambda,\Lambda',j}^{\lambda,\lambda'}(r,\tilde{\lambda}) \cap \{ |t|= (2-k/16) |x| \}$
adheres to the point $\tau_{k,+}(0,\lambda_{0})$ when $(r,\tilde{\lambda}) \to (0,\lambda_{0})$
for $k \in \{1,2\}$.
Let us study the properties of $\tau_{1,+}$ and $\tau_{2,+}$. Those of
$\tau_{1,-}$ and $\tau_{2,-}$ are analogous.
By applying several times prop. \ref{pro:hit} we obtain that
$\tau_{1,+}$ and $\tau_{2,+}$ are asymptotically continuous.
The points
$\tau_{1,+}(0,\tilde{\lambda})$ and $\tau_{2,+}(0,\tilde{\lambda})$
belong to the same element $\gamma_{\tilde{\lambda}}$ of
$Tr_{\leftarrow \infty} (\aleph_{{\mathcal E}_{+}} X_{\beta}(\tilde{\lambda}))$
for any $\tilde{\lambda} \in I_{\Lambda,\Lambda'}^{\lambda,\lambda'}$ (see def. \ref{def:pol}).
Consider the tangent section $T_{i X}^{\epsilon,j'}(x)$ containing the other
point in $T_{iX}^{\epsilon}(x) \cap \overline{H_{\Lambda,j}^{\lambda}}$.
Then we can replace $j$ with $j'$ to obtain asymptotically continuous sections
$\tau_{1,+}'(r,\tilde{\lambda})$ and
$\tau_{2,+}'(r,\tilde{\lambda})$. Corollary \ref{cor:hit} implies
$\tau_{k,+}'(0,\tilde{\lambda})=\tau_{k,+}(0,\tilde{\lambda})$
for all $\tilde{\lambda} \in I_{\Lambda,\Lambda'}^{\lambda,\lambda'}$ and $k \in \{1,2\}$.
Thus $H_{\Lambda,\Lambda',j}^{\lambda,\lambda'}(r,\tilde{\lambda}) \cap \{ |t|= (2-k/16) |x| \}$
adheres to the point $\tau_{k,+}(0,\lambda_{0})$ when $(r,\tilde{\lambda}) \to (0,\lambda_{0})$.

We define
\[ \Gamma_{k,x}= \Gamma(X, \tau_{k,+}(x), \overline{H_{\Lambda,j}^{\lambda}})
\ {\rm for} \ k \in \{1,2\}
\ {\rm and} \ x \in (0,\delta) I_{\Lambda,\Lambda'}^{\lambda,\lambda'} . \]
The vector field $\Re (X_{\beta}(\lambda_{0}))$ is transversal to $\gamma_{\lambda_{0}}$ at
$\tau_{1,+}(0,\lambda_{0})$ and $\tau_{2,+}(0,\lambda_{0})$.
Hence $\Gamma_{k,r,\tilde{\lambda}}$ tends to $\tau_{k,+}(0,\lambda_{0})$
when $(r,\tilde{\lambda}) \to (0,\lambda_{0})$ for all $\lambda_{0} \in I_{\Lambda,\Lambda'}^{\lambda,\lambda'}$ and
$k \in \{1,2\}$. Hence the trajectories $\Gamma_{1,x}$ and $\Gamma_{2,x}$ are contained in
${\mathcal E}_{+}' \setminus \tilde{\mathcal E}_{+}$ and then in
$H_{\Lambda,j}^{\lambda} \cap H_{\Lambda',j}^{\lambda'}$
for any $x \in (0,\delta) I_{\Lambda,\Lambda'}^{\lambda,\lambda'}$. Given $k \in \{1,2\}$ we define
\[  H_{j,k}^{\lambda_{0}}(x) = H_{\Lambda,j}^{\lambda_{0}}(x) \cap
\{ Im (\psi ) \in (Im [\psi (\tau_{k,-}(x))], Im [\psi (\tau_{k,+}(x))]) \}  \]
for $\lambda_{0} \in \{\lambda,\lambda'\}$ and $x \in (0,\delta) I_{\Lambda,\Lambda'}^{\lambda,\lambda'}$.
We define $Im (\psi \circ \tau_{k,\pm})= \pm \infty$ if
$e_{\Lambda,\Lambda',j}^{\lambda,\lambda',\pm} = \infty$.
The previous discussion implies
$H_{\Lambda,\Lambda',j}^{\lambda,\lambda'} \subset H_{j,k}^{\lambda} = H_{j,k}^{\lambda'}
\subset H_{\Lambda,j}^{\lambda} \cap H_{\Lambda',j}^{\lambda'}$
for $k \in \{1,2\}$.

We want to study the sets
$Im (\tilde{\psi}_{j,\lambda}^{\varphi})(H_{\Lambda,\Lambda',j}^{\lambda,\lambda'}
(r,\lambda_{0}))$
for $(r,\lambda_{0}) \in [0,\delta) \times I_{\Lambda,\Lambda'}^{\lambda,\lambda'}$ .
Let $\psi_{\lambda_{0}}^{\beta}$ be a Fatou coordinate
of $X_{\beta}(\lambda_{0})$ defined in the neighborhood of $\gamma_{\lambda_{0}}$.
Given $k \in \{1,2\}$, $l \in \{+,-\}$ consider the function
$F_{k,l}: (0,\delta) \times I_{\Lambda,\Lambda'}^{\lambda,\lambda'} \to {\mathbb C}$
defined by
\[ F_{k,l}(r,\tilde{\lambda})= |r|^{\iota({\mathcal E}_{l})} \psi (\tau_{k,l}(r,\tilde{\lambda})) . \]
It extends continuously to $[0,\delta) \times I_{\Lambda,\Lambda'}^{\lambda,\lambda'}$
by cor. \ref{cor:intbou}. Moreover
$F_{k,+}(\{0\} \times I_{\Lambda,\Lambda'}^{\lambda,\lambda'}) \subset {\mathbb H}$
and
$F_{k,-}(\{0\} \times I_{\Lambda,\Lambda'}^{\lambda,\lambda'}) \subset -{\mathbb H}$
for $k \in \{1,2\}$. We also have
\[ (F_{2,+}-F_{1,+})(0,\lambda_{0})=\psi_{\lambda_{0}}^{\beta}(\tau_{2,+}(0,\lambda_{0})) -
\psi_{\lambda_{0}}^{\beta}(\tau_{1,+}(0,\lambda_{0}))
\in \aleph_{{\mathcal E}_{+}} {\mathbb R}^{+} \subset {\mathbb H} \]
and $(F_{1,-}-F_{2,-})(0,\lambda_{0}) \in {\mathbb H} $
for any $\lambda_{0} \in I_{\Lambda,\Lambda'}^{\lambda,\lambda'}$.
There exists $\zeta \in {\mathbb R}^{+}$ such that
$Im(F_{2,+}-F_{1,+})(0,\lambda_{0}) \geq 2 \zeta$ and
$Im(F_{1,-}-F_{2,-})(0,\lambda_{0}) \geq 2 \zeta$
for any $\lambda_{0} \in I_{\Lambda,\Lambda'}^{\lambda,\lambda'}$.
Since $H_{\Lambda,\Lambda',j}^{\lambda,\lambda'}(r,\lambda_{0}) \subset
H_{j,1}^{\lambda}(r,\lambda_{0})$
for any $(r,\lambda_{0}) \in [0,\delta) \times I_{\Lambda,\Lambda'}^{\lambda,\lambda'}$ and
$\tilde{\psi}_{j,\lambda}^{\varphi} - \psi$ is bounded (prop. \ref{pro:bddcon})
we obtain
\[ Im (\tilde{\psi}_{j,\lambda}^{\varphi})(H_{\Lambda,\Lambda',j}^{\lambda,\lambda'}(r,\lambda_{0}))
\subset \left(
\frac{Im(F_{1,-})(0,\lambda_{0})-\zeta/2}{r^{e_{\Lambda,\Lambda',j}^{\lambda,\lambda',-}}},
\frac{Im(F_{1,+})(0,\lambda_{0})+\zeta/2}{r^{e_{\Lambda,\Lambda',j}^{\lambda,\lambda',+}}}
\right) \]
for all $\lambda_{0} \in I_{\Lambda,\Lambda'}^{\lambda,\lambda'}$ and $r \in [0,\delta)$.

Consider the domain $B_{\varphi}(T_{i X}^{\epsilon,j}(x))$ enclosed by $\Gamma_{x}$ and $\varphi(\Gamma_{x})$.
We define
\[ B_{\varphi}^{2}(T_{i X}^{\epsilon,j}(x)) = B_{\varphi}(T_{i X}^{\epsilon,j}(x)) \cap H_{j,2}^{\lambda}(x) \]
for $x \in (0,\delta) I_{\Lambda,\Lambda'}^{\lambda,\lambda'}$.
%
%
The  function $\psi_{j,\lambda,\lambda'}^{\varphi}$ is then well-defined in
\[ \cup_{(r,\tilde{\lambda}) \in [0,\delta) \times I_{\Lambda,\Lambda'}^{\lambda,\lambda'}} (\{r \tilde{\lambda} \} \times
e^{2 \pi i \tilde{\psi}_{j,\lambda}^{\varphi}}[B_{\varphi}^{2}(T_{i X}^{\epsilon,j}(x))])  \]
since $B_{\varphi}(T_{i X}^{\epsilon,j}(x)) \setminus \varphi(\Gamma_{x})$
is a fundamental domain of $\varphi_{|H_{\Lambda,j}^{\lambda}(x)}$.
Moreover we have
\[ e^{2 \pi i \tilde{\psi}_{j,\lambda}^{\varphi}}(B_{\varphi}^{2}(T_{i X}^{\epsilon,j}(r,\tilde{\lambda}))
\supset {\mathbb S}^{1} e^{-2 \pi z} \left(
\frac{Im(F_{2,-})(0,\tilde{\lambda})+\zeta/2}{r^{e_{\Lambda,\Lambda',j}^{\lambda,\lambda',-}}},
\frac{Im(F_{2,+})(0,\tilde{\lambda})-\zeta/2}{r^{e_{\Lambda,\Lambda',j}^{\lambda,\lambda',+}}}
\right) \]
for all $\tilde{\lambda} \in I_{\Lambda,\Lambda'}^{\lambda,\lambda'}$ and $r \in (0,\delta)$.
It remains to show that
$\psi_{j,\lambda,\lambda'}^{\varphi}$ is well-defined in
$(x, e^{2 \pi i \tilde{\psi}_{j,\lambda}^{\varphi}})(H_{\Lambda,\Lambda',j}^{\lambda,\lambda'})$.
More precisely,
given $Q=(x_{0},y_{0}) \in H_{\Lambda,\Lambda',j}^{\lambda,\lambda'}$ there exist
$Q_{0} \in B_{\varphi}(T_{i X}^{\epsilon,j}(x_{0}))$ and $k \geq 0$ such that
$Q_{0}, \hdots, \varphi^{k}(Q_{0}) \in H_{\Lambda,j}^{\lambda}$
and $Q=\varphi^{k}(Q_{0})$. It suffices to prove that
$\{Q_{0}, \hdots, \varphi^{k}(Q_{0})\}$ is contained
in  $H_{\Lambda',j}^{\lambda'}$ and
$Q_{0} \in B_{\varphi}^{2}(T_{i X}^{\epsilon,j}(x_{0}))$ since then
$\tilde{\psi}_{j,\Lambda,\lambda}^{\varphi} - \tilde{\psi}_{j,\Lambda',\lambda'}^{\varphi}$ is
constant along orbits of $\varphi$.

We have $|\tilde{\psi}_{j,\lambda}^{\varphi} - \psi| \leq M$ in $H_{\Lambda,j}^{\lambda}$
for some $M>0$ by prop. \ref{pro:bddcon}.  Thus we obtain
$|\psi \circ \varphi^{l}(Q_{0}) - (\psi(Q_{0})+l)| \leq 2M$ for any $0 \leq l \leq k$.
Since $Q \in H_{j,1}^{\lambda}(x_{0})$ and
\[ \lim_{x \in (0,\delta) I_{\Lambda,\Lambda'}^{\lambda,\lambda'}, \ x \to 0}
 Im (\psi (\tau_{2,\pm}(x))) -  Im (\psi (\tau_{1,\pm}(x))) = \pm \infty \]
for $e_{\Lambda,\Lambda',j}^{\lambda,\lambda',\pm} \neq \infty$ we deduce that
$Q_{0}, \hdots, \varphi^{k}(Q_{0}) \in
H_{j,2}^{\lambda} \subset H_{\Lambda,j}^{\lambda} \cap H_{\Lambda',j}^{\lambda'}$.
\subsubsection{Case $R_{\Lambda,\Lambda',j}^{\lambda,\lambda'}=1$}
\label{subsec:rj1}
First we introduce the setup that we use for $R_{\Lambda,\Lambda',j}^{\lambda,\lambda'}=1$. Then we
discuss the topological properties that allow to adapt the proof for the case
$R_{\Lambda,\Lambda',j}^{\lambda,\lambda'}=2$
to the new setting. Such properties are key to prove the results in
subsections \ref{subsec:extfatcor} and \ref{subsec:asyfatcor}.

{\bf Step 1.} Let us introduce some definitions.
If $Ex_{\Lambda,\Lambda',j}^{\lambda,\lambda'}= \emptyset$
we have $H_{\Lambda,j}^{\lambda}=H_{\Lambda',j}^{\lambda'}$ and
$\tilde{\psi}_{j,\lambda}^{\varphi} \equiv  \tilde{\psi}_{j,\lambda'}^{\varphi}$.
Otherwise we denote ${\mathcal E}={\mathcal E}_{+}={\mathcal E}_{-}$ and
$e_{\Lambda,\Lambda',j}^{\lambda,\lambda'}=e_{\Lambda,\Lambda',j}^{\lambda,\lambda',+}=e_{\Lambda,\Lambda',j}^{\lambda,\lambda',-}$
(see def. \ref{def:exexr}). Denote $\aleph = \aleph_{\Lambda,\lambda}$. We have
\[ {\mathcal E}_{\beta} = {\mathcal E}=
\{(x,t) \in B(0,\delta) \times {\mathbb C} :  \eta \geq |t| \geq \rho |x| \} . \]
Denote ${\mathcal C}={\mathcal C}_{\beta}$. We define
\[ {\mathcal E}'= \{(x,t) \in B(0,\delta) \times {\mathbb C} :  \eta \geq |t| > (2-1/4) \rho |x| \} .\]
We define $\rho^{0} = 2 \rho$, $\rho^{4}= (2-1/4) \rho$,
$H_{j,0}^{\lambda,\lambda'}=H_{\Lambda,\Lambda',j}^{\lambda,\lambda'}$ and
\[ H_{j,4}^{\lambda, \lambda'} = [(H_{\Lambda,j}^{\lambda})^{L_{j}} \cap
({\mathcal E}' \cup \cup_{{\mathcal B} \in En_{\Lambda,\Lambda',j}^{\lambda,\lambda'}} {\mathcal B})]
\cap \{ x \in [0,\delta) I_{\Lambda,\Lambda'}^{\lambda,\lambda'} \}  . \]
We have
$H_{\Lambda,\Lambda',j}^{\lambda,\lambda'} \subset H_{j,4}^{\lambda, \lambda'} \subset H_{\Lambda,j}^{\lambda} \cap H_{\Lambda',j}^{\lambda'}$.
Denote $\Gamma_{x}=\Gamma(\aleph_{\Lambda,\lambda}X ,T_{i X}^{\epsilon,j}(x),T_{0})$.
Let $x \in (0,\delta) I_{\Lambda,\Lambda'}^{\lambda,\lambda'}$. The boundary of $H_{j,k}^{\lambda,\lambda'}(x)$
is composed of a piece of trajectory $\Gamma_{x}[s_{k}^{-}(x),s_{k}^{+}(x)]$ of $\Gamma_{x}$
and a closed arc $arc_{k}(x) \subset \{ |t| =   \rho^{k} |x| \}$ for $k \in \{0,4\}$.
The section $\tau_{k}^{l}(r,\tilde{\lambda})=\Gamma_{x}(s_{k}^{l}(r,\tilde{\lambda}))$ is asymptotically
continuous in $[0,\delta) \times I_{\Lambda,\Lambda'}^{\lambda,\lambda'}$ for $k \in \{0,4\}$ and $l \in \{-,+\}$.
Moreover $\tau_{0}^{-}(0,\lambda_{0})$ and $\tau_{4}^{-}(0,\lambda_{0})$
belong to an element $\gamma_{-,\lambda_{0}}$
of $Tr_{\to \infty}(\aleph_{\mathcal E} X_{\beta}(\lambda_{0}))$ whereas
$\tau_{0}^{+}(0,\lambda_{0})$ and $\tau_{4}^{+}(0,\lambda_{0})$
belong to an element $\gamma_{+,\lambda_{0}}$
of $Tr_{\leftarrow \infty}(\aleph_{\mathcal E} X_{\beta}(\lambda_{0}))$
for any $\lambda_{0} \in  I_{\Lambda,\Lambda'}^{\lambda,\lambda'}$.
The functions
$arc_{0}(r,\tilde{\lambda})$ and $arc_{4}(r,\tilde{\lambda})$ defined in $(0,\delta) \times I_{\Lambda,\Lambda'}^{\lambda,\lambda'}$
admit a continuous extension to $[0,\delta) \times I_{\Lambda,\Lambda'}^{\lambda,\lambda'}$ (we are considering
the Hausdorff topology for compact sets). The extrema of $arc_{k}(0,\lambda_{0})$
are the points $\tau_{k}^{-}(0,\lambda_{0})$ and $\tau_{k}^{+}(0,\lambda_{0})$.
We define
\[ \gamma_{+,\lambda_{0}}^{k} = \gamma_{+,\lambda_{0}} \cap ({\mathbb C} \setminus B(0,\rho^{k})) \
{\rm and} \ \gamma_{-,\lambda_{0}}^{k} = \gamma_{-,\lambda_{0}} \cap ({\mathbb C} \setminus B(0,\rho^{k})) . \]
Let $\tilde{H}_{j,k}^{\lambda,\lambda'}(0,{\lambda}_{0})$ be
the connected component of
${\mathbb C} \setminus [\gamma_{+,\lambda_{0}}^{k} \cup \gamma_{-,\lambda_{0}}^{k} \cup arc_{k}(0,\lambda_{0})]$
not containing $0$. We denote
$\psi_{\lambda_{0}}^{\mathcal C}$ the Fatou coordinate of $X_{\beta}(\lambda_{0})$
defined in the neighborhood of $\overline{\tilde{H}_{j,k}^{\lambda,\lambda'}(0,\lambda_{0})}$
such that $\psi_{\lambda_{0}}^{\mathcal C}(\infty)=0$. Denote $\psi = \psi_{H_{\Lambda,j}^{\lambda},L_{j}}^{X}$.
\begin{figure}[h]
\begin{center}
\includegraphics[height=6cm,width=12cm]{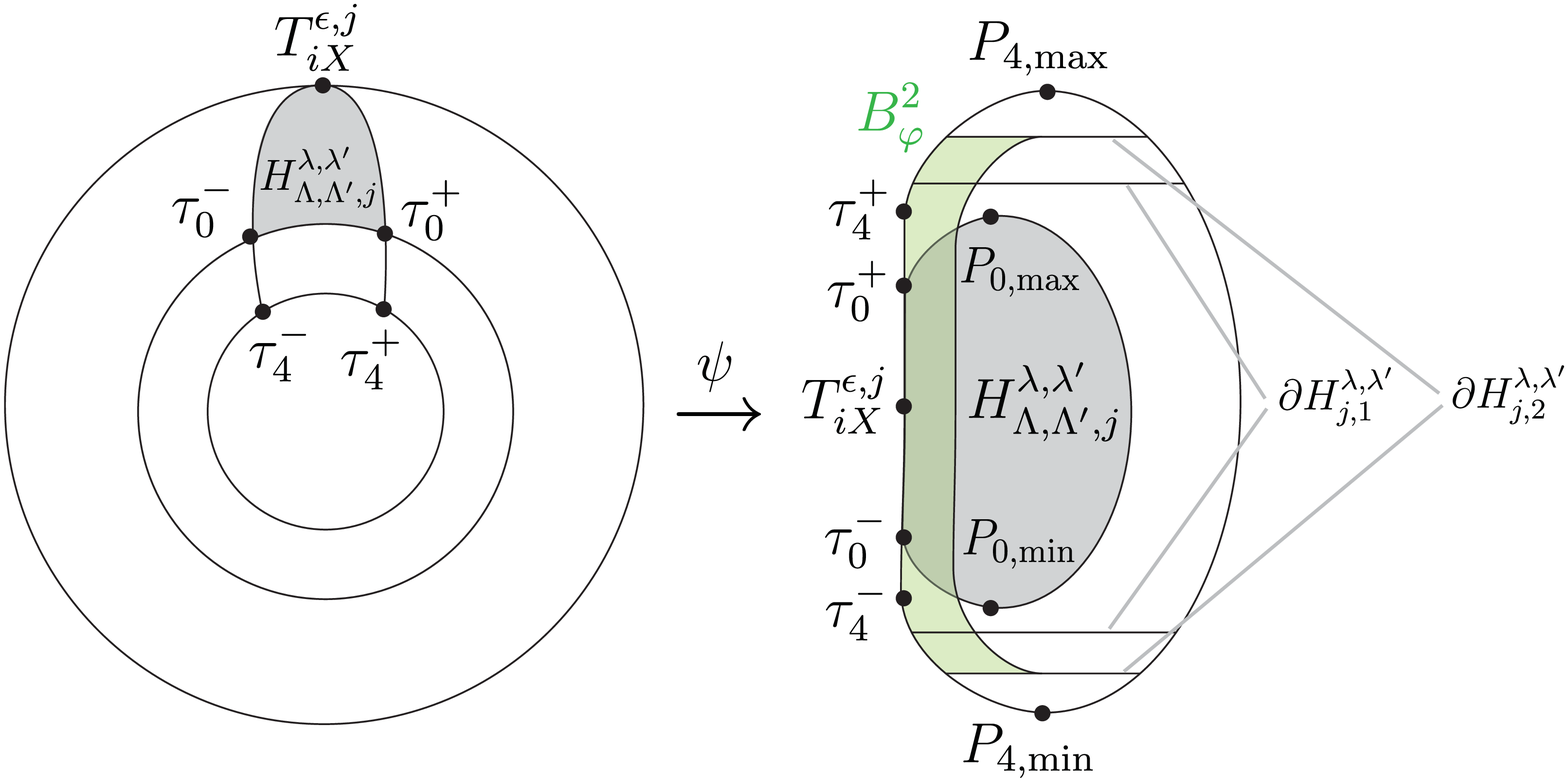}
\end{center}
\caption{Case $R_{\Lambda,\Lambda',j}^{\lambda,\lambda'}=1$}
\label{EVf7}
\end{figure}

{\bf Step 2.}
We present the main properties of the sets $H_{j,0}^{\lambda,\lambda'}$ and
$H_{j,k}^{\lambda,\lambda'}$. Our goal is showing that the qualitative shape of
$H_{\Lambda,\Lambda',j}^{\lambda,\lambda'}$ is as in figure (\ref{EVf7}).
Suppose that $\Re(X)$ points towards $B(0,\delta) \times B(0,\epsilon)$ at $T_{iX}^{\epsilon,j}(0)$
without lack of generality.
\begin{lem}
Let $X \in \Xt$. Consider $\Lambda, \Lambda' \in {\mathcal M}$,
$\lambda, \lambda' \in {\mathbb S}^{1}$ and $j \in {\mathbb Z}/(2 \nu({\mathcal E}_{0}) {\mathbb Z})$.
Assume $R_{\Lambda,\Lambda',j}^{\lambda,\lambda'}=1$.
Then $H_{j,k}^{\lambda,\lambda'}(x)$ is $\Re(X)$-convex for all
$k \in \{0,4\}$ and $x \in [0,\delta) I_{\Lambda,\Lambda'}^{\lambda,\lambda'}$.
\end{lem}
\begin{proof}
If $Ex_{\Lambda,\Lambda',j}^{\lambda,\lambda'}= \emptyset$ or $x=0$ the vector field $\Re (X)$
is transversal and then almost transversal to $\partial H_{j,k}^{\lambda,\lambda'}(x) \setminus Sing X$.
Thus we can suppose $Ex_{\Lambda,\Lambda',j}^{\lambda,\lambda'} \neq \emptyset$ and $x \neq 0$
by lemma \ref{lem:transv}.
Suppose that $\Re(X)$ points towards $B(0,\delta) \times B(0,\epsilon)$ at $T_{iX}^{\epsilon,j}(0)$
without lack of generality.

The vector field $\Re(X)$ is transversal to
$\partial H_{j,k}^{\lambda,\lambda'}(x) \setminus (\{ \tau_{k}^{+}(x), \tau_{k}^{-}(x)\}
\cup T{\mathcal C}_{X}^{\rho^{k}}(x))$
for $k \in \{0,4\}$ and $x \in (0,\delta) I_{\Lambda,\Lambda'}^{\lambda,\lambda'}$.
Moreover $\Re(X)$ is almost transversal to $\partial H_{j,k}^{\lambda,\lambda'}(x)$
at the points in $T{\mathcal C}_{X}^{\rho^{k}}(x)$ since they are convex.
The curve $\Gamma_{x}$ is transversal to the set
$\{ |t| =   \rho^{k} |x| \}$ at $\tau_{k}^{+}(x)$ and $\tau_{k}^{-}(x)$ for
$k \in \{0,4\}$ and $x \in (0,\delta) I_{\Lambda,\Lambda'}^{\lambda,\lambda'}$ since the points of
$T{\mathcal C}_{\aleph_{\Lambda,\lambda}X}^{\rho^{k}}(x)$ are convex (lemma \ref{lem:tgpt30}).
Moreover the curve $\psi(arc_{k}(x) \setminus \tau_{k}^{\pm}(x))$ is to the right of $\psi(\Gamma_{x})$
in the neighborhood of $\psi(\tau_{k}^{\pm}(x))$.
Thus a neighborhood of $\psi(\tau_{k}^{\pm}(x))$ in $\psi(\overline{H_{j,k}^{\lambda,\lambda'}(x)})$
does not contain points to the left of $\psi(\Gamma_{x})$.
The point $\psi({\rm exp}(s X)(\tau_{k}^{\pm}(x)))$ is to the left of $\psi(\Gamma_{x})$
and then ${\rm exp}(s X)(\tau_{k}^{\pm}(x)) \not \in \overline{H_{j,k}^{\lambda,\lambda'}(x)}$
for any $s<0$ in a neighborhood of $0$. Therefore
$\Re (X)$ is almost transversal to $\partial H_{j,k}^{\lambda,\lambda'}(x)$
at both $\tau_{k}^{+}(x)$ and $\tau_{k}^{-}(x)$ for all
$k \in \{0,4\}$ and $x \in (0,\delta) I_{\Lambda,\Lambda'}^{\lambda,\lambda'}$.
We obtain that $\Re (X)$ is almost transversal to $\partial H_{j,k}^{\lambda,\lambda'}(x)$
at any of its points.
Lemma \ref{lem:transv} implies that
$H_{j,k}^{\lambda,\lambda'}(x)$ is $\Re(X)$-convex for all
$k \in \{0,4\}$ and $x \in (0,\delta) I_{\Lambda,\Lambda'}^{\lambda,\lambda'}$.
\end{proof}
The set $\overline{H_{j,k}^{\lambda,\lambda'}(x)}$ does not contain
pieces of trajectories of $Re(X)$ for all
$k \in \{0,4\}$ and $x \in (0,\delta) I_{\Lambda,\Lambda'}^{\lambda,\lambda'}$.
Hence the maximum (resp. the minimum) of $Im(\psi)$ in
$\overline{H_{j,k}^{\lambda,\lambda'}(x)}$
is attained at a unique point $P_{k,\max}(x)$ (resp. $P_{k,\min}(x)$) by corollary \ref{cor:flconcl}.
Since $\Gamma_{x}$ is transversal to $\Re (X)$ the points
$P_{k,\max}(x)$ and $P_{k,\min}(x)$ belong to $\{|t|=\rho^{k}|x|\}$.
The points in $T{\mathcal C}_{X}^{\rho^{k}}(x)$ are convex and then they can not
be accessed by trajectories of $\Re (X)$ lying outside $\{|t| \leq \rho^{k}|x|\}$. We deduce that
\[ \overline{H_{j,k}^{\lambda,\lambda'}(x)} \cap T{\mathcal C}_{X}^{\rho^{k}}(x)
\subset \{P_{k,\min}(x), P_{k,\max}(x) \}
\ \forall k \in \{0,4\} \ \forall x \in (0,\delta) I_{\Lambda,\Lambda'}^{\lambda,\lambda'}. \]
We proceed in an analogous way with $\Re (X_{\beta}(\lambda_{0}))$ and
$\tilde{H}_{j,k}^{\lambda,\lambda'}(0,\lambda_{0})$ for $\lambda_{0} \in I_{\Lambda,\Lambda'}^{\lambda,\lambda'}$.
There exists a unique point $P_{k,\max}(0,\lambda_{0})$ (resp. $P_{k,\max}(0,\lambda_{0})$)
where is attained the maximum (resp. the minimum) of $Im(\psi_{\lambda_{0}}^{\mathcal C})$
in $\overline{\tilde{H}_{j,k}^{\lambda,\lambda'}(0,\lambda_{0})}$. We obtain
\[ \{P_{k,\min}(0,\lambda_{0}), P_{k,\max}(0,\lambda_{0}) \} \subset \partial B(0,\rho^{k}). \]
Moreover $\Re (X_{\beta}(\lambda_{0}))$ is transversal to
$arc_{k}(0,\lambda_{0})$ except maybe at $P_{k,\min}(0,\lambda_{0})$ and $P_{k,\max}(0,\lambda_{0})$.
The sections $P_{k,\min}$ and $P_{k,\max}$ are asymptotically continuous.
Given $k \in \{0,4\}$ and $l \in \{\min, \max\}$ we define the functions
\[ F_{k,l}(r,\lambda_{0}) = Im(|r|^{\iota({\mathcal E})} \psi(P_{k,l}(r,\lambda_{0}))) \]
for $(r,\lambda_{0}) \in (0,\delta) \times I_{\Lambda,\Lambda'}^{\lambda,\lambda'}$. Moreover they extend
continuously to $[0,\delta) \times I_{\Lambda,\Lambda'}^{\lambda,\lambda'}$ (prop. \ref{pro:hit}) by defining
\[ F_{k,\min}(0,\lambda_{0}) = Im(\psi_{\lambda_{0}}^{\mathcal C}(P_{k,\min}(0,\lambda_{0}))) \ {\rm and} \
F_{k,\max}(0,\lambda_{0}) = Im(\psi_{\lambda_{0}}^{\mathcal C}(P_{k,\max}(0,\lambda_{0})))  \]
for $k \in \{0,4\}$ and $\lambda_{0} \in  I_{\Lambda,\Lambda'}^{\lambda,\lambda'}$. We have
\[ F_{4,\max}(0,\lambda_{0}) > F_{0,\max}(0,\lambda_{0}) \geq
Im (\psi_{\lambda_{0}}^{\mathcal C}(\tau_{0}^{+}(0,\lambda_{0}))) >0 \]
for any $\lambda_{0} \in I_{\Lambda,\Lambda'}^{\lambda,\lambda'}$.
The last inequality is a consequence of prop. \ref{pro:hit}. Analogously we obtain
\[ F_{4,\min}(0,\lambda_{0}) < F_{0,\min}(0,\lambda_{0}) \leq
Im (\psi_{\lambda_{0}}^{\mathcal C}(\tau_{0}^{-}(0,\lambda_{0}))) < 0 \]
for any $\lambda_{0} \in I_{\Lambda,\Lambda'}^{\lambda,\lambda'}$.
There exists $\zeta \in {\mathbb R}^{+}$ such that
\[ (F_{4,\max}-F_{0,\max})(0,\lambda_{0}) \geq 2 \zeta \leq (F_{0,\min}-F_{4,\min})(0,\lambda_{0})
\ \forall \lambda_{0} \in I_{\Lambda,\Lambda'}^{\lambda,\lambda'}. \]
We define $\zeta_{1}=7\zeta/4$, $\zeta_{2}= \zeta/4$, $\zeta_{3}=\zeta/8$ and
\[ H_{j,k}^{\lambda,\lambda'}(r,\lambda_{0}) =
H_{j,4}^{\lambda,\lambda'}(r,\lambda_{0}) \cap
\left\{ Im(\psi) \in \left( \frac{F_{4,\min}(0,\lambda_{0}) + \zeta_{k}}{r^{e_{\Lambda,\Lambda',j}^{\lambda,\lambda'}}}
, \frac{F_{4,\max}(0,\lambda_{0})-\zeta_{k}}{r^{e_{\Lambda,\Lambda',j}^{\lambda,\lambda'}}} \right)  \right\} \]
for $(r, \lambda_{0}) \in (0,\delta) \times I_{\Lambda,\Lambda'}^{\lambda,\lambda'}$ and $k \in \{1,2,3\}$. We have
\[ H_{\Lambda,\Lambda',j}^{\lambda,\lambda'} = H_{j,0}^{\lambda,\lambda'}
\subset H_{j,1}^{\lambda,\lambda'} \subset H_{j,2}^{\lambda,\lambda'} \subset
 H_{j,3}^{\lambda,\lambda'} \subset H_{j,4}^{\lambda,\lambda'} \subset H_{\Lambda,j}^{\lambda} \cap H_{\Lambda',j}^{\lambda'} . \]
Let $x \in (0,\delta) I_{\Lambda,\Lambda'}^{\lambda,\lambda'}$.
The set $\overline{H_{j,3}^{\lambda,\lambda'}(x)} \cap \partial H_{j,4}^{\lambda,\lambda'}(x)$
is a union of two curves, namely a curve $\varpi_{x}$ containing $T_{iX}^{\epsilon,j}(x)$
a curve $\varpi_{x}'$ contained in $arc_{4}(x)$.

There exists $\theta'>0$ such that the angle between $\Re (X_{\beta}(\lambda_{0}))$ and
$arc_{4}(0,\lambda_{0})$ at $Q$ is greater than $\theta' \in {\mathbb R}^{+}$ for any
$\lambda_{0} \in  I_{\Lambda,\Lambda'}^{\lambda,\lambda'}$ and any $Q \in arc_{4}(0,\lambda_{0})$ such that
\[ Im(\psi_{\lambda_{0}}^{\mathcal C}(Q))) \in
[F_{4,\min}(0,\lambda_{0}) + \zeta/8, F_{4,\max}(0,\lambda_{0})-\zeta/8] . \]
Thus the angle between $\Re (X)$ and
$arc_{4}(x)$ at $Q$ is greater than $\theta'$ for any
$x$ in  $(0,\delta) I_{\Lambda,\Lambda'}^{\lambda,\lambda'}$
and any $Q \in arc_{4}(x) \cap (\varpi_{x} \cup \varpi_{x}')$.
The angle between $\Re (X)$ and $\Gamma_{x}$ at any point $Q \in \Gamma_{x}$
is bounded by below by a positive constant not depending on
$x \in (0,\delta) I_{\Lambda,\Lambda'}^{\lambda,\lambda'}$ or $Q$.
Thus the angle between $\Re (X)$ and $\varpi_{x} \cup \varpi_{x}'$ at any of its points
is greater than $\theta'' \in {\mathbb R}^{+}$ for any $x \in (0,\delta) I_{\Lambda,\Lambda'}^{\lambda,\lambda'}$.
Analogously as in Step 3 of subsection \ref{subsec:defatcor}
we can prove that $\varphi(\varpi_{x}) \cup \varphi^{-1}(\varpi_{x}')$ is transversal to
$\Re(X)$ at any of its points for any $x \in (0,\delta) I_{\Lambda,\Lambda'}^{\lambda,\lambda'}$.

{\bf Step 3.}
Let $x \in (0,\delta) I_{\Lambda,\Lambda'}^{\lambda,\lambda'}$. Let
$B_{\varphi}^{2}(T_{i X}^{\epsilon,j}(x))$ be the closure of the bounded connected component
of the complementary of
$\varpi_{x} \cup \varphi(\varpi_{x}) \cup (\partial H_{j,2}^{\lambda,\lambda'}(x) \setminus \varpi_{x}')$.
As in subsection \ref{subsec:rj2} we obtain that
given $Q \in H_{j,1}^{\lambda,\lambda'}(x)$ there exist
$Q_{0} \in B_{\varphi}^{2}(T_{i X}^{\epsilon,j}(x))$ and $k \geq 0$ such that
$Q_{0}, \hdots, \varphi^{k}(Q_{0}) \in H_{j,2}^{\lambda,\lambda'}$
and $Q=\varphi^{k}(Q_{0})$.
We can proceed as in subsection \ref{subsec:rj2} to prove prop. \ref{pro:flasp}.
\subsection{Flatness properties of Fatou coordinates}
\label{subsec:comfatcor}
Let $\varphi \in \diff{tp1}{2}$. Let $\Upsilon={\rm exp}(X)$ be a $2$-convergent normal form.
Consider $\Lambda=(\lambda_{1}, \hdots, \lambda_{\tilde{q}}) \in {\mathcal M}$ and
the dynamical splitting $\digamma_{\Lambda}$ in remark \ref{rem:unifspl}.

We want to prove that
$\tilde{\psi}_{j,\Lambda,\lambda}^{\varphi} - \tilde{\psi}_{j,\Lambda',\lambda'}^{\varphi}$
is exponentially small and then flat up to an additive function of $x$.
We use two ingredients, namely the boundness of
$\tilde{\psi}_{j,\Lambda,\lambda}^{\varphi} -  \psi_{H_{\Lambda,j}^{\lambda}, L_{j}}^{X}$
(subsection \ref{subsec:defatcor}) and the study of the shape of
$H_{\Lambda,\Lambda',j}^{\lambda,\lambda'}$
provided by prop. \ref{pro:flasp}.
Flatness is the key property to prove multi-summability of
Fatou coordinates of elements of $\diff{p1}{2}$.
\begin{defi}
Let $\lambda \in {\mathbb S}^{1}$ and $j \in {\mathcal D}(\varphi)$.
Denote $(0,y_{0})= T_{iX}^{\epsilon, j}(0)$. We define
\[ \ddot{\psi}_{j,\Lambda,\lambda}^{\varphi}(x,y) = \tilde{\psi}_{j,\Lambda,\lambda}^{\varphi}(x,y) -
\tilde{\psi}_{j,\Lambda,\lambda}^{\varphi}(x,y_{0}) \ \
{\rm for} \ (x,y) \in H_{\Lambda,j}^{\lambda} . \]
The function $\ddot{\psi}_{j,\Lambda,\lambda}^{\varphi}$ is a Fatou coordinate of $\varphi$ in $H_{\Lambda,j}^{\lambda}$
such that $\ddot{\psi}_{j,\Lambda,\lambda}^{\varphi}(x,y_{0}) \equiv 0$.
We denote $\ddot{\psi}_{j,\lambda}^{\varphi} = \ddot{\psi}_{j,\Lambda,\lambda}^{\varphi}$ if $\Lambda$ is implicit.
\end{defi}
Let us remark that the Fatou coordinate $\tilde{\psi}_{j,\lambda}^{\varphi}$ can be extended to a neighborhood of
$([0,\delta) I_{\Lambda}^{\lambda}) \times \{y_{0}\}$ by using the equation
$\tilde{\psi}_{j,\lambda}^{\varphi} \circ \varphi \equiv \tilde{\psi}_{j,\lambda}^{\varphi} + 1$.

We have the normalizing conditions
\[ \ddot{\psi}_{j,\lambda}^{\varphi}(x,y_{0}) \equiv 0  \ {\rm and} \
(\tilde{\psi}_{j,\lambda}^{\varphi} - \psi_{H_{\Lambda,j}^{\lambda}, L_{j}}^{X})
(\omega^{\aleph_{\Lambda, \lambda} X} (H_{\Lambda,j}^{\lambda}(x))) \equiv 0 . \]
The latter one is not a good choice to compare $\tilde{\psi}_{j,\lambda}^{\varphi}$
and $\tilde{\psi}_{j,\lambda'}^{\varphi}$ since for example we could have
$\omega^{\aleph_{\Lambda, \lambda} X} (H_{\Lambda,j}^{\lambda}(x))) \neq
\omega^{\aleph_{\Lambda',\lambda'} X} (H_{\Lambda',j}^{\lambda'}(x)))$
for any $x \in (0,\delta) I_{\Lambda,\Lambda'}^{\lambda,\lambda'}$.

In the next proposition we denote $e^{-K/|x|^{\infty}} \equiv 0$ by convention.
\begin{pro}
\label{pro:difFatflat}
Let $\varphi \in \diff{tp1}{2}$ with $2$-convergent normal form ${\rm exp}(X)$.
Let $\Lambda, \Lambda' \in {\mathcal M}$. Consider
$\lambda, \lambda' \in {\mathbb S}^{1}$ and $j \in {\mathcal D}(\varphi)$.
Then there exists $K \in {\mathbb R}^{+}$ such that
\[ |\ddot{\psi}_{j,\Lambda,\lambda}^{\varphi} - \ddot{\psi}_{j,\Lambda',\lambda'}^{\varphi}|(x,y) \leq
\frac{e^{-K/|x|^{e_{\Lambda,\Lambda',j}^{\lambda,\lambda',-}}}}{2} +    \frac{e^{-K/|x|^{e_{\Lambda,\Lambda',j}^{\lambda,\lambda',+}}}}{2}
\leq  e^{-K/|x|^{\tilde{e}_{d_{\Lambda,\Lambda'}^{\lambda,\lambda'}+1}}}  \]
for any $(x,y) \in H_{\Lambda,\Lambda',j}^{\lambda,\lambda'}$.
\end{pro}
\begin{proof}
The functions $\tilde{\psi}_{j,\lambda}^{\varphi} - \psi_{H_{\Lambda,j}^{\lambda}, L_{j}}^{X}$ and
$\tilde{\psi}_{j,\lambda'}^{\varphi} - \psi_{H_{\Lambda',j}^{\lambda'}, L_{j}}^{X}$ are bounded by
prop. \ref{pro:bddcon}. Thus there exists $M \in {\mathbb R}^{+}$ such that
$|\psi_{j,\lambda,\lambda'}^{\varphi}| \leq M$ in
\[ \cup_{(r,\tilde{\lambda}) \in [0,\delta) \times I_{\Lambda,\Lambda'}^{\lambda,\lambda'}} (\{r \tilde{\lambda}\} \times
[B(0,\tilde{\kappa}_{-}(r,\tilde{\lambda})) \setminus \overline{B}(0,\tilde{\kappa}_{+}(r,\tilde{\lambda}))]), \]
see prop. \ref{pro:flasp}.
We obtain
\[ \psi_{j,\lambda,\lambda'}^{\varphi}(x,z) = a_{0}(x) + \sum_{k \in {\mathbb N}} a_{k}(x) z^{k} +
\sum_{k \in {\mathbb N}} \frac{a_{-k}(x)}{z^{k}} \]
by considering the Laurent series of $\psi_{j,\lambda,\lambda'}^{\varphi}$. We have
\[ a_{k}(x) = \frac{1}{2 \pi i}
\int_{|z|=\tilde{\kappa}_{-}(x)} \frac{\psi_{j,\lambda,\lambda'}^{\varphi}(x,z)}{z^{k+1}}dz
\implies |a_{k}(x)| \leq \frac{M}{\tilde{\kappa}_{-}(x)^{k}} \]
for all $k \in {\mathbb N}$ and $x \in (0,\delta) I_{\Lambda,\Lambda'}^{\lambda,\lambda'}$. Analogously we deduce
\[ a_{-k}(x) = \frac{1}{2 \pi i}
\int_{|z|=\tilde{\kappa}_{+}(x)} \psi_{j,\lambda,\lambda'}^{\varphi}(x,z) z^{k-1} dz
\implies |a_{k}(x)| \leq M \tilde{\kappa}_{+}(x)^{k} \]
for all $k \in {\mathbb N}$ and $x \in (0,\delta) I_{\Lambda,\Lambda'}^{\lambda,\lambda'}$. We get
\[ |\tilde{\psi}_{j,\lambda}^{\varphi}(x,y) - \tilde{\psi}_{j,\lambda'}^{\varphi}(x,y) -
a_{0}(x) | \leq
M \sum_{k \in {\mathbb N}} {\left( {\frac{\kappa_{-}(x)}{\tilde{\kappa}_{-}(x)} }\right)}^{k} +
M \sum_{k \in {\mathbb N}}{\left( {\frac{\tilde{\kappa}_{+}(x)}{\kappa_{+}(x)} }\right)}^{k}   \]
in $H_{\Lambda,\Lambda',j}^{\lambda,\lambda'}$.
By plugging the values of
$\kappa_{-}$, $\tilde{\kappa}_{-}$,  $\kappa_{+}$ and $\tilde{\kappa}_{+}$
in the previous equation we obtain
\[ |\tilde{\psi}_{j,\lambda}^{\varphi}(x,y) - \tilde{\psi}_{j,\lambda'}^{\varphi}(x,y) -
a_{0}(x) | \leq
M \sum_{k \in {\mathbb N}}  e^{-k \zeta/|x|^{e_{\Lambda,\Lambda',j}^{\lambda,\lambda',-}}}     +
M \sum_{k \in {\mathbb N}}  e^{-k \zeta/|x|^{e_{\Lambda,\Lambda',j}^{\lambda,\lambda',+}}} \]
and then
\[  |\tilde{\psi}_{j,\lambda}^{\varphi}(x,y) - \tilde{\psi}_{j,\lambda'}^{\varphi}(x,y) -
a_{0}(x)  |  \leq
2M (e^{- \zeta/|x|^{e_{\Lambda,\Lambda',j}^{\lambda,\lambda',-}}}  + e^{- \zeta/|x|^{e_{\Lambda,\Lambda',j}^{\lambda,\lambda',+}}})   \]
in $H_{\Lambda,\Lambda',j}^{\lambda,\lambda'}$. Denote $(0,y_{0})= T_{i X}^{\epsilon,j}(0)$ and
$b_{0}(x)=a_{0}(x)- \tilde{\psi}_{j,\lambda}^{\varphi}(x,y_{0}) + \tilde{\psi}_{j,\lambda'}^{\varphi}(x,y_{0})$.
We have
\[  |\ddot{\psi}_{j,\lambda}^{\varphi}(x,y) - \ddot{\psi}_{j,\lambda'}^{\varphi}(x,y) - b_{0}(x)  |  \leq
2M (e^{- \zeta/|x|^{e_{\Lambda,\Lambda',j}^{\lambda,\lambda',-}}}  + e^{- \zeta/|x|^{e_{\Lambda,\Lambda',j}^{\lambda,\lambda',+}}})   \]
in $H_{\Lambda,\Lambda',j}^{\lambda,\lambda'}$. We obtain
\[  |b_{0}(x)  |  \leq
2M (e^{- \zeta/|x|^{e_{\Lambda,\Lambda',j}^{\lambda,\lambda',-}}}  + e^{- \zeta/|x|^{e_{\Lambda,\Lambda',j}^{\lambda,\lambda',+}}})   \]
for any  $x \in (0,\delta) I_{\Lambda,\Lambda'}^{\lambda,\lambda'}$ by evaluating at $(x,y_{0})$. This property implies
\[  |\ddot{\psi}_{j,\lambda}^{\varphi}(x,y) - \ddot{\psi}_{j,\lambda'}^{\varphi}(x,y)|  \leq
4M (e^{- \zeta/|x|^{e_{\Lambda,\Lambda',j}^{\lambda,\lambda',-}}}  + e^{- \zeta/|x|^{e_{\Lambda,\Lambda',j}^{\lambda,\lambda',+}}})   \]
for any $(x,y) \in H_{\Lambda,\Lambda',j}^{\lambda,\lambda'}$. It suffices to consider $K \in {\mathbb R}^{+}$ such that
$K < \zeta$.
\end{proof}
\subsection{Extending Fatou coordinates}
\label{subsec:extfatcor}
Let $\varphi \in \diff{tp1}{2}$. Let $\Upsilon={\rm exp}(X)$ be a $k$-convergent normal form.
Consider $\Lambda,\Lambda' \in {\mathcal M}$.
Let $\lambda, \lambda' \in {\mathbb S}^{1}$ and $j \in {\mathcal D}(\varphi)$.
The goal of this subsection is extending $\ddot{\psi}_{j,\Lambda,\lambda}^{\varphi}$
and $\ddot{\psi}_{j,\Lambda,\lambda}^{\varphi} - \ddot{\psi}_{j,\Lambda',\lambda'}^{\varphi}$
to domains slightly bigger
than $H_{\Lambda,j}^{\lambda}$ and $H_{\Lambda,\Lambda',j}^{\lambda,\lambda'}$,
namely
$H_{j,\theta}^{\epsilon,\rho,\lambda}$ and $H_{j,\theta}^{\epsilon,\rho,\lambda,\lambda'}$
respectively.
We intend to use such an extension to deduce properties of quasi-analytic
type for the infinitesimal generator of $\varphi$.
\begin{figure}[h]
\begin{center}
\includegraphics[height=5.5cm,width=12cm]{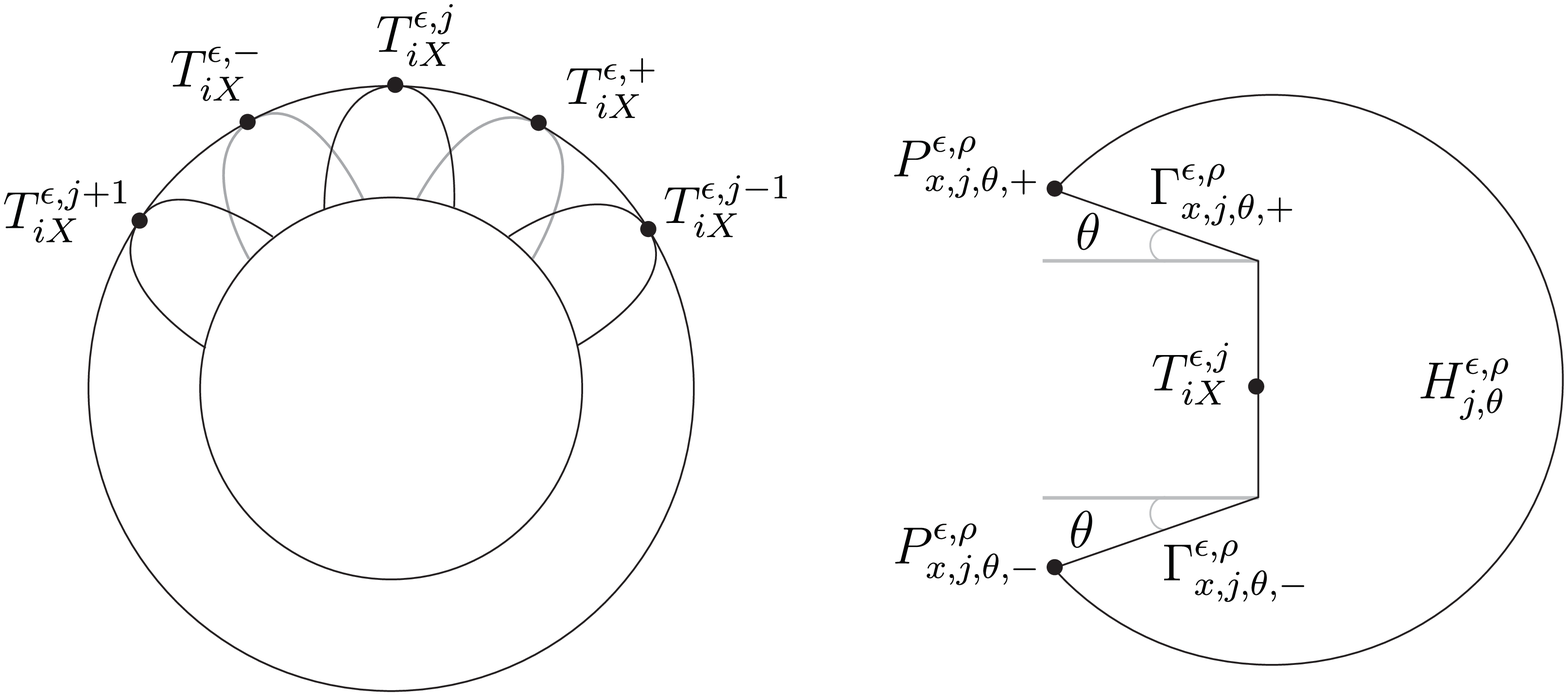}
\end{center}
\caption{Picture of $H_{j,\theta}^{\epsilon, \rho}$
in the $\psi_{L_{j}}^{X}$ coordinate on the right}
\label{EVf8}
\end{figure}

{\bf Step 1.} Fix $\theta \in (0,\pi/2]$.
We define the set $H_{j,\theta}^{\epsilon, \rho}$ where $\ddot{\psi}_{j,\lambda}^{\varphi}$
can be extended.

  Suppose without lack of generality that $\Re (X)$ points
towards $B(0,\delta) \times B(0,\epsilon)$ at $T_{i X}^{\epsilon, j}(0)$.
There exists a unique section $T_{X}^{\epsilon,+}$ of $T_{X}^{\epsilon}$ such that
$T_{X}^{\epsilon,+}(x)$ is in the arc in $\{x\} \times \partial B(0,\epsilon)$ going from
$T_{iX}^{\epsilon, j-1}(x)$ to $T_{iX}^{\epsilon,j}(x)$ in counter clockwise sense for any
$x \in B(0,\delta)$ (see def. \ref{def:lj}). Analogously we can define the section
$T_{X}^{\epsilon,-}$ of $T_{X}^{\epsilon}$ contained in the arc going from
$T_{iX}^{\epsilon, j}$ to $T_{iX}^{\epsilon,j+1}$ in counter clockwise sense.
The exterior set ${\mathcal E}_{0}$ is of the form $\{ \rho_{0} |x| \leq |y| \leq \epsilon\}$.
Given $\rho \geq 2 \rho_{0}$ we denote
${\mathcal E}_{0}^{\rho} = \{ \rho |x| \leq |y| \leq \epsilon\}$. We define
\[ \Gamma_{x,j,0}^{\epsilon,\rho} = \Gamma(iX, T_{iX}^{\epsilon,j}(x), {\mathcal E}_{0}^{\rho}), \
\Gamma_{x,j,-}^{\epsilon,\rho} = \Gamma(X, T_{X}^{\epsilon,-}(x), {\mathcal E}_{0}^{\rho}), \
\Gamma_{x,j,+}^{\epsilon,\rho} = \Gamma(X, T_{X}^{\epsilon,+}(x), {\mathcal E}_{0}^{\rho}) . \]
We define $H_{j,s}^{\epsilon,\rho}(x_{0})$ as the bounded component of
$\{ (x_{0}, y) \in {\mathbb C}^{2} : |y| > \rho |x_{0}| \} \setminus \Gamma_{x_{0},j,s}^{\epsilon,\rho}$
for $x_{0} \in B(0,\delta)$ and $s \in \{0,+,-\}$.

The sets $\Gamma_{0,j,0}^{\epsilon, \rho} \cap \Gamma_{0,j,+}^{\epsilon, \rho}$ and
$\Gamma_{0,j,0}^{\epsilon, \rho} \cap \Gamma_{0,j,-}^{\epsilon, \rho}$ are singletons.
Indeed we have
\[ \Gamma_{0,j,0}^{\epsilon, \rho} \cap \Gamma_{0,j,+}^{\epsilon, \rho} = \{ \Gamma_{0,j,0}^{\epsilon, \rho} (h_{+}) \}
\ {\rm and} \
\Gamma_{0,j,0}^{\epsilon, \rho} \cap \Gamma_{0,j,-}^{\epsilon, \rho} = \{ \Gamma_{0,j,0}^{\epsilon, \rho} (-h_{-}) \} \]
for some $h_{-}, h_{+} \in {\mathbb R}^{+}$.  Consider $M_{0} > \max (h_{-}, h_{+})$. Hence the point
$\Gamma_{0,j,0}^{\epsilon, \rho} (-M_{0})$ belongs to $H_{j,-}^{\epsilon,\rho}(0)$ whereas
$\Gamma_{0,j,0}^{\epsilon, \rho} (M_{0})$ belongs to $H_{j,+}^{\epsilon,\rho}(0)$. We obtain
\[ \Gamma_{x,j,0}^{\epsilon, \rho} (-M_{0}) \in H_{j,-}^{\epsilon,\rho}(x) \ {\rm and} \
\Gamma_{x,j,0}^{\epsilon, \rho} (M_{0}) \in H_{j,+}^{\epsilon,\rho}(x) \ \forall x \in B(0,\delta) . \]
By applying prop. \ref{pro:estext} to $X$ and $-iX$ we obtain that there exists
$\theta_{0} \in {\mathbb R}^{+}$ such that
\[ H_{j,0}^{\epsilon,2 \rho_{0}}(x) \cup H_{j,+}^{\epsilon,2 \rho_{0}}(x) \cup H_{j,-}^{\epsilon,2 \rho_{0}}(x)
\subset \{x\} \times (0,\epsilon) e^{i[-\theta_{0},\theta_{0}]} \]
for any $x \in B(0,\delta)$.
We extend
$\psi_{L_{j}}^{X}$ to
$H_{j,0}^{\epsilon,2 \rho_{0}} \cup H_{j,+}^{\epsilon,2 \rho_{0}} \cup H_{j,-}^{\epsilon,2 \rho_{0}}$
by analytic continuation.
We can apply prop. \ref{pro:ext1g} to $X$ and $-iX$ to obtain
\begin{equation}
\label{equ:pexpsi}
 \frac{1}{C_{6} |y|^{\nu({\mathcal E}_{0})}} \leq
|\psi_{L_{j}}^{X}| \leq \frac{C_{6}}{|y|^{\nu({\mathcal E}_{0})}} \ \ {\rm in} \
H_{j,0}^{\epsilon,2 \rho_{0}} \cup H_{j,+}^{\epsilon,2 \rho_{0}} \cup H_{j,-}^{\epsilon,2 \rho_{0}}
\end{equation}
for some $C_{6} \geq 1$.
Moreover $\Delta_{\varphi}=O(y^{k (\nu({\mathcal E}_{0})+1)})$ implies that there exists
$K_{2} \in {\mathbb R}^{+}$ such that
\begin{equation}
\label{equ:deltak2}
|\Delta_{\varphi}(x,y)| \leq
\frac{K_{2}}{(1+|\psi_{L_{j}}^{X}(x,y)|)^{\frac{k(\nu({\mathcal E}_{0})+1)}{\nu({\mathcal E}_{0})}}}
\ \ \forall (x,y) \in
H_{j,0}^{\epsilon,2 \rho_{0}} \cup H_{j,+}^{\epsilon,2 \rho_{0}} \cup H_{j,-}^{\epsilon,2 \rho_{0}}.
\end{equation}
Consider $d_{0} \in {\mathbb R}^{+}$ such that
\begin{equation}
\label{equ:chlint}
\frac{K_{2}}{d_{0}^{k}} < \min
\left({ \frac{\sin (\theta)}{2}, \frac{\tan(\theta)}{16 \sup_{\mathbb R} |\varrho|} }\right)
\end{equation}
where $\varrho$ is the function defined in Step 1 of subsection \ref{subsec:defatcor}.
Consider $M \geq M_{0}$ such that
$Im(\psi_{L_{j}}^{X}(\Gamma_{x,j,0}^{\epsilon, \rho} (-M))) < -d_{0}$ and
$Im(\psi_{L_{j}}^{X}(\Gamma_{x,j,0}^{\epsilon, \rho} (M))) > d_{0}$ for any $x \in B(0,\delta)$.
We define
\[ \Gamma_{x,j,\theta,\pm}^{\epsilon, \rho} =
\Gamma (- e^{i \mp \theta} X,\Gamma_{x,j,0}^{\epsilon, \rho} (\pm M), {\mathcal E}_{0}^{\rho})
({\mathcal I}(\Gamma_{x,j,\theta,\pm}^{\epsilon, \rho})) \]
where ${\mathcal I}(\Gamma_{x,j,\theta,\pm}^{\epsilon, \rho}) =
{\mathcal I}(- e^{i \mp \theta} X,\Gamma_{x,j,0}^{\epsilon, \rho} (\pm M), {\mathcal E}_{0}^{\rho}) \cap
({\mathbb R}^{+} \cup \{0\})$
and
\[ \Gamma_{x,j,\theta}^{\epsilon, \rho} =
\Gamma_{x,j,0}^{\epsilon, \rho}[-M,M] \cup \Gamma_{x,j,\theta,+}^{\epsilon, \rho} \cup
\Gamma_{x,j,\theta,-}^{\epsilon, \rho}, \ \
P_{x,j,\theta,\pm}^{\epsilon, \rho}=
\Gamma_{x,j,\theta,\pm}^{\epsilon, \rho}(\sup {\mathcal I}(\Gamma_{x,j,\theta,\pm}^{\epsilon, \rho})) . \]
The vector field $\Re (- e^{-i \theta} X)$ is transversal to $\Gamma_{x,j,+}^{\epsilon, \rho}$
and points towards $H_{j,+}^{\epsilon,\rho}$ at $T_{X}^{\epsilon,+}(x)$ for any $x \in B(0,\delta)$.
Therefore $P_{x,j,\theta,+}^{\epsilon, \rho}$ belongs to $\{ |y|=\rho |x| \}$. Analogously
$\Re (- e^{i \theta} X)$ is transversal to $\Gamma_{x,j,-}^{\epsilon, \rho}$
and points towards $H_{j,-}^{\epsilon,\rho}$ at $T_{X}^{\epsilon,-}(x)$ for any $x \in B(0,\delta)$.
Thus $P_{x,j,\theta,-}^{\epsilon, \rho}$ belongs to $\{ |y|=\rho |x| \}$.
\begin{defi}
We define $H_{j,\theta}^{\epsilon,\rho}(x_{0})$ as the bounded component of
\[ \{ (x_{0}, y) \in {\mathbb C}^{2} : |y| > \rho |x_{0}| \} \setminus
\Gamma_{x_{0},j,\theta}^{\epsilon, \rho} \]
for $x_{0} \in B(0,\delta)$.
We have $H_{j,\pi/2}^{\epsilon,\rho}(x) \subset H_{j,\theta}^{\epsilon,\rho}(x)$
for any choice of $\rho \geq 2\rho_{0}$, $\theta \in (0,\pi/2]$ and $x \in B(0,\delta)$.
We denote $\tilde{H}_{j,\theta}^{\epsilon,\rho}(x_{0})=H_{j,\theta}^{\epsilon,\rho}(x_{0})$
for $x_{0} \neq 0$.
\end{defi}
\begin{defi}
We define
\[ H_{j,\theta}^{\epsilon,\rho,\lambda} =H_{\Lambda,j}^{\lambda} \cup
\cup_{x \in [0,\delta) I_{\Lambda}^{\lambda}} H_{j,\theta}^{\epsilon,\rho}(x), \
H_{j,\theta}^{\epsilon,\rho,\lambda,\lambda'} =H_{\Lambda,\Lambda',j}^{\lambda,\lambda'} \cup
\cup_{x \in [0,\delta) I_{\Lambda,\Lambda'}^{\lambda,\lambda'}} H_{j,\theta}^{\epsilon,\rho}(x) . \]
\end{defi}
{\bf Step 2.}
The extension of Fatou coordinates is going to be obtained through iteration.
In order to obtain similar asymptotic behavior as in subsection \ref{subsec:defatcor}
we use lemma \ref{lem:techsum}. Next result assures that the hypotheses of
lemma  \ref{lem:techsum} are satisfied.
\begin{lem}
\label{lem:hjter}
Let $\varphi \in \diff{tp1}{2}$. Let $\Upsilon={\rm exp}(X)$ be a $k$-convergent normal form.
Let $j \in {\mathcal D}(\varphi)$  and $\theta \in (0,\pi/2]$.
Consider $\rho \geq 2 \rho_{0}$ and $\theta_{1} \in [-\theta,\theta]$.
Then $H_{j,\theta}^{\epsilon,\rho}(x)$ is $\Re(e^{i \theta_{1}} X)$-convex
for any $x \in B(0,\delta)$. In particular
$\psi_{L_{j}}^{X}(H_{j,\theta}^{\epsilon,\rho}(x))$ is contained in the set
$\psi_{L_{j}}^{X}(T_{iX}^{\epsilon,j}(x)) + W_{\theta, M}$ (see def. \ref{def:W}).
\end{lem}
\begin{proof}
The result is straightforward if $x=0$. Suppose $x \neq 0$.
Suppose without lack of generality that $j \in {\mathcal D}_{1}(\varphi)$ (see def. \ref{def:dirvf}).
Denote ${\mathcal C}={\mathcal C}_{0}$.
The vector field $\Re (e^{i \theta_{1}} X)$ is transversal to
$\partial H_{j,\theta}^{\epsilon,\rho}(x) \setminus
(\Gamma_{x,j,\theta}^{\epsilon, \rho} \cup T {\mathcal C}_{e^{i \theta_{1}} X}^{\rho}(x))$
for any $\theta_{1} \in [-\theta,\theta]$.
Moreover $\Re (e^{i \theta_{1}} X)$ is almost transversal to
$\partial H_{j,\theta}^{\epsilon,\rho}(x)$ at the set of convex points
$T {\mathcal C}_{e^{i \theta_{1}} X}^{\rho}(x)$
for any $\theta_{1} \in [-\theta,\theta]$.

The function $Im(\psi_{L_{j}}^{X})$ is injective in
$\Gamma_{x,j,\theta}^{\epsilon, \rho}$ by construction.
Since $\Re (e^{i \theta_{1}} X)$  is transversal to
$\Gamma_{x,j,0}^{\epsilon, \rho}[-M,M]$, $\Gamma_{x,j,\theta,+}^{\epsilon, \rho}$ and
$\Gamma_{x,j,\theta,-}^{\epsilon, \rho}$ at any of their points we deduce that
$Im(\psi_{L_{j}}^{X} e^{-i \theta_{1}})$ is injective in
$\Gamma_{x,j,\theta}^{\epsilon, \rho}$ for any $\theta_{1} \in (-\theta,\theta)$.
As a consequence given $\theta_{1} \in (-\theta,\theta)$ the vector field
$\Re (e^{i \theta_{1}} X)$ is almost transversal to
$\Gamma_{x,j,\theta}^{\epsilon, \rho}$ at any point in
$\Gamma_{x,j,\theta,-}^{\epsilon, \rho} \setminus
\{ P_{x,j,\theta,+}^{\epsilon, \rho}, P_{x,j,\theta,-}^{\epsilon, \rho} \}$.
Moreover in the neighborhood of $P_{x,j,\theta,\pm}^{\epsilon, \rho}$ the curve
$\psi_{L_{j}}^{X}(\partial H_{j,\theta}^{\epsilon,\rho}(x) \setminus \Gamma_{x,j,\theta}^{\epsilon, \rho})$
is to the right of $\psi_{L_{j}}^{X}(\Gamma_{x,j,\theta}^{\epsilon, \rho})$.
Analogously as in subsection \ref{subsec:rj1} we obtain that
$\Re (e^{i \theta_{1}} X)$ is almost transversal to $\partial H_{j,\theta}^{\epsilon,\rho}(x)$
at both $P_{x,j,\theta,-}^{\epsilon, \rho}$ and $P_{x,j,\theta,+}^{\epsilon, \rho}$
for any $\theta_{1} \in (-\theta,\theta)$.  Lemma \ref{lem:transv}
implies that $H_{j,\theta}^{\epsilon,\rho}(x)$ is $\Re(e^{i \theta_{1}} X)$-convex
for any $\theta_{1} \in (-\theta,\theta)$.

Let us consider the flow $\Re (e^{i \theta} X)$. The proof for $\Re (e^{-i \theta} X)$
is analogous. We can proceed as in the previous paragraphs to obtain that
$\Re (e^{i \theta} X)$ is almost transversal to $\partial H_{j,\theta}^{\epsilon,\rho}(x)$
at any point outside of $\Gamma_{x,j,\theta,-}^{\epsilon, \rho}$.
The analysis of the properties of $P_{x,j,\theta,-}^{\epsilon, \rho}$ also
implies
${\rm exp}(s e^{i \theta} X)(P_{x,j,\theta,-}^{\epsilon, \rho}) \not \in
\overline{H_{j,\theta}^{\epsilon,\rho}(x)}$ for any $s <0$ in a neighborhood of $0$.

Given $Q \in H_{j,\theta}^{\epsilon,\rho}(x)$ we denote
$\Gamma_{Q}=\Gamma(e^{i \theta} X, Q, H_{j,\theta}^{\epsilon,\rho})$ and
$(s_{-},s_{+})= {\mathcal I}(\Gamma_{Q})$. Clearly
$\Gamma_{Q}(s_{+}) \in \partial H_{j,\theta}^{\epsilon,\rho}(x) \setminus
\Gamma_{x,j,\theta,-}^{\epsilon, \rho}$, thus
$\Re (e^{i \theta} X)$ is almost transversal to $\partial H_{j,\theta}^{\epsilon,\rho}(x)$
at $\Gamma_{Q}(s_{+})$. Analogously if
$\Gamma_{Q}(s_{-}) \not \in \Gamma_{x,j,\theta,-}^{\epsilon, \rho}$ the vector field
$\Re (e^{i \theta} X)$ is almost transversal to $\partial H_{j,\theta}^{\epsilon,\rho}(x)$
at $\Gamma_{Q}(s_{-})$. Otherwise we obtain
$\Gamma_{Q}(s_{-}) = \Gamma_{x,j,0}^{\epsilon, \rho} (- M)$.
We deduce
\[ {\mathcal I}(\Gamma(e^{i \theta} X, Q, \overline{H_{j,\theta}^{\epsilon,\rho}}))=
[s_{-} - \sup {\mathcal I}(\Gamma_{x,j,\theta,-}^{\epsilon, \rho} ), s_{+}]  \]
and then
\[ (s_{-},s_{+})=
\{s \in {\mathcal I}(\Gamma(e^{i \theta} X, Q, \overline{H_{j,\theta}^{\epsilon,\rho}})) :
\Gamma(e^{i \theta} X, Q, \overline{H_{j,\theta}^{\epsilon,\rho}})(s) \in H_{j,\theta}^{\epsilon,\rho} \} . \]
Lemma \ref{lem:transv} implies that $H_{j,\theta}^{\epsilon,\rho}(x)$ is
$Re(e^{i \theta} X)$-convex.

There exists $s_{0} \in {\mathbb R}^{+}$ such that
${\rm exp}(s e^{i \theta_{1}} X)(Q) \in  H_{j,\theta}^{\epsilon,\rho}(x)$ for all
$s \in (0,s_{0})$, $\theta_{1} \in [-\theta,\theta]$ and
$Q \in \Gamma_{x,j,0}^{\epsilon, \rho}[-M,M]$.
Since the set $H_{j,\theta}^{\epsilon,\rho}(x)$ is
$Re(e^{i \theta_{1}} X)$-convex for $\theta_{1} \in [-\theta,\theta]$ we deduce that the sets
$\psi_{L_{j}}^{X}(H_{j,\theta}^{\epsilon,\rho}(x))$ and
$\psi_{L_{j}}^{X}(Q) + e^{i[-\theta,\theta]} ({\mathbb R}^{-} \cup \{0\})$ are disjoint for any
$Q \in \Gamma_{x,j,0}^{\epsilon, \rho}[-M,M]$.
Therefore $\psi_{L_{j}}^{X}(H_{j,\theta}^{\epsilon,\rho}(x))$ is contained in the set
$\psi_{L_{j}}^{X}(T_{iX}^{\epsilon,j}(x)) + W_{\theta, M}$ (see def. \ref{def:W}).
\end{proof}
\begin{figure}[h]
\begin{center}
\includegraphics[height=6cm,width=5.5cm]{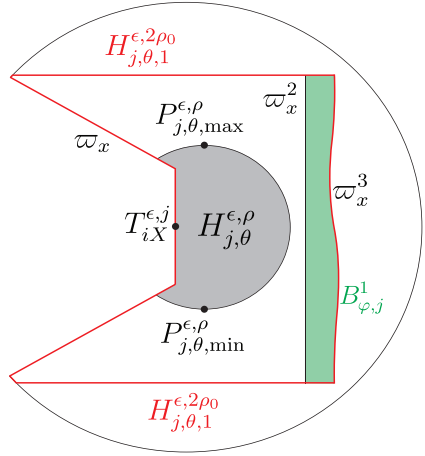}
\end{center}
\caption{Picture of $H_{j,\theta,1}^{\epsilon,2 \rho_{0}}$ and $B_{\varphi, j}^{1}$
in the $\psi_{L_{j}}^{X}$ coordinate}
\label{EVf9}
\end{figure}

{\bf Step 3.}
In this step we define a subset $H_{j,\theta,1}^{\epsilon,2 \rho_{0}}$ of
$H_{j,\theta}^{\epsilon, 2 \rho_{0}}$ and sort of a fundamental domain
$B_{\varphi, j}^{1}$ of $\varphi_{|H_{j,\theta,1}^{\epsilon,2 \rho_{0}}}$ such that
$B_{\varphi, j}^{1}(x) \subset H_{\Lambda,j}^{\lambda}(x)$ for
$x \in (0,\delta) I_{\Lambda}^{\lambda}$ and
$B_{\varphi, j}^{1}(x) \subset H_{\Lambda,\Lambda',j}^{\lambda,\lambda'}(x)$ for $x \in (0,\delta) I_{\Lambda,\Lambda'}^{\lambda,\lambda'}$.
Moreover we have $H_{j,\theta}^{\epsilon, \rho} \subset H_{j,\theta,1}^{\epsilon,2 \rho_{0}}$ for
$\rho >>0$. As a consequence a Fatou coordinate of $\varphi$ defined in $B_{\varphi,j}^{1}$
can be extended to $H_{j,\theta}^{\epsilon, \rho}$ (see Step 4). In other words to define a Fatou coordinate
of $\varphi$ in $H_{j,\theta}^{\epsilon, \rho}$ we can restrict ourselves to $B_{\varphi,j}^{1}$.
This is a key property in next subsection's results.

The fundamental domain $B_{\varphi,j}^{1}(x)$ has different properties than
$B_{\varphi}^{2}(T_{iX}^{\epsilon,j}(x))$ (see Step 3 of subsection \ref{subsec:comdifmflo}).
The former one tends to the origin when $x \to 0$ whereas the latter one tends to
$B_{\varphi}^{2}(T_{iX}^{\epsilon,j}(0))$. This discrepancy is justified by
the methods that we use in subsection \ref{subsec:asyfatcor} to
study the asymptotic developments of Lavaurs vector fields.

More precisely, let ${\rm exp}(Y_{k})$ be a $k$-convergent normal form of $\varphi$.
In subsection \ref{subsec:asyfatcor} we want to find Fatou coordinates
$\psi_{j,\lambda,k}^{\varphi}$ of $\varphi$ in $H_{j,\theta}^{\epsilon, \rho}$ such that
$\psi_{j,\lambda,k}^{\varphi}$ satisfies
\[ \psi_{j,\lambda,k}^{\varphi} - \psi_{L_{j},k} = O((y \circ \varphi -y)^{c_{k}}) \]
where $\psi_{L_{j},k}$ is a Fatou coordinate of $Y_{k}$ and $\lim_{k \to \infty} c_{k} = \infty$.
We obtain $\psi_{j,\lambda,k}^{\varphi}$ in a fundamental domain by
using synthesis and solving a $\overline{\partial}$
equation. Then we extend the result by iteration. If the fundamental domain is
$B_{\varphi}^{2}(T_{iX}^{\epsilon,j}(x))$ then $O(y \circ \varphi -y)$ is a $O(1)$
and such property does not improve by iteration. On the other hand the choice of
$B_{\varphi,j}^{1}$ turns out to be good since we have
\[ \sup_{(x,y') \in B_{\varphi,j}^{1}(x)} |y \circ \varphi - y|(x,y') =
O((y \circ \varphi - y)(x,y)) \]
when $x \to 0$ in $H_{j,\theta}^{\epsilon, \rho}$.

Let us explicit the construction.
We define $arc_{j,\theta}^{\epsilon,\rho}(x_{0})= \partial H_{j,\theta}^{\epsilon,\rho}(x_{0})
\cap \{|y|=\rho x_{0}\}$. By proposition \ref{pro:hit} the sections
$P_{x,j,\theta,+}^{\epsilon, \rho}$ and $P_{x,j,\theta,-}^{\epsilon, \rho}$
are asymptotically continuous in $B(0,\delta) \setminus \{0\}$.
Hence the arc $arc_{j,\theta}^{\epsilon,\rho}(r,\tilde{\lambda})$ converges in adapted coordinates
$(x,w)=(x,y/x)$ to an arc $arc_{j,\theta}^{\epsilon,\rho}(0,\lambda_{0})$ when
$(r,\tilde{\lambda}) \to (0,\lambda_{0})$. The set
$H_{j,\theta}^{\epsilon,\rho}(r,\tilde{\lambda})$ tends to some set
$\tilde{H}_{j,\theta}^{\epsilon,\rho}(0,\lambda_{0})$ in adapted coordinates
when $(r,\tilde{\lambda}) \to (0,\lambda_{0})$. More precisely
$\tilde{H}_{j,\theta}^{\epsilon,\rho}(0,\lambda_{0})$ is contained in $\{|w| > \rho\}$
and its boundary is contained in the union of
$arc_{j,\theta}^{\epsilon,\rho}(0,\lambda_{0})$, a trajectory in
$Tr_{\leftarrow  \infty}(-e^{-i \theta} X_{0}(\lambda_{0}))$ and a trajectory in
$Tr_{\leftarrow \infty}(-e^{i \theta} X_{0}(\lambda_{0}))$ for $\lambda_{0} \in {\mathbb S}^{1}$.
We denote
$\psi_{\lambda_{0}}^{{\mathcal C}_{0}}$ the Fatou coordinate of $X_{0}(\lambda_{0})$
defined in the neighborhood of $\overline{\tilde{H}_{j,\theta}^{\epsilon,\rho}(0,\lambda_{0})}$
such that $\psi_{\lambda_{0}}^{{\mathcal C}_{0}}(\infty)=0$. We obtain
\[ \lim_{(r,\tilde{\lambda}) \to (0,\lambda_{0})}
(r^{\nu({\mathcal E}_{0})} \psi_{L_{j}}^{X}) (\overline{H_{j,\theta}^{\epsilon,\rho}(r,\tilde{\lambda})}) =
\psi_{\lambda_{0}}^{{\mathcal C}_{0}}(\overline{\tilde{H}_{j,\theta}^{\epsilon,\rho}(0,\lambda_{0})}) \]
in the Hausdorff topology for any $\lambda_{0} \in {\mathbb S}^{1}$.

We proceed as in subsection \ref{subsec:rj1}.
The maximum (resp. the minimum) of the function
$Im(r^{\nu({\mathcal E}_{0})} \psi_{L_{j}}^{X})$ in
$\overline{\tilde{H}_{j,\theta}^{\epsilon,\rho}(r,\tilde{\lambda})}$
is attained at a unique point $P_{j,\theta, \max}^{\epsilon, \rho}(r,\tilde{\lambda})$
(resp. $P_{j,\theta, \min}^{\epsilon, \rho}(r,\tilde{\lambda})$)
contained in $arc_{j,\theta}^{\epsilon,\rho}(r,\tilde{\lambda})$ for any
$(r, \tilde{\lambda}) \in [0,\delta) \times {\mathbb S}^{1}$. Consider
$a_{0}, a_{1} \in {\mathbb R}^{+}$ such that
$a_{0} + 4i[-a_{1},a_{1}]$ is contained in
$(r^{\nu({\mathcal E}_{0})} \psi_{L_{j}}^{X})(\tilde{H}_{j,\pi/2}^{\epsilon,2 \rho_{0}}(r,\tilde{\lambda}))$
for any $(r,\tilde{\lambda})) \in [0,\delta) \times {\mathbb S}^{1}$.
We choose $\rho \geq 2 \rho_{0}$ such that
$C_{6}/\rho^{\nu({\mathcal E}_{0})} \leq \min (a_{0},a_{1})/2$.
The equation (\ref{equ:pexpsi}) implies that
$\psi_{L_{j}}^{X}({H}_{j,\theta}^{\epsilon,\rho}(x))$ is contained in
\[ \left[{ \frac{-a_{0}}{2|x|^{\nu({\mathcal E}_{0})}}, \frac{a_{0}}{2|x|^{\nu({\mathcal E}_{0})}} }\right] +
i \left[{ \frac{-a_{1}}{2|x|^{\nu({\mathcal E}_{0})}}, \frac{a_{1}}{2|x|^{\nu({\mathcal E}_{0})}} }\right] \]
for any $x \in B(0,\delta) \setminus \{0\}$. We have
\[ \overline{\tilde{H}_{j,\theta}^{\epsilon,2 \rho_{0}}(r,\tilde{\lambda})} \cap
T({\mathcal C}_{0})_{X}^{2 \rho_{0}}(r,\tilde{\lambda})
\subset \{P_{j,\theta, \min}^{\epsilon, 2 \rho_{0}}(r,\tilde{\lambda}),
P_{j,\theta, \min}^{\epsilon, 2 \rho_{0}}(r,\tilde{\lambda}) \}
\   \forall (r,\tilde{\lambda}) \in [0,\delta) \times {\mathbb S}^{1}. \]
As a consequence there exists $\theta' \in (0,\theta]$ such that the angle between
$Re (X)$ and $arc_{j,\theta}^{\epsilon,2 \rho_{0}}(x)$ is greater or equal than $\theta'$
at any point in
\[ \cup_{x \in B(0,\delta) \setminus \{0\}}
\left({ arc_{j,\theta}^{\epsilon,2 \rho_{0}}(x) \cap \left\{{ Im (\psi_{L_{j}}^{X}) \in
\left[{ \frac{-(3+1/2)a_{1}}{|x|^{\nu({\mathcal E}_{0})}},
\frac{(3+1/2) a_{1}}{|x|^{\nu({\mathcal E}_{0})}} }\right] }\right\} }\right). \]
The angle between
$Re (X)$ and $\partial H_{j,\theta}^{\epsilon,2 \rho_{0}}(x)$ is greater or equal than
$\theta' = \min(\theta',\theta)$
at any point in
\[ \partial H_{j,\theta}^{\epsilon,2 \rho_{0}}(x) \cap \left\{{ Im (\psi_{L_{j}}^{X}) \in
\left[{ \frac{-(3+1/2) a_{1}}{|x|^{\nu({\mathcal E}_{0})}},
\frac{(3+1/2) a_{1}}{|x|^{\nu({\mathcal E}_{0})}} }\right] }\right\} \ \forall x \in B(0,\delta). \]
The previous set is a union of two curves, namely a curve $\varpi_{x}$ containing
$T_{i X}^{\epsilon, j}(x)$ and a curve $\varpi_{x}'$ contained in $\{|y| = 2 \rho_{0} |x|\}$.
The curve $\varpi_{x}$ is parametrized by $Im (\psi_{L_{j}}^{X})$. Indeed given
$x \in B(0,\delta) \setminus \{0\}$ and
$s \in (7/2) a_{1}[-1/|x|^{\nu({\mathcal E}_{0})},1/|x|^{\nu({\mathcal E}_{0})}]$
there exists a unique $d(\varpi_{x},s) \in {\mathbb R}$ such that
$d(\varpi_{x},s) + is \in \psi_{L_{j}}^{X}(\varpi_{x})$. We define
\[ \varpi_{x}^{1} = \varphi(\varpi_{x}) \cap
\left\{{ Im (\psi_{L_{j}}^{X}) \in
\left[{ \frac{-3a_{1}}{|x|^{\nu({\mathcal E}_{0})}}, \frac{3a_{1}}{|x|^{\nu({\mathcal E}_{0})}}
}\right] }\right\} \ \forall x \in B(0,\delta) \setminus \{0\} . \]
We can use the equation (\ref{equ:chlint}) and the techniques in Step 3 of subsection
\ref{subsec:defatcor} to obtain that $\varpi_{x}^{1}$ is transversal to $\Re(X)$
for any $x \in B(0,\delta) \setminus \{0\}$; moreover we obtain
$d(\varpi_{x},s) \leq d(\varpi_{x}^{1},s)$ for all
$x \in B(0,\delta)  \setminus \{0\}$ and
$s \in [-3a_{1}/|x|^{\nu({\mathcal E}_{0})},3a_{1}/|x|^{\nu({\mathcal E}_{0})}]$.
The curve $\psi_{L_{j}}^{X}(\varpi_{x}^{1})$ is to the right of $\psi_{L_{j}}^{X}(\varpi_{x})$.
We define the curves $\varpi_{x}^{2}$ and $\varpi_{x}^{3}$ contained in
$H_{j,\theta}^{\epsilon,2 \rho_{0}}(x)$ such that
\[ (|x|^{\nu({\mathcal E}_{0})} \psi_{L_{j}}^{X})(\varpi_{x}^{2})=
a_{0} + \frac{7}{2} i [-a_{1},a_{1}] , \ \varpi_{x}^{3} = \varphi(\varpi_{x}^{2}) \cap
\left\{{ |Im (\psi_{L_{j}}^{X})|  \leq \frac{3a_{1}}{|x|^{\nu({\mathcal E}_{0})}} }\right\} . \]
We proceed as in Step 3 of subsection \ref{subsec:defatcor} to prove that
$\varpi_{x}^{3}$ is transversal to $\Re (X)$ for any $x \in B(0,\delta) \setminus \{0\}$
and $d(\varpi_{x}^{2},s) \leq d(\varpi_{x}^{3},s)$ for all
$x \in B(0,\delta)  \setminus \{0\}$ and
$s \in [-3a_{1}/|x|^{\nu({\mathcal E}_{0})},3a_{1}/|x|^{\nu({\mathcal E}_{0})}]$.
The curve $\psi_{L_{j}}^{X}(\varpi_{x}^{3})$ is to the right of $\psi_{L_{j}}^{X}(\varpi_{x}^{2})$.
\begin{defi}
We define $H_{j,\theta,1}^{\epsilon,2 \rho_{0}}(x)$ as the subset of
$H_{j,\theta}^{\epsilon,2 \rho_{0}}(x)$ such that
\[ \psi_{L_{j}}^{X} (H_{j,\theta,1}^{\epsilon,2 \rho_{0}}(x)) =
\cup_{s \in [-a_{1}/|x|^{\nu({\mathcal E}_{0})}, a_{1}/|x|^{\nu({\mathcal E}_{0})}]}
( (d(\varpi_{x},s), d(\varpi_{x}^{3},s)) + is ) . \]
\end{defi}
\begin{defi}
\label{def:bfj1}
Given $l \in \{1,2,3\}$ we denote $B_{\varphi,j}^{l}(x)$ the subset of
$H_{j,\pi/2}^{\epsilon,2 \rho_{0}}(x)$ such that
\[ \psi_{L_{j}}^{X} (B_{\varphi,j}^{l}(x)) =
\cup_{s \in [-l a_{1}/|x|^{\nu({\mathcal E}_{0})}, l a_{1}/|x|^{\nu({\mathcal E}_{0})}]}
( [d(\varpi_{x}^{2},s), d(\varpi_{x}^{3},s)) + is ) . \]
Clearly $B_{\varphi,j}^{1}(x)$ is contained in $H_{j,\theta,1}^{\epsilon,2 \rho_{0}}(x)$.
\end{defi}
{\bf Step 4.}
Next we prove that a Fatou coordinate of $\varphi$ in $B_{\varphi,j}^{1}$ extends by
iteration to a Fatou coordinate in $H_{j,\theta}^{\epsilon,\rho}$.
\begin{lem}
\label{lem:fundext}
Let $\varphi \in \diff{tp1}{2}$. Let $\Upsilon={\rm exp}(X)$ be a $k$-convergent normal form.
Let $j \in {\mathcal D}(\varphi)$, $\theta \in (0,\pi/2]$.
Consider $x \in B(0,\delta) \setminus \{0\}$ and $P \in H_{j,\theta}^{\epsilon,\rho}(x)$.
Then there exists $l_{0}(P) \in {\mathbb Z}$ such that
$\{P, \hdots, \varphi^{l_{0}}(P) \} \subset H_{j,\theta,1}^{\epsilon,2 \rho_{0}}(x)$ and
$\varphi^{l_{0}}(P) \in B_{\varphi,j}^{1}(x)$.
\end{lem}
\begin{proof}
Suppose that $j \in {\mathcal D}_{1}(\varphi)$ without lack of generality. Suppose that
$\{P, \hdots, \varphi^{l_{1}}(P) \}$ is contained in $H_{j,\theta,1}^{\epsilon,2 \rho_{0}}(x)$
for some $l_{1} \in {\mathbb N} \cup \{0\}$. We have
\[ \psi_{L_{j}}^{X}(\varphi^{l_{1}+1}(P)) = \psi_{L_{j}}^{X}(P) + (l_{1}+1) + \sum_{l=0}^{l_{1}}
\Delta_{\varphi}(\varphi^{l}(P)) . \]
The set $\psi_{L_{j}}^{X}(H_{j,\theta}^{\epsilon,\rho}(x))$ is contained in the set
$\psi_{L_{j}}^{X}(T_{iX}^{\epsilon,j}(x)) + W_{\theta, M}$ (lemma \ref{lem:hjter}).
Thus equations (\ref{equ:deltak2}), (\ref{equ:chlint}) and lemma \ref{lem:techsum} imply that
there exists a constant $K_{3} \in {\mathbb R}^{+}$ independent of $P$, $x$ and $l_{1}$ such that
\begin{equation}
\label{equ:context}
| \psi_{L_{j}}^{X}(\varphi^{l_{1}+1}(P)) - \psi_{L_{j}}^{X}(P)- (l_{1}+1)| \leq
\frac{K_{3}}{(1+|\psi_{L_{j}}^{X}(P)|)^{k-1}} .
\end{equation}
Suppose that the lemma is false. Then
Step 3 implies that there exists $l_{2} \in {\mathbb N} \cup \{0\}$ such that
\[ \{P, \hdots, \varphi^{l_{2}}(P) \} \subset H_{j,\theta,1}^{\epsilon,2 \rho_{0}}(x) \ {\rm and} \
|Im (\psi_{L_{j}}^{X}(\varphi^{l_{2}+1}(P)))| > a_{1}/|x|^{\nu({\mathcal E}_{0})}  . \]
The choice of $\rho$ implies
\[ \frac{a_{1}}{2 |x|^{\nu({\mathcal E}_{0})}} \leq
 | Im(\psi_{L_{j}}^{X}(\varphi^{l_{2}+1}(P))) - Im( \psi_{L_{j}}^{X}(P))| \leq K_{3}. \]
We obtain a contradiction.
\end{proof}
We can extend $\ddot{\psi}_{j, \lambda}^{\varphi}$ and
$\ddot{\psi}_{j, \lambda}^{\varphi} - \ddot{\psi}_{j, \lambda'}^{\varphi}$
to $H_{j,\theta}^{\epsilon, \rho}$ by defining
\begin{equation}
\label{equ:exttheta}
\ddot{\psi}_{j, \lambda}^{\varphi}(P) = \ddot{\psi}_{j, \lambda}^{\varphi}(\varphi^{l_{0}}(P)) - l_{0} \ {\rm and} \
(\ddot{\psi}_{j, \lambda}^{\varphi} - \ddot{\psi}_{j, \lambda'}^{\varphi})(P)=
(\ddot{\psi}_{j, \lambda}^{\varphi} - \ddot{\psi}_{j, \lambda'}^{\varphi})(\varphi^{l_{0}}(P)) .
\end{equation}
The lemma \ref{lem:fundext} and property (\ref{equ:context}) imply
\begin{equation}
\label{equ:equext0}
|(\ddot{\psi}_{j,\lambda}^{\varphi}(P)- \psi_{L_{j}}^{X}(P)) -
(\ddot{\psi}_{j,\lambda}^{\varphi}(\varphi^{l_{0}}(P))-  \psi_{L_{j}}^{X}(\varphi^{l_{0}}(P)))|  \leq
\frac{K_{3}}{(1+|\psi_{L_{j}}^{X}(P)|)^{k-1}}.
\end{equation}
Since  $Re (\psi_{L_{j}}^{X}(Q)) \geq a_{0}/|x|^{\nu({\mathcal E}_{0})}$ for any
$Q \in B_{\varphi,j}^{1}(x)$ we can use equation (\ref{equ:pexpsi}) to obtain
$B_{\varphi,j}^{1}(x) \subset \{x\} \times \overline{B}(0,\sqrt[\nu({\mathcal E}_{0})]{C_{6}/a_{0}} |x|)$.
We deduce
\[ \lim_{x \in (0,\delta) I_{\Lambda}^{\lambda}, \
Q \in B_{\varphi,j}^{1}(x), \ Q \to (0,0)}
\ddot{\psi}_{j,\lambda}^{\varphi}(Q)-  \psi_{L_{j}}^{X}(Q) =
(\ddot{\psi}_{j,\lambda}^{\varphi}-  \psi_{L_{j}}^{X})(0,0) \]
by proposition \ref{pro:bddcon} since
$B_{\varphi,j}^{1}(x) \subset H_{\Lambda, j}^{\lambda}(x)$ for any
$x  \in (0,\delta) I_{\Lambda}^{\lambda}$ . This implies
\[ \lim_{Q \in H_{j,\theta}^{\epsilon,\rho,\lambda}, \ Q \to (0,0)}
\ddot{\psi}_{j,\lambda}^{\varphi}(Q)-  \psi_{L_{j}}^{X}(Q) =
(\ddot{\psi}_{j,\lambda}^{\varphi}-  \psi_{L_{j}}^{X})(0,0) . \]
as a consequence of equation (\ref{equ:equext0}).
The previous discussion leads us to
\begin{pro}
\label{pro:bddconext}
Let $\varphi \in \diff{tp1}{2}$. Let $\Upsilon$ be a $2$-convergent normal form.
Consider $\Lambda=(\lambda_{1}, \hdots, \lambda_{\tilde{q}}) \in {\mathcal M}$,
$\lambda \in {\mathbb S}^{1}$,
$j \in {\mathcal D}(\varphi)$ and $\theta \in (0,\pi/2]$.
Then there exists $\rho \geq 2 \rho_{0}$ such that the function
$\ddot{\psi}_{j,\Lambda,\lambda}^{\varphi} - \psi_{L_{j}}^{X}$ is continuous in
$\overline{H_{j,\theta}^{\epsilon,\rho,\lambda}}$ and holomorphic in the interior of
$H_{j,\theta}^{\epsilon,\rho,\lambda}$.
Moreover $\rho$ depends only on $X$, $\varphi$ and $\Lambda$.
\end{pro}
The definition of
$\ddot{\psi}_{j,\lambda}^{\varphi} - \ddot{\psi}_{j,\lambda'}^{\varphi}$ in equation (\ref{equ:exttheta})
and proposition \ref{pro:difFatflat} immediately imply
\begin{pro}
\label{pro:difFatflatext}
Let $\varphi \in \diff{tp1}{2}$ with $2$-convergent normal form ${\rm exp}(X)$.
Let $\Lambda,\Lambda' \in {\mathcal M}$. Consider
$\lambda, \lambda' \in {\mathbb S}^{1}$, $j \in {\mathcal D}(\varphi)$
and $\theta \in (0,\pi/2]$.
Then there exist $K \in {\mathbb R}^{+}$ and $\rho \geq 2 \rho_{0}$ such that
\[ |\ddot{\psi}_{j,\Lambda,\lambda}^{\varphi} - \ddot{\psi}_{j,\Lambda',\lambda'}^{\varphi}|(x,y) \leq
\frac{e^{-K/|x|^{e_{\Lambda,\Lambda',j}^{\lambda,\lambda',-}}}}{2} +    \frac{e^{-K/|x|^{e_{\Lambda,\Lambda',j}^{\lambda,\lambda',+}}}}{2}
\leq  e^{-K/|x|^{\tilde{e}_{d_{\Lambda,\Lambda'}^{\lambda,\lambda'}+1}}}  \]
for any $(x,y) \in H_{j,\theta}^{\epsilon,\rho,\lambda,\lambda'}$.
Moreover $\rho$ depends only on $X$, $\varphi$ and $\Lambda$.
\end{pro}
\subsection{Asymptotics of Fatou coordinates}
\label{subsec:asyfatcor}
Let $\varphi \in \diff{tp1}{2}$. Consider a $2$-convergent normal form $\Upsilon={\rm exp}(X)$.
Consider $\Lambda=(\lambda_{1}, \hdots, \lambda_{\tilde{q}}) \in {\mathcal M}$ and
the dynamical splitting $\digamma_{\Lambda}$ in remark \ref{rem:unifspl}.
Let $j \in {\mathbb Z}/(2 \nu ({\mathcal E}_{0}) {\mathbb Z})$. We use the notations in
the previous subsection.

We denote $f = y \circ \varphi -y$,
$\log \varphi = \hat{u}  f \partial/\partial y$ and
$X = u  f \partial/\partial y$
for some units $\hat{u} \in {\mathbb C}[[x,y]]$ and $u \in {\mathbb C}\{x,y\}$.
We obtain $\hat{u}-u \in (f^{2})$.
Fix $k \geq \max(5,4 \nu({\mathcal E}_{0}))$. Let ${\rm exp}(Y_{k})$ be a
$k$-convergent normal form. It is of the form ${\rm exp}(u_{k} f \partial/\partial y)$
for some unit $u_{k} \in {\mathbb C}\{x,y\}$.
We choose $\epsilon_{k} \in (0,\epsilon]$ such that $u_{k}$ does not vanish at any point of
$B(0,\delta) \times B(0,\epsilon_{k})$.
By construction we have $\hat{u}-u \in (f^{2})$ and $\hat{u}-u_{k} \in (f^{k})$.
We obtain $u - u_{k} \in (f^{2})$.

The set
$H_{j,0}^{\epsilon_{k},2 \rho_{0}} \cup H_{j,+}^{\epsilon_{k},2 \rho_{0}} \cup H_{j,-}^{\epsilon_{k},2 \rho_{0}}$
is contained in
$H_{j,0}^{\epsilon,2 \rho_{0}} \cup H_{j,+}^{\epsilon,2 \rho_{0}} \cup H_{j,-}^{\epsilon,2 \rho_{0}}$.
A Fatou coordinate  $\psi_{k}$ of $Y_{k}$ is a solution of the differential
equation $u_{k} f \partial \psi_{k}/\partial y=1$. Thus we can obtain a Fatou coordinate $\psi_{k}$ of $Y_{k}$
in
$H_{j,0}^{\epsilon_{k},2 \rho_{0}} \cup H_{j,+}^{\epsilon_{k},2 \rho_{0}} \cup H_{j,-}^{\epsilon_{k},2 \rho_{0}}$
of the form $\psi_{L_{j},k}=\psi_{L_{j}}^{X} +h$ where
$h \in {\mathcal O}(B(0,\delta) \times B(0,\epsilon_{k}))$ is a solution of
\[ \frac{\partial h}{\partial y} = \frac{1}{u_{k} f} - \frac{1}{u f} = \frac{1}{u u_{k}} \frac{u-u_{k}}{f} . \]
Such a solution exists since $u-u_{k} \in (f^{2})$. Moreover we can suppose $h(0,0)=0$.
Denote $\psi_{k} = \psi_{L_{j},k}$.

Denote
$I_{\lambda}= (0,\delta) \lambda e^{i(-\pi/(4 \nu({\mathcal E}_{0})), \pi/(4 \nu({\mathcal E}_{0})))}$.
In this subsection we build Fatou coordinates $\psi_{j,\lambda,k}^{\varphi}$ of $\varphi$
in $\cup_{x \in I_{\lambda}} (H_{j,\theta}^{\epsilon,\rho}(x) \cap H_{j,\theta/2}^{\epsilon_{k},\rho}(x))$
such that
\[ \psi_{j,\lambda,k}^{\varphi} - \psi_{L_{j},k} =
O(f^{k - 5 \nu({\mathcal E}_{0})/(\nu({\mathcal E}_{0})+1)}) . \]
Let $X_{j,\lambda,k}^{\varphi} = v_{k} f \partial /\partial y$ be the holomorphic
vector field such that $X_{j,\lambda,k}^{\varphi}(\psi_{j,\lambda,k}^{\varphi}) \equiv 1$.
The previous equation implies $v_{k} - u_{k} = O(f^{k - 5 \nu({\mathcal E}_{0})/(\nu({\mathcal E}_{0})+1)})$.
Thus the asymptotic development of $X_{j,\lambda,k}^{\varphi}$ is
$\log \varphi$ up to order $f^{k - 5 \nu({\mathcal E}_{0})/(\nu({\mathcal E}_{0})+1)}$.
The flatness properties in subsection \ref{subsec:comfatcor} imply that the difference
$X_{j,\lambda,k}^{\varphi} - X_{H_{\Lambda,j}^{\lambda}}^{\varphi}$ is exponentially small
(see corollary \ref{cor:Lavvf} and def. \ref{def:hjlam} for the definition of
$X_{H_{\Lambda,j}^{\lambda}}^{\varphi}$).
In this way we deduce in section \ref{sec:mulsuminf} that $\log \varphi$
is an asymptotic development of $X_{H_{\Lambda,j}^{\lambda}}^{\varphi}$.
The construction of $\psi_{j,\lambda,k}^{\varphi}$ is based on finding a well-behaved
$C^{\infty}$ Fatou coordinate and then deducing the existence of a holomorphic one
by solving a $\overline{\partial}$ equation.
\begin{defi}
\label{def:dkpm}
We define
\[ \Delta_{\varphi^{l},k} = \psi_{k} \circ \varphi^{l} - (\psi_{k} +l) \in {\mathbb C}\{x,y\} \cap (f^{k})
\ {\rm for} \  l \in \{-1,1\} . \]
\end{defi}
By considering a smaller $\epsilon_{k}>0$ we can suppose that $\Delta_{\varphi^{-1},k}$
is defined in a neighborhood of $\overline{B}(0,\delta) \times \overline{B}(0,\epsilon_{k})$.
We define the coordinates
\[ \left\{{ \begin{array}{l}
z= \psi_{k}(x,y) \\
\xi = \frac{1}{x^{\nu ({\mathcal E}_{0})}}
\end{array}
}\right. \  {\rm in} \
H_{j,0}^{\epsilon_{k},2 \rho_{0}} \cup H_{j,+}^{\epsilon_{k},2 \rho_{0}} \cup H_{j,-}^{\epsilon_{k},2 \rho_{0}}. \]
We define $U'= \cup_{x \in B(0,\delta) \setminus \{0\}, \ s \in B(0,1/4)} {\rm exp}(sX)(B_{\varphi,j}^{2}(x))$.
Analogously as for equation (\ref{equ:deltak2}) there exists  $K_{4} \in  {\mathbb R}^{+}$
such that
\begin{equation}
\label{equ:deltak4}
|\Delta_{\varphi^{-1},k}| \leq
\frac{K_{4}}{(1+|\psi_{L_{j}}^{X}|)^{k(1+1/\nu({\mathcal E}_{0}))}}
\ {\rm in} \
H_{j,0}^{\epsilon_{k},2 \rho_{0}} \cup H_{j,+}^{\epsilon_{k},2 \rho_{0}} \cup H_{j,-}^{\epsilon_{k},2 \rho_{0}}.
\end{equation}
By construction we have
$|\psi_{L_{j}}^{X}|(x,y) \geq a_{0}/(2 |x|^{\nu({\mathcal E}_{0})})$ in $U'$. We obtain
\begin{equation}
\label{equ:deltamk}
|\Delta_{\varphi^{-1},k}|(x,y) \leq  K_{5}
|x|^{k(\nu({\mathcal E}_{0})+1)} \ \forall (x,y) \in U'.
\end{equation}
where $K_{5}=K_{4} (2/a_{0})^{k(1+1/\nu({\mathcal E}_{0}))}$.
\begin{defi}
We define the mapping
\[ \sigma(\xi, z)= (\xi, z + \varrho (Re (z) - a_{0} |\xi|) \Delta_{\varphi^{-1}, k} (\xi,z)) \]
where $\varrho$ is the function defined in Step 1 of subsection \ref{subsec:defatcor}.
\end{defi}
The mapping $\sigma$ is defined in $U'$.
In fact it conjugates the diffeomorphisms $\varphi(\xi,z)$ and $(\xi,z+1)$.
Proceeding as in Step 3 of subsection \ref{subsec:defatcor} we obtain a constant
$K_{6} \in {\mathbb R}^{+}$ such that
\begin{equation}
\label{equ:deltamkd}
\left|{ \frac{\partial \Delta_{\varphi^{-1},k}}{\partial z} }\right|(\xi,z) \leq
\frac{K_{6}}{|\xi|^{k(1+1/\nu({\mathcal E}_{0}))}} \ \forall (\xi,z) \in U' .
\end{equation}
We have
\[ |z \circ \sigma(\xi, z) - z|   \leq
\frac{K_{5}}{|\xi|^{k(1+1/\nu({\mathcal E}_{0}))}} \ {\rm and} \
||({\mathcal J} \sigma)(\xi, z)  - Id|| <
\frac{K_{5} \sup_{\mathbb R} |\partial \varrho / \partial t|
+2 K_{6}}{|\xi|^{k(1+1/\nu({\mathcal E}_{0}))}} \]
for any $(\xi,z) \in U'$. Indeed $\sigma$ is a $C^{\infty}$ diffeomorphism from $U'$ to $\sigma(U')$.

Fix $\lambda_{0} \in {\mathbb S}^{1}$. Consider the sector
$\tilde{I}_{\lambda_{0}}= (0,\delta) \lambda_{0} e^{i(-7\pi/(24 \nu({\mathcal E}_{0})), 7\pi/(24 \nu({\mathcal E}_{0})))}$
in the $x$ coordinate. It corresponds to
$(\delta^{-\nu({\mathcal E}_{0})} , \infty) \lambda_{0}^{-\nu({\mathcal E}_{0})} e^{i[-7\pi/24,7\pi/24]}$
in the $\xi$ coordinate.
We define $U_{\lambda_{0}}' = \cup_{x \in \tilde{I}_{\lambda_{0}}} U'(x)$.
\begin{defi}
We define the functions
\[ h_{1}(\xi,z) = Re (2 a_{1} \xi \lambda_{0}^{\nu({\mathcal E}_{0})}) - Im(z) \ {\rm and} \
h_{2}(\xi,z) = Re (2 a_{1} \xi \lambda_{0}^{\nu({\mathcal E}_{0})}) + Im(z) \]
in $H_{j,0}^{\epsilon_{k},2 \rho_{0}} \cup H_{j,+}^{\epsilon_{k},2 \rho_{0}} \cup H_{j,-}^{\epsilon_{k},2 \rho_{0}}$.
We define the functions $\tau_{1}=h_{1} \circ \sigma$ and $\tau_{2}=h_{2} \circ \sigma$ in $U_{\lambda_{0}}'$.
\end{defi}
The functions $\tau_{1}$, $\tau_{2}$ satisfy
$\tau_{1} \circ \varphi = \tau_{1}$ and  $\tau_{2} \circ \varphi = \tau_{2}$ in
$U' \cap \varphi^{-1}(U')$. We have
\[  a_{1} |\xi| < 2 \cos (7 \pi/24) a_{1} |\xi| \leq Re (2 a_{1} \xi \lambda_{0}^{\nu({\mathcal E}_{0})})
\leq 2 a_{1} |\xi| \ \forall \xi \in \tilde{I}_{\lambda_{0}}. \]
The set
$\{ (\xi,z) \in B_{\varphi,j}^{3}: \xi \in \tilde{I}_{\lambda_{0}} \ {\rm and} \ h_{1}(\xi,z)>0<h_{2}(\xi,z) \}$
is contained in $B_{\varphi,j}^{2}$ and contains $\cup_{\xi \in \tilde{I}_{\lambda_{0}}} B_{\varphi,j}^{1}(\xi)$.
 The set
\[ \{ (\xi,z) \in U_{\lambda_{0}}' \cap B_{\varphi,j}^{3}: \xi \in \tilde{I}_{\lambda_{0}} \ {\rm and} \
\tau_{1}(\xi,z)>0< \tau_{2}(\xi,z) \} \]
contains $\cup_{\xi \in \tilde{I}_{\lambda_{0}}} B_{\varphi,j}^{1}(\xi)$. We define
\[ \tilde{U}_{\lambda_{0}}'=
\{ (\xi,z) \in U_{\lambda_{0}}': \tau_{1}(\xi,z) \tau_{2}(\xi,z) > 1 \} \]
and
\[ \tilde{B}_{\varphi,j}^{2} =
\{ (\xi,z) \in U_{\lambda_{0}}'  \cap B_{\varphi,j}^{3} : \xi \in \tilde{I}_{\lambda_{0}} \ {\rm and} \
\tau_{1}(\xi,z) \tau_{2}(\xi,z) > 1 \}. \]
The set $\tilde{B}_{\varphi,j}^{2}$ contains both
$\cup_{\xi \in \tilde{I}_{\lambda_{0}}} B_{\varphi,j}^{1}(\xi)$
and a fundamental domain of $\varphi_{|\tilde{U}_{\lambda_{0}}'}$.
We denote $U_{\lambda_{0}}^{*}$ the space of orbits of $\varphi_{|\tilde{U}_{\lambda_{0}}'}$.
\subsubsection{The $\overline{\partial}$ equation}
Our goal is defining a Fatou coordinate $\psi_{j,\lambda_{0}, k}^{\varphi}$ of $\varphi$ in
$\tilde{U}_{\lambda_{0}}'$.  We can define
$\psi_{j,\lambda_{0}, k}^{\varphi}  =  \sigma \circ  \psi_{k}$  if $\sigma$ is holomorphic.
In general we look for a function $v$ such that
\[ \sigma_{v}(\xi,z) =
(\xi, z + \varrho (Re (z) - a_{0} |\xi|) \Delta_{\varphi^{-1}, k} (\xi,z) + v(\xi,z))\]
is a holomorphic mapping conjugating $\varphi(\xi,z)$ and $(\xi,z+1)$.
The latter condition is equivalent to $v \circ \varphi \equiv v$ whereas the former condition
is equivalent to the equation
$\overline{\partial} v = \Omega$ where
$\Omega=- \overline{\partial} (z \circ \sigma(\xi,z))$ is
a $(0,1)$ form. Since $z \circ \sigma \circ \varphi(\xi,z) = z \circ \sigma +1$ we obtain
$\varphi^{*} \Omega \equiv \Omega$. The form $\Omega$
represents an element of $H^{0,1}(U_{\lambda_{0}}^{*})$.
It suffices to find a function $v$ defined in
$U_{\lambda_{0}}^{*}$ such that $\overline{\partial} v = \Omega$.

We have
\[ \Omega = - \Delta_{\varphi^{-1}, k}(\xi,z) \overline{\partial} (\varrho (Re (z) - a_{0} |\xi|)) . \]
Moreover $\Omega$ is of the form $\Omega = A(\xi,z) d \overline{\xi} + B(\xi,z) d \overline{z}$
where $A, B$ satisfy
\[ |A(\xi,z)| \leq
\frac{a_{0} K_{5} \sup_{\mathbb R}|\partial \varrho /  \partial t |}
{2 |\xi|^{k(1+1/\nu({\mathcal E}_{0}))}}, \
|B(\xi,z)| \leq \frac{K_{5} \sup_{\mathbb R}|\partial \varrho /  \partial t |}
{2 |\xi|^{k(1+1/\nu({\mathcal E}_{0}))}} \ \ \forall (\xi,z) \in U_{\lambda_{0}}' . \]
\begin{defi}
Since $\tau_{l}$ is defined in $U_{\lambda_{0}}^{*}$ we define
$\omega_{l}= \partial \tau_{l} \in H^{1,0}(U_{\lambda_{0}}^{*})$ for $l \in \{1,2\}$.
We denote
$\overline{\partial} {\omega}_{l} = \sum {c_{bd}^{l} \overline{{\omega}_{b}} \wedge {\omega}_{d}}$
for $l \in \{1,2\}$.
We consider the decomposition
\begin{equation}
\label{equ:dercon}
dc_{bd}^{l} = \sum_{g=1}^{2} {\partial}_{g} (c_{bd}^{l}) w_{g} +
\sum_{g=1}^{2} \overline{{\partial}_{g}} (c_{bd}^{l}) \overline{w_{g}} .
\end{equation}
We define the volume elements
\[ dV_{0}= {(i/2)}^{2} d z \wedge d \overline{z} \wedge d \zeta \wedge d \overline{\zeta} \ {\rm and} \
dV = {(i/2)}^{2}  {\omega}_{1} \wedge  \overline{\omega}_{1} \wedge  {\omega}_{2} \wedge  \overline{\omega}_{2}.\]
\end{defi}
The forms $dV_{0}$ and $dV$ are defined in $U_{\lambda_{0}}'$. The form $dV$ is also defined in
$U_{\lambda_{0}}^{*}$.
\begin{lem}
\label{lem:dbset}
Denote $e=k(1+1/\nu({\mathcal E}_{0}))$. We have
\begin{itemize}
\item The forms ${\omega}_{1}$ and ${\omega}_{2}$ compose a base of the cotangent space in every
point of $U_{\lambda_{0}}'$ and $U_{\lambda_{0}}^{*}$.
\item There exists $K_{7} \geq 1$ such that
\[ \frac{1}{K_{7}} dV_{0} \leq   dV \leq K_{7} dV_{0} \ \ {\rm in} \ U_{\lambda_{0}}'. \]
\item We have $\partial {\omega}_{1} = \partial {\omega}_{2}=0$,
$\overline{\partial} ({\omega}_{1} + {\omega}_{2})=0$ and
\[ |c_{bd}^{l}| \leq K_{8} /|\xi^{e}| \ , \ |{\partial}_{g} c_{bd}^{l}| \leq K_{8} /|\xi^{e}| \
{\rm and} \ |\overline{\partial}_{g} c_{bd}^{l}| \leq K_{8} /|\xi^{e}| \]
for some $K_{8}>0$ and any $(l,b,d,g) \in {\{1,2\}}^{4}$.
\end{itemize}
\end{lem}
\begin{proof}
Denote $\rho(\xi,z)=\varrho(Re (z) - a_{0} |\xi|)$ and $\Delta= \Delta_{\varphi^{-1}, k}$.
We have
\[ \frac{\partial \tau_{1}}{\partial \xi} =a_{1} \lambda_{0}^{\nu({\mathcal E}_{0})} -
\frac{\partial \rho}{\partial \xi}  Im(\Delta) +
\frac{i}{2} \frac{\partial \Delta}{\partial \xi} \rho . \]
We have $\Delta =O(1/\xi^{e})$ by equation (\ref{equ:deltamk}).
We use Cauchy's integral formula as in Step 3 of subsection \ref{subsec:defatcor} to obtain
a constant $K_{9} \in{\mathbb R}^{+}$ such that
\[ |\Delta | ,
\left|{\frac{\partial \Delta}{\partial \xi} }\right|  ,
\left|{\frac{\partial \Delta}{\partial z} }\right|  ,
\left|{\frac{\partial^{2} \Delta}{\partial \xi \partial z} }\right|  ,
\left|{\frac{\partial^{2} \Delta}{\partial z^{2}} }\right|  ,
\left|{\frac{\partial^{2} \Delta}{\partial \xi^{2}} }\right| \leq
\frac{K_{9}}{|\xi|^{e}}  \]
in $U'$. Therefore we obtain
\[ \left|{\frac{\partial \tau_{l}}{\partial \xi} - a_{1} \lambda_{0}^{\nu({\mathcal E}_{0})} }\right|(\xi,z)
\leq \frac{1}{2} \frac{K_{9}}{|\xi|^{e}}
( a_{0}  \sup_{\mathbb R} |\partial \varrho / \partial t|  + 1) \ \forall (\xi,z) \in U_{\lambda_{0}}'  \
\forall l \in \{1,2\}. \]
The proof for $\tau_{2}$ is analogous. We also have
\[  \left|{\frac{\partial \tau_{1}}{\partial z} - \frac{i}{2} }\right|(\xi,z)
\leq \frac{K_{9}}{2 |\xi|^{e}} \left({
\sup_{\mathbb R} |\varrho'| + 1 }\right), \
 \left|{\frac{\partial \tau_{2}}{\partial z} + \frac{i}{2} }\right|(\xi,z)
\leq \frac{K_{9}}{2 |\xi|^{e}} \left({
\sup_{\mathbb R} |\varrho'| + 1 }\right) \]
for any $(\xi,z) \in U_{\lambda_{0}}'$. Thus the two first points of the lemma hold true.

Consider a composition $\partial_{l} \circ ... \circ \partial_{1}$ of operators where $l \leq 3$
and $\partial_{b}$ is either $\partial /\partial \xi$,
$\partial /\partial \overline{\xi}$, $\partial /\partial z$ or $\partial /\partial \overline{z}$ for
$b \leq 3$.
Consider an operator $\partial'= \partial_{b_{d}} \circ \hdots \circ \partial_{b_{1}}$ where
$d \leq l$ and $1 \leq b_{1} < \hdots < b_{d} \leq l$.
Let $l_{\xi}'$ be the number of times that we apply the operators $\partial /\partial \xi$
and $\partial /\partial \overline{\xi}$ in $\partial'$.
Suppose $l_{\xi}' \geq 1$.
We have that $\partial' (-a_{0}|\xi|) = O(\xi^{-l_{\xi}' +1})$. Since
$(\partial_{l} \circ ... \circ \partial_{1})(\rho)$
is a polynomial with rational coefficients in
$\varrho'$, $\varrho^{2)}$ and $\varrho^{3)}$ and functions of the form
$(\partial_{b_{d}} \circ \hdots \circ \partial_{b_{1}})(-a_{0}|\xi|)$
we deduce that $(\partial_{l} \circ ... \circ \partial_{1})(\rho)$ is bounded in $U_{\lambda_{0}}'$.

We have that the form  $\overline{\partial} \omega_{1}$ is equal to
\[  - Im \left({
\frac{\partial^{2} \rho}{\partial \xi \partial \overline{\xi}} \Delta  +
\frac{\partial \Delta}{\partial \xi} \frac{\partial \rho}{\partial \overline{\xi}}
}\right) d \overline{\xi} \wedge d \xi +
 \left({
- \frac{\partial^{2} \rho}{\partial \xi \partial \overline{z}} Im(\Delta) - \frac{i}{2}
\overline{\frac{\partial \Delta}{\partial z}} \frac{\partial \rho}{\partial \xi} +
\frac{i}{2}  \frac{\partial \Delta}{\partial \xi} \frac{\partial \rho}{\partial \overline{z}}
}\right) d \overline{z} \wedge d \xi + \]
\[  \left({
- \frac{\partial^{2} \rho}{\partial \overline{\xi} \partial z} Im(\Delta) - \frac{i}{2}
\overline{\frac{\partial \Delta}{\partial \xi}} \frac{\partial \rho}{\partial z} +
\frac{i}{2}  \frac{\partial \Delta}{\partial z} \frac{\partial \rho}{\partial \overline{\xi}}
}\right)  d \overline{\xi} \wedge d z
- Im \left({
\frac{\partial^{2} \rho}{\partial z \partial \overline{z}} \Delta  +
\frac{\partial \Delta}{\partial z} \frac{\partial \rho}{\partial \overline{z}}
}\right)  d \overline{z} \wedge d z . \]
Since
\[
\left({ \begin{array}{l}
\omega_{1} \\
\omega_{2}
\end{array}
}\right) =
\left({ \begin{array}{cc}
\frac{\partial \tau_{1}}{\partial \xi} &  \frac{\partial \tau_{1}}{\partial z} \\
\frac{\partial \tau_{2}}{\partial \xi} &  \frac{\partial \tau_{2}}{\partial z}
\end{array}
}\right)
\left({ \begin{array}{l}
d \xi \\
dz
\end{array}
}\right) \implies
\left({ \begin{array}{l}
d \xi \\
dz
\end{array} }\right)
=
\frac{\left({ \begin{array}{cc}
\frac{\partial \tau_{2}}{\partial z} & -  \frac{\partial \tau_{1}}{\partial z} \\
- \frac{\partial \tau_{2}}{\partial \xi} &  \frac{\partial \tau_{1}}{\partial \xi}
\end{array}
}\right)}{2  a_{1} \lambda_{0}^{-\nu({\mathcal E}_{0})} \frac{\partial \tau_{2}}{\partial z} }
\left({ \begin{array}{l}
\omega_{1} \\
\omega_{2}
\end{array} }\right)
\]
the coefficients $c_{bd}^{l}$ are of the form
$\tilde{c}_{bd}^{l}/(2 a_{1} |\partial \tau_{2}/\partial z|)^{2}$ where
numerator and denominator are polynomials with complex coefficients in
$\Delta$, $\partial \Delta/\partial \xi$, $\partial \Delta/\partial z$,
their complex conjugates and
functions of the form $(\partial_{g} \circ ... \circ \partial_{1})(\rho)$ with $g \leq 2$. There is no
independent term in the $\Delta$ variables for $\tilde{c}_{bd}^{l}$ and
$(2 a_{1} |\partial \tau_{2}/\partial z|)^{2} - a_{1}^{2}$. Therefore
$d c_{bd}^{l}$ is a $O(1/|x|^{e})$ in the base $\{d \xi, d \overline{\xi}, dz, d \overline{z}\}$
and then in the base $\{ \omega_{1}, \omega_{2}, \overline{\omega}_{1}, \overline{\omega}_{2}\}$ for any
$(l,b,d) \in {\{1,2\}}^{3}$.
\end{proof}
\subsubsection{Stein character of the space of orbits}
Next step is proving that $U_{\lambda_{0}}^{*}$ is pseudo-convex. We define the function
\[ u_{2} = u_{2}^{0} + u_{2}^{1} + u_{2}^{2} + u_{2}^{3} \]
where $u_{2}^{0} = -\ln (h_{1} h_{2}-1)$, $u_{2}^{1}= 5 \ln |\xi|$,
$u_{2}^{2}=- \ln ({|\xi|}^{2}-1/\delta^{2 \nu({\mathcal E}_{0})})$
and
\[ u_{2}^{3}=
-\ln \left({
[Im(\ln (\xi \lambda_{0}^{\nu({\mathcal E}_{0})})) + 7\pi/24]
[7\pi/24 - Im(\ln (\xi \lambda_{0}^{\nu({\mathcal E}_{0})}))]
}\right) .
\]
 Since none of the functions $u_{2}^{l}$ depends on $Re(z)$ then
$u_{2}$ is a $C^{\infty}$ function of $\sigma(U_{\lambda_{0}}^{*})$.
The function $u_{2}$ is a good candidate to be a $C^{\infty}$ p.s.h. exhaustion of
$\sigma(U_{\lambda_{0}}^{*})$. We define $u_{1} = u_{2} \circ \sigma$.
\begin{lem}
\label{lem:p2upsh}
The function $u_{1}$ is a p.s.h. $C^{\infty}$ exhaustion of
$U_{\lambda_{0}}^{*}$ for $\delta>0$ small enough.
In particular $U_{\lambda_{0}}^{*}$ is a Stein manifold.
\end{lem}
\begin{proof}
 The function $u_{1}$ is $C^{\infty}$ and in order to prove that it is an exhaustion it suffices
to prove that $u_{2}$ is an exhaustion of $\sigma(U_{\lambda_{0}}^{*})$.
  The set $\sigma(U_{\lambda_{0}}^{*})$ is biholomorphic to the subset
$S = (\xi, {e}^{2 \pi i z})(\sigma(\tilde{U}_{\lambda_{0}}'))$ of ${\mathbb C}^{2}$
since we have to identify $(\xi,z)$ and $(\xi,z+1)$.
The functions $\ln |\xi|$
and $u_{2}^{3}$ are bounded by below. The function $2 \ln |\xi| + u_{2}^{2}$ is also bounded by
below. The functions $h_{1}$ and $h_{2}$ are smaller than $4 a_{1} |\xi|$, we have
\[ u_{2}^{0} + 2 \ln |\xi| \geq \ln ({|\xi|}^{2}/(16 a_{1}^{2} {|\xi|}^{2})) = - 2 \ln (4 a_{1}) \]
and then $u_{2}^{0} + 2 \ln |\xi|$ is also bounded by below. Since
\[ u_{2} = (u_{2}^{0} + 2 \ln |\xi|) + \ln |\xi| + (u_{2}^{2}+ 2 \ln |\xi|) + u_{2}^{3} \]
the function $u_{2}$ extends continuously to $\overline{S}$ by defining
$(u_{2})_{|\partial S} \equiv \infty$. It suffices to prove the compactness of
$S_{K} = \{ (\xi,w) \in S : u_{2}(\xi,w) \leq K \}$ for any $K \in {\mathbb R}^{+}$.
The set $S_{K}$ is closed in ${\mathbb C}^{2}$ since $(u_{2})_{|\partial S} \equiv \infty$.
There exists $K_{1}' \in {\mathbb R}$ such that
$\ln |\xi| \leq K_{1}'$ for any $(\xi,w) \in S_{K}$.
The set $S_{K}$ is contained in $|\xi| \leq {e}^{K_{1}'}$.
Moreover since $-2 a_{1} |\xi|<Im(z)<2 a_{1} |\xi|$ in $\sigma(\tilde{U}_{\lambda_{0}}')$ then
\[ |\xi| \leq {e}^{K_{1}'} \ {\rm and} \ |{e}^{2 \pi i z}| \leq
{e}^{4 \pi a_{1} {\rm exp}(K_{1}')} \ \ \forall (\xi,z) \in \sigma(\tilde{U}_{\lambda_{0}}') \cap
\{u_{2} \leq K\}. \]
The set $S_{K}$ is closed and bounded and then compact.
Thus $\sigma({U}_{\lambda_{0}}^{*}) \cap \{u_{2} \leq K\}$ is a compact set for any $K \in {\mathbb R}^{+}$.

  The remainder of the proof is devoted to show that $u_{1}$ is a p.s.h. funtion in $U_{\lambda_{0}}^{*}$.
It suffices to prove that $u=-\ln(\tau_{1} \tau_{2}-1)$ is strictly p.s.h in $U_{\lambda_{0}}^{*}$
since it is straightforward to prove that $u_{2}^{l} \circ \sigma$ is p.s.h. for $l \in \{1,2,3\}$. We have
\[ \overline{\partial} u = \frac{\partial{u}}{\partial{\tau_{1}}} \overline{\partial} \tau_{1}
+  \frac{\partial{u}}{\partial{\tau_{2}}} \overline{\partial} \tau_{2} \]
and then
\[ \partial \overline{\partial} u = \sum_{bl}
{\frac{{\partial}^{2} u}{\partial \tau_{b} \partial \tau_{l}}} {\omega}_{b} \wedge \overline{\omega}_{l}
+ \frac{\partial{u}}{\partial{\tau_{1}}} \partial \overline{\partial} \tau_{1} +
\frac{\partial{u}}{\partial{\tau_{2}}} \partial \overline{\partial} \tau_{2} . \]
We denote
$\overline{\partial} {\omega}_{l} = \sum_{bd} {c_{bd}^{l} \overline{{\omega}_{b}} \wedge {\omega}_{d}}$
for $l \in \{1,2\}$.
We can express $\partial \overline{\partial} u$ as
\[ \sum_{bl} {\omega}_{b} \wedge \overline{\omega}_{l}
{\left({
{\frac{{\partial}^{2} u}{\partial \tau_{b} \partial \tau_{l}}} + c_{lb}^{1} \frac{\partial{u}}{\partial{\tau_{1}}}
+ c_{lb}^{2} \frac{\partial{u}}{\partial{\tau_{2}}}
}\right)} .
\]
We define the coefficients $u_{bl}$ to satisfy
$\partial \overline{\partial} u = \sum_{bl} u_{bl} {\omega}_{b} \wedge \overline{\omega}_{l}$. Then we have
\[ u_{bl} = \frac{{\partial}^{2} u}{\partial \tau_{b} \partial \tau_{l}}
+ c_{lb}^{1}
\left({
\frac{\partial{u}}{\partial{\tau_{1}}}- \frac{\partial{u}}{\partial{\tau_{2}}}
}\right) . \]
We define now $\rho= \tau_{1}\tau_{2}-1$. We can calculate the first derivatives of $u$ to obtain
\[ \frac{\partial{u}}{\partial{\tau_{1}}} = -\frac{\tau_{2}}{\rho} \ , \
\frac{\partial{u}}{\partial{\tau_{2}}} = -\frac{\tau_{1}}{\rho} \ \ {\rm and} \ \
\frac{\partial{u}}{\partial{\tau_{1}}}- \frac{\partial{u}}{\partial{\tau_{2}}} =
\frac{\tau_{1}-\tau_{2}}{\rho} \]
and then the second derivatives
\[ \frac{{\partial}^{2} u}{\partial {\tau_{1}}^{2}} = \frac{{\tau_{2}}^{2}}{{\rho}^{2}} \ , \
\frac{{\partial}^{2} u}{\partial {\tau_{2}}^{2}} = \frac{{\tau_{1}}^{2}}{{\rho}^{2}} \ \ {\rm and} \ \
\frac{{\partial}^{2} u}{\partial \tau_{1} \partial \tau_{2}} = \frac{1}{{\rho}^{2}} . \]
We can apply these calculations to show that $\sum_{bl} u_{bl} {\zeta}_{b} \overline{\zeta}_{l}$ is
greater or equal than
\[
\frac{(\tau_{2}^{2} |{\zeta}_{1}|^{2} + \tau_{1}^{2} |{\zeta}_{2}|^{2} +
 ({\zeta}_{1} \overline{\zeta}_{2} + \overline{\zeta}_{1} {\zeta}_{2}))}{{\rho}^{2}}
- \frac{2 K_{8} |\tau_{1}-\tau_{2}| {|\xi|}^{-k(1+1/\nu({\mathcal E}_{0}))}}{\rho}
({|{\zeta}_{1}|}^{2} + {|{\zeta}_{2}|}^{2}) . \]
where $K_{8} \in {\mathbb R}^{+}$ is the constant in lemma \ref{lem:dbset}. We have
\[ \frac{(\tau_{2}^{2} |{\zeta}_{1}|^{2} + \tau_{1}^{2} |{\zeta}_{2}|^{2} +
({\zeta}_{1} \overline{\zeta}_{2} + \overline{\zeta}_{1} {\zeta}_{2}))}{{\rho}^{2}}
\geq \frac{1}{{\rho}^{2}} \frac{{\tau_{1}}^{2} {\tau_{2}}^{2} - 1}{ {\tau_{1}}^{2} + \tau_{2}^{2} }
({|{\zeta}_{1}|}^{2} + {|{\zeta}_{2}|}^{2}) . \]
The right-hand side is equal to
\[ \frac{\tau_{1}\tau_{2}+1}{\rho ({\tau_{1}}^{2}+{\tau_{2}}^{2})} ({|{\zeta}_{1}|}^{2} + {|{\zeta}_{2}|}^{2}) \]
and since $(\tau_{1}\tau_{2}+1)/({\tau_{1}}^{2}+{\tau_{2}}^{2}) \geq 2 / {(\tau_{1}+\tau_{2})}^{2}$
if $\tau_{1}\tau_{2} >1$ then
\[ \frac{(\tau_{2}^{2} |{\zeta}_{1}|^{2} + \tau_{1}^{2} |{\zeta}_{2}|^{2} +
 ({\zeta}_{1} \overline{\zeta}_{2} + \overline{\zeta}_{1} {\zeta}_{2}))}{{\rho}^{2}} \geq
\frac{1}{\rho} \frac{2}{{(\tau_{1}+\tau_{2})}^{2}} ({|{\zeta}_{1}|}^{2} + {|{\zeta}_{2}|}^{2}) . \]
We remark that $\tau_{1}+\tau_{2}=4 a_{1} Re(\xi \lambda_{0}^{\nu({\mathcal E}_{0})})$ and
$|\tau_{2}-\tau_{1}| < 4 a_{1} |\xi|$, we deduce
\[ \sum_{bl} u_{bl} {\zeta}_{b} \overline{\zeta}_{l} \geq
\frac{1}{\rho}
\left({
\frac{1}{8 a_{1}^{2} Re(\xi \lambda_{0}^{\nu({\mathcal E}_{0})})^{2}}
- 8 a_{1} K_{8}  {|\xi|}^{-k(1+1/\nu({\mathcal E}_{0}))+1}
}\right)
({|{\zeta}_{1}|}^{2} + {|{\zeta}_{2}|}^{2})
. \]
Since $k(1+1/\nu({\mathcal E}_{0})) >3$ we get
\[ \sum_{bl} u_{bl}(\xi,z) {\zeta}_{b} \overline{\zeta}_{l} \geq
\frac{1}{\rho} \frac{1}{9 a_{1}^{2} Re(\xi \lambda_{0}^{\nu({\mathcal E}_{0})})^{2}}
({|{\zeta}_{1}|}^{2} + {|{\zeta}_{2}|}^{2}) \ \forall (\xi,z) \in \tilde{U}_{\lambda_{0}}' \
\forall (\zeta_{1},\zeta_{2}) \in {\mathbb C}^{2} \]
for $\delta>0$ small enough. The function $u$ is strictly p.s.h. in $U_{\lambda_{0}}^{*}$.
\end{proof}
\subsubsection{Estimates for the solution of the $\overline{\partial}$ equation}
\label{subsec:estsoldbe}
The equation $\overline{\partial} v=\Omega$ has a solution in $U_{\lambda_{0}}^{*}$ since
$U_{\lambda_{0}}^{*}$ is a Stein manifold. We could define
\[ \psi_{j,\lambda_{0},k}^{\varphi} = \psi_{L_{j},k} + \varrho (Re (z) - a_{0} |\xi|) \Delta_{\varphi^{-1}, k}+v \]
as a Fatou coordinate of $\varphi$ in $\tilde{U}_{\lambda_{0}}'$. Unfortunately
such a solution does not necessarily satisfy good estimates as the one we will prove later on
(see proposition \ref{pro:fatcorork}). Thus we need a well-behaved solution of the
$\overline{\partial}$ equation.
We will use the following theorem due to H\"{o}rmander (see \cite{PM-Yo}):
\begin{teo}
\label{teo:hor}
Let ${\mathcal U}$ be a Stein manifold of dimension $n$ with an hermitian metric.
Let 
$\omega_{1}, \hdots, \omega_{n}\ \in \Omega^{1,0}({\mathcal U})$ be
$C^{\infty}$ forms such that
$\{ \omega_{1}, \hdots, \omega_{n} \}$ is an orthonormal base of 
$T^{*} {\mathcal U}$ at any point. Denote
$dV = {(i/2)}^{n}  {\omega}_{1} \wedge  \overline{\omega}_{1} \wedge  \hdots \wedge
{\omega}_{n} \wedge  \overline{\omega}_{n}$ the volume element.
Denote
${\partial} {\omega}_{d} = \frac{1}{2} 
\sum_{bl} {a_{bl}^{d} {\omega}_{b} \wedge {\omega}_{l}}$ 
and
$\overline{\partial} {\omega}_{d} = \sum_{bl} {c_{bl}^{d} \overline{{\omega}_{b}} \wedge {\omega}_{l}}$
for $1 \leq d \leq n$. Suppose that there exist
$\theta_{0}, \theta_{1}: {\mathcal U} \to {\mathbb R}$ with
$|c_{bl}^{d}| \leq \theta_{0} \geq |a_{bl}^{d}| $,
$|{\partial}_{g} c_{bl}^{d}| \leq \theta_{1} \geq |{\partial}_{g} a_{bl}^{d}| $ 
and
$|\overline{\partial}_{g} c_{bl}^{d}| \leq \theta_{1} \geq |\overline{\partial}_{g} a_{bl}^{d}|$
in ${\mathcal U}$
for any $(g,b,l,d) \in {\{1,\hdots,n\}}^{4}$.
Let $\phi: {\mathcal U} \to {\mathbb R}$  be a $C^{2}$  strictly p.s.h. function.  
Suppose that there exists
a continuous $\theta: {\mathcal U} \to {\mathbb R}^{+}$ such that
\[ \sum_{bl} {\phi}_{bl} {\zeta}_{b} \overline{\zeta}_{l}
\geq (\theta + A({\theta}_{0}^{2} + {\theta}_{1}))  \sum_{d=1}^{n} {|{\zeta}_{d}|}^{2}  \]
in ${\mathcal U}$ where
$\partial \overline{\partial} \phi = \sum_{bl} \phi_{bl} {\omega}_{b} \wedge \overline{\omega}_{l}$
and $A \in {\mathbb R}^{+}$ is a universal constant depending only on the dimension of ${\mathcal U}$.
Consider a  $\overline{\partial}$-closed $(0,1)$ form $\Omega \in C^{\infty}({\mathcal U})$ 
such that
\[ {\int}_{{\mathcal U}} {\theta}^{-1} {|\Omega|}^{2} {e}^{-\phi} dV < \infty . \]
Then there exists a complex function 
$v \in C^{\infty} ({\mathcal U})$ such that $\overline{\partial} v = \Omega$ and
\[ {\int}_{{\mathcal U}} {{|v|}^{2} {e}^{-\phi}} dV \leq
{\int}_{{\mathcal U}} {\theta}^{-1} {|\Omega|}^{2} {e}^{-\phi} dV. \]
\end{teo}
\begin{pro}
There exists a solution $v : U_{\lambda_{0}}^{*} \to {\mathbb C}$ of
$\overline{\partial} v = \Omega$ such that
\[ {\int}_{U_{\lambda_{0}}^{*}} {{|v|}^{2} {|\xi|}^{2k(1+1/\nu({\mathcal E}_{0}))-10}
(\tau_{1}\tau_{2}-1)} dV < K_{13}  \]
for some $K_{13} \in {\mathbb R}^{+}$.
\end{pro}
\begin{proof}
Let us apply theorem \ref{teo:hor}.
In our situation we define ${\mathcal U}=U_{\lambda_{0}}^{*}$ and
\[ \phi = -2 (k(1+1/\nu({\mathcal E}_{0})) -5) \ln |\xi| - \ln(\tau_{1}\tau_{2}-1) . \]
The function $\phi$ is p.s.h. in $U_{\lambda_{0}}^{*}$.
We can choose $\theta_{0}=\theta_{1}=K_{8}/|\xi|^{k(1+1/\nu({\mathcal E}_{0}))}$
by lemma \ref{lem:dbset}.
%
Denote $\rho= \tau_{1}\tau_{2}-1$.
We define $\theta=1/(10 a_{1}^{2} \rho Re(\xi \lambda_{0}^{\nu({\mathcal E}_{0})})^{2})$. We have
\[ \sum_{bl} {\phi}_{bl} {\zeta}_{b} \overline{\zeta}_{l} \geq \sum_{bl} {u}_{bl} {\zeta}_{b} \overline{\zeta}_{l}
\geq
\left({
\frac{1}{9 a_{1}^{2} \rho Re(\xi \lambda_{0}^{\nu({\mathcal E}_{0})})^{2}}
}\right) ({|{\zeta}_{1}|}^{2} + {|{\zeta}_{2}|}^{2}) .
\]
Now we use
$\rho Re(\xi \lambda_{0}^{\nu({\mathcal E}_{0})})^{2} A ({\theta}_{0}^{2} + {\theta}_{1})
=O({\xi}^{4-k(1+1/\nu({\mathcal E}_{0}))})$
to prove
\[ \frac{1}{9 a_{1}^{2} \rho Re(\xi \lambda_{0}^{\nu({\mathcal E}_{0})})^{2}} \geq
\theta + A({\theta}_{0}^{2} + {\theta}_{1}). \]
Next step is proving
\[ {\int}_{U_{\lambda_{0}}^{*}} {\theta}^{-1} {|\Omega|}^{2} {e}^{-\phi} dV < \infty . \]
Since $\Omega= - \overline{\partial}(\varrho (Re (z) - a_{0} |\xi|)) \Delta_{\varphi^{-1}, k}$
and  $\overline{\partial}(\varrho (Re (z) - a_{0} |\xi|))$
is bounded (see proof of lemma \ref{lem:dbset}) then we obtain
$\Omega =O(|\xi|^{-e})$.
The integral is smaller or equal than
\[ K_{10} {\int}_{U_{\lambda_{0}}^{*}}
{  {|\xi|}^{4}  {|\xi|}^{-2e}  {|\xi|}^{2(e-5)}  {|\xi|}^{2}} dV \]
for $e=k(1+1/\nu({\mathcal E}_{0}))$ and some positive constant $K_{10}$. We deduce
\[ {\int}_{U_{\lambda_{0}}^{*}} {\theta}^{-1} {|\Omega|}^{2} {e}^{-\phi} dV \leq
K_{11} {\int}_{U_{\lambda_{0}}'} {\frac{1}{{|\xi|}^{4}}} dV_{0}  \]
for some $K_{11}>0$. The area of $U_{\lambda_{0}}'(\xi)$ is
a $O(\xi)$ when $\xi \to \infty$. Thus we have
\[ {\int}_{U_{\lambda_{0}}^{*}} {\theta}^{-1} {|\Omega|}^{2} {e}^{-\phi} dV \leq
K_{12} {\int}_{\tilde{I}_{\lambda_{0}}} {\frac{1}{{|\xi|}^{3}}} (i/2) d\xi \wedge d \overline{\xi} < K_{13}  \]
for some positive constants $K_{12}$ and $K_{13}$. We obtain the existence of a function $v$
defined in $U_{\lambda_{0}}^{*}$ such that $\overline{\partial} v = \Omega$ and
${\int}_{U_{\lambda_{0}}^{*}} {{|v|}^{2} {e}^{-\phi}} dV < K_{13}$. Hence we obtain
\[ {\int}_{U_{\lambda_{0}}^{*}} {{|v|}^{2} {|\xi|}^{2e-10} (\tau_{1}\tau_{2}-1)} dV < K_{13}  \]
as intended.
\end{proof}
\begin{defi}
We define
$U= \cup_{x \in B(0,\delta) \setminus \{0\}, \ s \in B(0,1/8)} {\rm exp}(sX)(B_{\varphi,j}^{2}(x))$,
\[ I_{\lambda_{0}}= (0,\delta) \lambda_{0} e^{i(-\pi/(4 \nu({\mathcal E}_{0})), \pi/(4 \nu({\mathcal E}_{0})))}
\ {\rm and} \  U_{\lambda_{0}} = \cup_{x \in I_{\lambda_{0}}} U(x). \]
We denote
$\tilde{U}_{\lambda_{0}}=
\{ (\xi,z) \in U_{\lambda_{0}}: \tau_{1}(\xi,z) >3< \tau_{2}(\xi,z) \}$.
\end{defi}
\begin{lem}
Let $L= \varrho (Re (z) - a_{0} |\xi|) \Delta_{\varphi^{-1}, k}+v$.
There exists $K_{15} \in {\mathbb R}^{+}$ such that
$|L| \leq K_{15} |\psi_{k}|^{5-k(1+1/\nu({\mathcal E}_{0}))}$ in $\tilde{U}_{\lambda_{0}}$.
\end{lem}
\begin{proof}
The function $L$
is holomorphic in $\tilde{U}_{\lambda_{0}}'$. It is not defined in general in $U_{\lambda_{0}}^{*}$.
The next step is estimating the modulus of the function $L$. We have
\[ {\int}_{U_{\lambda_{0}}^{*} \cap \{ \tau_{1} \geq 2\} \cap \{\tau_{2} \geq 2\}}
{{|v|}^{2} {|\xi|}^{2e-10}} dV \leq \frac{K_{13}}{3} . \]
Since
${|L|}^{2} \leq 2{|v|}^{2} + 2 {|\Delta_{\varphi^{-1}, k}|}^{2}$
we deduce that
\[ {\int}_{\tilde{U}_{\lambda_{0}}' \cap \{ \tau_{1} \geq 2\} \cap \{\tau_{2} \geq 2\}}
{|L|^{2} {|\xi|}^{2e-10}} d V_{0} \leq K_{14} \]
for some $K_{14}>0$.

Let $(\xi_{0},z_{0}) \in \tilde{U}_{\lambda_{0}}$.
Consider $\kappa \in {\mathbb R}^{+}$ such that $a_{0} \kappa \leq 1/32$ and
$a_{1} \kappa \leq 1/4$.
The polydisk $D_{\xi_{0},z_{0}}$ of center $(\xi_{0},z_{0})$ and poli-radius $(\kappa,1/16)$ is contained in
the set
$\tilde{U}_{\lambda_{0}}' \cap \{ \tau_{1} \geq 2\} \cap \{\tau_{2} \geq 2\}$. We have
\[ L(\xi_{0},z_{0}) \xi_{0}^{e-5} =
\frac{16^{2}}{\pi^{2} \kappa^{2}} \int_{D_{\xi_{0},z_{0}}} L(\xi,z) \xi^{e-5} dV_{0} . \]
This implies
\[ |L(\xi_{0},z_{0}) \xi_{0}^{e-5}| \leq
\frac{16^{2}}{\pi^{2} \kappa^{2}} \sqrt{\int_{D_{\xi_{0},z_{0}}} |L(\xi,z) \xi^{e-5}|^{2} dV_{0}}
\sqrt{\int_{D_{\xi_{0},z_{0}}}   dV_{0}} \leq \frac{16}{\pi \kappa} \sqrt{K_{14}} \]
for any $(\xi_{0},z_{0}) \in \tilde{U}_{\lambda_{0}}$.
As a consequence $L$ is a $O(\xi^{-e+5})=O(\psi_{k}^{5-e})$ in $\tilde{U}_{\lambda_{0}}$.
\end{proof}
We define
\[ \psi_{j,\lambda_{0},k}^{\varphi} = \psi_{L_{j},k} + L(x,y) . \]
It is a holomorphic Fatou coordinate of $\varphi$ in $\tilde{U}_{\lambda_{0}}$. We obtain
\begin{lem}
\label{lem:fatcorork}
Let $\varphi \in \diff{tp1}{2}$. Let $\Upsilon$ be a $2$-convergent normal form.
Consider $\Lambda=(\lambda_{1}, \hdots, \lambda_{\tilde{q}}) \in {\mathcal M}$,
$\lambda \in {\mathbb S}^{1}$,
$j \in {\mathcal D}(\varphi)$ and
$k \geq \max(5,4 \nu({\mathcal E}_{0}))$.
Then there exists $K_{16} \in {\mathbb R}^{+}$ such that
\[ |\psi_{j,\lambda,k}^{\varphi} - \psi_{L_{j},k}|(x,y) \leq
\frac{K_{16}}{(1+|\psi_{L_{j}}^{X}(x,y)|)^{k(1+1/\nu({\mathcal E}_{0}))-5}} \]
for any  $(x,y) \in  \tilde{U}_{\lambda}$.
\end{lem}
\subsubsection{Well-behaved Fatou coordinates}
\begin{pro}
\label{pro:fatcorork}
Let $\varphi \in \diff{tp1}{2}$. Let $\Upsilon$ be a $2$-convergent normal form.
Consider $\Lambda \in {\mathcal M}$,
$\lambda \in {\mathbb S}^{1}$,
$j \in {\mathcal D}(\varphi)$, $\theta \in (0,\pi/2]$ and
$k \geq \max(5,4 \nu({\mathcal E}_{0}))$.
Then there exists $\rho \geq 2 \rho_{0}$ such that the function
$\psi_{j,\lambda,k}^{\varphi}$ satisfies
\begin{equation}
\label{equ:bedb}
|\psi_{j,\lambda,k}^{\varphi} - \psi_{L_{j},k}|(x,y) \leq
\frac{K_{17}}{(1+|\psi_{L_{j}}^{X}(x,y)|)^{k-1}}
\end{equation}
for some $K_{17} \in {\mathbb R}^{+}$ and any
$(x,y) \in \cup_{x' \in I_{\lambda}} (H_{j,\theta}^{\epsilon,\rho}(x') \cap H_{j,\theta/2}^{\epsilon_{k},\rho}(x'))$.
Moreover $\rho$ depends only on $X$, $\varphi$ and $\Lambda$.
\end{pro}
\begin{proof}
Let $\rho$ be the same constant defined in proposition \ref{pro:bddconext}.

We claim that $\tilde{U}_{\lambda}(x)$ contains
a neighborhood of $\overline{B_{\varphi,j}^{1}(x)}$ for any $x \in I_{\lambda}$.
In fact we have
\[ |Im(z \circ \sigma)| \leq |Im(z)| +1 =|Im(\psi_{k})|+1 \leq |Im(\psi_{L_{j}}^{X})| + 2 \leq a_{1}  |\xi| +2 \]
in $B_{\varphi,j}^{1}$. Since
$Re(\xi \lambda^{\nu({\mathcal E}_{0})}) > |\xi| \cos (\pi/4)$ in $I_{\lambda}$
we obtain
\[ a_{1}  |\xi| +2 < 2 a_{1} Re(\xi \lambda^{\nu({\mathcal E}_{0})}) - 3
\ \forall \xi \in I_{\lambda} . \]
As a consequence $\overline{B_{\varphi,j}^{1}(x)}$ is contained in $\tilde{U}_{\lambda}(x)$
for any $x \in I_{\lambda}$. Therefore $\psi_{j,\lambda,k}^{\varphi}$ is defined in
$\cup_{x \in I_{\lambda}} B_{\varphi,j}^{1}(x)$.
We can suppose that $j \in {\mathcal D}_{1}(\varphi)$ without lack of generality.

Given $x \in B(0,\delta) \setminus \{0\}$ and
$P \in H_{j,\theta}^{\epsilon,\rho}(x) \cap H_{j,\theta/2}^{\epsilon_{k},\rho}(x)$ there exists
$l_{0}(P) \in {\mathbb N} \cup \{0\}$ such that
$\{P, \hdots, \varphi^{l_{0}}(P) \} \subset
H_{j,\theta,1}^{\epsilon,2 \rho_{0}}(x) \cap H_{j,\theta/2}^{\epsilon_{k},2 \rho_{0}}(x)$ and
$\varphi^{l_{0}}(P) \in B_{\varphi,j}^{1}(x)$ by lemma
\ref{lem:fundext}. By defining
$\psi_{j,\lambda,k}^{\varphi}(P) = \psi_{j,\lambda,k}^{\varphi}(\varphi^{l_{0}}(P)) - l_{0}$
we extend the function $\psi_{j,\lambda,k}^{\varphi}$ to
$H_{j,\theta}^{\epsilon,\rho}(x) \cap H_{j,\theta/2}^{\epsilon_{k},\rho}(x)$. We have
\[ (\psi_{j,\lambda,k}^{\varphi} - \psi_{L_{j},k})(P)  =
(\psi_{j,\lambda,k}^{\varphi} - \psi_{L_{j},k})(\varphi^{l_{0}}(P)) +
(\psi_{L_{j},k}(\varphi^{l_{0}}(P)) - \psi_{L_{j},k}(P) - l_{0}) \]
and then we obtain (see def. \ref{def:dkpm})
\[ (\psi_{j,\lambda,k}^{\varphi} - \psi_{L_{j},k})(P)  =
O(\xi^{1-k}) + \sum_{l=0}^{l_{0}-1} \Delta_{\varphi,k}(\varphi^{l}(P)) \]
by lemma \ref{lem:fatcorork} and $k \geq 4 \nu({\mathcal E}_{0})$.
Since $\Delta_{\varphi,k}=O(1/(1+| \psi_{L_{j}}^{X}|)^{k})$ in
$H_{j,\theta,1}^{\epsilon,2 \rho_{0}}$
lemma \ref{lem:techsum} implies
\[ \psi_{j,\lambda,k}^{\varphi} - \psi_{L_{j},k}  =
O(\xi^{-(k-1)}) + O \left({ \frac{1}{(1+|\psi_{L_{j}}^{X}|)^{k-1}} }\right) \ {\rm in} \
\cup_{x \in I_{\lambda}} (H_{j,\theta}^{\epsilon,\rho}(x) \cap H_{j,\theta/2}^{\epsilon_{k},\rho}(x)). \]
We have $\psi_{L_{j}}^{X} \sim 1/y^{\nu({\mathcal E}_{0})}$ by equation (\ref{equ:pexpsi}).
Thus $\psi_{L_{j}}^{X}$ is a $O(1/x^{\nu({\mathcal E}_{0})})=O(\xi)$
in $H_{j,\theta}^{\epsilon_{k},\rho}$. We deduce $\xi^{-1} =O(|\psi_{L_{j}}^{X}|^{-1})$
in $H_{j,\theta}^{\epsilon,\rho}(x) \cap H_{j,\theta/2}^{\epsilon_{k},\rho}(x)$.
Hence equation (\ref{equ:bedb}) holds true
for some $K_{17} \in {\mathbb R}^{+}$ and any
$(x,y) \in \cup_{x' \in I_{\lambda}} (H_{j,\theta}^{\epsilon,\rho}(x') \cap H_{j,\theta/2}^{\epsilon_{k},\rho}(x'))$.
\end{proof}
\section{Multi-summability of the infinitesimal generator}
\label{sec:mulsuminf}
The goal of this section is proving the multi-summable nature (with respect to the parameter $x$)
of the Fatou coordinates
and Lavaurs vector fields of an element $\varphi$ of $\diff{tp1}{2}$. In the latter case
we explain how the infinitesimal generator of $\varphi$ is summable.

The subsection \ref{subsec:mulsumfps} is a fast review of the results of summability theory
that we are going to use. In subsection \ref{subsec:mulsumfc} we study the extensions of the
Ecalle-Voronin invariants
\[ \ddot{\xi}_{j,\Lambda,\lambda}^{\varphi}(x,z) =
\ddot{\psi}_{j+1,\Lambda,\lambda}^{\varphi} \circ (x, \ddot{\psi}_{j,\Lambda,\lambda}^{\varphi})^{-1}(x,z) . \]
for $j \in {\mathcal D}(\varphi)$, $\Lambda \in {\mathcal M}$ and $\lambda \in {\mathbb S}^{1}$.
At first sight the definition does not make sense since
$H_{\Lambda,j}^{\lambda} \cap H_{\Lambda,j+1}^{\lambda} = \emptyset$ but this problem can be
solved by extending slightly the domains of definition of
$ \ddot{\psi}_{j,\Lambda,\lambda}^{\varphi}$ and $ \ddot{\psi}_{j+1,\Lambda,\lambda}^{\varphi}$.
We prove in theorem \ref{teo:difccfl} that the family
$\{  \ddot{\psi}_{j,\Lambda,\lambda}^{\varphi} \}_{(j,\Lambda,\lambda) \in
{\mathcal D}(\varphi) \times {\mathcal M} \times {\mathbb S}^{1}}$ represents a multi-summable
object. The proof is based on the estimates provided in prop. \ref{pro:difFatflat}.
The summable nature is concentrated in the $x$ variable since all our estimates are exponentially
small functions of $x$. We study the properties of the Lavaurs vector fields in
subsection \ref{subsec:mulsumlvf}.
Given $j \in {\mathcal D}(\varphi)$ the Lavaurs vector fields
$\{ X_{j,\Lambda,\lambda}^{\varphi}\}_{(\Lambda,\lambda) \in {\mathcal M} \times {\mathbb S}^{1}}$
(def. \ref{def:lavaurs}) represent a multi-summable object whose asymptotic development is
of the form
$\hat{X}_{j}^{\varphi} = \left({ \sum_{k=0}^{\infty} g_{j,k}^{\varphi}(y) x^{k} }\right)
\partial / \partial y$
where the coefficients $g_{j,k}^{\varphi}$ are defined in the petal of order $j$
of $\varphi_{|x=0}$ or more precisely in
$\cup_{\theta \in (0,\pi/2]} H_{j,\theta}^{\epsilon,\rho,\lambda}(0)$. The proof is based on
the estimates of prop. \ref{pro:difFatflatext}. The infinitesimal generator
$\log \varphi$ is of the form
$\left({ \sum_{k=0}^{\infty} \hat{g}_{k}^{\varphi}(y) x^{k}}\right) \partial / \partial y$.
Then given $k \in {\mathbb N} \cup \{0\}$ it is natural to ask whether the power series
$\hat{g}_{k}^{\varphi}$ is a $\nu({\mathcal E}_{0})$-summable function whose sums are the functions
$g_{j,k}^{\varphi}$ for $j \in {\mathcal D}(\varphi)$.
The answer is affirmative (theorem \ref{teo:infgenad} of subsection \ref{subsec:anainfgen}).
\subsection{Multi-summability of formal power series}
\label{subsec:mulsumfps}
For the sake of completeness we introduce the notations in \cite{MalRam:Fou} and \cite{RamSib:Fou}.

We consider ${\mathcal V}_{\lambda}$ the set of open subsets of ${\mathbb C}^{*}$ containing
a set of the form $(0,\zeta) \lambda e^{i(-\zeta,\zeta)}$ for some $\zeta \in {\mathbb R}^{+}$.

Let us introduce some sheafs defined in ${\mathbb S}^{1}$. Let ${\mathcal A}$ be the sheaf
of rings such that ${\mathcal A}_{\lambda}$ is the set of holomorphic functions $f$ defined
in some $W \in {\mathcal V}_{\lambda}$ and admitting an asymptotyc development
$\hat{f} = \sum_{l \geq 0} a_{l} x^{l}$ at the origin, i.e. we have
\[ |f(x) - \sum_{l=0}^{b-1} a_{l} x^{l}| \leq c_{b} |x|^{b} \ {\rm in} \ W \ {\rm for \ some} \
c_{b} \in {\mathbb R}^{+}. \]
We denote ${\mathcal A}^{<0}$ the subsheaf of ${\mathcal A}$ whose elements $f$ satisfy $\hat{f} \equiv 0$.

Given $s \in {\mathbb R}^{+}$ we define ${\mathcal A}_{(s)}$ the subsheaf of ${\mathcal A}$
such that an element $f$ of ${\mathcal A}_{(s),\lambda}$ defined in some $W \in {\mathcal V}_{\lambda}$
satisfies
\[ |f(x) - \sum_{l=0}^{b-1} a_{l} x^{l}| \leq c^{b} (b!)^{s} |x|^{b} \]
for any $x \in W$ and some $c \in {\mathbb R}^{+}$ independent of $b$.
The sheaf ${\mathcal A}_{(s)}$ is the sheaf of functions admitting a Gevrey asymptotic
expansion of order $s$. If $\hat{f}=\sum_{l=0}^{\infty} a_{l} x^{l}$ is a formal power series
such that $|a_{b}| \leq c^{b} (b!)^{s}$ for any $b \in {\mathbb N}$ and some $c \in {\mathbb R}^{+}$
we say that $\hat{f}$ is a formal power series of Gevrey order $s$.
We denote ${\mathbb C}[[x]]_{s}$ the set of formal power series
of Gevrey order $s$.

Given $k \geq 0$ we define ${\mathcal A}^{\leq -k}$ the subsheaf of ${\mathcal A}^{<0}$
such that an element $f$ of ${\mathcal A}_{\lambda}^{\leq -k}$ defined in some $W \in {\mathcal V}_{\lambda}$
satisfies $|f(x)| \leq A e^{-B/|x|^{k}}$ for any $x \in W$ and some $A,B \in {\mathbb R}^{+}$.
By convention ${\mathcal A}^{\leq - \infty}$ only contains the zero function.
%
%
\begin{defi}
Let $\tilde{e}=(\tilde{e}_{1}, \hdots, \tilde{e}_{\tilde{q}}) \in ({\mathbb R}^{+})^{\tilde{q}}$ and
$\Lambda=(\lambda_{1}, \hdots, \lambda_{\tilde{q}}) \in ({\mathbb S}^{1})^{\tilde{q}}$. We define
$I_{l}(\lambda,\upsilon)=
\lambda e^{i[-\frac{\pi}{2 \tilde{e}_{l}}- \upsilon, \frac{\pi}{2 \tilde{e}_{l}}+ \upsilon]}$
for $\lambda \in {\mathbb S}^{1}$, $\upsilon \in {\mathbb R}^{+} \cup \{0\}$ and
$1 \leq l \leq \tilde{e}_{\tilde{q}}$.
We say that the pair $(\tilde{e},\Lambda)$ is admissible if
\begin{itemize}
\item We have $0 < \tilde{e}_{1} < \hdots < \tilde{e}_{\tilde{q}}$
\item $I_{l+1}(\lambda_{l+1},0) \subset I_{l}(\lambda_{l},0)$ for any $1 \leq l < \tilde{q}$.
\end{itemize}
\end{defi}
\begin{defi}
\label{def:RamSib}
\cite{RamSib:Fou}
Let $\tilde{e}=(\tilde{e}_{1}, \hdots, \tilde{e}_{\tilde{q}}) \in ({\mathbb R}^{+})^{\tilde{q}}$ and
$\Lambda=(\lambda_{1}, \hdots, \lambda_{\tilde{q}}) \in ({\mathbb S}^{1})^{\tilde{q}}$.
We set $I_{0}(\lambda_{0},0)={\mathbb S}^{1}$ and $\tilde{e}_{\tilde{q}+1}= \infty$.
Assume that $(\tilde{e},\Lambda)$ is admissible. Let $\hat{\phi} \in {\mathbb C}[[x]]$ be a formal
power series expansion. We will say that $\hat{\phi}$ is
$(\tilde{e}_{1}, \hdots, \tilde{e}_{\tilde{q}})$-summable in the multi-direction
$\Lambda$, with sum $\phi_{\tilde{q}}$, if:
\begin{itemize}
\item[(i)] $\hat{\phi} \in {\mathbb C}[[x]]_{\frac{1}{\tilde{e}_{1}}}$,
\item[(ii)] there exists a sequence $(\phi_{0}, \hdots, \phi_{\tilde{q}})$ where:
\begin{itemize}
\item[a)] $\phi_{0} \in \Gamma({\mathbb S}^{1};{\mathcal A}/{\mathcal A}^{\leq -\tilde{e}_{1}})$
and $\phi_{0}$ corresponds to $\hat{\phi}$ by the natural isomorphism
\[ \Gamma({\mathbb S}^{1};{\mathcal A}/{\mathcal A}^{\leq -\tilde{e}_{1}}) \to
{\mathbb C}[[x]]_{\frac{1}{\tilde{e}_{1}}}, \]
\item[b)] $\phi_{j} \in \Gamma(I_{j}(\lambda_{j},0); {\mathcal A}/{\mathcal A}^{-\tilde{e}_{j+1}})$
($j =0,\hdots,\tilde{q}$),  and $\phi_{j|I_{j+1}} = \phi_{j+1}$ modulo ${\mathcal A}^{\leq - \tilde{e}_{j+1}}$,
for $j =0 , \hdots, \tilde{q}-1$.
\end{itemize}
\end{itemize}
\end{defi}
The next proposition (see \cite{Balser:LN}, page 69) is a criterium to identify multi-summable functions.
\begin{pro}
\label{pro:sumbal}
Let $\tilde{e}=(\tilde{e}_{1}, \hdots, \tilde{e}_{\tilde{q}}) \in ({\mathbb R}^{+})^{\tilde{q}}$ and
$\Lambda=(\lambda_{1}, \hdots, \lambda_{\tilde{q}}) \in ({\mathbb S}^{1})^{\tilde{q}}$.
We set $I_{0}(\lambda_{0},0)={\mathbb S}^{1}$ and $\tilde{e}_{\tilde{q}+1}= \infty$.
Assume that $\tilde{e}_{1} > 1/2$ and $(\tilde{e},\Lambda)$ is admissible.
For $\zeta, \tau \in {\mathbb R}^{+}$, assume existence of
$f(z;\lambda)$ (for every $\lambda \in {\mathbb S}^{1}$), analytic in the sector
$I^{\lambda}=(0,\tau) \lambda e^{i(-\zeta/2,\zeta/2)}$ and bounded at the origin,
such that for every $\lambda_{1}$, $\lambda_{2}$ with $I^{\lambda_{1}} \cap I^{\lambda_{2}} \neq \emptyset$
we have: If $\lambda_{1}, \lambda_{2} \in I_{l}(\lambda_{l},0)$ for some
$l$, $0 \leq l \leq \tilde{q}$, then
\[ f(z; \lambda_{1}) - f(z; \lambda_{2})
\in {\mathcal A}^{\leq - \tilde{e}_{l+1}}(I^{\lambda_{1}} \cap I^{\lambda_{2}}). \]
Then there exists a (unique) $(\tilde{e}_{1}, \hdots, \tilde{e}_{\tilde{q}})$-summable power series
$\hat{\phi}$ in $\Lambda$ with sum $f(z;\lambda_{\tilde{q}})$.
\end{pro}
Let $e \in {\mathbb R}^{+}$.
A power series $\hat{\phi} \in {\mathbb C}[[x]]$ is $e$-summable if it is $e$-summable in
any direction outside of a finite set.
A power series $\hat{\phi} \in {\mathbb C}[[x]]$ is $(\tilde{e}_{1}, \hdots, \tilde{e}_{\tilde{q}})$-summable if
it has at most finitely many singular directions of each level $\tilde{e}_{l}$, $1 \leq l \leq \tilde{q}$
(see \cite{Balser:LN}).
We will use the following result:
\begin{lem}
\label{lem:regmd}
Let $\hat{\phi} \in {\mathbb C}[[x]]$ and
$\tilde{e}=(\tilde{e}_{1}, \hdots, \tilde{e}_{\tilde{q}}) \in ({\mathbb R}^{+})^{\tilde{q}}$.
Fix $\lambda_{j} \in {\mathbb S}^{1}$.
Suppose that there exists a sequence
$({\lambda}_{1,n}, \hdots, {\lambda}_{j-1,n},
{\lambda}_{j+1,n}, \hdots, {\lambda}_{\tilde{q},n}) \in
({\mathbb S}^{1})^{\tilde{q}-1}$ such that
\begin{itemize}
\item Given ${\Lambda}_{n} = ({\lambda}_{1,n}, \hdots, {\lambda}_{j-1,n},
\lambda_{j}, {\lambda}_{j+1,n}, \hdots, {\lambda}_{\tilde{q},n})$
the pair  $(\tilde{e}, {\Lambda}_{n})$ is admissible  for any $n \in {\mathbb N}$
\item $\lim_{n \to \infty}  {\lambda}_{k,n} = \lambda_{k}$ for any $k < j$
\item $\hat{\phi}$ is $\tilde{e}$-summable in
${\Lambda}_{n}$ for any $n \in {\mathbb N}$
\end{itemize}
for any admissible pair $(\tilde{e}, \Lambda)$ with
$\Lambda=(\lambda_{1}, \hdots, \lambda_{\tilde{q}})$.
Then $\lambda_{j}$ is
a regular direction of level $\tilde{e}_{j}$.
\end{lem}
Next lemma is of technical type. It will be used to identify the asymptotic development
of the Lavaurs vector fields associated to an element of $\diff{tp1}{2}$.
\begin{lem}
\label{lem:techgevrey}
Fix $\nu \geq 2$, $\lambda_{0} \in {\mathbb S}^{1}$, $a,n \in {\mathbb N}$,
$c_{1},c_{2} \in {\mathbb R}^{+}$ and $b \in {\mathbb R}^{+}$ with $b>\pi/a$.
Denote $\lambda_{k}=\lambda_{0} e^{2 \pi i k/a}$, $\lambda_{k}'=\lambda_{k} e^{-i \pi/a}$
and $I_{k}=\lambda_{k} e^{i(-b,b)}$ for $k \in {\mathbb Z}/ a {\mathbb Z}$.
Consider $c \in {\mathbb R}^{+}$ such that the function
$t \mapsto e^{-c_{2}t^{-\nu}} t^{-(n+1)}$ is increasing in $(0,c)$.
Let ${\{h_{k}\}}_{k \in {\mathbb Z}/ a {\mathbb Z}}$ be a family of holomorphic functions satisfying
\begin{itemize}
\item $h_{k}$ is holomorphic in $(0,c) I_{k}$ for any $k \in {\mathbb Z}/ a {\mathbb Z}$.
\item $|h_{k}-h_{k-1}| \leq c_{1} e^{-c_{2}/|x|^{\nu}}$ in
$(0,c) \lambda_{k}' e^{i(-(b-\pi/a), b -\pi/a)}$ for any $k \in {\mathbb Z}/ a {\mathbb Z}$.
\end{itemize}
Suppose that there exists $\tau \in {\mathbb R}^{+}$ such that $|h_{k}| \leq \tau$ in
$(0,c) I_{k}$ for any $k \in {\mathbb Z}/ a {\mathbb Z}$.
The function $h_{k}$ has a $1/\nu$ Gevrey asymptotic development
$\sum_{l=0}^{\infty} \hat{h}_{l} x^{l}$ independent of $k$. Then we obtain
\begin{equation}
\label{equ:cch}
|\hat{h}_{n}| \leq \frac{2^{n} \tau}{c^{n}} + e^{- c_{2}/(2 c^{\nu})}
\end{equation}
if $c \in {\mathbb R}^{+}$ is small enough.
\end{lem}
\begin{proof}
The functions $h_{k}$ share a $1/\nu$ Gevrey asymptotic development since they define a section $h$ of
$({\mathcal A}/{\mathcal A}^{\leq -\nu})({\mathbb S}^{1})$ and
$\Gamma({\mathbb S}^{1};{\mathcal A}/{\mathcal A}^{\leq -\nu}) \to {\mathbb C}[[x]]_{\frac{1}{\nu}}$
is an isomorphism.

Denote $\theta=(b-\pi/a)/2$ and $c_{0}=\min (\sin (\theta),1/2)$.
By replacing $b$ with $b-\theta$ we can suppose that $h_{k}-h_{k-1}$ is defined in
$(0,c) \lambda_{k}' e^{i(-(b-\pi/a+\theta), b -\pi/a+\theta)}$ for any $k \in {\mathbb Z}/ a {\mathbb Z}$.

Let us construct, for $j \in {\mathbb Z}/ a {\mathbb Z}$, holomorphic functions
$\tilde{h}_{k}$ defined in $(0,c/2) I_{k}$, defining the same element
$h \in \Gamma({\mathbb S}^{1};{\mathcal A}/{\mathcal A}^{\leq -\nu})$ and
whose Gevrey development is easier to estimate. We use the Cauchy-Heine transform
(see \cite{Balser:LN}, chapter 4) to define the function
\begin{equation}
\label{equ:ch}
h_{k}^{\sharp}(x)=\frac{1}{2 \pi i} \int_{0}^{c \lambda_{k}'} (h_{k}-h_{k-1})(w) (w-x)^{-1} dw .
\end{equation}
It is defined in $(0,c/2) ({\mathbb S}^{1} \setminus \lambda_{k}' e^{i[-\theta,\theta]})$.
Moreover we obtain $|w-x|/|w| \geq \sin (\theta)$ for all
$w \in (0,c) \lambda_{k}'$ and $x \in (0,c/2) ({\mathbb S}^{1} \setminus \lambda_{k}' e^{i[-\theta,\theta]})$.
Given $\lambda \in \lambda_{k}' e^{i(-(b-\pi/a+\theta), b -\pi/a+\theta)}$ we can extend
$h_{k}^{\sharp}$ to $(0,c/2) ({\mathbb S}^{1} \setminus \lambda e^{i[-\theta,\theta]})$
by replacing the path of integration $[0,c] \lambda_{k}'$ in equation (\ref{equ:ch})
with the union $\gamma_{\lambda}$ of $[0,c] \lambda$ and
an arc in $\partial B(0,c)$ joining $c \lambda$ and $c \lambda_{k}'$.
We obtain $|w-x|/|w| \geq c_{0}$ for all
$w \in \gamma_{\lambda}$ and $x \in (0,c/2) ({\mathbb S}^{1} \setminus \lambda e^{i[-\theta,\theta]})$.
The function $h_{k}^{\sharp}$ is holomorphic in
$(0,c/2) (-\lambda_{k}') e^{i(-(b- \pi/a +\pi),b- \pi/a +\pi)}$. It is
multi-valuated and satisfies $h_{k}^{\sharp}(x)- h_{k}^{\sharp}(e^{2 \pi i}x)=(h_{k}-h_{k-1})(x)$.
We define
\[ \tilde{h}_{k}(x) = \sum_{l=1}^{k} h_{l}^{\sharp}(x) + \sum_{l=k+1}^{a} h_{l}^{\sharp}(e^{2 \pi i} x) \]
in $(0,c/2)I_{k}$ for $k \in {\mathbb Z}/ a {\mathbb Z}$.
By construction we obtain $\tilde{h}_{k}-\tilde{h}_{k-1} = h_{k}-h_{k-1}$ and
\[ |\tilde{h}_{k}(x)| \leq \frac{1}{2 \pi} \frac{a}{c_{0}} c (b - \pi/a + \theta +1)
\sup_{t \in (0,c)} (c_{1} e^{-c_{2} t^{-\nu}} t^{-1}) \]
for all $k \in {\mathbb Z}/ a {\mathbb Z}$ and $x \in (0,c/2) I_{k}$. By simple calculus
we can see that if $e^{-c_{2} t^{-\nu}} t^{-(n+1)}$ is increasing in $(0,c)$ then so is
$e^{-c_{2} t^{-\nu}} t^{-1}$. As a consequence we obtain
\[ |\tilde{h}_{k}(x)| \leq \frac{1}{2 \pi} \frac{a c_{1}}{c_{0}}  (b - \pi/a + \theta +1)
e^{-c_{2} c^{-\nu}} \]
for all $k \in {\mathbb Z}/ a {\mathbb Z}$ and $x \in (0,c/2) I_{k}$.

Let $\sum_{l=0}^{\infty} {h}_{l}^{\flat} x^{l}$ be the $1/\nu$ Gevrey development associated to
the element of $\Gamma({\mathbb S}^{1};{\mathcal A}/{\mathcal A}^{\leq -\nu})$ defined by
the functions $\tilde{h}_{k}$ for $k \in {\mathbb Z}/ a {\mathbb Z}$. We have
\[ h_{n}^{\flat} =
\frac{1}{2 \pi i} \sum_{k=1}^{a} \int_{0}^{c \lambda_{k}'} (h_{k}-h_{k-1})(w) w^{-(n+1)} dw \]
by the properties of the Cauchy-Heine transform. Since $e^{-c_{2}t^{-\nu}} t^{-(n+1)}$ is increasing
in $(0,c)$ we deduce
$|h_{n}^{\flat}| \leq (2 \pi)^{-1} a c c_{1} e^{-c_{2} c^{-\nu}} c^{-(n+1)}$.
Let $h-\tilde{h}$ be the function  defined in $B(0,c/2)$ such that
$(h-\tilde{h})_{|(0,c/2)I_{k}} =h_{k} - \tilde{h}_{k}$ for $k \in {\mathbb Z}/ a {\mathbb Z}$.
We have
\[ |(h-\tilde{h})(x)| \leq \tau + \frac{a c_{1} (b - \pi/a + \theta +1)}{2 c_{0} \pi}
e^{-c_{2} c^{-\nu}} \ \forall x \in B(0,c/2) . \]
Cauchy's integral formula implies
\[ |\hat{h}_{n}| \leq \frac{2^{n}}{c^{n}} \left({
\tau + \frac{a c_{1} (b - \pi/a + \theta +1)}{2 c_{0} \pi}
e^{-c_{2} c^{-\nu}} }\right) +  \frac{a  c_{1}}{2 c^{n} \pi}  e^{-c_{2} c^{-\nu}}. \]
A straightforward argument provides the estimate (\ref{equ:cch}).
\end{proof}
\subsection{Multi-summability of Fatou coordinates}
\label{subsec:mulsumfc}
Let $\varphi \in \diff{tp1}{2}$ with $2$-convergent normal form $\Upsilon={\rm exp}(X)$.
Consider $\Lambda=(\lambda_{1}, \hdots, \lambda_{\tilde{q}}) \in {\mathcal M}$ and
the dynamical splitting $\digamma_{\Lambda}$ in remark \ref{rem:unifspl}.

Let $j \in {\mathcal D}_{s}(\varphi)$ and $\lambda \in {\mathbb S}^{1}$.
Our aim is to define
\[ \ddot{\xi}_{j,\Lambda,\lambda}^{\varphi}(x,z) =
\ddot{\psi}_{j+1,\Lambda,\lambda}^{\varphi} \circ
(x, \ddot{\psi}_{j,\Lambda,\lambda}^{\varphi})^{-1}(x,z) . \]
We can interpret the set $(x,\ddot{\psi}_{j, \lambda}^{\varphi})(H_{\Lambda,j}^{\lambda})$
as a chart
coordinate system in which $\varphi$ is of the form $(x,z+1)$. The map
$\ddot{\xi}_{j,\lambda}^{\varphi}$
is then a transition function commuting with $(x,z+1)$. The main problem in order to define
$\ddot{\xi}_{j,\lambda}^{\varphi}$ is that
$H_{\Lambda,j}^{\lambda} \cap H_{\Lambda,j+1}^{\lambda} = \emptyset$.
Anyway in next lemma we extend the Fatou coordinates by iteration in order to obtain
a common domain of
definition for both $\ddot{\psi}_{j, \lambda}^{\varphi}$ and $\ddot{\psi}_{j+1,\lambda}^{\varphi}$.
\begin{rem}
The family $\{\ddot{\xi}_{j,\Lambda,\lambda}^{\varphi}\}_{j \in {\mathcal D}(\varphi)}$
for $\Lambda \in {\mathcal M}$ and $\lambda \in {\mathbb S}^{1}$ is an extension of
the Ecalle-Voronin invariants of $\varphi_{|x=0}$ in $I_{\Lambda}^{\lambda}$.
\end{rem}
\begin{lem}
\label{lem:chacha}
Let $\varphi \in \diff{tp1}{2}$ with $2$-convergent normal form $\Upsilon={\rm exp}(X)$.
Consider $\Lambda=(\lambda_{1}, \hdots, \lambda_{\tilde{q}}) \in {\mathcal M}$, $\lambda \in {\mathbb S}^{1}$,
$j \in {\mathcal D}_{s}(\varphi)$ and $s \in \{-1,1\}$. There exists $M \in {\mathbb R}^{+}$ such that the function
$\ddot{\xi}_{j,\Lambda,\lambda}^{\varphi}(x,z)$ commutes with $z \mapsto z+1$ and is defined in
$[0,\delta) I_{\Lambda}^{\lambda} \times s({\mathbb R} + i (-\infty,-M))$. Moreover
there exists
\[ \lim_{(x,Im(z)) \to (x_{0}, -s \infty)} \ddot{\xi}_{j,\Lambda,\lambda}^{\varphi}(x,z) - z \]
for any $x_{0} \in [0,\delta) I_{\Lambda}^{\lambda}$. The function
$e^{2 \pi i z} \circ \ddot{\xi}_{j,\Lambda,\lambda}^{\varphi}(x,(\ln z)/(2 \pi i))$
is holomorphic in
$(0,\delta) \lambda e^{i(-\upsilon_{\Lambda},\upsilon_{\Lambda})} \times \{ z \in \pn{1} : |z|^{s} > e^{2 \pi M} \}$.
The constant $M$ does not depend on $\Lambda$, $\lambda$ or $j$.
\end{lem}
\begin{proof}
Consider $\Gamma_{x,k,\lambda}=\Gamma(\aleph_{\Lambda,\lambda}X, T_{iX}^{\epsilon,k}(x), T_{0})$
for $k \in {\mathcal D}(\varphi)$ and
$x \in [0,\delta) I_{\Lambda}^{\lambda}$.
Suppose $s=1$ without lack of generality. We denote
\[ B_{X,j,\lambda}(x)= \cup_{t \in [0,2]} {\rm exp}(tX)(\Gamma_{x,j,\lambda}) \ {\rm and} \
B_{X,j+1,\lambda}(x)= \cup_{t \in [-2,0]} {\rm exp}(tX)(\Gamma_{x,j+1,\lambda}) \]
for $x \in [0,\delta) I_{\Lambda}^{\lambda}$.
Let $H_{j,j+1}^{\lambda}$ be the element of the set
$Reg^{*}(\epsilon, \aleph_{\Lambda, \lambda} X,I_{\Lambda}^{\lambda})$ such that
$T_{iX}^{\epsilon,j}(0) \in \overline{H_{j,j+1}^{\lambda}(0)} \ni T_{iX}^{\epsilon,j+1}(0)$.
We define
\[ E_{X,j,\lambda}(x)=B_{X,j,\lambda}(x) \cup B_{X,j+1,\lambda}(x) \cup H_{j,j+1}^{\lambda}(x)
\ {\rm for} \  x \in [0,\delta) I_{\Lambda}^{\lambda}. \]
Since $E_{X,j,\lambda}(x)$ is simply connected for $x \in [0,\delta) I_{\Lambda}^{\lambda}$
we can extend $\psi_{L_{j}}^{X}$ to $E_{X,j,\lambda}$.

There exists $M \in {\mathbb R}^{+}$ such that any trajectory
$\Gamma(X, P, E_{X,j,\lambda})$ for
$P \in B_{X,j,\lambda}(x)$, $x \in [0,\delta) I_{\Lambda}^{\lambda}$
and $Im(\psi_{L_{j}}^{X}(P))<-M$ intersects $B_{X,j+1,\lambda}(x)$.
More precisely, there exists $t_{0}(P) \in {\mathbb R}^{+}$ such that
${\rm exp}(tX)(P) \in E_{X,j,\lambda}$ for any $t \in [-t_{0}(P),0]$ and
${\rm exp}(-t_{0}(P) X)(P) \in B_{X,j+1,\lambda}(x)$. The constant $M$ does not depend on
$\Lambda$, $\lambda$ or $j$. We have
\[ |\Delta_{\varphi}| \leq \frac{K_{18}}{(1+|\psi_{L_{j}}^{X}|)^{2}} \ {\rm in} \ E_{X,j,\lambda} \]
for some $K_{18} \in {\mathbb R}^{+}$ by proposition \ref{pro:bddconf}. We have
$\psi_{L_{j}}^{X} \circ \varphi^{-1} = \psi_{L_{j}}^{X} - 1 - \Delta_{\varphi} \circ \varphi^{-1}$.
Lemma \ref{lem:techsum} implies
\[ |\psi_{L_{j}}^{X} \circ \varphi^{-l}  - (\psi_{L_{j}}^{X} - l)|(P) \leq
 \frac{K_{19}}{1+|\psi_{L_{j}}^{X}(P)|} \]
for all orbit $\{ P, \varphi^{-1}(P), \hdots, \varphi^{-l}(P)\}$ contained in $E_{X,j,\lambda}(x)$ and
$x \in [0,\delta) I_{\Lambda}^{\lambda}$.
We can use the previous inequality to show that there exists $t_{1}(P) \in {\mathbb N}$ such that
\[ \varphi^{-l}(P) \in E_{X,j,\lambda} \ \forall 1 \leq l \leq t_{1}(P) \ {\rm and} \
\varphi^{-t_{1}(P)}(P) \in B_{X,j+1,\lambda} \]
for any  $P \in \cup_{x \in [0,\delta) I_{\Lambda}^{\lambda}} B_{X,j,\lambda}(x)$ with
$Im(\psi_{L_{j}}^{X}(P)) < - M$. We consider a greater $M \in {\mathbb R}^{+}$ if necessary.
Thus we can extend $\ddot{\psi}_{j+1, \lambda}^{\varphi}$ to
$B_{X,j,\lambda}(x) \cap \{ Im(\psi_{L_{j}}^{X}) < -M\}$ for any
$x \in [0,\delta) I_{\Lambda}^{\lambda}$. Since we have
\[ |(\ddot{\psi}_{j+1,\lambda}^{\varphi}- \psi_{L_{j}}^{X})(P) -
(\ddot{\psi}_{j+1,\lambda}^{\varphi}- \psi_{L_{j}}^{X})(\varphi^{-t_{1}(P)}(P))| \leq
\frac{K_{19}}{1+ |\psi_{L_{j}}^{X}(P)|} \]
for any  $P \in \cup_{x \in [0,\delta) I_{\Lambda}^{\lambda}} B_{X,j,\lambda}(x)$ with
$Im(\psi_{L_{j}}^{X}(P)) < - M$  then
$\ddot{\psi}_{j+1, \lambda}^{\varphi}- \psi_{L_{j}}^{X}$ is continuous in
$\cup_{x \in [0,\delta) I_{\Lambda}^{\lambda}} \overline{B_{X,j,\lambda}(x)} \cap
 \{ Im(\psi_{L_{j}}^{X}) < -M\}$ and so is the mapping
$\ddot{\psi}_{j+1, \lambda}^{\varphi}- \ddot{\psi}_{j, \lambda}^{\varphi}$.
The mapping $\ddot{\xi}_{j,\lambda}^{\varphi}(x,z)$
commutes with $(x,z+1)$. Therefore it is defined in
$[0,\delta) I_{\Lambda}^{\lambda} \times ({\mathbb R} + i (-\infty,-M))$
up to consider a greater $M \in {\mathbb R}^{+}$. Moreover we obtain
\[ \lim_{(x,Im(z)) \to (x_{0}, -\infty)}
\ddot{\xi}_{j,\lambda}^{\varphi}(x,z) -z =
(\ddot{\psi}_{j+1, \lambda}^{\varphi} - \ddot{\psi}_{j, \lambda}^{\varphi})
(\alpha^{\aleph_{\Lambda,\lambda}X}(H_{\Lambda,j}^{\lambda}(x))). \]
Thus
$e^{2 \pi i z} \circ \ddot{\xi}_{j,\lambda}^{\varphi}(x,(\ln z)/(2 \pi i))$ is holomorphic in
$[0,\delta) I_{\Lambda}^{\lambda} \times \{|z| > e^{2 \pi M} \}$.
\end{proof}
Next we study the dependence of $\ddot{\xi}_{j,\Lambda,\lambda}^{\varphi}$
on $\Lambda \in {\mathcal M}$ and $\lambda \in {\mathbb S}^{1}$. We provide
the estimates implying the multi-summability of the extensions of the Ecalle-Voronin
invariants.
\begin{teo}
\label{teo:difccfl}
Let $\varphi \in \diff{tp1}{2}$ with $2$-convergent normal form $\Upsilon={\rm exp}(X)$.
Consider $\Lambda, \Lambda' \in {\mathcal M}$,
$\lambda, \lambda' \in {\mathbb S}^{1}$, $j \in {\mathcal D}_{s}(\varphi)$ and $s \in \{-1,1\}$. Then
\begin{equation}
\label{equ:change}
|\ddot{\xi}_{j,\Lambda,\lambda}^{\varphi}(x,z) -  \ddot{\xi}_{j,\Lambda',\lambda'}^{\varphi}(x,z)| \leq K_{20}
e^{-K/|x|^{\tilde{e}_{d_{\Lambda,\Lambda'}^{\lambda,\lambda'}+1}}}
\end{equation}
in $[0,\delta) I_{\Lambda,\Lambda'}^{\lambda,\lambda'} \times s({\mathbb R} + i (-\infty,-M))$
for some $K_{20} \in {\mathbb R}^{+}$.
\end{teo}
\begin{proof}
Consider the notations in the proof of lemma \ref{lem:chacha}.
Denote $d=d_{\Lambda,\Lambda'}^{\lambda,\lambda'}$. We have
\[  |\ddot{\psi}_{j,\Lambda,\lambda}^{\varphi} - \ddot{\psi}_{j,\Lambda',\lambda'}^{\varphi}| \leq
e^{-K/|x|^{\tilde{e}_{d+1}}} \ {\rm and} \
|\ddot{\psi}_{j+1,\Lambda,\lambda}^{\varphi} - \ddot{\psi}_{j+1,\Lambda',\lambda'}^{\varphi}| \leq
e^{-K/|x|^{\tilde{e}_{d+1}}}  \]
in $H_{\Lambda,\Lambda',j}^{\lambda,\lambda'}$ and $H_{\Lambda,\Lambda',j+1}^{\lambda,\lambda'}$ respectively
by prop. \ref{pro:difFatflat}.

Given $x \in [0,\delta) I_{\Lambda,\Lambda'}^{\lambda,\lambda'}$ consider a connected path
$\gamma_{x}$ contained in $B_{X,j,\lambda}(x)$
such that $\ddot{\psi}_{j,\lambda}^{\varphi}(\gamma_{x})$
is of the form $[a(x), a(x)+1] - i s M$.
Given $x_{0} \in [0,\delta) I_{\Lambda,\Lambda'}^{\lambda,\lambda'}$ and
$z_{0} \in \ddot{\psi}_{j,\lambda}^{\varphi}(\gamma_{x_{0}})$
we consider the point $(x_{0},y_{0}) \in \gamma_{x_{0}}$ such that
$\ddot{\psi}_{j,\lambda}^{\varphi}(x_{0},y_{0})=z_{0}$.
We consider the point $(x_{0},y_{1})$ such that $\ddot{\psi}_{j,\lambda'}^{\varphi}(x_{0},y_{1})=z_{0}$.
Since
\[ |\ddot{\psi}_{j,\lambda}^{\varphi}(x_{0},y_{1}) - \ddot{\psi}_{j,\lambda'}^{\varphi}(x_{0},y_{1})| \leq
e^{-K/|x_{0}|^{\tilde{e}_{d+1}}} \]
by proposition \ref{pro:difFatflat} we obtain
$|\ddot{\psi}_{j,\lambda}^{\varphi}(x_{0},y_{1}) - \ddot{\psi}_{j,\lambda}^{\varphi}(x_{0},y_{0})| \leq
e^{-K/|x_{0}|^{\tilde{e}_{d+1}}}$.
There exists $K_{1}' \in {\mathbb R}^{+}$ such that
$|\ddot{\psi}_{j,\lambda}^{\varphi} \circ (x, \psi_{L_{j}}^{X})^{-1}(x,z) - z| \leq K_{1}'$ in
a neighborhood of
$\cup_{x \in [0,\delta) I_{\Lambda,\Lambda'}^{\lambda,\lambda'}} \{x\} \times \psi_{L_{j}}^{X}(B_{X,j,\lambda}(x))$
since $\ddot{\psi}_{j,\lambda}^{\varphi} - \psi_{L_{j}}^{X}$ is bounded.
By using Cauchy's integral formula we obtain $K_{2}'>0$ such that
\[ \left|{ \frac{\partial (\ddot{\psi}_{j,\lambda}^{\varphi} \circ (x, \psi_{L_{j}}^{X})^{-1})}{\partial z} - 1 }\right|
\leq K_{2}'  \]
in a neighborhood of
$\cup_{x \in [0,\delta) I_{\Lambda,\Lambda'}^{\lambda,\lambda'}} \{x\} \times \psi_{L_{j}}^{X}(B_{X,j,\lambda}(x))$.
We can suppose $K_{2}' \leq 1/2$ by considering a greater $M \in {\mathbb R}^{+}$.
Thus we obtain
$|\partial (\psi_{L_{j}}^{X} \circ (x, \ddot{\psi}_{j,\lambda}^{\varphi})^{-1}) / \partial z| \leq 2$
in a neighborhood of
$\cup_{x \in [0,\delta) I_{\Lambda,\Lambda'}^{\lambda,\lambda'}}
\{x\} \times \ddot{\psi}_{j,\lambda}^{\varphi}(B_{X,j,\lambda}(x))$.
We deduce
\[ |\psi_{L_{j}}^{X} (x_{0},y_{1}) - \psi_{L_{j}}^{X} (x_{0},y_{0})| \leq
2 e^{-K/|x_{0}|^{\tilde{e}_{d+1}}} . \]
We have
\[ |\ddot{\psi}_{j+1, \lambda}^{\varphi} (x_{0},y_{0}) -
\ddot{\psi}_{j+1, \lambda'}^{\varphi}  (x_{0},y_{1})| \leq
|\ddot{\psi}_{j+1, \lambda}^{\varphi} (x_{0},y_{0}) -
\ddot{\psi}_{j+1, \lambda}^{\varphi}  (x_{0},y_{1})| + e^{-K/|x_{0}|^{\tilde{e}_{d+1}}}. \]
There exists $K_{3}' \in {\mathbb R}^{+}$ such that
$| \partial (\ddot{\psi}_{j+1,\lambda}^{\varphi} \circ (x, \psi_{L_{j}}^{X})^{-1}) / \partial z - 1| \leq K_{3}'$
in a neighborhood of
$\cup_{x \in [0,\delta) I_{\Lambda,\Lambda'}^{\lambda,\lambda'}} \{x\} \times
[\psi_{L_{j}}^{X}(B_{X,j,\lambda}(x)) \cap s({\mathbb R} + i (-\infty,-M))]$.
The proof is analogous to the one for $\ddot{\psi}_{j,\lambda}^{\varphi}$. This property implies
\[ |\ddot{\psi}_{j+1, \lambda}^{\varphi} (x_{0},y_{0}) -
\ddot{\psi}_{j+1, \lambda'}^{\varphi}  (x_{0},y_{1})| \leq
2(1+K_{3}') e^{-K/|x_{0}|^{\tilde{e}_{d+1}}} +
e^{-K/|x_{0}|^{\tilde{e}_{d+1}}}. \]
Denote $K_{20}= 2(1+K_{3}') +1$. The function
$(\ddot{\xi}_{j,\Lambda,\lambda}^{\varphi} -  \ddot{\xi}_{j,\Lambda',\lambda'}^{\varphi})(x_{0},(\ln z)/(2 \pi i))$
is defined in $ \{z \in \pn{1}: |z|^{s} > e^{2 \pi M} \}$ and is bounded by above by
$K_{20} e^{-K/|x_{0}|^{\tilde{e}_{d+1}}}$ in $\partial B(0, e^{2 \pi s M})$.
The modulus maximum principle implies equation (\ref{equ:change}).
\end{proof}
\begin{defi}
Let $j \in {\mathcal D}_{s}(\varphi)$.
Since $\ddot{\xi}_{j,\Lambda,\lambda}^{\varphi}$ commutes with $(x,z+1)$, it is of the form
\[ \ddot{\xi}_{j,\Lambda,\lambda}^{\varphi}= z + \ddot{a}_{j,\Lambda,\lambda,0}^{\varphi}(x) +
\sum_{l=1}^{\infty} \ddot{a}_{j,\Lambda,\lambda,l}^{\varphi}(x) e^{-2 \pi i s l z} \]
where $\ddot{a}_{j,\Lambda,\lambda,l}^{\varphi}$ is continuous in
$[0,\delta) I_{\Lambda}^{\lambda}$ and holomorphic in $(0,\delta) \dot{I}_{\Lambda}^{\lambda}$
for any $l \geq 0$.
\end{defi}
The properties of the families
$\{\ddot{a}_{j,\Lambda,\lambda,l}^{\varphi}\}_{(\Lambda,\lambda) \in
{\mathcal M} \times {\mathbb S}^{1}}$
and
$\{\ddot{\xi}_{j,\Lambda,\lambda}^{\varphi}\}_{(\Lambda,\lambda) \in
{\mathcal M} \times {\mathbb S}^{1}}$
are analogous.
\begin{teo}
\label{teo:coechachas}
Let $\varphi \in \diff{tp1}{2}$ with $2$-convergent normal form $\Upsilon={\rm exp}(X)$.
Consider $\Lambda=(\lambda_{1}, \hdots, \lambda_{\tilde{q}}) \in {\mathcal M}$,
$j \in {\mathcal D}_{s}(\varphi)$ and $s \in \{-1,1\}$. Then
the function $\ddot{a}_{j,\Lambda,\lambda,l}^{\varphi}$ is
$(\tilde{e}_{1}, \hdots, \tilde{e}_{\tilde{q}})$-summable in the multi-direction $\Lambda$
for any $l \in {\mathbb N} \cup \{0\}$.
Moreover its asymptotic development $\ddot{a}_{j,l}^{\varphi}$ does not depend on $\Lambda$.
In particular  $\ddot{a}_{j,l}^{\varphi}$ is a $(\tilde{e}_{1}, \hdots, \tilde{e}_{\tilde{q}})$-summable
power series for any $l \in {\mathbb N} \cup \{0\}$.
\end{teo}
\begin{proof}
Fix $l \in {\mathbb N} \cup \{0\}$. Let $\Lambda' \in {\mathcal M}$ and $\lambda, \lambda' \in {\mathbb S}^{1}$.
Consider the curve $\gamma_{x}$ defined in the proof of theorem \ref{teo:difccfl} for
$x \in [0,\delta) I_{\Lambda,\Lambda'}^{\lambda,\lambda'}$. We have
\[ |\ddot{a}_{j,\Lambda,\lambda,l}^{\varphi} - \ddot{a}_{j,\Lambda',\lambda',l}^{\varphi}|(x) \leq
\left|{
\int_{\ddot{\psi}_{j,\lambda}^{\varphi}(\gamma_{x})}
e^{2 \pi i s l z}(\ddot{\xi}_{j,\Lambda,\lambda}^{\varphi}(x,z) - \ddot{\xi}_{j,\Lambda',\lambda'}^{\varphi}(x,z)) dz
}\right| \]
and then
\[ |\ddot{a}_{j,\Lambda,\lambda,l}^{\varphi} - \ddot{a}_{j,\Lambda',\lambda',l}^{\varphi}|(x) \leq
e^{2 \pi l M}   K_{20} e^{-K/|x|^{\tilde{e}_{d_{\Lambda,\Lambda'}^{\lambda,\lambda'}+1}}}. \]
We obtain
\begin{equation}
\label{equ:asydev}
\ddot{a}_{j,\Lambda,\lambda,l}^{\varphi} - \ddot{a}_{j,\Lambda',\lambda',l}^{\varphi} \in
{\mathcal A}^{\leq - \tilde{e}_{1}}(\dot{I}_{\Lambda,\Lambda'}^{\lambda,\lambda'}).
\end{equation}
Suppose $\Lambda'=\Lambda$.
If $\lambda, \lambda' \in I_{k}(\lambda_{k},0)$ for some $k \in \{0, \hdots, \tilde{q} \}$ we obtain
$d_{\Lambda}^{\lambda} \geq k \leq d_{\Lambda}^{\lambda'}$ and $d_{\Lambda,\Lambda}^{\lambda,\lambda'} \geq k$.
We obtain
\[ \ddot{a}_{j,\Lambda,\lambda,l}^{\varphi} - \ddot{a}_{j,\Lambda,\lambda',l}^{\varphi} \in
{\mathcal A}^{\leq - \tilde{e}_{k+1}}(\dot{I}_{\Lambda,\Lambda}^{\lambda,\lambda'}).  \]
Proposition \ref{pro:sumbal} implies that $\ddot{a}_{j,\Lambda,\lambda,l}^{\varphi}$ is
$(\tilde{e}_{1}, \hdots, \tilde{e}_{\tilde{q}})$-summable in $\Lambda$. Denote
$\ddot{a}_{j,\Lambda,l}^{\varphi}$ its asymptotic development. Property (\ref{equ:asydev})
for $\lambda, \lambda' \in {\mathbb S}^{1}$ implies
$\ddot{a}_{j,\Lambda,l}^{\varphi}=\ddot{a}_{j,\Lambda',l}^{\varphi}$
for all $\Lambda,\Lambda' \in {\mathcal M}$ (see condition (ii) a) in def. \ref{def:RamSib}).

Consider $(\lambda_{1}, \hdots, \lambda_{\tilde{q}}) \in ({\mathbb S}^{1})^{\tilde{q}}$ where
$\lambda_{k} \not \in \tilde{\Xi}_{X}^{k}$ (see def. \ref{def:msing}).
There exist $s_{1}, s_{2} \in \{-1,1\}$ such that
\[ (\lambda_{1} e^{i s_{1} \zeta}, \hdots, \lambda_{k-1} e^{i s_{1} \zeta}, \lambda_{k},
\lambda_{k+1} e^{i s_{2} \zeta}, \hdots, \lambda_{\tilde{q}} e^{i s_{2} \zeta}) \in {\mathcal M} \]
for any $\zeta$ in a neighborhood of $0$ in ${\mathbb R}^{+}$.
The singular directions of order $\tilde{e}_{k}$ of $\ddot{a}_{j,l}^{\varphi}$ are contained
in the finite set $\tilde{\Xi}_{X}^{k}$ for any $1 \leq k \leq \tilde{q}$ by lemma \ref{lem:regmd}.
Therefore $\ddot{a}_{j,l}^{\varphi} \in {\mathbb C}[[x]]$ is
$(\tilde{e}_{1}, \hdots, \tilde{e}_{\tilde{q}})$-summable.
\end{proof}
\begin{defi}
We define $\zeta_{\varphi}(x) = - \pi i / \nu({\mathcal E}_{0})^{-1}
\sum_{P \in (Sing X)(x)} Res(X,P)$ (see def. \ref{def:noinfcon}).
The function $\zeta_{\varphi}(x)$ is holomorphic in the neighborhood of $x=0$.
It does not depend on the choice of $X$.
\end{defi}
Let us review the normalizing conditions for Fatou coordinates
introduced in subsection 8.1 of \cite{JR:mod}.
By making analytic extension along the arc going from $T_{iX}^{\epsilon,j}(0)$ to
$T_{iX}^{\epsilon,j+1}(0)$ in counter clock wise sense we can define
$\psi_{L_{j+1}}^{X}-\psi_{L_{j}}^{X}$. It depends only on $x$ and it is holomorphic in a neighborhood of
$x=0$. By considering Fatou coordinates of the form $\dot{\psi}_{j}^{X} = \psi_{L_{j}}^{X} + c_{j}(x)$
for convenient holomorphic functions $c_{1}, \hdots, c_{2 \nu({\mathcal E}_{0})}$ defined in $B(0,\delta)$
we can obtain $(\dot{\psi}_{j+1}^{X} - \dot{\psi}_{j}^{X})(x,y)= \zeta_{\varphi}(x)$ for any
$j \in {\mathcal D}(\varphi)$.
Let $y=\gamma_{1}(x)$, $\hdots$, $y=\gamma_{p}(x)$ the irreducible components of $Fix (\varphi)$.
There exist continuous functions
\[ b_{j}:[0,\delta) I_{\Lambda}^{\lambda} \to {\mathbb C} \ {\rm for} \
1 \leq j \leq 2 \nu({\mathcal E}_{0}) \ {\rm and} \
c_{\Lambda,\lambda,k}^{\varphi}:[0,\delta) I_{\Lambda}^{\lambda} \to {\mathbb C} \ {\rm for} \
1 \leq k \leq p,  \]
holomorphic in $(0,\delta) \dot{I}_{\Lambda}^{\lambda}$, such that
\begin{equation}
\label{equ:priv}
(\ddot{\psi}_{j,\Lambda,\lambda}^{\varphi} + b_{j}(x)
- \dot{\psi}_{j}^{X})(x,\gamma_{k}(x)) = c_{\Lambda,\lambda,k}^{\varphi}(x)
\end{equation}
for all $x \in [0,\delta) I_{\Lambda}^{\lambda}$, $j \in {\mathcal D}(\varphi)$
and $1 \leq k \leq p$
such that $(x,\gamma_{k}(x)) \in \overline{H_{\Lambda,j}^{\lambda}}$ (see \cite{JR:mod}).
The relevant property is that a function
$c_{\Lambda,\lambda,k}^{\varphi}$
does not depend on the choice of $j \in {\mathcal D}(\varphi)$.
We define
\[ \dot{\psi}_{j,\Lambda,\lambda}^{\varphi}(x,y) = \ddot{\psi}_{j,\Lambda,\lambda}^{\varphi}(x,y)
+ b_{j}(x) -
c_{\Lambda,\lambda,1}^{\varphi}(x) \ \ {\rm for} \ j \in {\mathcal D}(\varphi). \]
We obtain $(\dot{\psi}_{j,\Lambda,\lambda}^{\varphi} - \dot{\psi}_{j}^{X})(x,\gamma_{1}(x)) = 0$
for all $x \in [0,\delta) I_{\Lambda}^{\lambda}$ and $j \in {\mathcal D}(\varphi)$
such that $(x,\gamma_{1}(x)) \in \overline{H_{\Lambda,j}^{\lambda}}$.
\begin{defi}
(section 8.1, prop. 8.1   \cite{JR:mod})
\label{def:prifat}
The family
${\{ \dot{\psi}_{j,\Lambda,\lambda}^{\varphi} \}}_{(j,\Lambda,\lambda) \in
{\mathcal D}(\varphi) \times {\mathcal M} \times {\mathbb S}^{1}}$
is called a homogeneous
privileged (with respect to $y=\gamma_{1}(x)$) system of Fatou coordinates of $\varphi$. We define the extension of the Ecalle-Voronin invariants
\[  \dot{\xi}_{j,\Lambda,\lambda}^{\varphi}(x,z) =
\dot{\psi}_{j+1,\Lambda,\lambda}^{\varphi} \circ (x, \dot{\psi}_{j,\Lambda,\lambda}^{\varphi})^{-1}(x,z) \]
for $(j,\Lambda,\lambda) \in {\mathcal D}(\varphi) \times {\mathcal M} \times {\mathbb S}^{1}$.
 We have
\[ \dot{\xi}_{j,\Lambda,\lambda}^{\varphi}(x,z)= z +  \zeta_{\varphi}(x) + \sum_{l=1}^{\infty}
\dot{a}_{j,\Lambda,\lambda,l}^{\varphi}(x)  e^{-2 \pi i s l z} \]
for $j \in {\mathcal D}_{s} (\varphi)$ and $s \in \{-1,1\}$.
Since the system ${\{\dot{\psi}_{j}^{X}\}}_{j \in {\mathcal D}(\varphi)}$ is unique up to
a holomorphic additive function then so is
${\{ \dot{\psi}_{j,\Lambda,\lambda}^{\varphi} \}}_{(j,\Lambda,\lambda) \in
{\mathcal D}(\varphi) \times {\mathcal M} \times {\mathbb S}^{1}}$.
More precisely, any other homogeneous privileged system of Fatou coordinates of $\varphi$
is of the form
${\{ \dot{\psi}_{j,\Lambda,\lambda}^{\varphi} + c \}}_{(j,\Lambda,\lambda) \in
{\mathcal D}(\varphi) \times {\mathcal M} \times {\mathbb S}^{1}}$
where $c(x)$ is a holomorphic function defined in $B(0,\delta)$.
\end{defi}
Let us remark that given $\Lambda \in {\mathcal M}$ and $\lambda \in {\mathbb S}^{1}$
we have
\[  \sum_{k=1}^{2 \nu({\mathcal E}_{0})} \ddot{a}_{k,\Lambda,\lambda,0}^{\varphi} \equiv
2 \nu({\mathcal E}_{0})  \zeta_{\varphi}. \]
The analogous result holds true for
${\{ \dot{\psi}_{j,\Lambda,\lambda}^{\varphi} \}}_{j \in {\mathcal D}(\varphi) }$ and
the property is preserved if we replace $\breve{\psi}_{j_{0},\Lambda,\lambda}^{\varphi}$
with $\breve{\psi}_{j_{0},\Lambda,\lambda}^{\varphi} + c_{j_{0}}(x)$ in a
a system of Fatou coordinates
${\{ \breve{\psi}_{j,\Lambda,\lambda}^{\varphi} \}}_{j \in {\mathcal D}(\varphi) }$.
Since $\psi_{j,\Lambda,\lambda}^{\varphi} - \dot{\psi}_{j,\Lambda,\lambda}^{\varphi}$
is a function of $x$ for any $j \in {\mathcal D}(\varphi)$
the result follows. This discussion motivates the next choice of normalizing condition
for multi-summable Fatou coordinates.
\begin{defi}
\label{def:msfatcoor}
We define $\psi_{1,\Lambda,\lambda}^{\varphi} \equiv \ddot{\psi}_{1,\Lambda,\lambda}^{\varphi}$.
We define
\[ \psi_{j,\Lambda,\lambda}^{\varphi}(x,y)=\ddot{\psi}_{j,\Lambda,\lambda}^{\varphi}(x,y)
- \sum_{k=1}^{j-1} \ddot{a}_{k,\Lambda,\lambda,0}^{\varphi}(x)+ (j-1) \zeta_{\varphi}(x) \]
for $j=2,\hdots,2 \nu({\mathcal E}_{0})$. We define the extension of the Ecalle-Voronin
invariants
\[  \xi_{j,\Lambda,\lambda}^{\varphi}(x,z) =
\psi_{j+1,\Lambda,\lambda}^{\varphi} \circ (x, \psi_{j,\Lambda,\lambda}^{\varphi})^{-1}(x,z) \]
for $j \in {\mathcal D}(\varphi)$. The family
${\{ \psi_{j,\Lambda,\lambda}^{\varphi} \}}_{(j,\Lambda,\lambda)
\in {\mathcal D}(\varphi) \times {\mathcal M} \times {\mathbb S}^{1}}$
is called a homogeneous
$(\tilde{e}_{1}, \hdots, \tilde{e}_{\tilde{q}})$-summable system of Fatou coordinates of $\varphi$.
It is defined up to a $(\tilde{e}_{1}, \hdots, \tilde{e}_{\tilde{q}})$-summable function,
i.e. any other homogeneous
$(\tilde{e}_{1}, \hdots, \tilde{e}_{\tilde{q}})$-summable system of Fatou coordinates of $\varphi$
is of the form
${\{ \psi_{j,\Lambda,\lambda}^{\varphi} + c_{\Lambda,\lambda} \}}_{(j,\Lambda,\lambda) \in
{\mathcal D}(\varphi) \times {\mathcal M} \times {\mathbb S}^{1}}$
where $c_{\Lambda,\lambda}(x)$ is a $(\tilde{e}_{1}, \hdots, \tilde{e}_{\tilde{q}})$-summable function.
\end{defi}
\begin{rem}
Different choices of $2$-convergent normal forms can provide different
homogeneous $(\tilde{e}_{1}, \hdots, \tilde{e}_{\tilde{q}})$-summable systems
of Fatou coordinates of $\varphi$.
\end{rem}
\begin{rem}
\label{rem:chacha}
The definition implies
\[ \xi_{j,\Lambda,\lambda}^{\varphi}= z +  \zeta_{\varphi}(x) +
\sum_{l=1}^{\infty} a_{j,\Lambda,\lambda,l}^{\varphi}(x) e^{-2 \pi i s l z} \]
for $j=1,\hdots,2 \nu({\mathcal E}_{0})$. The function
$\psi_{j,\Lambda,\lambda}^{\varphi} - \dot{\psi}_{j,\Lambda,\lambda}^{\varphi}$
depends only on $x$ for any $j \in \{1,2\}$. It is continuous in
$[0,\delta) I_{\Lambda}^{\lambda}$ and holomorphic in $(0,\delta) \dot{I}_{\Lambda}^{\lambda}$.
Moreover it does not depend on $j \in {\mathcal D}(\varphi_{1})$ since both systems of
Fatou coordinates are homogeneous.
\end{rem}
\begin{teo}
\label{teo:difFatflatext2}
Let $\varphi \in \diff{tp1}{2}$ with $2$-convergent normal form ${\rm exp}(X)$.
Let $\Lambda, \Lambda' \in {\mathcal M}$,
$\lambda, \lambda' \in {\mathbb S}^{1}$ and $j \in {\mathcal D}(\varphi)$.
Then there exists $K \in {\mathbb R}^{+}$ such that
\[ |\psi_{j,\Lambda,\lambda}^{\varphi} - \psi_{j,\Lambda',\lambda'}^{\varphi}|(x,y)
\leq  e^{-K/|x|^{\tilde{e}_{d_{\Lambda,\Lambda'}^{\lambda,\lambda'}+1}}}  \]
for any $(x,y) \in H_{\Lambda,\Lambda',j}^{\lambda,\lambda'}$.
\end{teo}
\begin{teo}
\label{teo:difFatflatext3}
Let $\varphi \in \diff{tp1}{2}$ with $2$-convergent normal form ${\rm exp}(X)$.
Let $\Lambda,\Lambda' \in {\mathcal M}$,
$\lambda, \lambda' \in {\mathbb S}^{1}$, $j \in {\mathcal D}(\varphi)$
and $\theta \in (0,\pi/2]$.
Then there exist $K \in {\mathbb R}^{+}$ and $\rho \geq 2 \rho_{0}$ such that
\[ |{\psi}_{j,\Lambda,\lambda}^{\varphi} - {\psi}_{j,\Lambda',\lambda'}^{\varphi}|(x,y)
\leq  e^{-K/|x|^{\tilde{e}_{d_{\Lambda,\Lambda'}^{\lambda,\lambda'}+1}}}  \]
for any $(x,y) \in H_{j,\theta}^{\epsilon,\rho,\lambda,\lambda'}$.
Moreover $\rho$ depends only on $X$, $\varphi$ and $\Lambda$.
\end{teo}
\begin{rem}
As in theorem \ref{teo:coechachas} the combinatorics of sectors in the $x$ variable
of the family
$\{ \psi_{j,\Lambda,\lambda}^{\varphi}\}_{(\Lambda,\lambda)
\in {\mathcal M} \times {\mathbb S}^{1}}$
corresponds to a multi-summable function. It is justified to say that
$\{ \psi_{j,\Lambda,\lambda}^{\varphi}\}_{(\Lambda,\lambda)
\in {\mathcal M} \times {\mathbb S}^{1}}$ is
$(\tilde{e}_{1}, \hdots, \tilde{e}_{\tilde{q}})$-summable in the $x$-variable.
It would be more rigorous to say that the family
$\{ \psi_{j,\Lambda,\lambda}^{\varphi} - \psi_{L_{j}}^{X} \}_{(\Lambda,\lambda)
\in {\mathcal M} \times {\mathbb S}^{1}}$ is multi-summable since its
elements are bounded.
\end{rem}
\begin{rem}
The previous theorems are deduced from
prop. \ref{pro:difFatflat}, \ref{pro:difFatflatext} and the
$(\tilde{e}_{1}, \hdots, \tilde{e}_{\tilde{q}})$-summability of the power series
$- \sum_{k=1}^{j-1} \ddot{a}_{k,0}^{\varphi}(x)+ (j-1) \zeta_{\varphi}(x)$ for any
$j \leq 2 \nu({\mathcal E}_{0})$.
Thus the lemma
\ref{lem:chacha} and the theorems \ref{teo:difccfl} and \ref{teo:coechachas}
hold true for
${\{ \xi_{j,\Lambda,\lambda}^{\varphi} \}}_{(j,\Lambda,\lambda) \in
{\mathcal D}(\varphi) \times {\mathcal M} \times {\mathbb S}^{1}}$
and
${\{ a_{j,\Lambda,\lambda,k}^{\varphi} \}}_{(j,\Lambda,\lambda,k) \in
{\mathcal D}(\varphi) \times {\mathcal M} \times {\mathbb S}^{1} \times {\mathbb N}}$.
Then Ecalle-Voronin invariants are
$(\tilde{e}_{1}, \hdots, \tilde{e}_{\tilde{q}})$-summable in the $x$-variable.
\end{rem}
\begin{rem}
Let $\varphi \in \diff{tp1}{2}$ with convergent normal form ${\rm exp}(X)$.
Suppose that $Fix (\varphi)$ is a curve $y=\gamma(x)$. The sets of unstable directions
${\mathcal U}_{X}^{j}$ (see subsection \ref{subsec:polvecfield}) are empty for any $1 \leq j \leq q$.
Therefore the Fatou coordinates ${\{ \psi_{j}^{\varphi}\}}_{j \in {\mathcal D}(\varphi)}$ and
Ecalle-Voronin invariants
${\{ \xi_{j}^{\varphi}\}}_{j \in {\mathcal D}(\varphi)}$ are analytic. Indeed we have
\[ \xi_{j}^{\varphi}= z +  \zeta_{\varphi}(x) +
\sum_{l=1}^{\infty} a_{j,l}^{\varphi}(x) e^{-2 \pi i s l z} \]
where $a_{j,l}^{\varphi} \in {\mathcal O}(B(0,\delta))$ for all $j \in {\mathcal D}_{s}(\varphi)$,
$s \in \{-1,1\}$ and $l \in {\mathbb N}$.
\end{rem}
\subsection{Multi-summability of Lavaurs vector fields}
\label{subsec:mulsumlvf}
Let $\varphi \in \diff{tp1}{2}$ with $2$-convergent normal form $\Upsilon={\rm exp}(X)$.
Consider $\Lambda \in {\mathcal M}$, $\lambda \in {\mathbb S}^{1}$, $j \in {\mathcal D}(\varphi)$.
\begin{defi}
\label{def:lavaurs}
We define the Lavaurs vector field (see cor. \ref{cor:Lavvf})
\[ X_{j,\Lambda,\lambda}^{\varphi} =
\frac{1}{\partial \psi_{j,\Lambda,\lambda}^{\varphi} / \partial y} \frac{\partial}{\partial y} . \]
It is defined in $H_{\Lambda,j}^{\lambda}$ and in $H_{j,\theta}^{\epsilon,\rho,\lambda}$
(see Step 1 of subsection \ref{subsec:extfatcor}) for any $\theta \in (0,\pi/2]$.
We denote $g_{j,\Lambda,\lambda}^{\varphi}= 1/(\partial \psi_{j,\Lambda,\lambda}^{\varphi} / \partial y)$.
\end{defi}
\begin{teo}
\label{teo:infgenms}
Let $\varphi \in \diff{tp1}{2}$ with $2$-convergent normal form $\Upsilon={\rm exp}(X)$.
Consider $j \in {\mathcal D}_{s}(\varphi)$ and $s \in \{-1,1\}$. There exists a development
\[ \hat{X}_{j}^{\varphi} = \left({ \sum_{k=0}^{\infty} g_{j,k}^{\varphi}(y) x^{k} }\right)
\frac{\partial}{\partial y} \]
such that
\begin{itemize}
\item $g_{j,k}^{\varphi}$ is holomorphic in $\cup_{\theta \in (0,\pi/2]} H_{j,\theta}^{\epsilon,\rho,\lambda}(0)$
for any $k \in {\mathbb N} \cup \{0\}$.
\item The series $\sum_{k=0}^{\infty} g_{j,k}^{\varphi}(y_{0}) x^{k}$ is the asymptotic development of
the  $(\tilde{e}_{1}, \hdots, \tilde{e}_{\tilde{q}})$-summable function $g_{j,\Lambda,\lambda}^{\varphi}(x,y_{0})$
for all $(0,y_{0}) \in \cup_{\theta \in (0,\pi/2]} H_{j,\theta}^{\epsilon,\rho,\lambda}(0)$ and
$\Lambda \in {\mathcal M}$.
\end{itemize}
\end{teo}
Let us remark that the set $\cup_{\theta \in (0,\pi/2]} H_{j,\theta}^{\epsilon,\rho,\lambda}(0)$ does not
depend on $\rho$ or $\lambda$. It only depends on $j$ and $\epsilon$.
\begin{proof}
Let us prove that $g_{j,\Lambda,\lambda}^{\varphi}$ is
$(\tilde{e}_{1}, \hdots, \tilde{e}_{\tilde{q}})$-summable in the $x$ variable.

Fix $\theta \in (0,\pi/2]$. Denote $X=g(x,y) \partial / \partial y$.
Let $\Lambda,\Lambda' \in {\mathcal M}$ and $\lambda, \lambda' \in {\mathbb S}^{1}$.
We have
\[ |{\psi}_{j,\Lambda,\lambda}^{\varphi} - {\psi}_{j,\Lambda',\lambda'}^{\varphi}|(x,y)
\leq  e^{-K/|x|^{\tilde{e}_{d_{\Lambda,\Lambda'}^{\lambda,\lambda'}+1}}}  \]
for any $(x,y) \in H_{j,\theta}^{\epsilon,\rho,\lambda,\lambda'}$ by theorem
\ref{teo:difFatflatext3}. The definition of $X_{j,\Lambda,\lambda}^{\varphi}$ implies
\[ g_{j,\Lambda,\lambda}^{\varphi} - g_{j,\Lambda',\lambda'}^{\varphi} =
\frac{1}{\partial {\psi}_{j,\Lambda,\lambda}^{\varphi}/\partial y} -
\frac{1}{\partial {\psi}_{j,\Lambda',\lambda'}^{\varphi} / \partial y}=
\frac{\partial {\psi}_{j,\Lambda',\lambda'}^{\varphi} / \partial y -\partial {\psi}_{j,\Lambda,\lambda}^{\varphi}/\partial y}
{(\partial {\psi}_{j,\Lambda,\lambda}^{\varphi}/\partial y)
(\partial {\psi}_{j,\Lambda',\lambda'}^{\varphi} / \partial y)} .\]
Let us consider the function
$h_{j,\Lambda,\lambda}=\partial ({\psi}_{j,\Lambda,\lambda}^{\varphi} - \psi_{L_{j}}^{X})/\partial y$. It satisfies
\[ h_{j,\Lambda,\lambda} =
\frac{\partial ({\psi}_{j,\Lambda,\lambda}^{\varphi} - \psi_{L_{j}}^{X})}{\partial \psi_{L_{j}}^{X}}
\frac{\partial \psi_{L_{j}}^{X}}{\partial y} =
\frac{\partial ({\psi}_{j,\Lambda,\lambda}^{\varphi} - \psi_{L_{j}}^{X})}{\partial \psi_{L_{j}}^{X}}
\frac{1}{g} . \]
By using proposition \ref{pro:bddconext} and Cauchy's integral formula we obtain that
$g h_{j,\Lambda,\lambda}$  is a continuous function
defined in in $\overline{H_{j,\theta}^{\epsilon,\rho,\lambda}}$ whose value
at $(0,0)$ is equal to $0$. Moreover the restriction
$(gh_{j,\Lambda,\lambda})(0,y)$ to $x=0$ does not depend on
$\Lambda$ or $\lambda$. Thus there exists $K_{1}' \geq 1$ such that
\[ \frac{1}{K_{1}'} \leq \frac{g_{j,\Lambda,\lambda}^{\varphi}}{g} \leq K_{1}' \ {\rm and} \
\frac{1}{K_{1}'} \leq \frac{g_{j,\Lambda',\lambda'}^{\varphi}}{g} \leq K_{1}' \]
in $H_{j,\theta}^{\epsilon,\rho,\lambda}$ and $H_{j,\theta}^{\epsilon,\rho,\lambda'}$ respectively.
Since
\[ \frac{\partial {\psi}_{j,\Lambda',\lambda'}^{\varphi}}{\partial y} -
\frac{\partial {\psi}_{j,\Lambda,\lambda}^{\varphi}}{\partial y} =
\frac{\partial ({\psi}_{j,\Lambda',\lambda'}^{\varphi} -
{\psi}_{j,\Lambda,\lambda}^{\varphi})}{\partial \psi_{L_{j}}^{X}}
\frac{\partial \psi_{L_{j}}^{X}}{\partial y} =
\frac{\partial ({\psi}_{j,\Lambda',\lambda'}^{\varphi} -
{\psi}_{j,\Lambda,\lambda}^{\varphi})}{\partial \psi_{L_{j}}^{X}}
\frac{1}{g}  \]
we deduce analogously that
\[ \left|{
\frac{\partial {\psi}_{j,\Lambda',\lambda'}^{\varphi}}{\partial y} -
\frac{\partial {\psi}_{j,\Lambda,\lambda}^{\varphi}}{\partial y}
}\right| \leq K_{2}' e^{-K/|x|^{\tilde{e}_{d_{\Lambda,\Lambda'}^{\lambda,\lambda'}+1}}} \frac{1}{|g|} \]
in $H_{j,\theta}^{\epsilon,\rho,\lambda,\lambda'}$ for some $K_{2}' \in {\mathbb R}^{+}$. Therefore
we obtain
\[ |g_{j,\Lambda,\lambda}^{\varphi} - g_{j,\Lambda',\lambda'}^{\varphi}| \leq |g| K_{2}' (K_{1}')^{2}
e^{-K/|x|^{\tilde{e}_{d_{\Lambda,\Lambda'}^{\lambda,\lambda'}+1}}} \]
in $H_{j,\theta}^{\epsilon,\rho,\lambda,\lambda'}$. Notice that
$H_{j,\theta}^{\epsilon,\rho,\lambda}(0)=H_{j,\theta}^{\epsilon,\rho,\lambda,\lambda'}(0)$.
The result is obtained by proceeding as in the proof of theorem \ref{teo:coechachas}.
\end{proof}
\subsection{Analyzing the infinitesimal generator}
\label{subsec:anainfgen}
Let $\varphi \in \diff{tp1}{2}$ with $2$-convergent normal form $\Upsilon={\rm exp}(X)$.
\begin{defi}
The infinitesimal generator $\log \varphi$ of $\varphi$ is of the form
\[ \log \varphi = \left({ \sum_{k=0}^{\infty} \hat{g}_{k}^{\varphi}(y) x^{k}}\right) \frac{\partial}{\partial y} \]
where $\hat{g}_{k}^{\varphi} \in {\mathbb C}[[x]]$ for any $k \in {\mathbb N} \cup \{0\}$.
\end{defi}
Analogously as in th. \ref{teo:infgenms} we prove that the family
$\{g_{j,k}^{\varphi}\}_{j \in {\mathcal D}(\varphi)}$ represents a
$\nu({\mathcal E}_{0})$-summable function for any $k \in {\mathbb N} \cup \{0\}$
(th. \ref{teo:infgenad}). We claim that this function coincides with
$\hat{g}_{k}^{\varphi}$. Roughly speaking the Lavaurs vector fields
$(1/\partial \psi_{j,\lambda,k}^{\varphi}) \partial /\partial y$ are very tangent to
$\hat{u} f \partial / \partial y$ by prop.  \ref{pro:fatcorork}.
The proof is completed by using that the functions of the form
$\psi_{j,\Lambda,\lambda}^{\varphi}(x,y) - \psi_{j,\lambda,k}^{\varphi}(x,y)$
are (up to an additive function of $x$) exponentially small in the $x$ variable
(prop. \ref{pro:coegida}).

Denote $f = y \circ \varphi -y$,
$\log \varphi = \hat{u}  f \partial/\partial y$ and
$X = u  f \partial/\partial y$
for some units $\hat{u} \in {\mathbb C}[[x,y]]$, $u \in {\mathbb C}\{x,y\}$.
We obtain $\hat{u}-u \in (f^{2})$.
Fix $k \geq \max(5,4 \nu({\mathcal E}_{0}))$. Let ${\rm exp}(Y_{k})$ be a
$k$-convergent normal form and $B(0,\delta) \times B(0,\epsilon_{k})$ as defined
in subsection \ref{subsec:asyfatcor}. The vector field $Y_{k}$ is of the form
$u_{k} f \partial / \partial y$ where $u_{k} \in {\mathbb C}\{x,y\}$ is a unit
such that $\hat{u}-u \in (f^{k})$.
\begin{pro}
\label{pro:coegida}
Let $\varphi \in \diff{tp1}{2}$ with $2$-convergent normal form $\Upsilon={\rm exp}(X)$.
Then $\hat{g}_{b}^{\varphi}$ is an asymptotic
development in $H_{j,\theta}^{\epsilon,\rho}(0)$
at $x=0$ of $g_{j,b}^{\varphi}$ for all $j \in {\mathcal D}(\varphi)$, $\theta \in (0,\pi/2]$
and $b \in {\mathbb N} \cup \{0\}$.
\end{pro}
\begin{proof}
Denote $\nu=\nu({\mathcal E}_{0})$.
Fix $k \geq \max(5,4 \nu)$. Consider $\lambda \in {\mathbb S}^{1}$. We have
\[ \frac{1}{\partial \psi_{j,\lambda,k}^{\varphi}/ \partial y} -
\frac{1}{\partial \psi_{L_{j},k} / \partial y} =
\frac{(\partial (\psi_{L_{j},k}-\psi_{j,\lambda,k}^{\varphi})/\partial \psi_{L_{j},k})
(\partial \psi_{L_{j},k} / \partial y)}
{(\partial \psi_{L_{j},k} / \partial y)(\partial \psi_{j,\lambda,k}^{\varphi}/ \partial y)},  \]
see subsection \ref{subsec:asyfatcor} for the definition of $\psi_{L_{j},k} $.
We denote $g_{j,\lambda,k}^{\varphi}= 1/(\partial \psi_{j,\lambda,k}^{\varphi} / \partial y)$.
We obtain
\[ \left|{
g_{j,\lambda,k}^{\varphi}  - u_{k} f }\right| =
O \left({ \frac{f}{(1+|\psi_{L_{j}}^{X}(x,y)|)^{k-1}} }\right)=
O(f^{\frac{\nu k +1}{\nu+1}})  \]
in $\cup_{x \in I_{\lambda}} (H_{j,\theta}^{\epsilon,\rho}(x) \cap H_{j,\theta/2}^{\epsilon_{k},\rho}(x))$
by proposition \ref{pro:fatcorork}.
Consider $\lambda' \in {\mathbb S}^{1}$ such that $I_{\lambda} \cap I_{\lambda'} \neq \emptyset$.
The function
$\psi_{j,\lambda,k}^{\varphi} - \psi_{j,\lambda',k}^{\varphi}$
is defined in $\cup_{x \in I_{\lambda'} \cap I_{\lambda}} B_{\varphi,j}^{1}(x)$
(see Step 3 of subsection \ref{subsec:extfatcor}).
We argue as in the proof of prop. \ref{pro:difFatflat} to obtain $K_{1}' \in {\mathbb R}^{+}$ such that
\[  |\psi_{j,\lambda,k}^{\varphi} - \psi_{j,\lambda',k}^{\varphi}| \leq
 e^{-K_{1}'/|x|^{\nu}}  \]
in $\cup_{x \in I_{\lambda'} \cap I_{\lambda}}
(H_{j,\theta}^{\epsilon,\rho}(x) \cap H_{j,\theta/2}^{\epsilon_{k},\rho}(x))$.
This implies
\begin{equation}
\label{equ:gevflat}
g_{j,\lambda,k}^{\varphi} - g_{j,\lambda',k}^{\varphi} =
O \left({
f \frac{\partial (\psi_{j,\lambda',k}^{\varphi}-
\psi_{j,\lambda,k}^{\varphi})}{\partial \psi_{L_{j}}^{X}}
}\right) = O \left({  e^{-K_{1}'/|x|^{\nu}} f }\right)
\end{equation}
in $\cup_{x \in I_{\lambda'} \cap I_{\lambda}}
(H_{j,\theta}^{\epsilon,\rho}(x) \cap H_{j,\theta/2}^{\epsilon_{k},\rho}(x))$.

We have $H_{j,\theta}^{\epsilon,\rho} \cap H_{j,\theta/2}^{\epsilon_{k},\rho} \subset
B(0,\delta) \times e^{i[-\zeta,\zeta]} (0,\epsilon)$ for some $\zeta \in {\mathbb R}^{+}$
(prop. \ref{pro:estext}). Lemma \ref{lem:itf} and remark \ref{rem:psiext} imply
$|\psi_{L_{j}}^{X}| \leq C_{1}'/|y|^{\nu}$ in
${\mathcal E}_{0} \cap (B(0,\delta) \times e^{i[-\zeta,\zeta]} (0,\epsilon))$ for some
$C_{1}' \in {\mathbb R}^{+}$. Cauchy's integral formula provides
\[ \left|{ \frac{\partial \psi_{L_{j}}^{X}}{\partial x} }\right|(x_{0},y) =
\frac{1}{2 \pi} \left|{
\int_{x \in \partial B(x_{0},|y|/(2 \rho_{0}))} \frac{\psi_{L_{j}}^{X}(x,y)}{(x-x_{0})^{2}} dx
}\right| \leq \frac{C_{1}' 2 \rho_{0}}{|y|^{\nu+1}}\]
for any $(x_{0},y) \in \tilde{\mathcal E}_{0} \cap (B(0,\delta) \times e^{i[-\zeta,\zeta]} (0,\epsilon))$
(see section \ref{subsec:dynsplit}). Consider points
$(0,y)  \in H_{j,\theta}^{\epsilon,\rho} \cap H_{j,\theta/2}^{\epsilon_{k},\rho}$
and $x \in \overline{B}(0,|y|^{\nu+2})$. We have
\[ |\psi_{L_{j}}^{X}(x,y) - \psi_{L_{j}}^{X}(0,y)| \leq C_{1}' 2 \rho_{0}
\frac{|x|}{|y|^{\nu+1}} \leq  C_{1}' 2 \rho_{0} |y|. \]
Therefore $(x,y)$ is in the neighborhood of
$H_{j,\theta}^{\epsilon,\rho} \cap H_{j,\theta/2}^{\epsilon_{k},\rho}$
if $y$ is in a neighborhood of $0$. As a consequence the property (\ref{equ:gevflat})
holds true for points $(x,y) \in T_{0}$ such that
$(0,y)$ is in a neighborhood of $(0,0)$ in $H_{j,\theta}^{\epsilon,\rho} \cap H_{j,\theta/2}^{\epsilon_{k},\rho}$
and $x \in \overline{B}(0,|y|^{\nu+2})$. The function $f$
is a $O(y^{\nu+1}) + O(x)$ and then a $O(y^{\nu+1})$
in $\{ (x,y) \in {\mathbb C}^{2} : |x| \leq |y|^{\nu+2} \}$.

Fix $(0,y_{0})  \in H_{j,\theta}^{\epsilon,\rho} \cap H_{j,\theta/2}^{\epsilon_{k},\rho}$ and the
ball $B(0,|y_{0}|^{\nu+2})$.
The equation (\ref{equ:gevflat}) implies that the family of functions
${\{ g_{j,\lambda,k}^{\varphi}(x,y_{0}) \}}_{\lambda \in {\mathbb S}^{1}}$
has a common $1/\nu$ Gevrey asymptotic development
$\sum_{b=0}^{\infty} g_{j,b,k}^{\varphi}(y_{0}) x^{b}$.
We apply lemma \ref{lem:techgevrey} to
the functions $(g_{j,\lambda,k}^{\varphi}-u_{k} f)(x,y_{0})$ to obtain that
$g_{j,b,k}^{\varphi}(y_{0})$ satisfies
\[ \left|{ g_{j,b,k}^{\varphi}(y_{0})-\frac{1}{b!} \frac{\partial^{b} (u_{k}f)}{\partial x^{b}}(0,y_{0}) }\right|
=  O \left({ \frac{(y_{0}^{\nu+1})^{\frac{\nu k +1}{\nu+1}}}
{y_{0}^{b(\nu+2)}} }\right) + e^{\frac{-K_{1}'}{2 |y_{0}|^{\nu(\nu+2)}}} =
O \left({ y_{0}^{\nu k +1 - b(\nu+2)} }\right). \]
Consider $\Lambda \in {\mathcal M}$, $\lambda \in {\mathbb S}^{1}$. Then
$\psi_{j,\Lambda,\lambda}^{\varphi} - \psi_{j,\lambda,k}^{\varphi}$
is defined in $\cup_{x \in I_{\lambda} \cap (0,\delta) I_{\Lambda}^{\lambda}} B_{\varphi,j}^{1}(x)$.
The functions $\psi_{j,\lambda,k}^{\varphi} - \psi_{L_{j}}^{X}$ and
$\psi_{j,\Lambda,\lambda}^{\varphi} - \psi_{L_{j}}^{X}$ are bounded in
$\cup_{x \in I_{\lambda} \cap (0,\delta) I_{\Lambda}^{\lambda}} B_{\varphi,j}^{1}(x)$
by lemma \ref{lem:fatcorork} and proposition \ref{pro:bddcon} respectively.
Therefore $\psi_{j,\Lambda,\lambda}^{\varphi}  - \psi_{j,\lambda,k}^{\varphi} $ is
bounded. By using
$(Im \ \psi_{L_{j}}^{X})(B_{\varphi,j}^{1}(x)) = [-a_{1}/|x|^{\nu},a_{1}/|x|^{\nu}]$
and proceeding as in prop. \ref{pro:difFatflat} we obtain that there exist
 $K_{2}' \in {\mathbb R}^{+}$ and a holomorphic function
 $a_{0}(x) \in {\mathcal O}( I_{\lambda} \cap (0,\delta) I_{\Lambda})$
such that
\[ |\psi_{j,\Lambda,\lambda}^{\varphi}(x,y) - \psi_{j,\lambda,k}^{\varphi}(x,y) - a_{0}(x)|
\leq e^{-K_{2}'/|x|^{\nu({\mathcal E}_{0})}} \ {\rm in} \
\cup_{x \in I_{\lambda} \cap (0,\delta) I_{\Lambda}^{\lambda}} B_{\varphi,j}^{1}(x). \]
The estimate holds true in
$\cup_{x \in I_{\lambda'} \cap (0,\delta) I_{\Lambda}^{\lambda}}
(H_{j,\theta}^{\epsilon,\rho}(x) \cap H_{j,\theta/2}^{\epsilon_{k},\rho}(x))$
too since the orbits by $\varphi$ of points in this set intersect
$\cup_{x \in I_{\lambda} \cap (0,\delta) I_{\Lambda}^{\lambda}} B_{\varphi,j}^{1}(x)$
(lemma \ref{lem:fundext}) and both sides are invariant by $\varphi$. We obtain
\[  g_{j,\lambda,k}^{\varphi} - g_{j,\Lambda,\lambda}^{\varphi} =
O \left({  e^{-K_{2}'/|x|^{\nu({\mathcal E}_{0})}} f }\right)  \]
in $\cup_{x \in I_{\lambda'} \cap (0,\delta) I_{\Lambda}^{\lambda}}
(H_{j,\theta}^{\epsilon,\rho}(x) \cap H_{j,\theta/2}^{\epsilon_{k},\rho}(x))$.
This implies $g_{j,b}^{\varphi}= g_{j,b,k}^{\varphi}$
in $(H_{j,\theta}^{\epsilon,\rho} \cap H_{j,\theta/2}^{\epsilon_{k},\rho})(0)$.
Since $(\sum_{k=0}^{\infty} \hat{g}_{k}^{\varphi}(y) x^{k})- u_{k} f \in (f^{k+1})$ we
deduce that $\hat{g}_{b}^{\varphi}$ is an asymptotic development of
$g_{j,b}^{\varphi}$ in $(H_{j,\theta}^{\epsilon,\rho} \cap H_{j,\theta/2}^{\epsilon_{k},\rho})(0)$
and then in $H_{j,\theta}^{\epsilon,\rho}(0)$ since
$H_{j,\theta}^{\epsilon,\rho}$ and
$H_{j,\theta}^{\epsilon,\rho} \cap H_{j,\theta/2}^{\epsilon_{k},\rho}$ coincide in a neighborhood of $(0,0)$.
\end{proof}
\begin{teo}
\label{teo:infgenad}
Let $\varphi \in \diff{tp1}{2}$ with $2$-convergent normal form $\Upsilon={\rm exp}(X)$.
Then $\hat{g}_{b}^{\varphi}$ is a $\nu({\mathcal E}_{0})$-summable function whose sum in
$\cup_{\theta \in (0,\pi/2]} H_{j,\theta}^{\epsilon,\rho}(0)$ is equal to $g_{j,b}^{\varphi}$
for all $j \in {\mathcal D}(\varphi)$ and $b \in {\mathbb N} \cup \{0\}$.
\end{teo}
\begin{proof}
Denote $\nu=\nu({\mathcal E}_{0})$.
Fix $\theta \in (0,\pi/2]$. Fix $\Lambda \in {\mathcal M}$. Consider $\lambda \in {\mathbb S}^{1}$.
The function $(g_{j+1,\Lambda,\lambda}^{\varphi} - g_{j,\Lambda,\lambda}^{\varphi})(x,y_{0})$
represents a $(\tilde{e}_{1}, \hdots, \tilde{e}_{\tilde{q}})$-summable function
for any $(0,y_{0}) \in H_{j,\theta}^{\epsilon,\rho}(0) \cap H_{j+1,\theta}^{\epsilon,\rho}(0)$
by th. \ref{teo:infgenms}. In fact we have
\[ (g_{j+1,\Lambda,\lambda}^{\varphi} - g_{j,\Lambda,\lambda}^{\varphi})(x,y) -
(g_{j+1,\Lambda',\lambda'}^{\varphi} - g_{j,\Lambda',\lambda'}^{\varphi})(x,y) =
O( f e^{-K/|x|^{\tilde{e}_{d_{\Lambda,\Lambda'}^{\lambda,\lambda'}+1}}} ) \]
in $H_{j,\theta}^{\epsilon,\rho,\lambda,\lambda'} \cap H_{j+1,\theta}^{\epsilon,\rho,\lambda,\lambda'}$
for all $\Lambda' \in {\mathcal M}$ and
$\lambda,\lambda' \in {\mathbb S}^{1}$ (see proof of theorem \ref{teo:infgenms}).

We have
\[ \xi_{j,\Lambda,\lambda}^{\varphi} - z= \zeta_{\varphi}(x) +
\sum_{l=1}^{\infty} a_{j,\Lambda,\lambda,l}^{\varphi}(x) e^{-2 \pi i s l z} \]
where $\sum_{l=1}^{\infty} a_{j,\Lambda,\lambda,l}^{\varphi}(x) w^{l}$ is a bounded function in
$\{ (x,w) \in (0,\delta) I_{\Lambda}^{\lambda} \times  B(0,e^{-2 \pi M}) \}$. We deduce
that $\xi_{j,\Lambda,\lambda}^{\varphi} - z - \zeta_{\varphi}(x)$ is a $O(e^{-2 \pi i s z})$. We obtain
\[ \psi_{j+1,\Lambda,\lambda}^{\varphi} - \psi_{j,\Lambda,\lambda}^{\varphi} - \zeta_{\varphi}(x) =
O(e^{-2 \pi i s \psi_{j,\Lambda,\lambda}^{\varphi}})  = O(e^{-2 \pi  |Im(\psi_{L_{j}}^{X})|}) \]
in $H_{j,\theta}^{\epsilon,\rho} \cap H_{j+1,\theta}^{\epsilon,\rho} \cap \{ s Im(\psi_{L_{j}}^{X}) < -M\}$.
Since $\theta>0$ we get that there exists $c_{\theta} \in {\mathbb R}^{+}$
such that $|Im(\psi_{L_{j}}^{X})| \geq c_{\theta} |\psi_{L_{j}}^{X}|$ in
$H_{j,\theta}^{\epsilon,\rho} \cap H_{j+1,\theta}^{\epsilon,\rho} \cap \{ s Im(\psi_{L_{j}}^{X}) < -M\}$
and then
$O(e^{-2 \pi  |Im(\psi_{L_{j}}^{X})|})=O(e^{-2 \pi c_{\theta} |\psi_{L_{j}}^{X}|})$. We obtain
\[ g_{j+1,\Lambda,\lambda}^{\varphi} - g_{j,\Lambda,\lambda}^{\varphi} =
\frac{(\partial (\psi_{j,\Lambda,\lambda}^{\varphi}-\psi_{j+1,\Lambda,\lambda}^{\varphi})/\partial \psi_{L_{j}}^{X})
(\partial \psi_{L_{j}}^{X} / \partial y)}
{(\partial \psi_{j,\Lambda,\lambda}^{\varphi} / \partial y)(\partial \psi_{j,\Lambda,\lambda}^{\varphi}/ \partial y)}
 = O(f e^{-2 \pi  c_{\theta} |\psi_{L_{j}}^{X}|})\]
in $H_{j,\theta}^{\epsilon,\rho} \cap H_{j+1,\theta}^{\epsilon,\rho} \cap \{ s Im(\psi_{L_{j}}^{X}) < -M\}$.
Moreover $\psi_{L_{j}}^{X}$ satisfies
$|\psi_{L_{j}}^{X}| \geq C_{6}^{-1}/|y|^{\nu}$ in $H_{j,\theta}^{\epsilon,\rho}$
by equation (\ref{equ:pexpsi}). The previous discussion implies
\[ g_{j+1,\Lambda,\lambda}^{\varphi} - g_{j,\Lambda,\lambda}^{\varphi} =
O \left({ f e^{- \frac{2 \pi c_{\theta}}{C_{6}} \frac{1}{|y|^{\nu}}} }\right) \]
in $H_{j,\theta}^{\epsilon,\rho} \cap H_{j+1,\theta}^{\epsilon,\rho} \cap \{ s Im(\psi_{L_{j}}^{X}) < -M\}$.

We proceed as in proposition \ref{pro:coegida} (using lemma \ref{lem:techgevrey}) to obtain
\begin{equation}
\label{equ:coccond}
 (g_{j+1,b}^{\varphi} - g_{j,b}^{\varphi})(y) =
O \left({ \frac{f e^{- \frac{2 \pi c_{\theta}}{C_{6}} \frac{1}{|y|^{\nu}}}}{y^{b(\nu+2)}} }\right)
+ O \left({ e^{-K/(2 |y|^{\nu (\nu+2)})} }\right) =
O \left({  e^{- \frac{\pi c_{\theta}}{C_{6}} \frac{1}{|y|^{\nu}}}   }\right)
\end{equation}
in $H_{j,\theta}^{\epsilon,\rho}(0) \cap H_{j+1,\theta}^{\epsilon,\rho}(0) \cap \{ s Im(\psi_{L_{j}}^{X}) < -M\}$.

The curve $\{ y \in H_{j,\theta}^{\epsilon,\rho}(0): \psi_{L_{j}}^{X}(0,y) \in {\mathbb R} \}$
adheres a direction $\lambda_{j} \in {\mathbb S}^{1}$ at $y=0$.
We have $\lambda_{k'} =\lambda_{k} e^{\pi i \nu^{-1} (k'-k)}$ for all
$k,k' \in {\mathbb Z}/(2 \nu {\mathbb Z})$. Since $\psi_{L_{j}}^{X} \sim 1/y^{\nu}$ in
$H_{j,\theta}^{\epsilon,\rho}(0)$ we deduce that $H_{j,\theta}^{\epsilon,\rho}(0)$
contains a sector
$\{ y \in (0,\epsilon_{\theta}) \lambda_{j}
e^{i \left[{ -   \frac{\pi - 2 \theta}{\nu} , \frac{\pi - 2 \theta}{\nu} }\right]} \}$. Indeed
it contains sectors of the form
$(0,\epsilon') \lambda_{j}
e^{i [ - (\pi - \theta) \nu^{-1} + \eta, (\pi - \theta) \nu^{-1} - \eta ]}$ for any
$\eta \in {\mathbb R}^{+}$. Analogously
$H_{j,\theta}^{\epsilon,\rho}(0) \cap H_{j+1,\theta}^{\epsilon,\rho}(0) \cap \{ s Im(\psi_{L_{j}}^{X}) < -M\}$
contains a sector $\{ y \in (0,\epsilon_{\theta}') \lambda_{j} e^{\frac{\pi i}{2 \nu}}
e^{i \left[{ -   \frac{\pi - 3 \theta}{2 \nu} , \frac{\pi - 3 \theta}{2 \nu} }\right]} \}$
in which property (\ref{equ:coccond}) holds true.
The function $g_{j,b}$ is bounded at the origin in $H_{j,\theta}^{\epsilon,\rho}(0)$
for $j \in {\mathbb Z}/(2 \nu {\mathbb Z})$ since it has
an asymptotic development (prop. \ref{pro:coegida}).
Therefore prop. \ref{pro:sumbal} implies that there exists a unique $\nu$-summable function
$\hat{\phi}_{b}$ such that its sum in $H_{j,\theta}^{\epsilon,\rho}(0)$ is $g_{j,b}$
for all $j \in {\mathbb Z}/(2 \nu {\mathbb Z})$ and $\theta \in (0,\pi/2]$.
Moreover $\{ \lambda_{1} e^{\frac{\pi i}{\nu}}, \hdots, \lambda_{2 \nu} e^{\frac{\pi i}{\nu}} \}$
is the set of singular directions of $\hat{\phi}_{b}$.
Since asymptotic developments are unique we obtain $\hat{g}_{b}^{\varphi} = \hat{\phi}_{b}$.
In particular $\hat{g}_{b}^{\varphi}$ is a $\nu$-summable function for any
$b \in {\mathbb N} \cup \{0\}$.
\end{proof}
\begin{rem}
Let $\varphi \in \diff{tp1}{2}$ and $\Upsilon={\rm exp}(g \partial / \partial y)$
a $2$-convergent normal form
of $\varphi$. We can describe the asymptotic behavior of the Lavaurs vector fields in
the neighborhood of $Fix (\varphi) \cup \{x=0\}$.
\begin{itemize}
\item The vector field
$X_{j,\Lambda,\lambda}^{\varphi}=g_{j,\Lambda,\lambda}^{\varphi} \partial / \partial y$
coincides with $\log \varphi$
until the first non-zero term in the neighborhood of $Fix (\varphi)$. More
precisely, $g_{j,\Lambda,\lambda}^{\varphi} /g -1$ is a continuous function in
$H_{j,\theta}^{\epsilon,\rho,\lambda}$
vanishing at $\overline{H_{j,\theta}^{\epsilon,\rho,\lambda}} \cap Fix (\varphi)$.
This result is corollary 7.3 in \cite{JR:mod} whose proof is based in prop.
7.3 \cite{JR:mod}. Since prop. \ref{pro:bddconext} is the analogue of
prop. 7.3 \cite{JR:mod} for multi-transversal flows the same result holds true.
\item Fixed $j \in {\mathcal D}(\varphi)$ the family
$\{X_{j,\Lambda,\lambda}^{\varphi}\}_{(\Lambda,\lambda)
\in {\mathcal M} \times {\mathbb S}^{1}}$ is multi-summable in the variable $x$.
\end{itemize}
\end{rem}
\section{Applications}
\label{sec:app}
In this section we adapt the analytic conjugation theorem in \cite{JR:mod} to the
multi-summable setting. We also express the analytic class of $\varphi \in \diff{p1}{2}$
as a function of the analytic classes of the one dimensional germs in the family
${\{\varphi_{|x=x_{n}}\}}_{n \in {\mathbb N}}$ where $x_{n}$ is a sequence of elements of
$B(0,\delta) \setminus \{0\}$ such that $\lim_{n \to \infty} x_{n}=0$.
\subsection{Actions of conjugations on Fatou coordinates}
Given $\varphi, \eta \in \diff{p1}{2}$, we are interested on studying
whether the analytic conjugacy
of $\varphi_{|x=x_{n}}$ and $\eta_{|x=x_{n}}$ by an analytic mapping
$\kappa_{n}$ defined in an open set in $\{x_{n}\} \times {\mathbb C}$ for a sequence
$x_{n} \to 0$ implies
that $\varphi$ and $\eta$ are analytically conjugated. No continuous dependence
with respect to the parameter $x$ of the mapping $(x_{n},y) \mapsto \kappa_{n}(y)$
is required. In this subsection we prove that even if the mappings $\kappa_{n}$ are
unrelated some of their properties behave uniformly.
\begin{defi} (see \cite{JR:mod})
We say that $\eta$ is a {\it s-mapping} if
$\eta$ is a biholomorphism from $B(0,s)$ onto
$\eta(B(0,s))$. If $\eta(B(0,s))$ is contained in $B(0,S)$ then
we say that $\eta$ is a (s,S)-mapping.
\end{defi}
Next we adapt the results in section 10.1 of \cite{JR:mod} to the context of slow decaying functions
(see def. \ref{def:slow}).
\begin{defi}
Let $X=x^{a} g(x,y) \partial /\partial y \in {\mathcal X} \cn{2}$ with
$g \in {\mathbb C} \{x,y\}$ and $g(0,y) \not \equiv 0$. The radical ideal $\sqrt{(g)}$ is
generated by an element $h \in {\mathbb C}\{x,y\}$. We define $N(X)$ as the order of
$h(0,y)$ at $y=0$. In fact $N(X)$ is the cardinal of the set
$Sing X \cap \{x=x_{0}\}$ for $x_{0} \neq 0$.
\end{defi}
%
%
\begin{defi}
Let $X \in {\mathcal X} \cn{2}$.
We denote $\kappa_{|(Sing X)(x_{0})}  \cong Id$ if $N(X)>1$ and $\kappa_{|(Sing X)(x_{0})} \equiv Id$.
Suppose $N(X)=1$. The set $Sing X$ has a unique component $y=\gamma(x)$ different than $x=0$.
We denote $\kappa_{|(Sing X)(x_{0})}  \cong Id$ if
$\kappa(\gamma(x_{0}))=\gamma(x_{0})$ and $\kappa'(\gamma(x_{0}))=1$.
\end{defi}
The previous definition implies that $\kappa - y$ vanishes in $(Sing X)(x_{0})$ and
the sum of the vanishing multiplicities is greater or equal than $2$.

Next we see that we can control the image of $s$-mappings.
\begin{lem}
\label{lem:quitR}
Let $X \in {\mathcal X} \cn{2}$ with $N(X) \geq 1$.
Let $s:B(0,\delta) \setminus \{0\} \to {\mathbb R}^{+}$ be a slow decaying function
and $\tau \in (0,1/4]$. There exists a neighborhood $V$
of $0$ in the $x$-line such that for any $x_{0} \in V \setminus \{0\}$ each $s(x_{0})$-mapping
satisfying $\kappa_{|(Sing X)(x_{0})} \cong Id$ is a $(s(x_{0})\tau, 2 s(x_{0})\tau)$-mapping.
\end{lem}
\begin{proof}
Let $\gamma_{1}(x_{0})$ and $\gamma_{2}(x_{0})$ be two different points of
$(Sing X)(x_{0})$ if $N(X)>1$. Otherwise we consider $\gamma_{1}(x_{0}) = \gamma_{2}(x_{0}) \in (Sing X)(x_{0})$.
We define
\[ \kappa_{1}(y) =
\frac{\kappa((s(x_{0})-|\gamma_{1}(x_{0})|)y + \gamma_{1}(x_{0})) -\gamma_{1}(x_{0})}
{(s(x_{0})-|\gamma_{1}(x_{0})|) (\partial \kappa/\partial y)(\gamma_{1}(x_{0}))} .\]
Then $\kappa_{1}$ is a Schlicht function.
Denote $\upsilon(x_{0}) =
(\gamma_{2}(x_{0}) - \gamma_{1}(x_{0}))/(s(x_{0})-|\gamma_{1}(x_{0})|)$. We have
$\kappa_{1} (\upsilon(x_{0})) = \upsilon(x_{0})/
(\partial \kappa/\partial y)(\gamma_{1}(x_{0}))$.
Koebe's distortion theorem (see \cite{Conway}, page 65) implies
$| (\partial \kappa / \partial y)(\gamma_{1}(x_{0})) | \leq
(1 + |\upsilon(x_{0})|)^{2}$. We have
\[ \sup_{y \in B(0,s(x_{0})\tau)} |\kappa(y)| \leq
(s(x_{0}) - |\gamma_{1}(x_{0})|)(1 + |\upsilon(x_{0})|)^{2}
\sup_{y \in B(0,A(x_{0},\tau))}
|\kappa_{1}(y)| + |\gamma_{1}(x_{0})| \]
where $A(x_{0},\tau)=(s(x_{0}) \tau +|\gamma_{1}(x_{0})|)/(s(x_{0})-|\gamma_{1}(x_{0})|)$.
By Koebe's distortion theorem we obtain
$\sup_{y \in B(0,A(x_{0},\tau))} |\kappa_{1}(y)| \leq
A(x_{0},\tau)/(1-A(x_{0},\tau))^{2}$ and
\[ \sup_{y \in B(0,s(x_{0})\tau)} |\kappa(y)| \leq
(s(x_{0}) - |\gamma_{1}(x_{0})|)(1 + |\upsilon(x_{0})|)^{2}
\frac{A(x_{0},\tau)}{(1-A(x_{0},\tau))^{2}}  + |\gamma_{1}(x_{0})|. \]
Since
\[ \lim_{x_{0} \to 0} A(x_{0},\tau) = \lim_{x_{0} \to 0}
\tau \frac{1 +\frac{|\gamma_{1}(x_{0})|}{s(x_{0}) \tau}}{1-\frac{|\gamma_{1}(x_{0})|}{s(x_{0})}} = \tau . \]
we get $\sup_{y \in B(0,s(x_{0})\tau)} |\kappa(y)| \leq 2 s(x_{0}) \tau$ for any $x_{0}$ in
a neighborhood of $0$.
\end{proof}
\begin{lem}
\label{lem:adjcon}
Let $\varphi_{1}, \varphi_{2} \in \diff{tp1}{2}$ with convergent
normal forms ${\rm exp}(X_{1})$ and ${\rm exp}(X_{2})$ respectively.
Suppose $Fix (\varphi_{1})=Fix (\varphi_{2})$ is a curve $y=\gamma(x)$.
Consider an analytic diffeomorphism $\sigma(x,y)=(x,\sigma_{2}(x,y)) \in \diff{p}{2}$
conjugating $X_{1}$ and $X_{2}$.
Let $s:B(0,\delta) \setminus \{0\} \to {\mathbb R}^{+}$ be a slow decaying function
and $\tau \in (0,1/4]$. There exists a neighborhood $V$
of $0$ in the $x$-line such that for any $x_{0} \in V \setminus \{0\}$ each $s(x_{0})$-mapping
$\kappa$ conjugating $(\varphi_{1})_{|x=x_{0}}$ and $(\varphi_{2})_{|x=x_{0}}$
is a $(s(x_{0})\tau, s(x_{0})\tau')$-mapping where
$\tau'=2 \sup_{x \in B(0,\delta)} | \partial(y \circ \sigma)/\partial y|(x,\gamma(x))$.
\end{lem}
\begin{proof}
Denote $\zeta_{1}$, $\zeta_{2}$ and $\upsilon$ the germs of
diffeomorphisms induced by $(\varphi_{1})_{|x=x_{0}}$,
$(\sigma^{-1} \circ \varphi_{2} \circ \sigma)_{|x=x_{0}}$  and
$\sigma_{|x=x_{0}}^{-1} \circ \kappa$ respectively
in the neighborhood of $\gamma(x_{0})$.
The mapping $\upsilon$ conjugates
diffeomorphisms $\zeta_{1}$, $\zeta_{2}$ with common convergent normal form
${\rm exp}(X_{1})_{|x=x_{0}}$.
Denote $\nu=\nu({\mathcal E}_{0})$.
The diffeomorphism $\upsilon$ is of the form
\[ \upsilon =
\sigma_{\lambda} \circ {\rm exp}(t \log \zeta_{2}) \circ \hat{\sigma}(\zeta_{1},\zeta_{2}) \]
where $\hat{\sigma}(\zeta_{1},\zeta_{2})$ is the unique element of
$\widehat{\rm Diff} ({\mathbb C},\gamma(x_{0}))$ conjugating $\zeta_{1}$, $\zeta_{2}$
of the form $y+O((y-\gamma(x_{0}))^{\nu+2})$, $t \in {\mathbb C}$ and $\sigma_{\lambda}$
is a periodic element of $\widehat{\rm Diff} ({\mathbb C},\gamma(x_{0}))$
commuting with $\zeta_{2}$ such that
$\upsilon'(\gamma(x_{0}))=\sigma_{\lambda}'(\gamma(x_{0}))=\lambda \in e^{2 \pi i {\mathbb Q}}$.
We obtain $|\upsilon'(\gamma(x_{0}))| = 1$ and then
$|\kappa'(\gamma(x_{0}))| =|\partial(y \circ \sigma)/\partial y|(x_{0},\gamma(x_{0}))$.

We define
\[ \kappa_{1}(y) =
\frac{\kappa((s(x_{0})-|\gamma(x_{0})|)y + \gamma(x_{0})) -\gamma(x_{0})}
{(s(x_{0})-|\gamma(x_{0})|) (\partial \kappa/\partial y)(\gamma(x_{0}))} .\]
Then $\kappa_{1}$ is a Schlicht function. By the Koebe's distortion
theorem (see \cite{Conway}, page 65) we get
\[ \sup_{y \in B(0,s(x_{0})\tau)} |\kappa(y)| \leq
(s  - |\gamma|)(x_{0}) \left|{ \frac{\partial(y \circ \sigma)}{\partial y} }\right|(x_{0},\gamma(x_{0}))
\sup_{y \in B(0,A(x_{0},\tau))}
|\kappa_{1}(y)| + |\gamma(x_{0})| \]
where $A(x_{0},\tau)=(s(x_{0}) \tau +|\gamma(x_{0})|)/(s(x_{0})-|\gamma(x_{0})|)$.
By arguing as in lemma \ref{lem:quitR} we obtain that
$\kappa$ is a $(s(x_{0}) \tau, s(x_{0}) \tau')$-mapping.
\end{proof}
Now we provide uniform quantitative estimates for $s$-mapping conjugacies.
\begin{lem}
\label{lem:modbdd}
Let $\varphi_{1}, \varphi_{2} \in \diff{tp1}{2}$ with common convergent
normal form ${\rm exp}(X)$ where $X=u(x,y) \prod_{j=1}^{N} (y-\gamma_{j}(x))^{n_{j}} \partial/\partial y$
and $u \in {\mathbb C} \{x,y\}$ is a unit.
Denote $\nu=\nu({\mathcal E}_{0})$.
Let $s:B(0,\delta) \setminus \{0\} \to {\mathbb R}^{+}$ be a slow decaying function.
There exists an open set $0 \in V \subset {\mathbb C}$
such that for any  $x_{0} \in V \setminus \{ 0 \}$ and any $s(x_{0})$-mapping $\kappa$ conjugating
$(\varphi_{1})_{|x=x_{0}}$,
$(\varphi_{2})_{|x=x_{0}}$ with $\kappa_{|(Sing X)(x_{0})} \cong Id$
then $\kappa$ is of the form
$y + J_{\kappa}(y) \prod_{j=1}^{N} (y-\gamma_{j}(x_{0}))^{n_{j}}$ where
$\sup_{B(0,s(x_{0})/4)} |J_{\kappa}| < (8^{\nu} 6)/s^{\nu}(x_{0})$.
\end{lem}
\begin{proof}
We have $\kappa = y + (y-\gamma_{1}(x_{0})) \hdots (y - \gamma_{N}(x_{0})) A(y)$
for some $A \in {\mathcal O}(B(0,s(x_{0})))$ by hypothesis.
By lemma \ref{lem:quitR} and the modulus maximum principle we obtain
\[ \sup_{B(0,s(x_{0})/4)} |A| =   \sup_{y \in B(0,s(x_{0})/4)}
\frac{|\kappa(y) -y|}{|(y-\gamma_{1}(x_{0})) \hdots (y - \gamma_{N}(x_{0}))|}
\leq \frac{3 s(x_{0})/4}{((s(x_{0})/4) /2)^{N}}\]
for any $x_{0}$ in a pointed neighborhood of $0$. We have that
\[ \left|{ \frac{\partial \kappa}{\partial y }(\gamma_{j}(x_{0})) -1 }\right|
\leq \frac{8^{N-1} 6}{s^{N-1}(x_{0})}
\prod_{k \in \{1,\hdots, N\} \setminus \{ j \}}
|\gamma_{j}(x_{0}) - \gamma_{k}(x_{0})| . \]
Fix $j \in \{1,\hdots,N\}$. We claim that
$(y-\gamma_{j}(x_{0}))^{n_{j}}$ divides $\kappa -y$.
We can suppose $n_{j}>1$. Denote $\zeta_{1}$,
$\zeta_{2}$ and $\upsilon$ the germs of
diffeomorphisms induced by $(\varphi_{1})_{|x=x_{0}}$,
$(\varphi_{2})_{|x=x_{0}}$  and $\kappa$ respectively
in the neighborhood of $\gamma_{j}(x_{0})$. The diffeomorphism $\upsilon$ is of the form
\[ \upsilon = \sigma_{\lambda} \circ {\rm exp}(t \log \zeta_{2}) \circ \hat{\sigma}(\zeta_{1},\zeta_{2}) \]
where $\hat{\sigma}(\zeta_{1},\zeta_{2})$ is the unique element of
$\widehat{\rm Diff} ({\mathbb C},\gamma_{j}(x_{0}))$ conjugating $\zeta_{1}$, $\zeta_{2}$
of the form $y+O((y-\gamma_{j}(x_{0}))^{n_{j}+1})$, $t \in {\mathbb C}$ and
$\sigma_{\lambda}$ is a periodic
element of $\widehat{\rm Diff} ({\mathbb C},\gamma_{j}(x_{0}))$ commuting
with $\zeta_{2}$ such that
$\upsilon'(\gamma_{j}(x_{0}))=\sigma_{\lambda}'(\gamma_{j}(x_{0}))=
\lambda \in <e^{2 \pi i/(n_{j}-1)}>$.
We obtain $\lambda = 1$ for $N(X)=1$ by hypothesis. We obtain $\lambda=1$ for $N(X)>1$ and
$x_{0}$ in a neighborhood of $0$ since $N \geq 2$
implies $\lim_{x_{0} \to 0} (\partial \kappa/\partial y)(\gamma_{j}(x_{0})) =1$.
Thus $\sigma_{\lambda} \equiv Id$ and then
$\kappa -y$ belongs to $(y-\gamma_{j}(x_{0}))^{n_{j}}$. The function
$J_{\kappa}$ belongs to ${\mathcal O}(B(0,s(x_{0})))$.
Analogously as for $A$ we obtain
$\sup_{B(0,s(x_{0})/4)} |J_{\kappa}| \leq 8^{\nu} 6/s^{\nu}(x_{0})$ for any $x_{0} \neq 0$.
\end{proof}
We want to interpret the estimates in lemma  \ref{lem:modbdd}
in terms of the Fatou coordinates of the common normal form.
We define $\kappa_{t}(y) = y + t (\kappa(y)-y)$
for $y \in B(0,s(x_{0}))$ and $t \in {\mathbb C}$.
Let $\psi^{X}$ be a holomorphic integral of the time form of $X$.
We can define the function $\psi^{X} \circ \kappa(x,y)  - \psi^{X}(x,y)$
analogously as $\Delta_{\varphi}$.
The continuous path that we use to extend $\psi^{X}$ is parametrized by
$t \to \kappa_{t}(x,y)$ for $t \in [0,1]$. The function
$\psi^{X} \circ \kappa - \psi^{X}$ is well-defined
and holomorphic in $B(0,s(x_{0})) \setminus Sing X$.
\begin{lem}
\label{lem:bddfldi}
Let $\varphi_{1}, \varphi_{2} \in \diff{tp1}{2}$ with common convergent
normal form ${\rm exp}(X)$.
Let $s$ be a slow decaying function. Then there exists $\tau \in {\mathbb R}^{+}$
and $C' \in {\mathbb R}^{+}$ such that we have
$\sup_{B(0,s(x_{0}) \tau)} |\psi^{X} \circ \kappa - \psi^{X}| \leq C'/s^{\nu}(x_{0})$
for any $s(x_{0})$-mapping $\kappa$ conjugating $\varphi_{1}(x_{0},y)$ and
$\varphi_{2}(x_{0},y)$ with $\kappa_{|(Sing X)(x_{0})} \cong Id$ and any
$x_{0}$ in a pointed neighborhood of $0$. In particular we obtain that
$\psi^{X} \circ \kappa - \psi^{X}$ belongs to ${\mathcal O}(B(0,s(x_{0}) \tau))$.
\end{lem}
\begin{proof}
Denote
$X(y)=u(x,y) (y-\gamma_{1}(x))^{n_{1}} \hdots (y - \gamma_{N}(x))^{n_{N}}$
where $u \in {\mathbb C} \{x,y\}$ is a unit.
Denote $\nu=\nu({\mathcal E}_{0})$. Consider $\tau \in (0,1/4)$.
The function $\kappa_{t}(y)=y + t(\kappa(y)-y)$ satisfies
$\kappa_{t}(B(0,s(x_{0}) \tau)) \subset B(0,3 s(x_{0}) \tau)$
for any $t \in [0,1]$ by lemma \ref{lem:quitR}.
By lemma \ref{lem:modbdd} we have
\[ \left|{ \frac{\partial \kappa_{t}}{\partial t}(y) }\right| = |\kappa(y)-y|
\leq \frac{(8^{\nu} 6)/s^{\nu}(x_{0})}{|u \circ \kappa_{t}(y)|} |X(y) \circ \kappa_{t}(y)|
\left|{ \frac{\prod_{j=1}^{N} (y-\gamma_{j}(x_{0}))^{n_{j}}}
{\prod_{j=1}^{N} (y-\gamma_{j}(x_{0}))^{n_{j}} \circ \kappa_{t}(y)} }\right| \]
for all $y \in B(0,s(x_{0}) \tau)$ and $t \in [0,1]$. We have
\[ \frac{(y-\gamma_{j}(x_{0})) \circ \kappa_{t}(y)}{y -\gamma_{j}(x_{0})} =
1 + t \frac{\kappa(y)-y}{y-\gamma_{j}(x_{0})}. \]
Lemma \ref{lem:modbdd} implies
\[ \left|{ \frac{\kappa(y)-y}{y-\gamma_{j}(x_{0})} }\right| \leq  \frac{8^{\nu} 6}{s^{\nu}(x_{0})}
(2 s(x_{0}) \tau)^{\nu} = 16^{\nu} 6 \tau^{\nu} \ \ \forall y \in B(0,s(x_{0})\tau). \]
%
By considering a smaller $\tau \in {\mathbb R}^{+}$ we can suppose
$16^{\nu} 6 \tau^{\nu} < 1/2$. We obtain
\[ \left|{ \frac{\partial \kappa_{t}}{\partial t}(y) }\right| \leq
\frac{1}{s^{\nu}(x_{0})} \frac{(8^{\nu} 6)2}{|u(0,0)|} 2^{\nu+1} |X(y) \circ \kappa_{t}(y)| \]
for all $y \in B(0,s(x_{0}) \tau)$ and $t \in [0,1]$.
Denote $C'=(2^{4 \nu +3}  3)/|u(0,0)|$. We deduce that
\[ |\psi^{X} \circ \kappa_{t} - \psi^{X}|(y)  =
\left|{ \int_{0}^{t} \frac{\partial (\psi^{X} \circ \kappa_{\upsilon}(y))}{\partial \upsilon} d \upsilon }\right| =
\left|{ \int_{0}^{t} \frac{\partial \psi^{X}}{\partial y} \circ \kappa_{\upsilon}(y)
\frac{\partial \kappa_{\upsilon}(y)}{\partial \upsilon} d \upsilon }\right| \leq \frac{C' t}{s^{\nu}(x_{0})} \]
for all $y \in B(0,s(x_{0}) \tau) \setminus (Sing X)(x_{0})$ and $t \in [0,1]$. By Riemann's theorem
$\psi^{X} \circ \kappa -\psi^{X}$ belongs to ${\mathcal O}(B(0,s(x_{0}) \tau))$
and $|\psi^{X} \circ \kappa - \psi^{X}| \leq C'/s^{\nu}(x_{0})$.
\end{proof}
Let $\varphi_{1}, \varphi_{2} \in \diff{tp1}{2}$ with common convergent
normal form. Given a formal conjugation $\hat{\eta} \in \diffh{p}{2}$
we express the condition $\hat{\eta} \in \diff{}{2}$ in terms of the
extensions of Ecalle-Voronin invariants
$\{\xi_{j,\Lambda,\lambda}^{\varphi_{1}}\}_{(j,\Lambda,\lambda) \in
{\mathcal D}(\varphi_{1}) \times {\mathcal M} \times {\mathbb S}^{1}}$
and
$\{\xi_{j,\Lambda,\lambda}^{\varphi_{2}}\}_{(j,\Lambda,\lambda) \in
{\mathcal D}(\varphi_{2}) \times {\mathcal M} \times {\mathbb S}^{1}}$.
\begin{pro}
\label{pro:conimeq}
  Let $\varphi_{1}, \varphi_{2} \in \diff{tp1}{2}$ with common convergent
normal form ${\rm exp}(X)$. Let $s$ be a slow decaying function.
Fix $\Lambda=(\lambda_{1}, \hdots, \lambda_{\tilde{q}}) \in {\mathcal M}$
and $\lambda \in {\mathbb S}^{1}$.
Consider $x_{0} \in (0,\delta) I_{\Lambda}^{\lambda}$ and
a $s(x_{0})$-mapping $\kappa$ conjugating $(\varphi_{1})_{|x=x_{0}}$ and
$(\varphi_{2})_{|x=x_{0}}$  with $\kappa_{|(Sing X)(x_{0})} \cong Id$. Then there exists
$c(x_{0}) \in {\mathbb C}$ such that
\[ \xi_{j,\Lambda,\lambda}^{\varphi_{2}} (x_{0},z) =
(z + c(x_{0})) \circ \xi_{j,\Lambda,\lambda}^{\varphi_{1}} (x_{0},z)
\circ (z-c(x_{0}))
\ \ \forall j \in {\mathcal D}(\varphi_{1}) \]
and $|c(x_{0})| \leq C'/s^{\nu({\mathcal E}_{0})}(x_{0})$.
The constant $C'$ depends on $\Lambda$ and $\lambda$ but it does not depend on $x_{0}$.
\end{pro}
\begin{proof}
Denote $\nu=\nu({\mathcal E}_{0})$.
Since $s$ is bounded we can suppose $s < \epsilon$ by replacing $s$ with $s \tau_{0}$ for some
$\tau_{0} \in (0,1)$.
By lemma \ref{lem:bddfldi} there exist $C_{1}' \in {\mathbb R}^{+}$ and $\tau_{1} \in (0,1)$
such that $\sup_{B(0,s(x_{0}) \tau_{1})} |\psi^{X} \circ \kappa - \psi^{X}| \leq C_{1}'/s^{\nu}(x_{0})$.

Let $j \in {\mathcal D}(\varphi_{1})$ and $\tau \in (0,1)$. We define $H_{s \tau, j}$ the element of
$Reg(s \tau, \aleph_{\Lambda,\lambda} X, I_{\Lambda}^{\lambda})$ contained in $H_{\Lambda,j}^{\lambda}$
(see def. \ref{def:hjlam}).
We obtain that $\kappa(H_{s \tau, j}(x_{0})) \subset H_{\Lambda,j}^{\lambda}(x_{0})$
for $\tau \in {\mathbb R}^{+}$  small enough and any $j \in {\mathcal D}(\varphi_{1})$
by proposition \ref{pro:widcon}.

Consider an irreducible component $y=\gamma(x)$ of $Sing X$.
We are in the situation of prop. 10.1 in \cite{JR:mod}. By considering
homogeneous privileged (with respect to $y=\gamma(x)$) systems of Fatou coordinates
\[ {\{ \dot{\psi}_{j',\Lambda',\lambda'}^{\varphi_{1}} \}}_{(j',\Lambda',\lambda') \in
{\mathcal D}(\varphi_{1}) \times {\mathcal M} \times {\mathbb S}^{1}} \ \ {\rm and} \ \
{\{ \dot{\psi}_{j',\Lambda',\lambda'}^{\varphi_{2}} \}}_{(j',\Lambda',\lambda') \in
{\mathcal D}(\varphi_{2}) \times {\mathcal M} \times {\mathbb S}^{1}} \]
of $\varphi_{1}$ and $\varphi_{2}$ respectively we obtain
\begin{equation}
\label{equ:conjmm}
\dot{\psi}_{j,\Lambda,\lambda}^{\varphi_{2}} \circ \kappa -
\dot{\psi}_{j,\Lambda,\lambda}^{\varphi_{1}} = \tilde{c}(x_{0}) \ {\rm in} \ H_{s \tau, j}(x_{0})
\ \ \forall j \in {\mathcal D}(\varphi_{1})
\end{equation}
where $\tilde{c} (x_{0}) \equiv (\psi^{X} \circ \kappa - \psi^{X})(x_{0},\gamma(x_{0}))$
(proposition 10.1 of \cite{JR:mod}). We obtain the inequality
$| \tilde{c}(x_{0})| \leq  C_{1}'/s^{\nu}(x_{0})$.
Let $c_{\Lambda,\lambda}^{\varphi_{l}}: [0,\delta) I_{\Lambda}^{\lambda} \to {\mathbb C}$
be the function defined by $c_{\Lambda,\lambda}^{\varphi_{l}} =
{\psi}_{j,\Lambda,\lambda}^{\varphi_{l}} - \dot{\psi}_{j,\Lambda,\lambda}^{\varphi_{l}}$
for $l \in \{1,2\}$ (see remark \ref{rem:chacha}).
%
%
Equation (\ref{equ:conjmm}) leads us to
\[ {\psi}_{j,\Lambda,\lambda}^{\varphi_{2}} \circ \kappa -
{\psi}_{j,\Lambda,\lambda}^{\varphi_{1}} = c(x_{0}) \ {\rm in} \ H_{s \tau, j}(x_{0})
\ \ \forall j \in {\mathcal D}(\varphi_{1}) \]
by defining
$c(x_{0})=\tilde{c} (x_{0}) +
c_{\Lambda,\lambda}^{\varphi_{2}}(x_{0}) - c_{\Lambda,\lambda}^{\varphi_{1}}(x_{0})$.
We obtain (see def. \ref{def:msfatcoor})
\[ \xi_{j,\Lambda,\lambda}^{\varphi_{2}} (x_{0},z) =
(z + c(x_{0})) \circ \xi_{j,\Lambda,\lambda}^{\varphi_{1}} (x_{0},z)
\circ (z-c(x_{0})) \ \ \forall j \in {\mathcal D}(\varphi_{1}) . \]
Thus there exists $C' \in {\mathbb R}^{+}$ independent of
$x_{0}$ such that
$|c(x_{0})| \leq C'/s^{\nu}(x_{0})$.
%
\end{proof}
\begin{pro}
\label{pro:iinvic} \cite{JR:mod}
  Let $\varphi_{1}, \varphi_{2} \in \diff{tp1}{2}$ with common convergent
normal form ${\rm exp}(X)$.
Fix $\Lambda=(\lambda_{1}, \hdots, \lambda_{\tilde{q}}) \in {\mathcal M}$
and $\lambda \in {\mathbb S}^{1}$. Fix a constant $M>0$. Suppose that
\[ \xi_{j,\Lambda,\lambda}^{\varphi_{2}} (x_{0},z) =
(z + c(x_{0})) \circ \xi_{j,\Lambda,\lambda}^{\varphi_{1}} (x_{0},z)
\circ (z-c(x_{0}))
\ \ \forall j \in {\mathcal D}(\varphi_{1}) \]
for some  $x_{0} \in [0,\delta) I_{\Lambda}^{\lambda}$ and $c(x_{0}) \in B(0,M)$.
Then there exists a s-mapping $\kappa$ such that
$\kappa \circ (\varphi_{1})_{|x=x_{0}} =(\varphi_{2})_{|x=x_{0}} \circ \kappa$
and $\kappa_{|(Sing X)(x_{0})} \cong Id$.
The constant $s \in {\mathbb R}^{+}$ does not depend on $x_{0}$.
\end{pro}
\subsection{Theorem of analytic conjugation}
In this subsection we prove an analogue (theorem \ref{teo:mulcon}) of theorem
10.1 in \cite{JR:mod} for the multi-summable setting. Roughly speaking
two elements in $\diff{tp1}{2}$ with common normal form are analytically
conjugated if and only if their homogeneous multi-summable systems of
extensions of Ecalle-Voronin invariants
coincide up to the action of a multi-summable family of
transformations of the form $(x,z+c(x))$.
\begin{defi}
\label{def:norcon}
Let $\varphi_{1}, \varphi_{2} \in \diff{p1}{2}$.
We denote $\varphi_{1} \sim \varphi_{2}$ if there exists
$\sigma \in \diff{}{2}$ conjugating $\varphi_{1}$ and $\varphi_{2}$
such that $\sigma_{|Fix (\varphi_{1})} \equiv Id$.
\end{defi}
\begin{teo}
\cite{JR:mod}
  Let $\varphi_{1}, \varphi_{2} \in \diff{tp1}{2}$ with common convergent
normal form ${\rm exp}(X)$. Suppose $N(X)=1$.
Then  $\varphi_{1} \sim \varphi_{2}$ if and only if there exist
$c \in {\mathbb C}\{x\}$ and $k \in {\mathbb Z}/(\nu({\mathcal E}_{0}) {\mathbb Z})$
such that
\[ \xi_{j}^{\varphi_{2}} (x,z) \equiv
\xi_{j+2k}^{\varphi_{1}} (x,z-c(x))   + c(x)  \]
for any $j \in {\mathcal D}(\varphi_{1})$.
\end{teo}
\begin{teo}
\label{teo:mulcon}
  Let $\varphi_{1}, \varphi_{2} \in \diff{tp1}{2}$ with common convergent
normal form ${\rm exp}(X)$. Suppose $N(X)>1$. Assume $\varphi_{1} \sim \varphi_{2}$. Then there exist
a $(\tilde{e}_{1}, \hdots, \tilde{e}_{\tilde{q}})$-summable function
${\{c_{\Lambda,\lambda}(x)\}}_{(\Lambda,\lambda) \in {\mathcal M} \times {\mathbb S}^{1}}$,
where $c_{\Lambda,\lambda}$ is defined in $(0,\delta) \dot{I}_{\Lambda}^{\lambda}$,
such that
\[ \xi_{j,\Lambda,\lambda}^{\varphi_{2}} (x,z) \equiv
\xi_{j,\Lambda,\lambda}^{\varphi_{1}} (x,z-c_{\Lambda,\lambda}(x))   + c_{\Lambda,\lambda}(x)  \]
for any $j \in {\mathcal D}(\varphi_{1})$.
\end{teo}
The reciprocal of the theorem is also true. In fact it is a direct consequence of
proposition \ref{pro:iinvic}  and the theorem \ref{teo:modi}.
\begin{proof}
Let $\sigma(x,y)=(x,\sigma_{2}(x,y) \in \diff{p}{2}$ be the mapping conjugating $\varphi_{1}$ and
$\varphi_{2}$ such that $\sigma_{|Fix (\varphi_{1})} \equiv Id$.
The mapping $\sigma$ is a diffeomorphism defined in a neighborhood $B(0,\delta) \times B(0,s)$
of the origin for some $s \in {\mathbb R}^{+}$.
By considering a smaller $s \in {\mathbb R}^{+}$ if necessary we obtain that
the function $\psi^{X} \circ \sigma - \psi^{X}$ is bounded in $B(0,\delta) \times B(0,s)$
(lemma \ref{lem:bddfldi}).
Fix $j' \in {\mathcal D}(\varphi_{1})$, we define $L_{j'}^{s}$ the unique element of
${\mathcal P}_{X}^{s}$ contained in $L_{j'}$.
The set $\sigma (L_{j}^{s})$ is contained in $L_{j}$ for some $s \in {\mathbb R}^{+}$ and any
$j \in {\mathcal D}(\varphi_{1})$ by proposition \ref{pro:widcon}. Consider a point $(0,y_{0}) \in L_{j'}^{s}$.

Proposition \ref{pro:conimeq} implies that
\[ c_{\Lambda,\lambda} = \psi_{j,\Lambda,\lambda}^{\varphi_{2}} \circ \sigma -  \psi_{j,\Lambda,\lambda}^{\varphi_{1}} \]
defines a function $c_{\Lambda,\lambda}: [0,\delta) I_{\Lambda}^{\lambda} \to {\mathbb C}$ of $x$.
The definition of $c_{\Lambda,\lambda}$ does not depend on $j \in {\mathcal D}(\varphi_{1})$.
Moreover $c_{\Lambda,\lambda}$ is continuous in  $[0,\delta) I_{\Lambda}^{\lambda}$ and holomorphic in
$(0,\delta) \dot{I}_{\Lambda}^{\lambda}$. We obtain
\[ \xi_{j,\Lambda,\lambda}^{\varphi_{2}} (x,z) \equiv
\xi_{j,\Lambda,\lambda}^{\varphi_{1}} (x,z-c_{\Lambda,\lambda}(x))   + c_{\Lambda,\lambda}(x)  \ \
\forall j \in {\mathcal D}(\varphi_{1}). \]
It suffices to prove that for all $\Lambda,\Lambda' \in {\mathcal M}$ and $\lambda,\lambda' \in {\mathbb S}^{1}$
we have
\[ |c_{\Lambda,\lambda} - c_{\Lambda',\lambda'}|(x) =
O \left({ e^{-K/|x|^{\tilde{e}_{d_{\Lambda,\Lambda'}^{\lambda,\lambda'}+1}}} }\right)  \]
for some $K \in {\mathbb R}^{+}$. We have
\[ (c_{\Lambda,\lambda} - c_{\Lambda',\lambda'})(x)=
(\psi_{j',\Lambda,\lambda}^{\varphi_{2}}-\psi_{j',\Lambda',\lambda'}^{\varphi_{2}})(\sigma(x,y_{0}))
-  (\psi_{j',\Lambda,\lambda}^{\varphi_{1}} - \psi_{j',\Lambda',\lambda'}^{\varphi_{1}})(x,y_{0})  \]
in a neighborhood of $x=0$ in
$(0,\delta) (\dot{I}_{\Lambda}^{\lambda} \cap \dot{I}_{\Lambda'}^{\lambda'})$.
The points $(x,y_{0})$ and $\sigma(x,y_{0})$ belong to
$H_{\Lambda,\Lambda',j'}^{\lambda,\lambda'}$ for any
$x \in (0,\delta) (\dot{I}_{\Lambda}^{\lambda} \cap \dot{I}_{\Lambda'}^{\lambda'})$
in a neighborhood of $0$.
Theorem \ref{teo:difFatflatext2} implies that
$c_{\Lambda,\lambda} - c_{\Lambda',\lambda'}$ is exponentially small of the proper order.
\end{proof}
\subsection{Isolated zeros theorem for analytic conjugacy}
\label{subsec:isozeros}
This subsection is devoted to the proof of theorem \ref{teo:modi}.
Let us notice that theorem \ref{teo:modi} stands by itself; it
makes no reference to Fatou coordinates or other multi-summable objects
in the paper.
\begin{defi}
Let $\varphi \in \diff{p1}{2}$. Consider a convergent normal form ${\rm exp}(X)$ of
$\varphi$. Denote $X=g(x,y) \partial / \partial y$.
We define $\nu(\varphi) = \nu(g(0,y)) -1$ where $\nu(g(0,y))$ is the order of the series
$g(0,y)$ at $y=0$. The definition does not depend on the choice of $X$.
\end{defi}
\begin{teo}
\label{teo:modi}
Let $\varphi, \eta \in \diff{p1}{2}$
with $Fix (\varphi)= Fix (\eta)$. Then $\varphi \sim \eta$
if and only if there exists a $\nu(\varphi)$ slow decaying function
$s$ and a sequence $x_{n} \to 0$ contained in $B(0,\delta) \setminus \{0\}$ such that
for any $n \in {\mathbb N}$ the
restrictions $\varphi_{|x=x_{n}}$ and $\eta_{|x=x_{n}}$ are conjugated by an injective
holomorphic mapping $\kappa_{n}$ defined in $B(0,s(x_{n}))$ and fixing the points in
$Fix (\varphi) \cap \{ x=x_{n}\}$.
\end{teo}
\begin{proof}
We start proving some normalizing conditions.
We can suppose that $\varphi, \eta \in \diff{tp1}{2}$
up to replace $\varphi$ and $\eta$ with $(x^{1/k},y) \circ \varphi \circ (x^{k},y)$
and $(x^{1/k},y) \circ \eta \circ (x^{k},y)$ respectively for some $k \in {\mathbb N}$.
It suffices to prove the theorem in this context since given $k \in {\mathbb N}$ we have
\[ \varphi \sim \eta \Leftrightarrow
(x^{1/k},y) \circ \varphi \circ (x^{k},y) \sim (x^{1/k},y) \circ \eta \circ (x^{k},y) \]
by lemma 10.9 in \cite{JR:mod}. Let us stress that the $\nu(\varphi)$ slow decaying
character is invariant by ramification $x \to x^{k}$.

Let ${\rm exp}(X)$ and ${\rm exp}(Y)$ be convergent normal forms of $\varphi$ and $\eta$
respectively. Since $Fix (\varphi)= Fix (\eta)$ we obtain
\[ X(y) = u(x,y) \prod_{j=1}^{p} (y-\gamma_{j}(x))^{n_{j}} \ {\rm and} \
Y(y) = v(x,y) \prod_{j=1}^{p} (y-\gamma_{j}(x))^{n_{j}'} \]
where $u, v \in {\mathbb C}\{x,y\}$ are units. The mapping $\kappa_{n}$ conjugates
the germs of $\varphi_{|x=x_{n}}$ and $\eta_{|x=x_{n}}$ in the neighborhood of $y=\gamma_{j}(x_{n})$.
Since the analytic conjugation $\kappa_{n}$ preserves the order of tangency with the identity
we get $n_{j}=n_{j}'$ for any $1 \leq j \leq p$.
Denote $w_{n,j}=(x_{n},\gamma_{j}(x_{n}))$. An analytic conjugation preserves residues too
(prop. 5.8 \cite{UPD}), thus we obtain
\[ Res(X,w_{n,j})=  Res(\varphi,w_{n,j})  =   Res(\eta, w_{n,j})= Res(Y, w_{n,j}) \]
for all $1 \leq j \leq p$ and $n \in {\mathbb N}$ (see def. \ref{def:res2} and
\ref{def:resd}).
Given $1 \leq j \leq p$ the functions
\[ x \to Res(\varphi, (x,\gamma_{j}(x))) \ {\rm and} \ x \to Res(\eta, (x,\gamma_{j}(x))) \]
are meromorphic (prop. 5.2 \cite{UPD}). We obtain
$Res(\varphi, (x,\gamma_{j}(x))) \equiv Res(\eta, (x,\gamma_{j}(x)))$ for any $1 \leq j \leq p$.
The equality $(X(y))=(Y(y))$ of ideals of ${\mathbb C}\{x,y\}$ and
\[ Res(X, (x,\gamma_{j}(x))) \equiv Res(Y, (x,\gamma_{j}(x))) \ \forall 1 \leq j \leq p \]
imply that the diffeomorphisms ${\rm exp}(X)$ and ${\rm exp}(Y)$
are analytically conjugated by a mapping
$\sigma(x,y)=(x,\sigma_{2}(x,y))$ such that $\sigma_{|Sing X} \equiv Id$
(prop. 5.10 \cite{UPD}).
Up to replace $s$ with $s \tau$
for some $\tau \in {\mathbb R}^{+}$ we obtain that $\sigma^{-1} \circ \kappa_{n}$
is a $s(x_{n})$-mapping for $n \in {\mathbb N}$
(lemmas \ref{lem:quitR} and \ref{lem:adjcon}).
Up to replace $\eta$ with $\sigma^{-1} \circ \eta \circ \sigma$
and $\kappa_{n}$ with $\sigma^{-1} \circ \kappa_{n}$ for $n \in {\mathbb N}$
we can suppose that $\varphi$ and $\eta$ have common normal form ${\rm exp}(X)$.

Suppose $N(X)=1$. We obtain that
$\kappa_{n}'(\gamma_{1}(x_{n})) \in <e^{2 \pi i/\nu}>$ for any $n \in {\mathbb N}$
by arguing as in lemma \ref{lem:modbdd}.
Up to taking a subsequence of ${\{x_{n}\}}_{n \in {\mathbb N}}$ we can suppose that
there exists $\mu \in <e^{2 \pi i/\nu}>$ such that
$\kappa_{n}'(\gamma_{1}(x_{n}))=\mu$ for any $n \in {\mathbb N}$. There exists
$\sigma_{0} \in \diff{p}{2}$ conjugating the vector fields $X$ and
\[ X_{0}=\frac{(y-\gamma(x))^{\nu+1}}{1+(y-\gamma(x))^{\nu} Res(X,(x,\gamma(x)))} \frac{\partial}{\partial y}. \]
Since $(x,\mu (y-\gamma(x))+ \gamma(x))$ preserves $X_{0}$
then $\sigma=\sigma_{0}^{-1} \circ (x, \mu (y-\gamma(x))+ \gamma(x)) \circ \sigma_{0}$ conjugates $X$ with
itself. Moreover $\sigma$ satisfies
$(\partial (y \circ \sigma)/\partial y)(x,\gamma(x)) \equiv \mu$.
As in the previous paragraph up to replace $\eta$ with $\sigma^{-1} \circ \eta \circ \sigma$
and $\kappa_{n}$ with $\sigma^{-1} \circ \kappa_{n}$ for $n \in {\mathbb N}$
we can suppose $\kappa_{n}'(\gamma(x_{n}))=1$ for any $n \in {\mathbb N}$.
In this way we obtain that $\kappa_{n} \cong Id$ for any $n \in {\mathbb N}$
independently on whether $N(X)=1$ or $N(X)>1$.

The goal is proving that there exists $s_{0} \in {\mathbb R}^{+}$ such that
$\varphi_{x=x_{0}}$ and $\eta_{|x=x_{0}}$ are conjugated by a $s_{0}$-mapping
$\phi_{x_{0}}$ with $\phi_{x_{0}} \cong Id$ for any $x_{0}$ in a pointed neighborhood of $0$.
Then the Main Theorem of \cite{JR:mod} implies $\varphi \sim \eta$.

By taking a subsequence of ${\{x_{n}\}}_{n \in {\mathbb N}}$ we can suppose that
$x_{n}$ adheres to a direction $\lambda \in {\mathbb S}^{1}$.
Fix $\Lambda = (\lambda_{1}, \hdots, \lambda_{\tilde{q}}) \in {\mathcal M}$ with
$\lambda_{q}$ close to $\lambda$.
Consider the set
\[ E_{\upsilon}^{\lambda}(\varphi)=\{ (j,l) \in {\mathcal D}_{\upsilon}(\varphi) \times {\mathbb N} :
a_{j,\Lambda,\lambda,l}^{\varphi} \not \equiv 0\} \ {\rm  for} \  \upsilon \in \{-1,1\} \]
and
$E^{\lambda}(\varphi)=E_{-1}^{\lambda}(\varphi) \cup E_{1}^{\lambda}(\varphi)$.
Proposition \ref{pro:conimeq} implies that there exists a sequence
${\{c_{n}\}}_{n \in {\mathbb N}}$ of complex numbers such that
\[ \xi_{j,\Lambda,\lambda}^{\eta} (x_{n},z) =
(z + c_{n}) \circ \xi_{j,\Lambda,\lambda}^{\varphi} (x_{n},z) \circ (z-c_{n})
\ \ \forall j \in {\mathcal D}(\varphi). \]
Moreover we have $|c_{n}| \leq C'/s^{\nu}(x_{n})$ for some $C' \in {\mathbb R}^{+}$
and any $n \in {\mathbb N}$. We have
\begin{equation}
\label{equ:conevi}
a_{j,\Lambda,\lambda,l}^{\eta}(x_{n})= a_{j,\Lambda,\lambda,l}^{\varphi}(x_{n}) e^{2 \pi i \upsilon l c_{n}}
\end{equation}
for all $n \in {\mathbb N}$, $\upsilon \in \{-1,1\}$ and $j \in {\mathcal D}_{\upsilon}(\varphi)$.
The multi-summable character implies for both $a_{j,\Lambda,\lambda,l}^{\varphi}$
and $a_{j,\Lambda,\lambda,l}^{\eta}$ that either they are identically $0$ or never vanishing in
a neighborhood of $0$ in $(0,\delta)I_{\Lambda}^{\lambda}$.
We obtain $E^{\lambda}(\varphi)=E^{\lambda}(\eta)$.

Consider $(j,l) \in E_{\upsilon}^{\lambda}(\varphi)$. We obtain
\[   e^{-\frac{2 \pi l C'}{s^{\nu}(x_{n})}} \leq
\left|{ \frac{a_{j,\Lambda,\lambda,l}^{\eta}}{a_{j,\Lambda,\lambda,l}^{\varphi}} }\right| (x_{n})
\leq e^{\frac{2 \pi l C'}{s^{\nu}(x_{n})}} \ \forall n \in {\mathbb N} . \]
Given $\iota \in {\mathbb R}^{+}$ we have
$2 \pi l C'/s^{\nu}(x_{n}) < \iota |\ln|x_{n}||$ for any $n \in {\mathbb N}$ big enough since
$s$ is a $\nu$ slow decaying function. This implies
\begin{equation}
\label{equ:stdev}
|x_{n}|^{\iota} \leq
\left|{ \frac{a_{j,\Lambda,\lambda,l}^{\eta}}{a_{j,\Lambda,\lambda,l}^{\varphi}} }\right| (x_{n})
\leq \frac{1}{|x_{n}|^{\iota}} \ \forall n >>1 .
\end{equation}
The multi-summability of $a_{j,\Lambda,\lambda,l}^{\varphi}$ and $a_{j,\Lambda,\lambda,l}^{\eta}$
implies that $a_{j,\Lambda,\lambda,l}^{\eta}/a_{j,\Lambda,\lambda,l}^{\varphi}$ is a function of
the form $x^{e} h_{j,\Lambda,\lambda,\lambda}$ where $e$ belongs to ${\mathbb Z}$ and
$h_{j,l,\Lambda,\lambda}$ is a  $(\tilde{e}_{1}, \hdots, \tilde{e}_{\tilde{q}})$-summable
function such that $h_{j,l,\Lambda,\lambda}(0) \neq 0$.
Equation (\ref{equ:stdev}) implies $e=0$.
We consider a continuous function $c_{\lambda}^{j,l}$
defined in a neighborhood of $0$ in $[0,\delta) \dot{I}_{\Lambda}^{\lambda}$ such that
$e^{2 \pi i \upsilon l c_{\lambda}^{j,l}} \equiv a_{j,\Lambda,\lambda,l}^{\eta}/a_{j,\Lambda,\lambda,l}^{\varphi}$.

Consider $(j,l) \in E_{\upsilon}^{\lambda}(\varphi)$ and $(j',l') \in E_{\upsilon'}^{\lambda}(\varphi)$.
The equation (\ref{equ:conevi}) implies
\[ {\left({ \frac{a_{j,\Lambda,\lambda,l}^{\eta}}{a_{j,\Lambda,\lambda,l}^{\varphi}} }\right)}^{\upsilon l'} (x_{n}) =
 {\left({ \frac{a_{j',\Lambda,\lambda,l'}^{\eta}}{a_{j',\Lambda,\lambda,l'}^{\varphi}} }\right)}^{\upsilon' l} (x_{n})
 \ \ \forall n \in {\mathbb N}. \]
Therefore we obtain
\begin{equation}
\label{equ:profms}
{\left({ \frac{a_{j,\Lambda,\lambda,l}^{\eta}}{a_{j,\Lambda,\lambda,l}^{\varphi}} }\right)}^{\upsilon l'} (x) =
 {\left({ \frac{a_{j',\Lambda,\lambda,l'}^{\eta}}{a_{j',\Lambda,\lambda,l'}^{\varphi}} }\right)}^{\upsilon' l} (x)
 \ \ \forall x \in (0,\delta) \dot{I}_{\Lambda}^{\lambda}
\end{equation}
by the multi-summable character of the functions involved.

Suppose $E^{\lambda}(\varphi) \neq \emptyset$. Let us consider
$(j_{1},l_{1})$, $\hdots$, $(j_{b},l_{b})$ such that $(j,l) \in E^{\lambda}(\varphi)$
implies that $l$ belongs to the ideal $(l_{1},\hdots,l_{b})$ of ${\mathbb Z}$.
Let $x_{n_{0}}$ be a point such that
$a_{j_{k},\Lambda,\lambda,l_{k}}^{\varphi}(x_{n_{0}}) \neq 0$ for any $1 \leq k \leq b$.
We can choose the function $c_{\lambda}^{j_{k},l_{k}}$ such that
$c_{\lambda}^{j_{k},l_{k}}(x_{n_{0}})=c_{n_{0}}$ for any $1 \leq k \leq b$.
Equation (\ref{equ:profms}) and
$c_{\lambda}^{j_{k},l_{k}}(x_{n_{0}}) = c_{\lambda}^{j_{d},l_{d}}(x_{n_{0}})$
imply
\[ e^{2 \pi i l_{k} l_{d} c_{\lambda}^{j_{k},l_{k}}} \equiv e^{2 \pi i l_{k} l_{d} c_{\lambda}^{j_{d},l_{d}}}
\implies c_{\lambda}^{j_{k},l_{k}} - c_{\lambda}^{j_{d},l_{d}} \in {\mathbb Z}/(l_{k}l_{d}) \implies
c_{\lambda}^{j_{k},l_{k}} \equiv c_{\lambda}^{j_{d},l_{d}} \]
for all $1 \leq k,d \leq b$. We denote $c_{\lambda}=c_{\lambda}^{j_{k},l_{k}}$ for $1 \leq k \leq b$.
We define $c_{\lambda} \equiv 0$ for the case $E^{\lambda}(\varphi) = \emptyset$.

We have $(j_{k},l_{k}) \in E_{\upsilon_{k}}^{\lambda}(\varphi)$ for $1 \leq k \leq b$.
Consider $(j,l) \in E_{\upsilon}^{\lambda}(\varphi)$. There exist
$m_{1}, \hdots, m_{b} \in {\mathbb Z}$ such that
$l=m_{1} l_{1} + \hdots + m_{b} l_{b}$. We have
\[ e^{2 \pi i \upsilon_{k} l_{k} c_{n}} =
\frac{a_{j_{k},\Lambda,\lambda,l_{k}}^{\eta}}{a_{j_{k},\Lambda,\lambda,l_{k}}^{\varphi}} (x_{n})
\ \forall k \in \{1,\hdots,b\} \ \forall n >>1. \]
We obtain
\[ \frac{a_{j,\Lambda,\lambda,l}^{\eta}}{a_{j,\Lambda,\lambda,l}^{\varphi}} (x_{n}) =
 e^{2 \pi i \upsilon l c_{n}} = \prod_{k=1}^{b}
{\left({  \frac{a_{j_{k},\Lambda,\lambda,l_{k}}^{\eta}}{a_{j_{k},\Lambda,\lambda,l_{k}}^{\varphi}}
}\right)}^{\upsilon_{k} \upsilon m_{k}} (x_{n}) \ \forall n>>1. \]
The multi-summability character of the functions involved in the previous equality implies
\[ \frac{a_{j,\Lambda,\lambda,l}^{\eta}}{a_{j,\Lambda,\lambda,l}^{\varphi}} \equiv \prod_{k=1}^{b}
{\left({  \frac{a_{j_{k},\Lambda,\lambda,l_{k}}^{\eta}}{a_{j_{k},\Lambda,\lambda,l_{k}}^{\varphi}}
}\right)}^{\upsilon_{k} \upsilon m_{k}}  . \]
We also have
\[  e^{2 \pi i \upsilon_{k} l_{k} c_{\lambda}} \equiv
\frac{a_{j_{k},\Lambda,\lambda,l_{k}}^{\eta}}{a_{j_{k},\Lambda,\lambda,l_{k}}^{\varphi}}
\ \forall k \in \{1,\hdots,b\}   \]
by construction. Therefore we obtain
\[ e^{2 \pi i \upsilon  l  c_{\lambda}} \equiv \prod_{k=1}^{b}
{\left({  \frac{a_{j_{k},\Lambda,\lambda,l_{k}}^{\eta}}{a_{j_{k},\Lambda,\lambda,l_{k}}^{\varphi}}
}\right)}^{\upsilon_{k} \upsilon m_{k}} \equiv
\frac{a_{j,\Lambda,\lambda,l}^{\eta}}{a_{j,\Lambda,\lambda,l}^{\varphi}} \]
for any $(j,l) \in E^{\lambda}(\varphi)$. We deduce
\[ \xi_{j,\Lambda,\lambda}^{\eta} (x,z) \equiv
\xi_{j,\Lambda,\lambda}^{\varphi} (x,z-c_{\lambda}(x))  + c_{\lambda}(x)
\ \ \forall j \in {\mathcal D}(\varphi). \]
By prop. \ref{pro:iinvic} there exists $s_{0} \in {\mathbb R}^{+}$ such that
for any $x_{0} \in [0,\delta) \dot{I}_{\Lambda}^{\lambda}$ there exists
a $s_{0}$-mapping $\kappa_{x_{0}}$ with $\kappa_{x_{0}} \cong Id$
conjugating $\varphi_{|x=x_{0}}$ and $\eta_{|x=x_{0}}$. Indeed it is possible
to choose ${\{\kappa_{x}\}}_{x \in [0,\delta) \dot{I}_{\Lambda}^{\lambda}}$ such that there exists
a continuous map $\kappa$ defined in $[0,\delta) \dot{I}_{\Lambda}^{\lambda} \times B(0,s_{0})$ and holomorphic
in $(0,\delta) \dot{I}_{\Lambda}^{\lambda} \times B(0,s_{0})$ satisfying
$\kappa \circ \varphi = \eta \circ \kappa$ and $\kappa_{|x=x_{0}}=\kappa_{x_{0}}$ for any
$x_{0} \in [0,\delta) \dot{I}_{\Lambda}^{\lambda}$.

By considering $\lambda' \in {\mathbb S}^{1}$ such that
$\dot{I}_{\Lambda}^{\lambda} \cap I_{\Lambda}^{\lambda'} \neq \emptyset$ we can repeat the
previous argument to enlarge the family ${\{\kappa_{x}\}}_{x \in [0,\delta) \dot{I}_{\Lambda}^{\lambda}}$
to obtain a family
${\{\kappa_{x}\}}_{x \in [0,\delta) (\dot{I}_{\Lambda}^{\lambda} \cup \dot{I}_{\Lambda}^{\lambda'}) }$
satisfying analogous properties for some smaller $s_{0} \in {\mathbb R}^{+}$.
By varying $\lambda' \in {\mathbb S}^{1}$ we obtain
a family ${\{\kappa_{x}\}}_{x \in B(0,\delta)}$ of $s_{0}$-mappings
such that $\kappa_{x_{0}}$ conjugates $\varphi_{|x=x_{0}}$ and $\eta_{|x=x_{0}}$
and satisfies $\kappa_{x_{0}} \cong Id$ for any $x_{0} \in B(0,\delta)$. The Main Theorem
in \cite{JR:mod} implies $\varphi \sim \eta$.
\end{proof}
\begin{rem}
\label{rem:res}
Theorem \ref{teo:modi} can be generalized for codimension finite resonant
diffeomorphisms. If the linear part of $\varphi_{|x=0}$ is $p$ periodic
we replace in the theorem the sets $Fix (\varphi)$ and $Fix(\eta)$ with
$Fix (\varphi^{p})$ and $Fix(\eta^{p})$ respectively.  Indeed we have
$(\partial (y \circ \varphi)/\partial y)(0,0) =  (\partial (y \circ \eta)/\partial y)(0,0) $
by continuity and $\varphi^{p} \sim \eta^{p}$ by theorem \ref{teo:modi}.
Proposition 5.4 of \cite{JR:mod} implies $\varphi \sim \eta$.
\end{rem}
\begin{cor}
\label{cor:modi}
Let $\varphi, \eta \in \diff{p1}{2}$
with $Fix (\varphi)= Fix (\eta)$. Suppose there exist $s \in {\mathbb R}^{+}$ and
a sequence $x_{n} \to 0$ contained in $B(0,\delta) \setminus \{0\}$ such that
for any $n \in {\mathbb N}$ the
restrictions $\varphi_{|x=x_{n}}$ and $\eta_{|x=x_{n}}$ are conjugated by an injective
holomorphic mapping $\kappa_{n}$ defined in $B(0,s)$ and fixing the points in
$Fix (\varphi) \cap \{ x=x_{n}\}$. Then we obtain $\varphi \sim \eta$.
\end{cor}
\begin{rem}
The theorem \ref{teo:modi} allows us to reduce the problem of analytic classification
of elements of $\diff{p1}{2}$ to well-behaved directions in the parameter space.
For instance consider a diffeomorphism of the form
\[ \varphi(x,y) = (x, y + u(x,y) (y-x)(y+x)) \]
for some unit $u \in {\mathbb C}\{x,y\}$ such that $u(0,0)=1$.
In any sector of the form
\[ S=(0,\delta) e^{i(\pi/2+\upsilon, 3\pi/2 -\upsilon)} \ {\rm or} \
S=(0,\delta) e^{i(- \pi/2 +\upsilon, \pi/2-\upsilon)} \]
for $\upsilon \in {\mathbb R}^{+}$ the diffeomorphism $\varphi_{|x=x_{0}}$
has an attracting and a repelling fixed point for any $x_{0} \in S$.
Indeed an analytic system of invariants in $S$ can be constructed by comparing the
linearizing mappings in both fixed points. This is Glutsyuk's point of view \cite{Gluglu}.
Roughly speaking, conjugation in
$(B(0,\delta) \setminus i {\mathbb R}^{*}) \times B(0,\epsilon)$ implies conjugation
in a neighborhood of the origin.
\end{rem}
\begin{rem}
Consider $X \in  {\mathcal X} \cn{2}$ such that ${\rm exp}(X) \in \diff{tp1}{2}$.
Prop. 11.1 in \cite{JR:mod} shows that there exist
$\varphi, \eta \in \diff{p1}{2}$ with normal form ${\rm exp}(X)$ such that
\begin{itemize}
\item $\varphi \not \sim \eta$
\item $\varphi$ and $\eta$ are
conjugated by an analytic injective multi-valued, in the $x$ variable, mapping $\sigma$.
\item $\sigma$ is
defined in a domain $|y|< C_{0}/\sqrt[\nu(X)]{|\ln x|}$ for some
$C_{0} \in {\mathbb R}^{+}$.
\item $\sigma$ satisfies
$\sigma_{|Fix (\varphi)} \equiv Id$ and
$\sigma(e^{2 \pi i}x,y) = \eta \circ \sigma(x,y)$.
\end{itemize}
The Main Theorem is optimal and it does not hold true for non-slow decaying functions.
Equivalently a domain of the form $|y|< C_{0}/\sqrt[\nu(X)]{|\ln x|}$
is maximal as a domain of definition of a mapping $\sigma$ satisfying the previous properties.
If $\sigma$ is defined in a substantially bigger domain $|y| < s(x)$, i.e
$\lim_{x \to 0} s(x)/(1/\sqrt[\nu(X)]{|\ln |x||})=\infty$, then we obtain $\varphi \sim \eta$.
\end{rem}
\bibliography{rendu}

\begin{thebibliography}{10}

\bibitem{Balser:LN}
Werner Balser.
\newblock {\em From divergent power series to analytic functions}, volume 1582
  of {\em Lecture Notes in Mathematics}.
\newblock Springer-Verlag, Berlin, 1994.
\newblock Theory and application of multisummable power series.

\bibitem{rousseau-christopher:modpar}
C.~Christopher and C.~Rousseau.
\newblock The moduli space of germs of generic families of analytic
  diffeomorphisms unfolding a parabolic fixed point.
\newblock {\em preprint CRM, arXiv:0809.2167}, 2008.

\bibitem{Conway}
J.B. Conway.
\newblock {\em Functions of one complex variable II}.
\newblock New York : Springer-Verlag, 1995.

\bibitem{DES}
A.~Douady, F.~Estrada, and P.~Sentenac.
\newblock Champs de vecteurs polyn\^{o}miaux sur $\plano$.
\newblock {\em To appear in the Proceedings of Boldifest}.

\bibitem{Ecalle}
J.~{\'E}calle.
\newblock Th\'eorie it\'erative: introduction \`a la th\'eorie des invariants
  holomorphes.
\newblock {\em J. Math. Pures Appl. (9)}, 54:183--258, 1975.

\bibitem{Gluglu}
A.~A. Glutsyuk.
\newblock Confluence of singular points and the nonlinear {Stokes} phenomenon.
\newblock {\em Trans. Moscow Math. Soc.}, pages 49--95, 2001.

\bibitem{Lavaurs}
P.~Lavaurs.
\newblock {\em Syst\ei mes dynamiques holomorphes: explosion de points p\ex
  riodiques paraboliques}.
\newblock PhD thesis, Université de Paris-Sud, 1989.

\bibitem{MalRam:Fou}
B.~Malgrange and J.-P. Ramis.
\newblock Fonctions multisommables.
\newblock {\em Ann. Inst. Fourier (Grenoble)}, 42(1-2):353--368, 1992.

\bibitem{MRR}
P.~Mardesic, R.~Roussarie, and C.~Rousseau.
\newblock Modulus of analytic classification of unfoldings of generic parabolic
  diffeomorphisms.
\newblock {\em Mosc. Math. J.}, 4(2):455--502, 2004.

\bibitem{MaRa:aen}
J.~Martinet and J.-P. Ramis.
\newblock Classification analytique des \ex quations differentielles non lin\ex
  aires r\ex sonnantes du premier ordre.
\newblock {\em Ann. Sci. Ecole Norm. Sup.}, 4(16):571--621, 1983.

\bibitem{Mar:Ast}
Jean Martinet.
\newblock Remarques sur la bifurcation noeud-col dans le domaine complexe.
  {S}ingularit\ex s d'\ex quations diff\ex rentielles ({D}ijon 1985).
\newblock {\em Asterisque}, (150-151):131--149, 1987.

\bibitem{Oudkerk}
R.~Oudkerk.
\newblock {\em The parabolic implosion for $f_{0}(z)=z+z^{\nu+1} +
  O(z^{\nu+2})$}.
\newblock PhD thesis, University of Warwick, 1999.

\bibitem{PM-Yo}
R.~P{\'{e}}rez-Marco and J.-C. Yoccoz.
\newblock Germes de feuilletages holomorphes à holonomie prescrite.
\newblock {\em Asterisque}, 7(222):345--371, 1994.

\bibitem{RamSib:Fou}
J.-P. Ramis and Y.~Sibuya.
\newblock A new proof of multisummability of formal solutions of nonlinear
  meromorphic differential equations.
\newblock {\em Ann. Inst. Fourier (Grenoble)}, 44(3):811--848, 1994.

\bibitem{UPD}
Javier Rib{\'o}n.
\newblock Formal classification of unfoldings of parabolic diffeomorphisms.
\newblock {\em Ergodic Theory Dynam. Systems}, 28(4):1323--1365, 2008.

\bibitem{JR:mod}
Javier Rib{\'o}n.
\newblock Modulus of analytic classification for unfoldings of resonant
  diffeomorphisms.
\newblock {\em Mosc. Math. J.}, 8(2):319--395, 400, 2008.

\bibitem{Rousseau:modres}
Christiane Rousseau.
\newblock The moduli space of germs of generic families of analytic
  diffeomorphisms unfolding of a codimension one resonant diffeomorphism or
  resonant saddle.
\newblock {\em J. Differential Equations}, 248(7):1794--1825, 2010.

\bibitem{Shishi}
M.~Shishikura.
\newblock Bifurcation of parabolic fixed points. {The Mandelbrot} set, theme
  and variations.
\newblock {\em London Math. Soc. Leture Note Ser.}, 274:325--363, 2000.

\end{thebibliography}
\end{document}